\documentclass[11pt,reqno,A4paper]{amsart}
 \usepackage[numbers,sort&compress]{natbib}
\usepackage{amsmath}
\usepackage{amssymb}
\usepackage{amsthm}
\usepackage{enumerate}
\usepackage{mathrsfs} %%\mathcal classic,  \mathscr decorative
\usepackage{eqlist}
\usepackage{array}

\usepackage[english]{babel}                             %kieliasetukset sun muut
\usepackage[latin1]{inputenc}
\usepackage{comment}
\usepackage{ae}
\usepackage[leqno]{amsmath}  %kaavat vasemmalle, int.alue alas

\usepackage[dotinlabels]{titletoc}     %piste sis?llysluetteloon

\usepackage{amssymb,latexsym,verbatim} %vaihtoehtoista s?l??
\usepackage{amstext}
\usepackage{amsfonts}
\usepackage{amsbsy}
\usepackage{amsopn}
\usepackage{enumerate}
\usepackage{txfonts}
\usepackage{amsmath, amsthm, amssymb}

\usepackage{color}
\usepackage[usenames,dvipsnames,svgnames,table]{xcolor}
\usepackage[numbers,sort&compress]{natbib}
\usepackage{geometry}
\allowdisplaybreaks
\newtheorem{theorem}{Theorem}
\newtheorem{definition}[theorem]{Definition}
\newtheorem{lemma}[theorem]{Lemma}

\newtheorem{proposition}[theorem]{Proposition}
\newtheorem{remark}[theorem]{Remark}
\numberwithin{equation}{section}
\numberwithin{theorem}{section}

\renewcommand{\epsilon}{\varepsilon}
\renewcommand{\rho}{\varrho}
\renewcommand{\widetilde}{\tilde}

\DeclareMathOperator{\supp}{supp}
\DeclareMathOperator*{\esssup}{ess\,sup}
\DeclareMathOperator*{\essinf}{ess\,inf}
\DeclareMathOperator*{\essosc}{ess\,osc}
\DeclareMathOperator*{\loc}{loc}

\def\XXint#1#2#3{{\setbox0=\hbox{$#1{#2#3}{\int}$ }
\vcenter{\hbox{$#2#3$ }}\kern-.6\wd0}}

\usepackage{graphicx} % for \rotatebox macro

\def\YYint#1#2#3{{\setbox0=\hbox{$#1{#2#3}{\iint}$}
    \vcenter{\hbox{$#2#3$}}\kern-.51\wd0}}

\numberwithin{equation}{section}
\begin{document}

\title[Continuity estimates for parabolic double phase equations]%
{Continuity estimates for degenerate parabolic double-phase equations via nonlinear potentials
}
\author[Qifan Li]%
{Qifan Li*}

\newcommand{\acr}{\newline\indent}

\address{\llap{*\,}Department of Mathematics\acr
                   School of Mathematics and Statistics\acr
                   Wuhan University of Technology\acr
                   430070, 122 Luoshi Road,
                   Wuhan, Hubei\acr
                   P. R. China}
\email{qifan\_li@yahoo.com,
qifan\_li@whut.edu.cn}

 %Secondary is optional

\begin{abstract}
In this article, we study regularity properties for
degenerate parabolic double-phase equations.
We establish continuity estimates for
bounded  weak solutions
in terms of elliptic Riesz potentials on the right-hand side of the equation.
\end{abstract}
\maketitle
\section{Introduction}
In this work, we are concerned with the regularity problem of bounded weak solutions to inhomogeneous parabolic double-phase equations of the type
\begin{equation}\label{zeroequation}\partial_t u-\operatorname{div}[|Du|^{p-2}Du+a(x,t)|Du|^{q-2}Du]
=f\end{equation}
in a space-time domain $\Omega_T=\Omega\times(0,T)$, where $\Omega\subset\mathbb{R}^n$ is a bounded open
set and $T>0$.
Here, the coefficient function $a(x,t)$ is nonnegative and locally H\"older continuous in $\Omega_T$. Double-phase problem
is a classical topic in partial differential equations,
and it is motivated by physical problems.
The elliptic double-phase equations
play a crucial role in the study of nonlinear elasticity and
the homogenization of strongly anisotropic materials
(see, for example \cite{zhikov1, zhikov2}). Moreover, regularity theory such as Harnack's inequality,
H\"older continuity, higher integrability and Calder\'on-Zygmund type estimates
for elliptic double-phase problems have been
studied in \cite{BCM, CM1,CM2,CM3, FM}. In recent years, there has been tremendous interest in developing regularity theory for
parabolic double-phase equations. The existence of weak solutions to the parabolic double-phase equations has been proved by
Chlebicka, Gwiazda, and Zatorska-Goldstein \cite{CGZ}. The gradient higher integrability results have been established by
Kim, Kinnunen, Moring and S\"arki\"o \cite{KKM, KS}; see also \cite{Asen} for the degenerate/singular parabolic
multi-phase problems. Furthermore, the Calder\'on-Zygmund type estimate has been studied in \cite{K}.
Recently, Kim, Moring and S\"arki\"o \cite{KMS} and independently Buryachenko and Skrypnik \cite{BS} proved the H\"older continuity
of weak solutions to parabolic double-phase equations. In \cite{BS}, the authors also established the Harnack's inequality.

On the other hand,
it is natural to try to relate the potential theory to the double-phase problems. However, limited work has been done in this field.
The first regularity result for nonlinear elliptic equations with measure data and
general growth was established by Baroni \cite{B}. In \cite{BS1}, the authors established local boundedness and Harnack's inequality
for double-phase elliptic equations in terms of the Wolff potentials. Subsequently, Buryachenko \cite{Bu} proved the local boundedness
of weak solutions to parabolic double-phase equations in terms of the time-dependent parabolic potentials.
Motivated by this work, we are interested in the continuity estimate
for the degenerate parabolic double-phase equations of the type \eqref{zeroequation}
 in terms of elliptic Riesz
potentials of the function $f$. Our main result states that any weak solution to degenerate parabolic double-phase equations
is continuous provided that $f$ belongs to a Kato-type class.

Our proof is in the spirit of \cite{BDG1, Bu, BS1, BS, KMS, LS, Qifanli}, which uses a phase analysis,
estimates in the intrinsic geometries, De Giorgi method and Kilpel\"ainen-Mal\'y technique. More precisely, our approach could be incorporated in the framework of
De Giorgi's method developed in \cite{Di93, KMS, U}.
In the context of degenerate parabolic double-phase problems, we have found
the intrinsic cylinder in the form
\begin{equation*}Q_{R,\Theta_\omega(R)}(x_0,t_0)=B_R(x_0)\times\left(t_0-\Theta_\omega(R),t_0\right],\quad\text{where}\quad
\Theta_\omega(R)=\omega^2\left(\left(\frac{\omega}{R}\right)^p+a(x_0,t_0)\left(\frac{\omega}{R}\right)^q\right)^{-1}
\end{equation*}
to be expedient, see \cite{KMS}.  It is essential in the proof of our main result. The strategy of the proofs of De Giorgi-type lemmas
makes use of the Kilpel\"ainen-Mal\'y technique. This technique was first invented by Kilpel\"ainen and Mal\'y \cite{KM}
in the study of the Wiener criteria of the $p$-Laplace equations. This technique has been devoted to the study of parabolic equations such as parabolic $p$-Laplace equations \cite{LS, LSS, Sk}, porous medium type equations \cite{BDG2, BDG1, z1, z2}, and doubly degenerate parabolic equations \cite{SS, Qifanli}. The treatment of the parabolic double-phase problem is considerably more delicate.
We have to work with the non-standard version of the intrinsic geometry and it would be necessary to distinguish between the cases $p$-phase and $(p,q)$-phase in the proof.

An outline of this paper is as follows. In Section 2, we set up notations and state the main result. We also discuss the Kato-type classes ensuring continuity of the weak solutions. Section 3 is devoted to the proof of the
Caccioppoli inequalities. Moreover, we set up an alternative argument in this section.
Section 4 deals with the analysis of the first alternative. Subsequently, Section 5 is devoted to the study of the second alternative.
Finally, in Section 6, we finish the proof of the main result.
\section{Statement of the main result}
In this section, we fix the notations and present our main result. In the following, we assume that
$\Omega$ is an open bounded set in $\mathbb{R}^n$ with $n\geq2$. For $T>0$, we set $\Omega_T=\Omega\times(0,T)$.
 Let $\rho>0$, $\theta>0$ and $z_0=(x_0,t_0)\in\Omega_T$ be a fixed point.
 We define $Q_{\rho,\theta}(z_0)=B_\rho(x_0)\times(t_0-\theta,t_0)$, where $B_\rho(x)=\left\{y\in\mathbb{R}^n:|y-x|\leq \rho\right\}$.
If $x_0$ is the origin, then we omit in our notation the point $x_0$ and write $B_\rho$ for $B_\rho(x_0)$.
Moreover, if $z_0=(0,0)$, then we abbreviate $Q_{\rho,\theta}:=Q_{\rho,\theta}(0,0)$.
Next, let us denote by
$$\partial_PQ_{\rho,\theta}(z_0)=
[\partial B_\rho(x_0)\times (t_0-\theta,t_0)]\cup [B_\rho(x_0)\times\{t_0-\theta\}]$$
the parabolic boundary of
$Q_{\rho,\theta}(z_0)$. Here, $\partial B_\rho(x_0)=\{y\in\mathbb{R}^n:|y-x_0|=\rho\}$.
In this paper, we consider quasilinear parabolic equations of the form
\begin{equation}\label{parabolic}\partial_t u-\operatorname{div}A(x,t,u,Du)=f,\end{equation}
in a bounded domain $\Omega_T\subset\mathbb{R}^{n+1}$.
Here, we assume that $f=f(x)\in L_{\loc}^1(\Omega)$.
The vector field $A(x,t,u,\xi)$ is measurable in $\Omega_T\times\mathbb{R}\times\mathbb{R}^n$
and satisfies the growth and ellipticity conditions:
 \begin{equation}\label{A}
	\begin{cases}
	 \big\langle A(x,t,u,\xi),\xi\big \rangle\geq C_0\left(|\xi|^p+a(x,t)|\xi|^q\right),\\
	|A(x,t,u,\xi)|\leq C_1\left(|\xi|^{p-1}+a(x,t)|\xi|^{q-1}\right)+g,
	\end{cases}
\end{equation}
where $C_0$, $C_1$ are given positive constants, $g\geq0$ and $g=g(x)\in L_{\loc}^1(\Omega)$. Throughout the paper,
we assume that $a(x,t)\geq0$ and $a(x,t)\in C_{\loc}^{\alpha,\frac{\alpha}{2}}(\Omega_T)$. More precisely, there exists a constant $[a]_\alpha>0$ such that
\begin{equation*}|a(x,t)-a(y,s)|\leq [a]_\alpha\left(|x-y|^\alpha+|t-s|^\frac{\alpha}{2}\right)\end{equation*}
holds for every $(x,t)\in\Omega_T$ and $(y,s)\in\Omega_T$. We now give the definition of a weak solution to the parabolic double-phase
equation \eqref{parabolic}-\eqref{A}.
\begin{definition}\label{weak solution}
A measurable function $u:\Omega_T\to \mathbb{R}$ is
said to be a local weak solution to \eqref{parabolic}-\eqref{A} if
$u\in C_{\loc}(-T,0;L_{\loc}^2(\Omega))\cap L_{\loc}^q(-T,0;W_{\loc}^{1,q}(\Omega))$
and
for every open set
$U\Subset \Omega$ and every subinterval $(t_1,t_2)\subset(-T,0)$
the identity
\begin{equation}\begin{split}\label{weaksolution}
\int_U &u(\cdot,t)\varphi(\cdot,t)dx\bigg|_{t=t_1}^{t_2}+\iint_{U\times(t_1,t_2)}-u\partial_t
 \varphi+\big\langle A(x,t,u,Du), D\varphi\big\rangle
\,\mathrm {d}x\mathrm {d}t
\\&=\iint_{U\times(t_1,t_2)}  f \varphi\,\mathrm {d}x\mathrm {d}t
\end{split}\end{equation}
holds for
all any function
$\varphi\in W_{\loc}^{1,2}(-T,0;L^2(U))\cap L_{\loc}^q(-T,0;W_0^{1,q}(U)).$
 \end{definition}
 We remark that the technical assumption $|Du|\in L^q_{\text{loc}}(\Omega_T)$ will be needed for the proof of Caccioppoli inequalities
 in Section 3. However, it is more natural to define the function space for weak solutions to the parabolic double-phase
equation \eqref{parabolic}-\eqref{A} by
\begin{equation*}\left\{u\in C_{\loc}\left(-T,0;L_{\loc}^2(\Omega)\right)\cap L_{\loc}^1\left(-T,0;W_{\loc}^{1,1}(\Omega)\right):
\iint_{\Omega_T}|Du|^p+a(x,t)|Du|^q\,\mathrm {d}x\mathrm {d}t<+\infty \right\}.\end{equation*}
It is possible to derive Caccioppoli inequalities
 in Section 3 under this function space assumption
by a parabolic Lipschitz truncation technique, see \cite{KKS}. However, this topic exceeds the scope of this paper.
 In this work, we assume that any weak solution to \eqref{parabolic}-\eqref{A} is locally bounded. More precisely, we can write
$\|u\|_\infty=\esssup_{\Omega_T}|u|<+\infty.$
The statement that a constant  $\gamma$ depends only upon the data means that it can be determined a priori only in terms
of $\left\{n,p,q,C_0,C_1,[a]_\alpha,\|u\|_\infty\right\}$.
For $1<s<n$, $\rho>0$ and $x\in\Omega$, we introduce the following quantities
\begin{equation*}\begin{split}
G_s(x,\rho)=\int_0^\rho\left(\frac{1}{r^{n-s}}\int_{B_r(x)\cap\Omega}g(y)^{\frac{s}{s-1}}
\,\mathrm {d}y\right)^{\frac{1}{s}}\frac{1}{r}\,\mathrm {d}r
\end{split}\end{equation*}
and
\begin{equation*}\begin{split}
F_s(x,\rho)=\int_0^\rho\left(\frac{1}{r^{n-s}}\int_{B_r(x)\cap\Omega}f(y)
\,\mathrm {d}y\right)^{\frac{1}{s-1}}\frac{1}{r}\,\mathrm {d}r,
\end{split}\end{equation*}
where $B_r(x)\cap\Omega=\left\{y\in\Omega:|y-x|\leq r\right\}$.
Such quantities are called elliptic Riesz potentials (see \cite{z1,z2}).
Furthermore, we define
\begin{equation*}\begin{split}
G_s(\rho)=\sup_{x\in\Omega}G_s(x,\rho)\qquad\text{and}\qquad F_s(\rho)=\sup_{x\in\Omega}F_s(x,\rho).
\end{split}\end{equation*}
We follow the terminology used in \cite{Bir1, LS, Sk, Qifanli} and introduce the definitions of the nonlinear Kato classes $K_s$ and $\tilde K_s$.
 \begin{definition}\label{Katoclass}
For $1<s<n$, the nonlinear Kato classes $K_s$ and $\tilde K_s$ are defined by
\begin{equation*}K_s=\left\{f\in L_{\loc}^1(\Omega):\lim_{\rho\to0}F_s(\rho)=0\right\}
\end{equation*}
and
\begin{equation*}\widetilde K_s=\left\{g\in L_{\loc}^1(\Omega):\lim_{\rho\to0}G_s(\rho)=0\right\}.
\end{equation*}
 \end{definition}
 Here, the condition $s<n$ is necessary in the study of the Kato class. In fact, this condition is always related to the
 use of fractional maximal functions in the study of the regularity problems for the quasilinear equation with measure data (see, for instance,
\cite{AD, Bir1}).
 Before giving the precise statement of the main result, we discuss an important property of Kato classes.
\begin{remark}We claim that if $1<p<q<n$ then $K_p\subset K_q$ and $\widetilde K_p\subset \widetilde K_q$.
For a fixed point $x\in\Omega$ and $\rho>0$, we use H\"older's inequality with
exponents $t=\frac{q-1}{p-1}$ and $t^\prime=\frac{q-1}{q-p}$ to obtain
\begin{equation*}\begin{split}
F_q(x,\rho)&=\int_0^\rho\left(\frac{1}{r^{n-p}}\int_{B_r(x)}f(y)
\,\mathrm {d}y\right)^{\frac{1}{q-1}}r^{\frac{q-p}{q-1}}\frac{1}{r}\,\mathrm {d}r
\leq \left[F_p(x,\rho)\right]^{\frac{p-1}{q-1}}\rho^{\frac{q-p}{q-1}}\leq F_p(\rho)^{\frac{p-1}{q-1}}.
\end{split}\end{equation*}
This proves the inclusion $K_p\subset K_q$. Next, we proceed to show that $\widetilde K_p\subset \widetilde K_q$.
We use H\"older's inequality with exponents $\nu=\frac{p(q-1)}{q(p-1)}$ and $\nu^\prime=\frac{p(q-1)}{q-p}$ to get
\begin{equation*}\begin{split}
\left(r^{q-n}\int_{B_r(x)}g(y)^\frac{q}{q-1}\,\mathrm {d}y\right)^\frac{1}{q}
\leq r^\sigma\left(r^{p-n}\int_{B_r(x)}g(y)^\frac{p}{p-1}\,\mathrm {d}y\right)^\frac{p-1}{p(q-1)},
\end{split}\end{equation*}
where
\begin{equation*}\begin{split}\sigma=\frac{n}{q\nu^\prime}+\frac{(p-1)(n-p)}{p(q-1)}+\frac{q-n}{q}
=\frac{q-p}{q-1}>0.
\end{split}\end{equation*}
From this inequality, we apply H\"older's inequality with exponents $t=\frac{q-1}{p-1}$ and $t^\prime=\frac{q-1}{q-p}$ to obtain
\begin{equation*}\begin{split}
G_q(x,\rho)&\leq\int_0^\rho\left(\frac{1}{r^{n-p}}\int_{B_r(x)}g(y)^\frac{p}{p-1}
\,\mathrm {d}y\right)^{\frac{p-1}{p(q-1)}}r^{\frac{q-p}{q-1}}\frac{1}{r}\,\mathrm {d}r
\leq \left[G_p(x,\rho)\right]^{\frac{p-1}{q-1}}\rho^{\frac{q-p}{q-1}}\leq G_p(\rho)^{\frac{p-1}{q-1}}.
\end{split}\end{equation*}
This proves that the inclusion $\widetilde K_p\subset \widetilde K_q$ holds.
\end{remark}
This result indicates that if $p<q$, then the conditions $K_p$ and $\widetilde K_p$ are stronger than $K_q$ and $\widetilde K_q$, respectively.
This motivates us to impose the assumptions that $f\in K_p$ and $g\in \widetilde K_p$ in the study of
inhomogeneous parabolic double-phase problems.
We are now in a position to state our main theorem, which also presents a continuity estimate for the weak solutions.
\begin{theorem}\label{main1}
 Let $u$ be a locally bounded weak solution to the parabolic double-phase equation \eqref{parabolic}
 in the sense of Definition \ref{weak solution}, where the vector field $A$ fulfills the structure conditions \eqref{A}.
 Assume that $2<p<q\leq p+\alpha$, $2<p<n$, $f\in K_p$ and $g\in \widetilde K_p$.
 Then, $u$ is locally continuous in $\Omega_T$. More precisely, for any fixed $z_0\in \Omega_T$, $R_0>0$, if $Q_{R_0,R_0^2}(z_0)\Subset \Omega_T$,
 then there exists a constant $\delta_1=\delta_1(\text{data},R_0)<1$, such that the oscillation
estimate
  \begin{equation}\begin{split}\label{theorem1oscillation}
 &\essosc_{Q_{r,b_0r^q}(z_0)}u\leq \gamma\left(\frac{r}{R_0}\right)^{\alpha_1}+\gamma
\left(F_p(cr^{\alpha_2}
)+G_p(cr^{\alpha_2})+r^{\alpha_2}\right)
 \end{split}\end{equation}
 holds for any $r<\delta_1R_0$. Here, the constants $b_0$, $\alpha_1$, $\alpha_2$, $\gamma$ and $c$
 depend only upon the data.
 \end{theorem}
 It is possible that this result also holds for the case $p=2$ and $q\leq 2+\alpha$ but we will not develop this point here.
 Moreover, it is not our purpose to study time-dependent parabolic potentials. For a treatment of the time-dependent parabolic potentials,
 we refer the reader to \cite{BDG2,BDG1,Bu, LSS, SS, z1, z2}.
 \begin{remark}\label{remarka>0}
 Let $f\in K_q$ and $g\in \widetilde K_q$.
 Let $(x_0,t_0)\in\Omega_T$ be a point such that $a_0=a(x_0,t_0)>0$. Since $a(x,t)$ is continuous, there exists $r>0$ such that
 \begin{equation*}\frac{a_0}{2}\leq \inf_{Q_{r,r^2}(x_0,t_0)}a(x,t)\leq \sup_{Q_{r,r^2}(x_0,t_0)}a(x,t)\leq 2a_0.\end{equation*}
 By Young's inequality, the structure condition \eqref{A} for the vector field $A$ becomes
  \begin{equation}\label{Avariant}
	\begin{cases}
	 \big\langle A(x,t,u,\xi),\xi\big \rangle\geq \frac{1}{2}C_0a_0|\xi|^q,\\
	|A(x,t,u,\xi)|\leq C_1^2(1+2a_0)|\xi|^{q-1}+1+g,
	\end{cases}
\end{equation}
where $(x,t)\in Q_{r,r^2}(x_0,t_0)$. In this case the quasilinear parabolic equation
\begin{equation*}\partial_t u-\operatorname{div}A(x,t,u,Du)=f,\end{equation*}
in the cylinder $Q_{r,r^2}(x_0,t_0)$
is actually the parabolic $q$-Laplace equation. Noting that for any $x\in\Omega$ there holds
\begin{equation*}\begin{split}
\int_0^\rho\left(\frac{1}{r^{n-q}}\int_{B_r(x)}1^{\frac{q}{q-1}}
\,\mathrm {d}y\right)^{\frac{1}{q}}\frac{1}{r}\,\mathrm {d}r=c(n,q)\rho\to0,\quad\text{as}\quad\rho\to0.
\end{split}\end{equation*}
It follows that $1+g\in\widetilde K_q$, since $g\in \widetilde K_q$. We conclude from \cite[Theorem 1.2]{LS} that the weak solution
$u$ is locally continuous in $Q_{r,r^2}(x_0,t_0)$. However, in the context of parabolic double-phase problem the continuity of $u$ in the level set $\left\{a(x,t)>0\right\}$
is not uniformly with respect to the point $(x,t)\in \left\{a(x,t)>0\right\}$ as $a(x,t)\to0^+$.
 \end{remark}
 Contrary to Remark \ref{remarka>0}, the continuity estimate \eqref{theorem1oscillation} indicates that under the stronger assumptions
 $f\in K_p$ and $g\in \widetilde K_p$, the continuity of $u$ is uniformly with respect to all the points in $Q_{R_0,R_0^2}(z_0)$. Finally,
 we will derive a continuity estimate similar to \eqref{theorem1oscillation} for the conditions $a(x_0,t_0)>0$, $f\in K_q$ and $g\in \tilde K_q$
 at the end of Section 6. It should be pointed out that the optimal conditions for $f$ and $g$
 which ensure the continuity of $u$
 are still unclear. We leave this problem for future study.
\section{Caccioppoli inequalities}
In this section, we provide some preliminary lemmas and establish the Caccioppoli inequalities for the weak solutions. To start with,
we follow the notations used in \cite{LSS} and introduce some auxiliary functions. Throughout the paper, we define
\begin{equation}\label{G}
	G(u)=\begin{cases}
u,&\quad \text{for}\quad u>1,\\
	u^2,&\quad \text{for}\quad 0\leq u\leq1
\\
	0,&\quad \text{for}\quad u<0.
	\end{cases}
\end{equation}
Let $s\in\mathbb{R}$ and $s_+=\max\{s,0\}$. We observe that $G(s)=
\min\{s_+,s_+^2\}$. Fix $0<\lambda<p-1$, we introduce the
function
\begin{equation}\label{ph+}\phi_+(s)=\int_0^{s_+}(1+\tau)^{-1-\lambda}\,\mathrm {d}\tau=\frac{1}{\lambda}
\left(1-\left(1+s_+\right)^{-\lambda}\right).
\end{equation}
It can be easily seen that there exist $\gamma_1$ and $\gamma_2$ depending only upon $\lambda$ such that $\gamma_1<\gamma_2$,
\begin{equation}\label{phiprime1}\gamma_1\min\{s_+,1\}\leq\phi_+(s)
\leq \gamma_2\min\{s_+,1\}\end{equation}
and
\begin{equation}\label{phi}
\gamma_1\frac{s_+}{s_++1}\leq\phi_+(s)
\leq\gamma_2\frac{s_+}{s_++1}.
\end{equation}
Next, we set
$\Phi(s)=\int_0^{s_+}\phi_+(\tau)\,\mathrm {d}\tau$
and it follows from \eqref{phiprime1} that
\begin{equation}\label{PhiG}
\gamma_1G(s)=\gamma_1\min\{s_+,s_+^2\}\leq \Phi(s)\leq\gamma_2 \min\{s_+,s_+^2\}=\gamma_2G(s).
\end{equation}
Furthermore, for $\nu>1$ and $0<\lambda<\nu-1$, we introduce the function
\begin{equation*}
\widetilde\psi_\nu(s)=\int_0^{s_+}(1+\tau)^{-\frac{1}{\nu}-\frac{\lambda}{\nu}}\,\mathrm {d}\tau=\frac{\nu}
{\nu-1-\lambda}\left[(1+s_+)^{1-\frac{1+\lambda}{\nu}}-1\right].
\end{equation*}
Let $d>0$, $l>0$ and $u$ be the weak solution
as in the Definition \ref{weak solution}. We introduce the quantity
\begin{equation}\label{psi-}
\psi_\nu=\left(\frac{1}{d}\int_u^l\left(1+\frac{l-s}{d}\right)^{-\frac{1}{\nu}-\frac{\lambda}{\nu}}
\,\mathrm {d}s\right)_+=\widetilde\psi_\nu\left(\frac{l-u}{d}\right).
\end{equation}
At this point, we state the following lemma which provides a characterization of the quantity $\psi_\nu$.
  \begin{lemma}\label{lemmainequalitypsi-}
  Let $\psi_\nu$ be the quantity defined in \eqref{psi-} and $0<\lambda<\nu-1$. Then, we have
    \begin{equation}\label{inequalitypsi-}
\left(\frac{\epsilon_1}{1+\epsilon_1}\right)^\frac{1+\lambda}{\nu}\left(\frac{l-u}{d}\right)^{\frac{\nu-1-\lambda}{\nu}}\leq\psi_\nu\leq  \frac{\nu}{\nu-1-\lambda}
\left(\frac{l-u}{d}\right)^{\frac{\nu-1-\lambda}{\nu}}
\qquad\text{if}\qquad \frac{l-u}{d}\geq \epsilon_1.
\end{equation}
\end{lemma}
The proof of this result is not difficult and so is omitted.
We denote by
$\varphi$ a piecewise smooth function in $Q_{\rho,\Theta}(z_0)$ such that
\begin{equation}\label{def zeta}0\leq\varphi\leq1, \quad|D\varphi|<\infty\quad\text{and}\quad\varphi=0\quad\text{
on}\quad\partial B_\rho(x_0).\end{equation}
Let $k$ be a fixed real number. For a function $v\in L_{\loc}^1(\Omega_T)$, the truncations are defined by
\begin{equation*}\begin{split}&(v-k)_+=\max\{v-k;0\}\\
&(v-k)_-=\max\{-(v-k);0\}.\end{split}\end{equation*}
The crucial result in our development of regularity property of weak solutions
to the parabolic double-phase equation will be the following Caccioppoli-type estimate.
  \begin{lemma}\label{Cac1}
  Let $d>0$, $l>0$, $\lambda<(q-1)^{-1}$ and let
  $u$ be a weak solution to the parabolic double-phase equation \eqref{parabolic}
 in the sense of Definition \ref{weak solution}.
 There exists a constant $\gamma$ depending only upon the data and $\lambda$, such that
 for every piecewise smooth cutoff function $\varphi$ satisfying \eqref{def zeta},
 we have
  \begin{equation}\begin{split}\label{Cacformula1}
  \esssup_{t_0-\Theta<t<t_0}&\int_{L^-(t)}G\left(\frac{l-u}{d}\right)\varphi^M\,\mathrm {d}x+
  \iint_{L^-}\left(d^{p-2}|D\psi_p|^p+a(x,t)d^{q-2}|D\psi_q|^q\right)\varphi^M\,\mathrm {d}x\mathrm {d}t
  \\
 \leq &\gamma\int_{L^-(t_0-\Theta)}G\left(\frac{l-u}{d}\right)\varphi^M\,\mathrm {d}x
 +\gamma \iint_{L^-}\frac{l-u}{d}\left|\partial_t\varphi\right|\varphi^{M-1}\,\mathrm {d}x\mathrm {d}t
  \\&+\gamma  d^{p-2}\iint_{L^-}\left(\frac{l-u}{d}\right)^{(1+\lambda)(p-1)}
  \varphi^{M-p}|D\varphi|^p\,\mathrm {d}x\mathrm {d}t
  \\&+\gamma  d^{q-2}\iint_{L^-}a(x,t)\left(\frac{l-u}{d}\right)^{(1+\lambda)(q-1)}
  \varphi^{M-q}|D\varphi|^q\,\mathrm {d}x\mathrm {d}t
\\&+\gamma \frac{\Theta}{d^2}\int_{B_\rho(x_0)}g^{\frac{p}{p-1}}\,\mathrm {d}x
  +\gamma \frac{\Theta}{d}\int_{B_\rho(x_0)}|f|\,\mathrm {d}x,
   \end{split}\end{equation}
   where $M>q$, $L^-=Q_{\rho,\Theta}(z_0)\cap \{u\leq l\}$ and $L^-(t)=B_\rho(x_0)\cap \{u(\cdot,t)\leq l\}$.
  \end{lemma}
  \begin{proof}
 For simplicity of notation, we write $B=B_\rho(x_0)$ and $Q=Q_{\rho,\Theta}(z_0)$.
 In the weak formulation \eqref{weaksolution} we choose the testing function
  \begin{equation*}\begin{split}\varphi_-=\left[\int_u^l\left(1+\frac{l-s}{d}\right)^{-1-\lambda}\,\mathrm {d}s\right]_+
  \varphi^M=d\phi_+\left(\frac{l-u}{d}\right)\varphi^M,\end{split}\end{equation*}
  where $\phi_+$ is the function defined in \eqref{ph+} and $M>q$.
Since $|Du|\in L_{\loc}^q(\Omega_T)$, the use of $\varphi_-$ as a testing function is justified, modulus a time mollification technique.
The weak formulation \eqref{weaksolution} can be rewritten as
\begin{equation}\begin{split}\label{weakformulaa}
   \iint_Q\varphi_-\partial_tu \,\mathrm {d}x\mathrm {d}t
   + \iint_Q\big\langle A(x,t,u,Du), D\varphi_-\big\rangle
\,\mathrm {d}x\mathrm {d}t=\iint_Qf \varphi_-\,\mathrm {d}x\mathrm {d}t.
    \end{split}\end{equation}
 We now proceed formally for the term involving the time derivative by assuming that the time derivative exists.
 According to \eqref{PhiG}, we find that
 for any $t\in (t_0-\Theta,t_0)$
   \begin{equation*}\begin{split}
   \int_{t_0-\Theta}^t&\int_B\varphi_-\partial_tu \,\mathrm {d}x\mathrm {d}t
 \\ & =-\int_{t_0-\Theta}^{t}\int_B\frac{\partial}{\partial t} \left[\left(\int_u^l\,\mathrm {d}w
   \int_w^l\left(1+\frac{l-s}{d}\right)^{-1-\lambda}\,\mathrm {d}s\right)_+\right]\varphi^M\,\mathrm {d}x\mathrm {d}t
  \\ &\leq -\gamma d^2\int_{L^-(t)}G\left(\frac{l-u}{d}\right)\varphi^M\,\mathrm {d}x
  \\&\quad+\gamma d^2
  \iint_{L^-}\frac{l-u}{d}\varphi^{M-1}\left|\partial_t\varphi\right|\,\mathrm {d}x\mathrm {d}t
  +\gamma d^2\int_{L^-(t_0-\Theta)}G\left(\frac{l-u}{d}\right)\varphi^M\,\mathrm {d}x.
   \end{split}\end{equation*}
   Next, we consider the second term on the left-hand side of \eqref{weakformulaa}. To this aim, we decompose
   \begin{equation*}\begin{split}
   \iint_Q&\big\langle A(x,t,u,Du), D\varphi_-\big\rangle
\,\mathrm {d}x\mathrm {d}t
\\= &Md\iint_Q A(x,t,u,Du)\varphi^{M-1}
\phi_+\left(\frac{l-u}{d}\right)D\varphi
\,\mathrm {d}x\mathrm {d}t
\\&-\iint_{L^-} A(x,t,u,Du)\cdot Du
\left(1+\frac{l-u}{d}\right)^{-1-\lambda}
\varphi^M
\,\mathrm {d}x\mathrm {d}t
=:T_1+T_2.
\end{split}\end{equation*}
We first consider the estimate for $T_2$. From the structure condition \eqref{A}$_1$ and \eqref{psi-}, we conclude that
 \begin{equation*}\begin{split}
 T_2&\leq-C_0\iint_{L^-}\left(|Du|^p+a(x,t)|Du|^q\right)\left(1+\frac{l-u}{d}\right)^{-1-\lambda}\varphi^M\,\mathrm {d}x\mathrm {d}t\\&=
  -C_0\iint_{L^-}\left(d^p|D\psi_p|^p+a(x,t)d^q|D\psi_q|^q\right)\varphi^M\,\mathrm {d}x\mathrm {d}t.
 \end{split}\end{equation*}
 To estimate $T_1$, we use the structure condition \eqref{A}$_2$ and \eqref{phi} to deduce that
  \begin{equation*}\begin{split}
 T_1&\leq d\iint_Q (|Du|^{p-1}+a(x,t)|Du|^{q-1}+g)\varphi^{M-1} |D\varphi|
\left(\frac{l-u}{d}\right)\left(1+\frac{l-u}{d}\right)^{-1}
\,\mathrm {d}x\mathrm {d}t
\\&=:T_3+T_4+T_5.
 \end{split}\end{equation*}
Next, we consider the estimate for $T_3$.
We apply Young's inequality to conclude that for any fixed
 $\epsilon>0$ there holds
  \begin{equation*}\begin{split}T_3\leq &\epsilon d^p\iint_{L^-}|D\psi_p|^p\varphi^M\,\mathrm {d}x\mathrm {d}t
\\ &+\gamma d^p\iint_{L^-}\varphi^{M-p}|D\varphi|^p\left(1+\frac{l-u}{d}\right)
^{\lambda(p-1)-1}\left(\frac{l-u}{d}\right)^{p}
\,\mathrm {d}x\mathrm {d}t
 \\ \leq &\epsilon d^p\iint_{L^-}|D\psi_p|^p\varphi^M\,\mathrm {d}x\mathrm {d}t
\\ &+\gamma(\lambda,\epsilon) d^p\iint_{L^-}\varphi^{M-p}|D\varphi|^p\left(\frac{l-u}{d}\right)^{(1+\lambda)(p-1)}
\,\mathrm {d}x\mathrm {d}t,
\end{split}\end{equation*}
since $\lambda<(q-1)^{-1}<(p-1)^{-1}$. Similarly, we conclude from $\lambda<(q-1)^{-1}$ that for any fixed
 $\epsilon>0$ there holds
 \begin{equation*}\begin{split}T_4\leq &\epsilon d^q\iint_{L^-}a(x,t)|D\psi_q|^q\varphi^M\,\mathrm {d}x\mathrm {d}t
\\ &+\gamma(\lambda,\epsilon) d^q\iint_{L^-}a(x,t)\varphi^{M-q}|D\varphi|^q\left(\frac{l-u}{d}\right)^{(1+\lambda)(q-1)}
\,\mathrm {d}x\mathrm {d}t.
\end{split}\end{equation*}
To estimate $T_5$, we apply Young's inequality to conclude that
 \begin{equation*}\begin{split}T_5=&\gamma \iint_{L^-}\left[d\varphi^{\frac{M}{p}-1}  \left(\frac{l-u}{d}\right)\left(1+\frac{l-u}{d}\right)^{-\frac{1}{p}
 +\frac{\lambda}{p^\prime}} |D\varphi|\right]
  \\&
  \qquad\times\left[\left(1+\frac{l-u}{d}\right)^{-\frac{1}{p^\prime}-\frac{\lambda}{p^\prime}}g\varphi^{\frac{M}{p^\prime}}\right]
\,\mathrm {d}x\mathrm {d}t
  \\ \leq &\gamma d^p\iint_{L^-}\varphi^{M-p}|D\varphi|^p\left(\frac{l-u}{d}\right)^{(1+\lambda)(p-1)}
\,\mathrm {d}x\mathrm {d}t
+\gamma\Theta\int_Bg^{\frac{p}{p-1}}
  \,\mathrm {d}x,
\end{split}\end{equation*}
since $\lambda<(p-1)^{-1}$.
Taking into account that $\varphi_-\leq \lambda^{-1} d$,
we obtain an estimate for the right-hand side of \eqref{weakformulaa} by
\begin{equation*}\begin{split}
\iint_Qf\varphi_-\,\mathrm {d}x\mathrm {d}t\leq \gamma d\Theta\int_B|f|
  \,\mathrm {d}x.
\end{split}\end{equation*}
Combining the above estimates and dividing by
$d^2$, we obtain the desired estimate \eqref{Cacformula1}. This completes the proof of the lemma.
\end{proof}
  We now turn to consider the logarithmic estimates for the weak solutions to the parabolic double-phase equations. To this end,
  we introduce the logarithmic function
  \begin{equation}\begin{split}\label{ln}\Psi^{\pm}(u)&=\ln^+\left(\frac{H_k^{\pm}}{H_k^{\pm}-(u-k)_{\pm}+c}\right),
  \qquad 0<c<H_k^{\pm},
\end{split}\end{equation}
where $H_k^{\pm}$ is a constant chosen such that $H_k^{\pm}\geq \esssup_{Q_{\rho,\Theta}}(u-k)_{\pm}$.
We are now in a position to state the following lemma.
\begin{lemma}\label{logestimatelemma}Let
  $u$ be a bounded weak solution to the parabolic double-phase equation \eqref{parabolic}
 in the sense of Definition \ref{weak solution}. Assume that $\Theta\leq \rho^2$.
 There exists a  constant $\gamma$ that can be determined a priori only in terms of the data such that
 \begin{equation}\begin{split}\label{lnCac}
&\esssup_{t_0-\Theta<t<t_0}\int_{B_\rho(x_0)\times\{t\}}[\Psi^{\pm}(u)]^2\phi^q\,\mathrm {d}x
\\ \leq &\int_{B_\rho(x_0)\times\{t_0-\Theta\}}[\Psi^{\pm}(u)]^2\phi^q\,\mathrm {d}x+\gamma
\iint_{Q_{\rho,\Theta}(z_0)}a(z_0)\Psi^{\pm}(u)|D\phi|^q\left[(\Psi^{\pm})^\prime(u)\right]^{2-q}
\,\mathrm {d}x\mathrm {d}t
\\&+\gamma
\iint_{Q_{\rho,\Theta}(z_0)}\Psi^{\pm}(u)\left(|D\phi|^p+\rho^\alpha|D\phi|^q\right)\left[(\Psi^{\pm})^\prime(u)\right]^{2-p}
\,\mathrm {d}x\mathrm {d}t
\\&+\gamma\ln\left(\frac{H}{c}\right)\left(\frac{1}{c}\iint_{Q_{\rho,\Theta}(z_0)}|f|\,\mathrm {d}x\mathrm {d}t+
\frac{1}{c^2}\iint_{Q_{\rho,\Theta}(z_0)}g^{\frac{p}{p-1}}\,\mathrm {d}x\mathrm {d}t\right),
\end{split}\end{equation}
where $\phi=\phi(x)$ is independent of $t$ and satisfies \eqref{def zeta}.
\end{lemma}
The proof of Lemma \ref{logestimatelemma} is quite similar to that of \cite[Lemma 3.2]{KMS} and so is omitted.
Next, we state in Lemma \ref{caclemma} an energy estimate, which will be used in the proof of Lemma \ref{DeGiorgi3}.
\begin{lemma}\label{caclemma}
Let $u$ be a weak solution to the parabolic double-phase equation \eqref{parabolic}
 in the sense of Definition \ref{weak solution}.
 There exists a positive constant
$\gamma$ depending only upon the data, such that for every piecewise smooth cutoff function
$\varphi$ and satisfying \eqref{def zeta}, there holds
\begin{equation}\begin{split}\label{Caccioppoli}
&\esssup_{t_0-\Theta<t<t_0}\int_{B_\rho(x_0)\times\{t\}}(u-k)_{\pm}^2\varphi^q \,\mathrm {d}x
\\&+\iint_{Q_{\rho,\Theta}(z_0)}\left(|D(u-k)_{\pm}|^p+a(x,t)|D(u-k)_{\pm}|^q \right)\varphi^q\,\mathrm {d}x\mathrm {d}t\\
\leq &\int_{B_\rho(x_0)\times\{t_0-\Theta\}}(u-k)_{\pm}^2\varphi^q \,\mathrm {d}x+\gamma
\iint_{Q_{\rho,\Theta}(z_0)} (u-k)_{\pm}^2|\partial_t\varphi^q|\,\mathrm {d}x\mathrm {d}t
\\&+\gamma\iint_{Q_{\rho,\Theta}(z_0)
} (u-k)_{\pm}^p|D\varphi|^p+a(x,t)(u-k)_{\pm}^q|D\varphi|^q
\,\mathrm {d}x\mathrm {d}t
\\&
+\gamma
\iint_{Q_{\rho,\Theta}(z_0)} |f|(u-k)_{\pm}\,\mathrm {d}x\mathrm {d}t+\gamma
\iint_{Q_{\rho,\Theta}(z_0)} g^{\frac{p}{p-1}}\,\mathrm {d}x\mathrm {d}t.
\end{split}\end{equation}
\end{lemma}
This is a standard result that can be proved by choosing the testing function $\pm(u-k)_{\pm}\varphi^q$
into \eqref{weaksolution} and the proof will be omitted.

We now turn our attention to the proof of Theorem \ref{main1}.
The continuity of the weak solution $u$ at a point $z_0$
will be a consequence of the following assertion. There exists a family of nested and shrinking
cylinders $Q_n$,
with vertex at $z_0$,
such that the essential oscillation of $u$ in $Q_n$ converges to
zero as the cylinders shrink to
the point $z_0$.

Without loss of generality
we may assume that $z_0=(0,0)$. We define $a_0=a(z_0)=a(0,0)$.
Let $R_0>0$ be such that $Q_{R_0,R_0^2}=B_{R_0}\times (-R_0^2,0)\Subset\Omega_T$.
We set
\begin{equation*}\begin{split}
\mu_+=\esssup_{Q_{R_0,R_0^2}}u,\qquad \mu_-=\essinf_{Q_{R_0,R_0^2}}u\qquad\text{and}\qquad\essosc_{Q_{R_0,R_0^2}}u=\mu_+-\mu_-.
\end{split}\end{equation*}
Henceforth, let $A>10$ be a constant which will be determined later.
At this stage, we set $R=\frac{9}{10}R_0$,
\begin{equation}\begin{split}\label{defomega}\omega=\max\left\{\mu_+-\mu_-,A^{\frac{q-2}{p-2}}R\right\},\qquad
\Theta_\omega(r)=\omega^2\left[\left(\frac{\omega}{r}
\right)^p+a_0\left(\frac{\omega}{r}\right)^q\right]^{-1},
\\
\Theta_A=\left(\frac{\omega}{A}\right)^2\left[\left(\frac{\omega}{AR}
\right)^p+a_0\left(\frac{\omega}{AR}\right)^q\right]^{-1}\quad\text{and}\quad Q_A=B_R\times(-\Theta_A,0),\end{split}\end{equation}
where $0<r\leq R$. Specifically, if $r=R$, we abbreviate $\Theta_\omega(R)$ to $\Theta_\omega$.
It follows from \eqref{defomega} that $A^{p-2}\Theta_\omega\leq \Theta_A\leq A^{q-2}\Theta_\omega$,
$\Theta_A\leq A^{p-q}R^2<R^2$ and hence that $Q_A\subset Q_{R,R^2}$.
For any fixed time level $-\frac{8}{9}\Theta_A\leq t\leq0$ we introduce
the intrinsic cylinder of the type
$Q_r^-(t)=B_r\times(t-\Theta_\omega(r),t)$, where $0<r\leq R$.
We remark that $Q_{R}^-(t)\subset Q_A$ for all $-\frac{8}{9}\Theta_A\leq t\leq0$, provided that we choose $A>9^{\frac{1}{p-2}}$.
Our task now is to establish a decay estimate of the type
\begin{equation}\label{osc1st}\essosc_{Q^\prime}u\leq \eta \omega+\gamma\left(R+G_p(R)+F_p(R)\right),\end{equation}
where $Q^\prime\subset Q_A$ is a smaller cylinder, $\frac{1}{2}\leq\eta<1$ and the constant $\gamma$ depends only upon the data and $A$. To this end,
we need only consider the case
\begin{equation}\label{initialupperbound}\mu_+-\mu_-\geq \frac{\omega}{2}.\end{equation}
Motivated by the work of DiBenedetto \cite[chapter III]{Di93}, we consider two complementary cases.
For a fixed constant $\nu_0>0$, we see that either
  \begin{itemize}
 \item[$\bullet$]
 \textbf{The first alternative}. There exists $-\frac{8}{9}\Theta_A\leq \bar t\leq0$ such that
\begin{equation}\label{1st}\left|\left\{(x,t)\in Q_{R}^-(\bar t)
:u\leq \mu_-+\frac{\omega}{4}\right\}\right|\leq \nu_0|Q_{R}^-(\bar t)|\end{equation}
\end{itemize}
or this does not hold. More precisely, if \eqref{1st} does not hold, we infer from \eqref{initialupperbound} that the
 following second alternative holds.
  \begin{itemize}
 \item[$\bullet$]
 \textbf{The second alternative}. For any $-\frac{8}{9}\Theta_A\leq \bar t\leq0$, there holds
\begin{equation}\label{2nd}\left|\left\{(x,t)\in Q_{R}
^-(\bar t):u>\mu_+-\frac{\omega}{4}\right\}\right|\leq (1-\nu_0)|Q_{R}^-(\bar t)|.\end{equation}
\end{itemize}
The constant $\nu_0>0$ will be determined in the course of the proof of Lemma \ref{lemmaDeGiorgi1} in Section 4,
while the value of $A$ will be
fixed during the proof of Lemma \ref{DeGiorgi3} in Section 5.
\begin{remark}\label{phase analysis 1}
Since $Q_A\subset Q_{R,R^2}$, we have $|a(x,t)-a(y,s)|\leq 2[a]_\alpha R^\alpha$ holds for all $(x,t)$, $(y,s)\in Q_A$.
In the $(p,q)$-phase, i.e., $a_0\geq 10[a]_\alpha R^\alpha$. For any $(x,t)\in Q_A$, we have $a(x,t)\leq a(0,0)+|a(x,t)-a(0,0)|\leq \frac{6}{5}a_0$ and
$a(x,t)\geq a(0,0)-|a(x,t)-a(0,0)|\geq \frac{4}{5}a_0$.
In the $p$-phase, i.e., $a_0< 10[a]_\alpha R^\alpha$. For any $(x,t)\in Q_A$, we have $a(x,t)\leq a(0,0)+|a(x,t)-a(0,0)|\leq 12[a]_\alpha R^\alpha$.
In summary, we conclude that the implications
 \begin{equation}\label{phase analysis}
	\begin{cases}
	(p,q)-\text{phase}\ (a_0\geq 10[a]_\alpha R^\alpha)\ \Rightarrow \ \frac{4}{5}a_0\leq a(x,t)\leq \frac{6}{5}a_0 ,\\
	p-\text{phase}\ (a_0< 10[a]_\alpha R^\alpha)\ \Rightarrow \ a(x,t)\leq 12[a]_\alpha R^\alpha,
	\end{cases}
\end{equation}
hold for all $(x,t)\in Q_A$.
\end{remark}
\section{The first alternative}
The aim of this section is to establish a decay estimate for essential oscillation of weak solutions similar to \eqref{osc1st} for the first alternative.
A key ingredient in the proof is
a De Giorgi type lemma. In principle the argument
of the proof of the De Giorgi type lemma is based on Kilpel\"ainen-Mal\'y technique from \cite{KM,LS}.
Moreover, we shall use phase analysis to address the double-phase problem.
 \begin{lemma}\label{lemmaDeGiorgi1}
 Let $u$ be a bounded weak solution to \eqref{parabolic}-\eqref{A} in $\Omega_T$.
 There exist constants $\nu_0\in(0,1)$ and $B>1$, depending only on the data, such that if
\begin{equation}\label{1st assumption}\left|\left\{(x,t)\in Q_{R}^-(\bar t):u\leq\mu_-+\frac{\omega}{4}\right\}\right
|\leq \nu_0| Q_{R}^-(\bar t)|,\end{equation}
then either
\begin{equation}
\label{DeGiorgi1}u(x,t)>\mu_-+\frac{\omega}{2^5}\qquad\text{for}\ \ \text{a.e.}\ \ (x,t)\in Q_{\frac{1}{4}R}^-(\bar t)\end{equation}
or
\begin{equation}
\label{omega1}\omega\leq B\left(F_p(R)+G_p(R)+100R\right).\end{equation}
Here, the time level $\bar t$ satisfies $-\frac{8}{9}\Theta_A\leq \bar t\leq0$.
 \end{lemma}
 \begin{proof}To start with, we set $\tilde u=u-\mu_-$ and observe that $\tilde u$ is nonnegative in $Q_A$. Note that $\tilde u$ is a bounded weak solution to the
parabolic double-phase equation
\begin{equation*}\partial_t \tilde u-\operatorname{div}A(x,t,\tilde u+\mu_-,D\tilde u)=f,\end{equation*}
where the vector field $A(x,t,\tilde u+\mu_-,\xi)$ satisfies the structure condition \eqref{A}. Moreover, we abbreviate $\tilde u$ by $u$.
 Let $B>10$ to be determined in the course of the proof.
 We first assume that \eqref{omega1} is violated, that is,
 \begin{equation}
\label{omega1violated}\omega>\frac{3}{4}B\left(\frac{1}{3B}\omega+F_p(R)+G_p(R)+100R
\right).\end{equation}
Fix $z_1=(x_1,t_1)\in Q_{\frac{1}{4}R}^-(\bar t)$ and assume that $(x_1,t_1)$ is a Lebesgue point of $u$.
The lemma will be proved by  showing that $u(x_1,t_1)>2^{-5}\omega$.  The proof of this inequality will be divided into four steps.

Step 1: \emph{Setting up the notations.}
Let $a_1=a(x_1,t_1)$, $r_j=4^{-j}C^{-1}R$
 and $B_j=B_{r_j}(x_1)$ where $C>4$ is to be determined.
For a sequence $\{l_j\}_{j=0}^\infty$ and a fixed $l>0$, we set $d_{j}(l)=l_j-l$. If $d_{j}(l)>0$, then we define
a parabolic cylinder $Q_j(l)=B_j\times (t_1-\Theta_j(l),t_1)$, where
 \begin{equation*}\Theta_j(l)=d_{j}(l)^2\left[\left(\frac{d_{j}(l)}{r_j}\right)^p+a_1\left(\frac{d_{j}(l)}{r_j}\right)^q\right]^{-1}.\end{equation*}
 Furthermore, we define $\varphi_j(l)=\phi_j(x)\theta_{j,l}(t)$, where
$\phi_j\in C_0^\infty(B_j)$, $\phi_j=1$ on $B_{j+1}$, $|D\phi_j|\leq r_j^{-1}$
 and $\theta_{j,l}(t)$ is a Lipschitz function
satisfies $\theta_{j,l}(t)=1$ in $t\geq t_1-\frac{4}{9}\Theta_j(l)$, $\theta_{j,l}(t)=0$ in $t\leq t_1-\frac{5}{9}\Theta_j(l)$
and
  \begin{equation*}
 \theta_{j,l}(t)=\frac{t-t_1-\frac{5}{9}\Theta_j(l)}{\frac{1}{9}\Theta_j(l)}\qquad\text{in}\qquad
 t_1-\frac{5}{9}\Theta_j(l)\leq t\leq t_1-\frac{4}{9}\Theta_j(l).
 \end{equation*}
From the definition of $\varphi_j(l)$, we see that $\varphi_j(l)=0$ on $\partial_PQ_j(l)$.
 Next, for  $j=-1,0,1,2,\cdots$, we define the sequence $\{\alpha_j\}$ by
  \begin{equation}\begin{split}\label{alpha}\alpha_j=&\frac{4^{-j-100
  }}{3B}\omega+ \frac{3}{4}\int_0^{r_j}\left(r^{p-n}\int_{B_r(x_1)}
  g(y)^{\frac{p}{p-1}} \,\mathrm {d}y
  \right)^{\frac{1}{p}}\frac{\mathrm {d}r}{r}\\&+
  \frac{3}{4}\int_0^{r_j}\left(r^{p-n}\int_{B_r(x_1)}|f(y)| \,\mathrm {d}y
  \right)^{\frac{1}{p-1}}\frac{\mathrm {d}r}{r}+75r_j.
  \end{split}\end{equation}
  According to the definition of $\alpha_j$, we see that $\alpha_j\to 0$ as $j\to\infty$ and there holds $B\alpha_{j-1}\leq\omega$,
  \begin{equation}\begin{split}\label{alpha1}
  \alpha_{j-1}-\alpha_j\geq &\frac{4^{-j-100}}{B}\omega+\gamma\left(r_j^{p-n}\int_{B_j} g(y)^{\frac{p}{p-1}}
  \,\mathrm {d}y\right)^{\frac{1}{p}}\\&+\gamma\left(r_j^{p-n}\int_{B_j}|f(y)|
  \,\mathrm {d}y\right)^{\frac{1}{p-1}}+200 r_j
  \end{split}\end{equation}
  and
   \begin{equation}\begin{split}\label{alpha2}
  \alpha_{j-1}-\alpha_j\leq & \frac{4^{-j-100}}{B}\omega+\gamma\left(r_{j-1}^{p-n}
  \int_{B_{j-1}} g(y)^{\frac{p}{p-1}}
  \,\mathrm {d}y\right)^{\frac{1}{p}}\\&+\gamma\left(r_{j-1}^{p-n}\int_{B_{j-1}}|f(y)|
  \,\mathrm {d}y\right)^{\frac{1}{p-1}}+\gamma r_{j-1}
  \end{split}\end{equation}
  for all $j=0,1,2,\cdots$, where the constant $\gamma$ depends only upon the data.
  Moreover, we define a quantity $A_j(l)$ by
 \begin{equation}\begin{split}\label{A_j}
 A_j(l)=&\frac{d_j(l)^{p-2}}{r_j^{n+p}}\iint_{L_j(l)}\left(\frac{l_j-u}{d_j(l)}\right)^{(1+\lambda)(p-1)}\varphi_j(l)^{M-p}
 \,\mathrm {d}x\mathrm {d}t
 \\&+\frac{d_j(l)^{q-2}}{r_j^{n+q}}\iint_{L_j(l)}a(x,t)\left(\frac{l_j-u}{d_j(l)}\right)^{(1+\lambda)(q-1)}\varphi_j(l)^{M-q}
 \,\mathrm {d}x\mathrm {d}t
 \\&+\esssup_t\frac{1}{r_j^n}\int_{B_j\times\{t\}}G\left(\frac{l_j-u}{d_j(l)}\right)\varphi_j(l)^M
 \,\mathrm {d}x,
 \end{split}\end{equation}
 where $k>q$, $G$ is defined in \eqref{G} and $L_j(l)=Q_j(l)\cap \{u\leq l_j\}\cap \Omega_T$.
 Throughout the proof, we keep $\lambda=\frac{p}{nq}$.
 It is easy to check that $A_j(l)$ is continuous in $l<l_j$. In the following, the sequence $\{l_j\}_{j=0}^\infty$ we constructed will satisfy $l_j\geq\frac{1}{16}\omega$
 for all $j=0,1,2,\cdots$. At this point, we assert that either $u(x_1,t_1)\geq\frac{1}{16}\omega$ or $A_j(l)\to+\infty$ as $l\to l_j$.
 For a fixed $j\geq0$, if $d_j(l)<1$, then $\Theta_j(l)\geq\left(r_j^{-p}+a_1r_j^{-q}\right)^{-1}=:\theta_j$ and we infer from \eqref{A_j} that
  \begin{equation}\begin{split}\label{assumption for Aj}
  A_j(l)\geq& \frac{d_j(l)^{p-2-(1+\lambda)(p+1)}}{r_j^{n+p}}\iint_{\tilde Q_j}(l_j-u)_+^{(1+\lambda)(p-1)}\,\mathrm {d}x\mathrm {d}t
  \\&+ \frac{d_j(l)^{q-2-(1+\lambda)(q+1)}}{r_j^{n+q}}\iint_{\tilde Q_j}a(x,t)(l_j-u)_+^{(1+\lambda)(q-1)}\,\mathrm {d}x\mathrm {d}t,
  \end{split}\end{equation}
  where $\tilde Q_j=B_{j+1}\times \left(t_1-\frac{4}{9}\theta_j,t_1\right)$. Since $\theta_j=\left(r_j^{-p}+a_1r_j^{-q}\right)^{-1}\leq r_j^p\leq r_j^2$,
  we see that $\tilde Q_j\subseteq Q_{r_j,r_j^2}(z_1)$.
 In the case $a_1<10[a]_\alpha r_j^{q-p}$, we have $\theta_j\geq (1+10[a]_\alpha)^{-1}r_j^p$. Now, we set $Q_j^\prime=B_{j+1}\times\left(t_1-\frac{4}{9}(1+10[a]_\alpha)^{-1}r_j^p,t_1\right)$.
 If
  \begin{equation*}\begin{split}
  \iint_{Q_j^\prime}(l_j-u)_+^{(1+\lambda)(p-1)}\,\mathrm {d}x\mathrm {d}t=0,
   \end{split}\end{equation*}
   then $u(x_1,t_1)\geq l_j\geq \frac{1}{16}\omega$, since $(x_1,t_1)$ is a Lebesgue point of $u$. On the other hand, if
    \begin{equation*}\begin{split}
  \iint_{Q_j^\prime}(l_j-u)_+^{(1+\lambda)(p-1)}\,\mathrm {d}x\mathrm {d}t>0,
   \end{split}\end{equation*}
   then we infer from \eqref{assumption for Aj} that
    \begin{equation*}\begin{split}
  A_j(l)\geq& \frac{d_j(l)^{p-2-(1+\lambda)(p+1)}}{r_j^{n+p}}\iint_{Q_j^\prime}(l_j-u)_+^{(1+\lambda)(p-1)}\,\mathrm {d}x\mathrm {d}t\to+\infty
  \end{split}\end{equation*}
  as $l\to l_j$ and this proves the assertion for the case $a_1<10[a]_\alpha r_j^{q-p}$. In the case $a_1\geq10[a]_\alpha r_j^{q-p}$, we have
  $\theta_j\geq (1+[a]_\alpha^{-1})^{-1}a_1^{-1}r_j^q$. Noting that $q\leq p+\alpha$
  and $r_j<1$, we have $a_1\geq10[a]_\alpha r_j^\alpha$ and therefore $\frac{4}{5}a_1\leq a(x,t)\leq \frac{6}{5}a_1$ for all $(x,t)\in \tilde Q_j$. At this point, we set
  $Q_j^{\prime\prime}=B_{j+1}\times\left(t_1-\frac{4}{9}(1+[a]_\alpha^{-1})^{-1}a_1^{-1}r_j^q,t_1\right)$.
  If
  \begin{equation*}\begin{split}
  \iint_{Q_j^{\prime\prime}}(l_j-u)_+^{(1+\lambda)(q-1)}\,\mathrm {d}x\mathrm {d}t=0,
   \end{split}\end{equation*}
   then $u(x_1,t_1)\geq l_j\geq \frac{1}{16}\omega$, since $(x_1,t_1)$ is a Lebesgue point of $u$. On the other hand, if
 \begin{equation*}\begin{split}
  \iint_{Q_j^{\prime\prime}}(l_j-u)_+^{(1+\lambda)(q-1)}\,\mathrm {d}x\mathrm {d}t>0,
   \end{split}\end{equation*}
   then we conclude from $Q_j^{\prime\prime}\subseteq \tilde Q_j$ and \eqref{assumption for Aj} that
    \begin{equation*}\begin{split}
  A_j(l)\geq& \frac{4}{5}a_1\frac{d_j(l)^{q-2-(1+\lambda)(q+1)}}{r_j^{n+q}}\iint_{Q_j^{\prime\prime}}(l_j-u)_+^{(1+\lambda)(q-1)}\,\mathrm {d}x\mathrm {d}t\to+\infty
  \end{split}\end{equation*}
  as $l\to l_j$. This proves the assertion for all the cases. It is obvious that the lemma holds if $u(x_1,t_1)\geq \frac{1}{16}\omega$. In the following, we assume that
  $A_j(l)\to+\infty$ as $l\to l_j$ holds for all $j=0,1,2,\cdots$.

 Step 2: \emph{Determine the values of $l_0$ and $l_1$.} Initially, we set $l_0=\frac{1}{4}\omega$ and $\bar l=\frac{1}{2}l_0+\frac{1}{16}
 B\alpha_0+\frac{1}{32}\omega$.
 Recalling that $B\alpha_0< \omega$, we deduce $\frac{1}{32}\omega\leq d_0(\bar l)\leq\frac{3}{32}\omega$ and
 this also implies that $d_0(\bar l)\leq\frac{3}{16}\|u\|_\infty$. Next, we claim that
 $Q_0(\bar l)\subseteq Q_R^-(\bar  t).$
In the $(p,q)$-phase, i.e., $a_0\geq 10[a]_\alpha R^\alpha$, we infer from \eqref{phase analysis} that
$\frac{4}{5}a_0\leq a_1\leq \frac{6}{5}a_0$. Recalling that $r_0=C^{-1}R$, we have
\begin{equation}\begin{split}\label{thetaupper}
\Theta_0(\bar l)&\leq d_{0}(\bar l)^2\left[\left(\frac{Cd_{0}(\bar l)}{R}\right)^p+\frac{4}{5}a_0\left(\frac{Cd_{0}(\bar l)}{R}\right)^q\right]^{-1}
\leq \left(\frac{1}{32}\right)^{2-q}\frac{5}{4}C^{-p}\Theta_\omega,
\end{split}\end{equation}
where $\Theta_\omega=\Theta_\omega(R)$ is defined in \eqref{defomega}. Since $4^{-q}\Theta_\omega\leq\Theta_\omega(\frac{1}{4}R)\leq 4^{-p}\Theta_\omega$, we infer
from \eqref{thetaupper} that
$\Theta_0(\bar l)\leq \Theta_\omega-\Theta_\omega(\frac{1}{4}R)$ and also $Q_0(\bar l)\subseteq Q_R^-(\bar  t)$, provided that we choose
\begin{equation}\label{C1condition}C\geq \max\left\{4,\ \ \left(\frac{5}{4}\frac{32^{q-2}}{1-4^{-p}}\right)^\frac{1}{p}\right\}.\end{equation}
Next, we consider the case of $p$-phase, i.e., $a_0<10[a]_\alpha R^\alpha$. It follows from \eqref{phase analysis} that $a_1\leq 12[a]_\alpha R^\alpha$.
Since $\omega\leq 2\|u\|_\infty$, $R<1$ and $q\leq p+\alpha$, we deduce that
\begin{equation*}a_0\left(\frac{\omega}{R}\right)^q\leq 10[a]_\alpha2^{q-p}\|u\|_\infty^{q-p}\frac{\omega^p}{R^{q-\alpha}}
\leq \tilde\gamma_1\left(\frac{\omega}{R}\right)^p.
\end{equation*}
Here, we set
$\tilde\gamma_1=10[a]_\alpha2^{q-p}\|u\|_\infty^{q-p}$.
Consequently, we infer that
$(1+\tilde\gamma_1)^{-1}\omega^{2-p}R^p\leq\Theta_\omega\leq\omega^{2-p}R^p$.
Since $d_0(\bar l)\geq \frac{1}{32}\omega$, we have
 \begin{equation*}\begin{split}
\Theta_0(\bar l)&\leq C^{-p}d_{0}(\bar l)^{2-p}R^p\leq C^{-p}\left(\frac{1}{32}\right)^{2-p}(1+\tilde\gamma_1)\Theta_\omega.
\end{split}\end{equation*}
At this stage, we choose
\begin{equation}\label{C}C=\max\left\{4,\ \ \left[(1+\tilde\gamma_1)\left(\frac{1}{32}\right)^{2-p}\left(\frac{1}{1-4^{-p}}\right)\right]^\frac{1}{p},
\ \ \left(\frac{5}{4}\frac{32^{q-2}}{1-4^{-p}}\right)^\frac{1}{p}\right\}.\end{equation}
This choice of $C$ implies that $\Theta_0(\bar l)\leq \Theta_\omega-\Theta_\omega(\frac{1}{4}R)$ and satisfies \eqref{C1condition}. Therefore, we see that the inclusion
$Q_0(\bar l)\subseteq Q_R^-(\bar  t)$ holds
in the case of $p$-phase.

At this point, we fix a number $\chi\in(0,1)$ which will be chosen later in a universal way.
We claim that $A_0(\bar l)\leq\frac{1}{2}\chi$, provided that we determine $\nu_0=\nu_0(\text{data},\chi)<1$ and $B=B(\text{data},\chi)>1$ in a suitable way.
Once again, we consider two cases: $p$-phase and $(p,q)$-phase. In the case of $p$-phase, i.e., $a_0< 10[a]_\alpha R^\alpha$,
we have $a(x,t)\leq a_0+|a(x,t)-a_0|\leq 12[a]_\alpha R^\alpha$ for all $(x,t)\in Q_0(\bar l)$, since $Q_0(\bar l)\subseteq Q_R^-(\bar  t)\subseteq Q_A\subseteq Q_{R,R^2}$.
In view of $l_0-u\leq \frac{1}{4}\omega$, $\frac{1}{32}\omega\leq d_0(\bar l)\leq\frac{3}{32}\omega\leq\frac{3}{16}\|u\|_\infty$ and $q\leq p+\alpha$, we infer from \eqref{1st assumption}
that
 \begin{equation}\begin{split}\label{A01}
&\frac{d_0(\bar l)^{p-2}}{r_0^{n+p}}\iint_{L_0(\bar l)}\left(\frac{l_0-u}{d_0(\bar l)}\right)^{(1+\lambda)(p-1)}\varphi_0(\bar l)^{M-p}
 \,\mathrm {d}x\mathrm {d}t
 \\&+\frac{d_0(\bar l)^{q-2}}{r_0^{n+q}}\iint_{L_0(\bar l)}a(x,t)\left(\frac{l_0-u}{d_0(\bar l)}\right)^{(1+\lambda)(q-1)}\varphi_0(\bar l)^{M-q}
 \,\mathrm {d}x\mathrm {d}t
 \\&\leq \gamma C^{n+q}\frac{|Q_0(\bar l)\cap \{u\leq l_0\}|}{\omega^{2-p}R^{n+p}}
 \leq \gamma C^{n+q}\frac{|Q_R^-(\bar  t)\cap \{u\leq l_0\}|}{|Q_R^-(\bar  t)|}\leq \gamma C^{n+q}\nu_0,
 \end{split}\end{equation}
 since $|Q_R^-(\bar  t)|\leq c_n\omega^{2-p}R^{n+p}$ and $Q_0(\bar l)\subseteq Q_R^-(\bar  t)$. In the case of $(p,q)$-phase, i.e., $a_0\geq10[a]_\alpha R^\alpha$,
 we have $\frac{4}{5}a_0\leq a(x,t)\leq \frac{6}{5}a_0$ for all $(x,t)\in Q_0(\bar l)$, since $Q_0(\bar l)\subseteq Q_{R,R^2}$. According to
 $l_0-u\leq \frac{1}{4}\omega$ and $\frac{1}{32}\omega\leq d_0(\bar l)\leq\frac{3}{32}\omega$, we infer from \eqref{1st assumption}
that
  \begin{equation}\begin{split}\label{A02}
&\frac{d_0(\bar l)^{p-2}}{r_0^{n+p}}\iint_{L_0(\bar l)}\left(\frac{l_0-u}{d_0(\bar l)}\right)^{(1+\lambda)(p-1)}\varphi_0(\bar l)^{M-p}
 \,\mathrm {d}x\mathrm {d}t
 \\&+\frac{d_0(\bar l)^{q-2}}{r_0^{n+q}}\iint_{L_0(\bar l)}a(x,t)\left(\frac{l_0-u}{d_0(\bar l)}\right)^{(1+\lambda)(q-1)}\varphi_0(\bar l)^{M-q}
 \,\mathrm {d}x\mathrm {d}t
 \\&\leq \gamma C^{n+q}d_0(\bar l)^{-2}\left(\frac{d_0(\bar l)^p}{R^{n+p}}+a_0\frac{d_0(\bar l)^q}{R^{n+q}}\right)|Q_0(\bar l)\cap \{u\leq l_0\}|
 \\&\leq \gamma C^{n+q}\frac{|Q_R^-(\bar  t)\cap \{u\leq l_0\}|}{|Q_R^-(\bar  t)|}\leq \gamma C^{n+q}\nu_0,
 \end{split}\end{equation}
 since $Q_0(\bar l)\subseteq Q_R^-(\bar  t)$ and
\begin{equation*}|Q_R^-(\bar  t)|=|B_R|\times\Theta_\omega\leq \gamma R^nd_0(\bar l)^2\left(\left(\frac{d_0(\bar l)}{R}\right)^p
+a_0\left(\frac{d_0(\bar l)}{R}\right)^q\right)^{-1}.\end{equation*}
According to the proofs of \eqref{A01} and \eqref{A02}, we conclude that
\begin{equation}\label{timederivativeL_0}\left(\frac{d_0(\bar l)^{p-2}}{r_0^{n+p}}+a_1\frac{d_0(\bar l)^{q-2}}{r_0^{n+q}}\right)
|L_0(\bar l)|\leq \gamma C^{n+q}\nu_0,\end{equation}
where $a_1=a(x_1,t_1)$. Furthermore, we use the Caccioppoli estimate \eqref{Cacformula1} with $(l,d,\Theta)$ replaced by $(l_0,d_0(\bar l),\Theta_0(\bar l))$
to deduce that
 \begin{equation*}\begin{split}
    \esssup_t& \frac{1}{r_0^n}\int_{B_0\times\{t\}}G\left(\frac{l_0-u}{d_0(\bar l)}\right)\varphi_0(\bar l)^M
 \,\mathrm {d}x \\ \leq &\gamma\frac{d_0(\bar l)^{p-2}}{r_0^{p+n}}\iint_{L_0(\bar l)}\left(\frac{l_0-u}{d_0(\bar l)}
 \right)^{(1+\lambda)(p-1)}
 \varphi_0(\bar l)^{M-p}\,\mathrm {d}x\mathrm {d}t
 \\ &+\gamma\frac{d_0(\bar l)^{q-2}}{r_0^{q+n}}\iint_{L_0(\bar l)}a(x,t)\left(\frac{l_0-u}{d_0(\bar l)}
 \right)^{(1+\lambda)(q-1)}
 \varphi_0(\bar l)^{M-q}\,\mathrm {d}x\mathrm {d}t
 \\&+\gamma\frac{1}{r_0^n} \iint_{L_0(\bar l)}\frac{l_0-u}{d_0(\bar l)}\varphi_0(\bar l)^{M-1}|\partial_t\varphi_0(\bar l)|\,\mathrm {d}x\mathrm {d}t
   \\&+\gamma \frac{\Theta_0(\bar l)}{r_0^nd_0(\bar l)^2}\int_{B_0}g^{\frac{p}{p-1}}\,\mathrm {d}x
  +\gamma \frac{\Theta_0(\bar l)}{r_0^nd_0(\bar l)}\int_{B_0}|f|\,\mathrm {d}x
  \\ =&L_1+L_2+L_3+L_4+L_5,
 \end{split}\end{equation*}
 since $|D\varphi_0(\bar l)|\leq r_0^{-1}$, $\varphi_0(\bar l)=0$ on $\partial_PQ_0(\bar l)$
 and the first term on the right-hand side of \eqref{Cacformula1} vanishes. We infer from \eqref{A01} and \eqref{A02} that $L_1+L_2\leq  \gamma C^{n+q}\nu_0$.
To estimate $L_3$, we use \eqref{timederivativeL_0} and
$|\partial_t\varphi_0(\bar l)|\leq 9\Theta_0(\bar l)^{-1}$ to deduce that
 \begin{equation*}\begin{split}
 L_3&\leq \gamma\frac{1}{r_0^n}\Theta_0(\bar l)^{-1}|L_0(\bar l)|
\leq
 \gamma\left(\frac{d_0(\bar l)^{p-2}}{r_0^{n+p}}+a_1\frac{d_0(\bar l)^{q-2}}{r_0^{n+q}}\right)
|L_0(\bar l)|\leq \gamma C^{n+q}\nu_0.
  \end{split}\end{equation*}
Noting that $\Theta_0(\bar l)\leq d_0(\bar l)^{2-p}r_0^p$, we infer from $d_0(\bar l)\geq \frac{1}{32}\omega$ and \eqref{omega1violated} that
 \begin{equation*}\begin{split}
 L_4\leq \gamma\frac{d_0(\bar l)^{-p}}{r_0^{n-p}}\int_{B_0}g^\frac{p}{p-1}\,\mathrm {d}x\leq \gamma C^{n-p}
 \frac{\omega^{-p}}{R^{n-p}}\int_{B_0}g^\frac{p}{p-1}\,\mathrm {d}x\leq
 \gamma C^{n-p}\frac{1}{B^p}
  \end{split}\end{equation*}
  and
  \begin{equation*}\begin{split}
 L_5\leq \gamma\frac{d_0(\bar l)^{1-p}}{r_0^{n-p}}\int_{B_0}|f|\,\mathrm {d}x\leq \gamma C^{n-p}
 \frac{\omega^{1-p}}{R^{n-p}}\int_{B_0}|f|\,\mathrm {d}x\leq
 \gamma C^{n-p}\frac{1}{B^{p-1}},
  \end{split}\end{equation*}
  where the constant $\gamma$ depends only upon the data.
  Consequently, we conclude from \eqref{A01}, \eqref{A02} and the estimates for $L_1$-$L_5$ that there exists a constant $\gamma^\prime$ such that
  \begin{equation}\begin{split}\label{gammaprime}
  A_0(\bar l)\leq \gamma^\prime C^{n+q}\nu_0+\gamma^\prime C^{n-p}\left(B^{-p}+B^{-(p-1)}\right),
   \end{split}\end{equation}
   where the constant $C$ is defined in \eqref{C}.
    At this point, we  fix a number $\chi\in(0,1)$ which will be chosen later in a universal way.
    Then, we choose $\nu_0=\nu_0(\text{data},\chi)<1$ and $B=B(\text{data},\chi)>1$ such that
     \begin{equation}\begin{split}\label{nu0}
  \gamma^\prime C^{n+q}\nu_0=\frac{\chi}{4}\qquad\text{and}\qquad\gamma^\prime C^{n-p}(B^{-p}+B^{1-p})<\frac{\chi}{4}.
    \end{split}\end{equation}
    This leads to $A_0(\bar l)\leq\frac{1}{2}\chi$. Taking into account that $A_0(l)\to+\infty$
    as $l\to l_0$ and $A_0(l)$ is continuous in $(\bar l, l_0)$, then there exists a number $\tilde l\in (\bar l, l_0)$ such that $A_0(\tilde l)=\chi$.
    From \eqref{omega1violated}, we infer that for $B>8$ there holds
  $l_0-\bar l\geq \frac{1}{32}\omega\geq \frac{1}{4
    B}\omega>\frac{1}{4}(\alpha_{-1}-\alpha_0),$
    since $B\alpha_{-1}<\omega$.
    At this point, we set
    \begin{equation}\label{l1}
	l_1=\begin{cases}
\tilde l,&\quad \text{if}\quad \tilde l<l_0-\frac{1}{4}(\alpha_{-1}-\alpha_0),\\
	l_0-\frac{1}{4}(\alpha_{-1}-\alpha_0),&\quad \text{if}\quad \tilde l\geq l_0-\frac{1}{4}(\alpha_{-1}-\alpha_0).
	\end{cases}
\end{equation}
Moreover, we define $Q_0=Q_0(l_1)$ and $d_0=l_0-l_1$. Since $B\alpha_0<\omega$, we have $l_1\geq  \bar l>\frac{1}{8}B\alpha_0+\frac{1}{16}\omega$.

    Step 3: \emph{Determine the sequence $\{l_j\}_{j=0}^{+\infty}$.} Assume that we have chosen two
   sequences $l_1,\cdots,l_j$ and $d_0,\cdots,d_{j-1}$ such that for $i=1,\cdots,j$, there holds
    \begin{equation}\label{li}\frac{1}{32}\omega+\frac{1}{2}l_{i-1}+\frac{1}{16}B\alpha_{i-1}<l_i\leq
    l_{i-1}-\frac{1}{4}(\alpha_{i-2}-\alpha_{i-1}),
    \end{equation}
    \begin{equation}\label{Aj-1}
    A_{i-1}(l_i)\leq \chi,
     \end{equation}
      \begin{equation}\label{lj}
      l_i>\frac{1}{8}B\alpha_{i-1}+\frac{1}{16}\omega.
       \end{equation}
       Then, we abbreviate $Q_i=Q_i(l_{i+1})$, $d_i=d_i(l_{i+1})=l_i-l_{i+1}$, $L_i=L_i(l_{i+1})$, $\varphi_i=\varphi_i(l_{i+1})$ and
       $\Theta_i=\Theta_i(l_{i+1})$ for $i=1,2,\cdots,j-1$. Moreover, we claim that
       \begin{equation}\label{Aj}
       A_j(\bar l)\leq \frac{1}{2}\chi,\qquad\text{where}\qquad \bar l=\frac{1}{2}l_j+\frac{1}{16}B\alpha_j+\frac{1}{32}\omega.
        \end{equation}
        Before proving \eqref{Aj}, we first claim that the inclusions
        $Q_i\subseteq Q_A$ and $Q_i\subseteq Q_{r_i,r_i^2}(z_1)$ hold for $i=0,1,\cdots,j-1$, where
        $Q_A$ is the cylinder defined in \eqref{defomega}.
        In the case of $p$-phase, i.e., $a_0<10[a]_\alpha R^\alpha$, we have
        \begin{equation}\begin{split}\label{thetastep2lower}
        \Theta_A\geq A^{p-2}\omega^2\left[\left(\frac{\omega}{R}
\right)^p+10[a]_\alpha R^\alpha\left(\frac{\omega}{R}\right)^q\right]^{-1}\geq \hat\gamma_1A^{p-2}\omega^{2-p}R^p,
        \end{split}\end{equation}
        where $\hat \gamma_1=(1+10[a]_\alpha\|u\|_\infty^{q-p})^{-1}$.
        In view of \eqref{alpha1}, we
        see that $\alpha_{i-1}-\alpha_i\geq\frac{1}{B}4^{-i-100}\omega$ and hence $d_i\geq \frac{1}{B}4^{-i-100}\omega$.
        It follows that
        \begin{equation*}\begin{split}
        \Theta_i=d_i^2\left[\left(\frac{d_i}{r_i}\right)^p+a_1\left(\frac{d_i}{r_i}\right)^q\right]^{-1}\leq d_i^{2-p}r_i^p\leq B^{p-2}4^{100(p-2)}\omega^{2-p}R^p,
        \end{split}\end{equation*}
        since $r_i\leq 4^{-i}r_0\leq4^{-i}R$. Since $-\frac{8}{9}\Theta_A\leq \bar t\leq0$ and $\bar t-t_1\leq\Theta_\omega$, we have $t_1-\Theta_i\geq-\Theta_A$,
        provided that we choose
         \begin{equation}\begin{split}\label{Afirstassumption}
         A>4^{100}\left(18\hat\gamma_1^{-1}\right)^{\frac{1}{p-2}}B.
        \end{split}\end{equation}
        This implies that $Q_i\subseteq Q_A$ for the $p$-phase. In the case of $(p,q)$-phase, i.e., $a_0\geq10[a]_\alpha R^\alpha$, we have
        $\frac{4}{5}a_0\leq a_1\leq \frac{6}{5}a_0$ and hence
         \begin{equation*}\begin{split}
        \Theta_i&\leq \frac{5}{4}d_i^2\left[\left(\frac{d_i}{r_i}\right)^p+a_0\left(\frac{d_i}{r_i}\right)^q\right]^{-1}
       \\&\leq\frac{5}{4}\left(\frac{1}{B}4^{-i-100}\omega\right)^2\left[\left(\frac{4^{-i-100}\omega}{4^{-i}BR}\right)^p+a_0
        \left(\frac{4^{-i-100}\omega}{4^{-i}BR}\right)^q\right]^{-1}
        \leq \frac{5}{4}(4^{100}B)^{q-2}\Theta_\omega,
        \end{split}\end{equation*}
        since $d_i\geq \frac{1}{B}4^{-i-100}\omega$. Taking into account that $\Theta_A\geq A^{p-2}\Theta_\omega$, we have $t_1-\Theta_i\geq-\Theta_A$,
        provided that we choose
         \begin{equation}\begin{split}\label{Asecondassumption}
         A>25^{\frac{1}{p-2}}(4^{100}B)^\frac{q-2}{p-2}.
        \end{split}\end{equation}
         This leads to $Q_i\subseteq Q_A$ for the $(p,q)$-phase. On the other hand, we infer from \eqref{alpha1} and \eqref{li} that
          $d_i=l_i-l_{i+1}\geq \frac{1}{4}(\alpha_{i-1}-\alpha_i)\geq 50r_i.$ This leads to
         $\Theta_i\leq(50r_i)^2\left(50^p+a_150^q\right)^{-1}\leq 50^{2-p}r_i^2<r_i^2$ and hence $Q_i\subseteq Q_{r_i,r_i^2}(z_1)$.
       We now turn our attention to the proof of \eqref{Aj}.
        To this end,
       we deduce from \eqref{li} and \eqref{lj} that
        \begin{equation}\begin{split}\label{upper bound for l}
        d_j(\bar l)\geq\frac{1}{4}d_{j-1}+\frac{1}{16}B(\alpha_{j-1}-\alpha_j).
         \end{split}\end{equation}
         At this point, we claim that the inclusions $Q_j(\bar l)\subseteq Q_{j-1}$ and $Q_j(\bar l)\subseteq Q_{r_j,r_j^2}(z_1)$ hold. In view of \eqref{upper bound for l},
         we find that
          \begin{equation*}\begin{split}
        \Theta_j(\bar l)=d_j(\bar l)^2\left[\left(\frac{d_j(\bar l)}{r_i}\right)^p+a_1\left(\frac{d_j(\bar l)}{r_i}\right)^q\right]^{-1}\leq
       4^{-2} d_{j-1}^2\left[\left(\frac{d_{j-1}}{r_{j-1}}\right)^p+a_1\left(\frac{d_{j-1}}{r_{j-1}}\right)^q\right]^{-1}=\frac{1}{16}\Theta_{j-1},
        \end{split}\end{equation*}
        which proves the inclusion $Q_j(\bar l)\subseteq Q_{j-1}$. Moreover, we infer from $\alpha_{j-1}-\alpha_j\geq 200r_j$ and \eqref{upper bound for l} that
        $d_j(\bar l)\geq 12r_j$. This leads to
       $\Theta_j(\bar l)\leq d_j(\bar l)^{2-p}r_j^p<r_j^2$ and hence $Q_j(\bar l)\subseteq Q_{r_j,r_j^2}(z_1)$.
       In view of $\Theta_j(\bar l)\leq\frac{1}{16}\Theta_{j-1}$, we see that $\varphi_{j-1}(x,t)=1$ for $(x,t)\in Q_j(\bar l)$.
       Since $u\leq l_j$ on $L_j(\bar l)$, we infer from \eqref{Aj-1} that
        \begin{equation}\begin{split}\label{step1 initial estimate}
        &\left(\frac{d_j(\bar l)^{p-2}}{r_j^{n+p}}+a_1\frac{d_j(\bar l)^{q-2}}{r_j^{n+q}}\right)|L_j(\bar l)|\leq \frac{1}{r_j^n}\esssup_t\int_{L_j(t)}\varphi_{j-1}^k(\cdot,t)
        \,\mathrm {d}x
        \\&\leq \frac{4^n}{r_{j-1}^n}\esssup_t\int_{B_{j-1}}G\left(\frac{l_{j-1}-u}{d_{j-1}}\right)\varphi_{j-1}^k \,\mathrm {d}x
        \leq 4^nA_{j-1}(l_j)\leq 4^n\chi,
         \end{split}\end{equation}
         where $L_j(t)=\{x\in B_j:u(\cdot,t)\leq l_j\}$.
         For a fixed $\epsilon_1>0$, we decompose
   $L_j(\bar l)=L^\prime_j(\bar l)\cup L^{\prime\prime}_j(\bar l)$, where
   \begin{equation}\begin{split}\label{Ldecomposition}
   L^\prime_j(\bar l)=L_j(\bar l)\cap \left\{\frac{l_j-u}{d_j(\bar l)}\leq\epsilon_1\right\}\qquad\text{and}\qquad
   L^{\prime\prime}_j(\bar l)=L_j(\bar l)\setminus L^\prime_j(\bar l).
   \end{split}\end{equation}
   To proceed further, we shall again distinguish two cases: $p$-phase ($a_1<10[a]_\alpha r_j^\alpha$) and $(p,q)$-phase ($a_1\geq10[a]_\alpha r_j^\alpha$).

We first consider the case of the $(p,q)$-phase, i.e., $a_1\geq10[a]_\alpha r_j^\alpha$. Since $Q_j(\bar l)\subseteq Q_{r_j,r_j^2}(z_1)$, we
follow the arguments in Remark \ref{phase analysis 1} to conclude that $\frac{4}{5}a_1\leq a(x,t)\leq \frac{6}{5}a_1$ for all $(x,t)\in Q_j(\bar l)$.
Furthermore, we apply \eqref{step1 initial estimate} to obtain
   \begin{equation}\begin{split}\label{step1 second estimate}
   &\frac{d_j(\bar l)^{p-2}}{r_j^{n+p}}\iint_{L^\prime_j(\bar l)}\left(\frac{l_j-u}{d_j(\bar l)}\right)^{(1+\lambda)(p-1)}\varphi_j(\bar l)^{M-p}
 \,\mathrm {d}x\mathrm {d}t
 \\&+\frac{d_j(\bar l)^{q-2}}{r_j^{n+q}}\iint_{L^\prime_j(\bar l)}a(x,t)\left(\frac{l_j-u}{d_j(\bar l)}\right)^{(1+\lambda)(q-1)}\varphi_j(\bar l)^{M-q}
 \,\mathrm {d}x\mathrm {d}t
 \\&\leq \epsilon_1^{(p-1)(1+\lambda)}\left(\frac{d_j(\bar l)^{p-2}}{r_j^{n+p}}+a_1\frac{d_j(\bar l)^{q-2}}{r_j^{n+q}}\right)|L_j(\bar l)|
 \leq 4^n\epsilon_1^{(p-1)(1+\lambda)}\chi.
 \end{split}\end{equation}
For a fixed $\epsilon_2>0$, we claim that there exists a constant $\gamma>0$ depending only upon the data, $\epsilon_1$ and $\epsilon_2$ such that
  \begin{equation}\begin{split}\label{double prime claim}
 &\frac{d_j(\bar l)^{p-2}}{r_j^{n+p}}\iint_{L^{\prime\prime}_j(\bar l)}\left(\frac{l_j-u}{d_j(\bar l)}\right)^{(1+\lambda)(p-1)}\varphi_j(\bar l)^{M-p}
 \,\mathrm {d}x\mathrm {d}t
 \\&+\frac{d_j(\bar l)^{q-2}}{r_j^{n+q}}\iint_{L^{\prime\prime}_j(\bar l)}a(x,t)\left(\frac{l_j-u}{d_j(\bar l)}\right)^{(1+\lambda)(q-1)}\varphi_j(\bar l)^{M-q}
 \,\mathrm {d}x\mathrm {d}t
 \\&\leq
  4^n\epsilon_2\chi+\gamma\left(\chi+B^{-p}+B^{-(p-1)}\right)^{1+\frac{p}{n}}
  +\gamma\left(\chi+B^{-p}+B^{-(p-1)}\right)^{1+\frac{q}{n}}.
  \end{split}\end{equation}
 In order to prove \eqref{double prime claim}, we treat the two integrals on the left-hand side of \eqref{double prime claim} separately.
To deal with the integral for the $p$-Laplace type, we first note that $2<p<n$.
 For the fixed $\epsilon_2>0$, we infer from Young's inequality and \eqref{step1 initial estimate}
 to obtain
  \begin{equation}\begin{split}\label{step1 second estimate p}
   &\frac{d_j(\bar l)^{p-2}}{r_j^{n+p}}\iint_{L^{\prime\prime}_j(\bar l)}\left(\frac{l_j-u}{d_j(\bar l)}\right)^{(1+\lambda)(p-1)}\varphi_j(\bar l)^{M-p}
 \,\mathrm {d}x\mathrm {d}t
 \\&\leq \epsilon_2\frac{d_j(\bar l)^{p-2}}{r_j^{n+p}}|L^{\prime\prime}_j(\bar l)|+c(\epsilon_2)
 \frac{d_j(\bar l)^{p-2}}{r_j^{n+p}}\iint_{L^{\prime\prime}_j(\bar l)}\left(\frac{l_j-u}{d_j(\bar l)}\right)^{(1+\lambda)(p-1)p_1}\varphi_j(\bar l)^{(M-p)p_1}
 \,\mathrm {d}x\mathrm {d}t
  \\&\leq 4^n\epsilon_2\chi+c(\epsilon_2)
 \frac{d_j(\bar l)^{p-2}}{r_j^{n+p}}\iint_{L^{\prime\prime}_j(\bar l)}\left(\frac{l_j-u}{d_j(\bar l)}\right)^{p\frac{n+h_1}{nh_1}}\varphi_j(\bar l)^{(M-p)p_1}
 \,\mathrm {d}x\mathrm {d}t,
 \end{split}\end{equation}
 where $M>q$, $h_1=\tfrac{p}{p-1-\lambda}>1$ and
$p_1=p\tfrac{n+h_1}{nh_1(1+\lambda)(p-1)}=\tfrac{p-1-
\lambda+\frac{p}{n}}{p-1+\lambda p-\lambda}>1,$
since $\lambda=\frac{p}{nq}<\frac{1}{n}$. At this stage, we use Lemma \ref{lemmainequalitypsi-} to estimate the integral on the right-hand side of \eqref{step1 second estimate p} by
\begin{equation}\begin{split}\label{step1secondestimatepvp}
&\frac{d_j(\bar l)^{p-2}}{r_j^{n+p}}\iint_{L^{\prime\prime}_j(\bar l)}\left(\frac{l_j-u}{d_j(\bar l)}\right)^{p\frac{n+h_1}{nh_1}}\varphi_j(\bar l)^{(M-p)p_1}
 \,\mathrm {d}x\mathrm {d}t
 \\&\leq \gamma\frac{d_j(\bar l)^{p-2}}{r_j^{n+p}}\iint_{L^{\prime\prime}_j(\bar l)} \psi_p(x,t)^{p\frac{n+h_1}{n}}
 \varphi_j(\bar l)^{(M-p)p_1}
 \,\mathrm {d}x\mathrm {d}t
 \\&=\gamma\frac{d_j(\bar l)^{p-2}}{r_j^{n+p}}\iint_{L^{\prime\prime}_j(\bar l)}v_p^{p\frac{n+h_1}{n}} \,\mathrm {d}x\mathrm {d}t,
 \end{split}\end{equation}
 where $v_p=\psi_p\varphi_j(\bar l)^{k_1}$, $k_1=\frac{(M-p)np_1}{p(n+h_1)}$ and $\psi_p$ is defined in \eqref{psi-} with $\nu=p$. To proceed further, we apply Sobolev's inequality
 and Lemma \ref{lemmainequalitypsi-}
 to deduce
 \begin{equation}\begin{split}\label{T22}
   \iint_{L_j^{\prime\prime}(\bar l)}&
v_p^{p+\frac{ph_1}{n}} \,\mathrm {d}x\mathrm {d}t
\leq \int_{t_1- \Theta_j(\bar l)}^{t_1}\left(\int_{B_j}
v_p^{\frac{np}{n-p}} \,\mathrm {d}x\right)^{\frac{n-p}{n}}
\left(\int_{L_j^{\prime\prime}(t)}
v_p^{h_1} \,\mathrm {d}x\right)^{\frac{p}{n}}
\mathrm {d}t
\\&\leq  \gamma\esssup_t\left(\int_{L_j^{\prime\prime}(t)}
\frac{l_j-u}{d_j(\bar l)}\varphi_j(\bar l)^{k_1h_1} \,\mathrm {d}x\right)^{\frac{p}{n}}\iint_{Q_j(\bar l)}
|Dv_p|^p \,\mathrm {d}x\mathrm {d}t,
   \end{split}\end{equation}
   where
 \begin{equation*}\begin{split}L_j^{\prime\prime}(t)=\{x\in B_j:u(\cdot,t)\leq l_j\}\cap \left\{x\in B_j:\frac{l_j-u(\cdot,t)}{d_j(\bar l)}>
   \epsilon_1\right\}.\end{split}\end{equation*}
     At this point, we apply \cite[Lemma 3.4]{Qifanli} and Lemma \ref{Cac1}
    with $(\rho,l,d,\theta)$ replaced by  $(r_j,l_j,d_j(\bar l),\Theta_j(\bar l))$
    to conclude that for any $t\in(t_1-\Theta_j(\bar l),t_1)$ there holds
\begin{equation}\begin{split}\label{T23}
 \frac{1}{r_j^n} \int_{L_j^{\prime\prime}(t)}&
\frac{l_j- u}{d_j(\bar l)}\varphi_j(\bar l)^{k_1h_1} \,\mathrm {d}x
\leq c(\epsilon_1)\frac{1}{r_j^n}
\int_{B_j}
G\left(\frac{l_j-u}{d_j(\bar l)} \right)\varphi_j(\bar l)^{k_1h_1}\,\mathrm {d}x
  \\
 \leq &\gamma \frac{1}{r_j^n}\iint_{L_j(\bar l)}\frac{l_j-u}{d_j(\bar l)}\left|\partial_t\varphi_j(\bar l)\right|\varphi_j(\bar l)^{k_1h_1-1}\,\mathrm {d}x\mathrm {d}t
  \\&+\gamma  \frac{d_j(\bar l)^{p-2}}{r_j^n}\iint_{L_j(\bar l)}\left(\frac{l_j-u}{d_j(\bar l)}\right)^{(1+\lambda)(p-1)}
  \varphi_j(\bar l)^{k_1h_1-p}|D\varphi_j(\bar l)|^p\,\mathrm {d}x\mathrm {d}t
  \\&+\gamma  \frac{d_j(\bar l)^{q-2}}{r_j^n}\iint_{L_j(\bar l)}a(x,t)\left(\frac{l_j-u}{d_j(\bar l)}\right)^{(1+\lambda)(q-1)}
  \varphi_j(\bar l)^{k_1h_1-q}|D\varphi_j(\bar l)|^q\,\mathrm {d}x\mathrm {d}t
\\&+\gamma \frac{\Theta_j(\bar l)}{r_j^nd_j(\bar l)^2}\int_{B_j}g^{\frac{p}{p-1}}\,\mathrm {d}x
  +\gamma \frac{\Theta_j(\bar l)}{r_j^nd_j(\bar l)}\int_{B_j}|f|\,\mathrm {d}x
  \\ =:&T_1+T_2+T_3+T_4+T_5,
    \end{split}\end{equation}
    since $\varphi_j(\bar l)=0$ on $\partial_PQ_j(\bar l)$.
We first consider the estimates for $T_2$ and $T_3$. Recalling that
$|D\varphi_j(\bar l)|\leq \gamma r_j^{-1}$, $Q_j(\bar l)\subseteq Q_{j-1}$,
$\varphi_{j-1}\equiv 1$ on $\supp \varphi_j(\bar l)$ and $d_j(\bar l)\geq\frac{1}{4}d_{j-1}$, we find that
\begin{equation}\begin{split}\label{T2T3}
T_2+T_3\leq &\gamma  \frac{d_{j-1}^{p-2}}{r_{j-1}^{n+p}}\iint_{L_{j-1}}\left(\frac{l_{j-1}-u}{d_{j-1}}\right)^{(1+\lambda)(p-1)}
  \varphi_{j-1}^{M-p}\,\mathrm {d}x\mathrm {d}t
  \\&+\gamma  \frac{d_{j-1}^{q-2}}{r_{j-1}^{n+q}}\iint_{L_{j-1}}a(x,t)\left(\frac{l_{j-1}-u}{d_{j-1}}\right)^{(1+\lambda)(q-1)}
  \varphi_{j-1}^{M-q}\,\mathrm {d}x\mathrm {d}t
  \\ \leq &\gamma A_{j-1}(l_j)\leq\gamma\chi.
\end{split}\end{equation}
We now turn our attention to the estimate of $T_1$. Taking into account that $p>2$, $q>2$,
$|\partial_t\varphi_j(\bar l)|\leq 9\Theta_j(\bar l)^{-1}$ and
$\frac{4}{5}a_1\leq a(x,t)\leq \frac{6}{5}a_1$
for all $(x,t)\in Q_j(\bar l)$,
we apply Young's inequality, \eqref{step1 initial estimate} and \eqref{T2T3} to conclude that
\begin{equation}\begin{split}\label{T1T2T3}
T_1\leq& \gamma \iint_{L_j(\bar l)}\left(\frac{l_j-u}{d_j(\bar l)}\right)\left(\frac{d_j(\bar l)^{p-2}}{r_j^{p+n}}+a_1
\frac{d_j(\bar l)^{q-2}}{r_j^{q+n}}\right)\,\mathrm {d}x\mathrm {d}t
\\ \leq &\left(\frac{d_j(\bar l)^{p-2}}{r_j^{p+n}}+a_1
\frac{d_j(\bar l)^{q-2}}{r_j^{q+n}}\right)|L_j(\bar l)|
 \\&+\gamma  \frac{d_j(\bar l)^{p-2}}{r_j^{n+p}}\iint_{L_j(\bar l)}\left(\frac{l_j-u}{d_j(\bar l)}\right)^{(1+\lambda)(p-1)}
 \,\mathrm {d}x\mathrm {d}t
  \\&+\gamma  a_1\frac{d_j(\bar l)^{q-2}}{r_j^{n+q}}\iint_{L_j(\bar l)}\left(\frac{l_j-u}{d_j(\bar l)}\right)^{(1+\lambda)(q-1)}
  \,\mathrm {d}x\mathrm {d}t
  \\ \leq &4^n\chi+\gamma  \frac{d_{j-1}^{p-2}}{r_{j-1}^{n+p}}\iint_{L_{j-1}}\left(\frac{l_{j-1}-u}{d_{j-1}}\right)^{(1+\lambda)(p-1)}
  \varphi_{j-1}^{M-p}\,\mathrm {d}x\mathrm {d}t
  \\&+\gamma  \frac{d_{j-1}^{q-2}}{r_{j-1}^{n+q}}\iint_{L_{j-1}}a(x,t)\left(\frac{l_{j-1}-u}{d_{j-1}}\right)^{(1+\lambda)(q-1)}
  \varphi_{j-1}^{M-q}\,\mathrm {d}x\mathrm {d}t
  \leq \gamma\chi.
\end{split}\end{equation}
To estimate $T_4$ and $T_5$, we
observe that $\Theta_j(\bar l)\leq d_j(\bar l)^{2-p}r_j^p$. Then, we infer from $d_0(\bar l)\geq \frac{1}{16}B(\alpha_{j-1}-\alpha_j)$ and \eqref{alpha1} that
there exists a constant $\gamma$ depending only upon the data such that
 \begin{equation}\begin{split}\label{T4}
 T_4\leq \gamma\frac{d_j(\bar l)^{-p}}{r_j^{n-p}}\int_{B_j}g^\frac{p}{p-1}\,\mathrm {d}x\leq \gamma B^{-p}
 \frac{(\alpha_{j-1}-\alpha_j)^{-p}}{r_j^{n-p}}\int_{B_j}g^\frac{p}{p-1}\,\mathrm {d}x\leq
 \gamma \frac{1}{B^p}
  \end{split}\end{equation}
  and
  \begin{equation}\begin{split}\label{T5}
 T_5\leq \gamma\frac{d_j(\bar l)^{1-p}}{r_j^{n-p}}\int_{B_j}|f|\,\mathrm {d}x\leq \gamma B^{-(p-1)}
 \frac{(\alpha_{j-1}-\alpha_j)^{1-p}}{r_j^{n-p}}\int_{B_j}|f|\,\mathrm {d}x\leq
 \gamma \frac{1}{B^{p-1}}.
  \end{split}\end{equation}
 Combining the estimates \eqref{T23}-\eqref{T5}, we obtain the estimate
  \begin{equation}\begin{split}\label{T1-T5}
   \frac{1}{r_j^n} \int_{L_j^{\prime\prime}(t)}
\frac{l_j- u}{d_j(\bar l)}\varphi_j(\bar l)^{k_1h_1} \,\mathrm {d}x\leq \gamma\chi+\gamma(B^{-p}+B^{-(p-1)}),
    \end{split}\end{equation}
   holds for all $t\in (t_1-\Theta_j(\bar l),t_1)$. Next, we proceed to treat the term involving the gradient. To this end, we decompose
    \begin{equation}\begin{split}\label{Dpsi}
    \frac{d_j(\bar l)^{p-2}}{r_j^n}&\iint_{Q_j(\bar l)}
|Dv_p|^p \,\mathrm {d}x\mathrm {d}t\leq  \frac{d_j(\bar l)^{p-2}}{r_j^n}\iint_{Q_j(\bar l)}|D\psi_p|^p\varphi_j(\bar l)^{k_1p}\,\mathrm {d}x\mathrm {d}t
\\& +k_1^p\frac{d_j(\bar l)^{p-2}}{r_j^n}\iint_{Q_j(\bar l)}\varphi_j(\bar l)^{(k_1-1)p}|D\varphi_j(\bar l)|^p\psi_p(x,t)^p\,\mathrm {d}x\mathrm {d}t
\\&=U_1+U_2,
    \end{split}\end{equation}
    with the obvious meaning of $U_1$ and $U_2$. To estimate $U_1$, we use the Caccioppoli estimate \eqref{Cacformula1} with $(\rho,l,d,\theta)$ replaced by
    $(r_j,l_j,d_j(\bar l),\Theta_j(\bar l))$.
    It follows from the estimates \eqref{T23}-\eqref{T5} that
    \begin{equation}\begin{split}\label{DpsiU1}
    U_1\leq \gamma\chi+\gamma(B^{-p}+B^{-(p-1)}),
    \end{split}\end{equation}
  where the constant $\gamma$ depends only upon the data, $\epsilon_1$ and $\epsilon_2$. Finally, we come to the estimate of $U_2$.
To this end, we observe that $\psi_p\leq \gamma d_j(\bar l)^{-\frac{1}{p^\prime}}(l_j-u)_+^{\frac{1}{p^\prime}}$, where $p^\prime=\tfrac{p}{p-1}$.
We now use $|D\varphi_j(\bar l)|\leq \gamma r_j^{-1}$, Young's inequality, \eqref{step1 initial estimate} and \eqref{T2T3} to get
 \begin{equation}\begin{split}\label{DpsiU2}
    U_2\leq &\gamma\frac{d_j(\bar l)^{p-2}}{r_j^{n+p}}\iint_{L_j(\bar l)}\left(\frac{l_j-u}{d_j(\bar l)}\right)^{p-1}\varphi_j(\bar l)^{(k_1-1)p}\,\mathrm {d}x\mathrm {d}t
    \\ \leq &\gamma\frac{d_j(\bar l)^{p-2}}{r_j^{n+p}}|L_j(\bar l)|
    \\&+\gamma\frac{d_j(\bar l)^{p-2}}{r_j^{n+p}}
    \iint_{L_j(\bar l)}\left(\frac{l_j-u}{d_j(\bar l)}\right)^{(p-1)(1+\lambda)}\varphi_j(\bar l)^{(k_1-1)p}\,\mathrm {d}x\mathrm {d}t
    \\ \leq&\gamma\chi.
    \end{split}\end{equation}
 Inserting the estimates \eqref{DpsiU1} and \eqref{DpsiU2} into \eqref{Dpsi}, we conclude that
  \begin{equation}\begin{split}\label{Dpsiresult}
   \frac{d_j(\bar l)^{p-2}}{r_j^n}\iint_{Q_j(\bar l)}
|Dv_p|^p \,\mathrm {d}x\mathrm {d}t\leq \gamma\chi+\gamma(B^{-p}+B^{-(p-1)}).
    \end{split}\end{equation}
    Combining the inequalities \eqref{step1 second estimate p}, \eqref{step1secondestimatepvp}, \eqref{T22}, \eqref{T1-T5} and \eqref{Dpsiresult}, we arrive at
 \begin{equation}\begin{split}\label{estimate for pintegral}
 \frac{d_j(\bar l)^{p-2}}{r_j^{n+p}}&\iint_{L^{\prime\prime}_j(\bar l)}\left(\frac{l_j-u}{d_j(\bar l)}\right)^{(1+\lambda)(p-1)}\varphi_j(\bar l)^{M-p}
 \,\mathrm {d}x\mathrm {d}t\\&\leq  4^n\epsilon_2\chi+\gamma\left(\chi+B^{-p}+B^{-(p-1)}\right)^{1+\frac{p}{n}},
  \end{split}\end{equation}
  where the constant $\gamma$ depends only upon the data, $\epsilon_1$ and $\epsilon_2$.
In order to treat the integral for the $q$-Laplace type, we distinguish between two cases $q<n$ and $q\geq n$.
In the case $q<n$, we use Young's inequality, \eqref{step1 initial estimate} and $\lambda<\frac{1}{n}$ to deduce an estimate similar to \eqref{step1 second estimate p},
  \begin{equation}\begin{split}\label{step1secondestimate q<n}
   &\frac{d_j(\bar l)^{q-2}}{r_j^{n+q}}\iint_{L^{\prime\prime}_j(\bar l)}a(x,t)\left(\frac{l_j-u}{d_j(\bar l)}\right)^{(1+\lambda)(q-1)}\varphi_j(\bar l)^{M-q}
 \,\mathrm {d}x\mathrm {d}t
 \\ \leq &\epsilon_2a_1\frac{d_j(\bar l)^{q-2}}{r_j^{n+q}}|L^{\prime\prime}_j(\bar l)|
 \\&+c(\epsilon_2)
 \frac{d_j(\bar l)^{q-2}}{r_j^{n+q}}\iint_{L^{\prime\prime}_j(\bar l)}a(x,t)\left(\frac{l_j-u}{d_j(\bar l)}\right)^{(1+\lambda)(q-1)q_1}\varphi_j(\bar l)^{(M-q)q_1}
 \,\mathrm {d}x\mathrm {d}t
  \\ \leq& 4^n\epsilon_2\chi+c(\epsilon_2)
 \frac{d_j(\bar l)^{q-2}}{r_j^{n+q}}\iint_{L^{\prime\prime}_j(\bar l)}a(x,t)\left(\frac{l_j-u}{d_j(\bar l)}\right)^{q\frac{n+h_2}{nh_2}}\varphi_j(\bar l)^{(M-q)q_1}
 \,\mathrm {d}x\mathrm {d}t,
 \end{split}\end{equation}
 where $M>q$,
$h_2=\tfrac{q}{q-1-\lambda}>1$ and $q_1=q\tfrac{n+h_2}{nh_2(1+\lambda)(q-1)}=\tfrac{q-1-
\lambda+\frac{q}{n}}{q-1+\lambda q-\lambda}>1.$ Taking into account that
$\frac{4}{5}a_1\leq a(x,t)\leq \frac{6}{5}a_1$
for all $(x,t)\in Q_j(\bar l)$, we use Lemma \ref{lemmainequalitypsi-} to estimate the integral on the right-hand side of \eqref{step1secondestimate q<n} by
\begin{equation}\begin{split}\label{step1secondestimatepvq}
 \frac{d_j(\bar l)^{q-2}}{r_j^{n+q}}&\iint_{L^{\prime\prime}_j(\bar l)}a(x,t)\left(\frac{l_j-u}{d_j(\bar l)}\right)^{q\frac{n+h_2}{nh_2}}\varphi_j(\bar l)^{(M-q)q_1}
 \,\mathrm {d}x\mathrm {d}t
 \\&\leq \gamma a_1\frac{d_j(\bar l)^{q-2}}{r_j^{n+q}}\iint_{L^{\prime\prime}_j(\bar l)}v_q^{q\frac{n+h_2}{n}} \,\mathrm {d}x\mathrm {d}t,
 \end{split}\end{equation}
 where $v_q=\psi_q\varphi_j(\bar l)^{k_2}$, $k_2=\frac{(M-q)nq_1}{q(n+h_2)}$ and $\psi_q$ is defined in \eqref{psi-} with $\nu=q$. Next, we use
 H\"older's inequality,
 Sobolev embedding theorem
 and Lemma \ref{lemmainequalitypsi-}
 to deduce
 \begin{equation}\begin{split}\label{T22q}
   \iint_{L_j^{\prime\prime}(\bar l)}&
v_q^{q+\frac{qh_2}{n}} \,\mathrm {d}x\mathrm {d}t
\leq \int_{t_1- \Theta_j(\bar l)}^{t_1}\left(\int_{B_j}
v_q^{\frac{nq}{n-q}} \,\mathrm {d}x\right)^{\frac{n-q}{n}}
\left(\int_{L_j^{\prime\prime}(t)}
v_q^{h_2} \,\mathrm {d}x\right)^{\frac{q}{n}}
\mathrm {d}t
\\&\leq  \gamma\esssup_t\left(\int_{L_j^{\prime\prime}(t)}
\frac{l_j-u}{d_j(\bar l)}\varphi_j(\bar l)^{k_2h_2} \,\mathrm {d}x\right)^{\frac{q}{n}}\iint_{Q_j(\bar l)}
|Dv_q|^q \,\mathrm {d}x\mathrm {d}t,
   \end{split}\end{equation}
since $q<n$. We now proceed along the lines of the proof of \eqref{T1-T5}
to conclude that the inequality
  \begin{equation}\begin{split}\label{T1-T5qlaplace}
   \frac{1}{r_j^n} \int_{L_j^{\prime\prime}(t)}
\frac{l_j- u}{d_j(\bar l)}\varphi_j(\bar l)^{k_2h_2} \,\mathrm {d}x\leq \gamma\chi+\gamma(B^{-p}+B^{-(p-1)}),
    \end{split}\end{equation}
   holds for all $t\in (t_1-\Theta_j(\bar l),t_1)$.
  To deal with the term involving the gradient of $v_q$. we decompose
    \begin{equation}\begin{split}\label{Dpsiq}
    a_1\frac{d_j(\bar l)^{q-2}}{r_j^n}&\iint_{Q_j(\bar l)}
|Dv_q|^q \,\mathrm {d}x\mathrm {d}t\leq  \frac{d_j(\bar l)^{q-2}}{r_j^n}\iint_{Q_j(\bar l)}a(x,t)|D\psi_q|^q\varphi_j(\bar l)^{k_2q}\,\mathrm {d}x\mathrm {d}t
\\& +\gamma a_1\frac{d_j(\bar l)^{q-2}}{r_j^n}\iint_{Q_j(\bar l)}\varphi_j(\bar l)^{(k_2-1)q}|D\varphi_j(\bar l)|^q\psi_p(x,t)^q\,\mathrm {d}x\mathrm {d}t
\\&=U_1^\prime+U_2^\prime,
    \end{split}\end{equation}
   since $\frac{4}{5}a_1\leq a(x,t)\leq \frac{6}{5}a_1$
for all $(x,t)\in Q_j(\bar l)$.
 To estimate $U_1^\prime$, we use the Caccioppoli estimate \eqref{Cacformula1} with $(\rho,l,d,\theta)$ replaced by
    $(r_j,l_j,d_j(\bar l),\Theta_j(\bar l))$. This together with the estimates \eqref{T23}-\eqref{T5}, we get
    \begin{equation}\begin{split}\label{DpsiU1q}
    U_1^\prime\leq \gamma\chi+\gamma(B^{-p}+B^{-(p-1)}).
    \end{split}\end{equation}
To estimate $U_2^\prime$, we observe that $\psi_q\leq \gamma d_j(\bar l)^{-\frac{1}{q^\prime}}(l_j-u)_+^{\frac{1}{q^\prime}}$, where $q^\prime=\tfrac{q}{q-1}$.
We now apply $|D\varphi_j(\bar l)|\leq \gamma r_j^{-1}$, Young's inequality, \eqref{step1 initial estimate} and \eqref{T2T3} to deduce
 \begin{equation}\begin{split}\label{DpsiU2q}
    U_2^\prime\leq &\gamma a_1\frac{d_j(\bar l)^{q-2}}{r_j^{n+q}}\iint_{L_j(\bar l)}\left(\frac{l_j-u}{d_j(\bar l)}\right)^{q-1}\varphi_j(\bar l)^{(k_2-1)q}\,\mathrm {d}x\mathrm {d}t
    \\ \leq &\gamma a_1\frac{d_j(\bar l)^{q-2}}{r_j^{n+q}}|L_j(\bar l)|
    \\&+\gamma \frac{d_j(\bar l)^{q-2}}{r_j^{n+q}}
    \iint_{L_j(\bar l)}a(x,t)\left(\frac{l_j-u}{d_j(\bar l)}\right)^{(q-1)(1+\lambda)}\varphi_j(\bar l)^{(k_2-1)q}\,\mathrm {d}x\mathrm {d}t
    \\ \leq&\gamma\chi,
    \end{split}\end{equation}
     since $\frac{4}{5}a_1\leq a(x,t)\leq \frac{6}{5}a_1$
for all $(x,t)\in Q_j(\bar l)$.
 Inserting the estimates \eqref{DpsiU1q} and \eqref{DpsiU2q} into \eqref{Dpsiq}, we infer that
  \begin{equation}\begin{split}\label{Dpsiresultq}
   a_1\frac{d_j(\bar l)^{q-2}}{r_j^n}\iint_{Q_j(\bar l)}
|Dv_q|^q \,\mathrm {d}x\mathrm {d}t\leq \gamma\chi+\gamma(B^{-p}+B^{-(p-1)}).
    \end{split}\end{equation}
    Combining the inequalities \eqref{step1secondestimate q<n}, \eqref{step1secondestimatepvq}, \eqref{T22q}, \eqref{T1-T5qlaplace} and \eqref{Dpsiresultq}, we arrive at
 \begin{equation}\begin{split}\label{estimate for q<nintegral}
 \frac{d_j(\bar l)^{q-2}}{r_j^{n+q}}&\iint_{L^{\prime\prime}_j(\bar l)}a(x,t)\left(\frac{l_j-u}{d_j(\bar l)}\right)^{(1+\lambda)(q-1)}\varphi_j(\bar l)^{M-q}
 \,\mathrm {d}x\mathrm {d}t\\&\leq  4^n\epsilon_2\chi+\gamma\left(\chi+B^{-p}+B^{-(p-1)}\right)^{1+\frac{q}{n}},
  \end{split}\end{equation}
  where the constant $\gamma$ depends only upon the data, $\epsilon_1$ and $\epsilon_2$.
In the case $q\geq n$, we apply Lemma \ref{lemmainequalitypsi-} and $\lambda q=\frac{p}{n}$ to deduce that
 \begin{equation}\begin{split}\label{step1secondestimate q>n}
   &\frac{d_j(\bar l)^{q-2}}{r_j^{n+q}}\iint_{L^{\prime\prime}_j(\bar l)}a(x,t)\left(\frac{l_j-u}{d_j(\bar l)}\right)^{(1+\lambda)(q-1)}\varphi_j(\bar l)^{M-q}
 \,\mathrm {d}x\mathrm {d}t
 \\ \leq &\frac{6a_1}{5} \frac{d_j(\bar l)^{q-2}}{r_j^{n+q}}\iint_{L^{\prime\prime}_j(\bar l)}\left(\frac{l_j-u}{d_j(\bar l)}\right)^{q-1-\lambda+\lambda q}\varphi_j(\bar l)^{M-q}
 \,\mathrm {d}x\mathrm {d}t
  \\ \leq &\frac{6a_1}{5} \frac{d_j(\bar l)^{q-2}}{r_j^{n+q}}\iint_{L^{\prime\prime}_j(\bar l)}\left(\frac{l_j-u}{d_j(\bar l)}\right)^{\frac{p}{n}}\psi_q(x,t)^q\varphi_j(\bar l)^{M-q}
 \,\mathrm {d}x\mathrm {d}t,
  \end{split}\end{equation}
  where $\psi_q$ is defined in \eqref{psi-} with $\nu=q$. Noting that $q\geq n$, we use Sobolev's inequality slice-wise (see for instance \cite[page 51, Corollary 3.7]{KLV}) to conclude that
  for all $t\in(t_1-\Theta_j(\bar l), t_1)$ there holds
   \begin{equation}\begin{split}\label{Sobolevq>n}
   \left(\int_{B_j}\left(\psi_q\varphi_j(\bar l)^s\right)(\cdot,t)^{\delta q}\,\mathrm {d}x\right)^\frac{1}{\delta q}
   \leq \gamma r_j^{1-\frac{n}{q}+\frac{n}{\delta q}} \left(\int_{B_j}|D(\psi_q\varphi_j(\bar l)^s)(\cdot,t)|^{ q}\,\mathrm {d}x\right)^\frac{1}{ q},
   \end{split}\end{equation}
   where $\delta=\frac{n}{n-p}>1$ and $s=\frac{\frac{1}{2}M-q}{q}>1$. We now proceed to estimate \eqref{step1secondestimate q>n}. To this end, we use H\"older's inequality with
   $\delta=\frac{n}{n-p}$ and $\delta^\prime=\frac{n}{p}$ to obtain
    \begin{equation}\begin{split}\label{step1secondestimateq>nq}
&a_1\frac{d_j(\bar l)^{q-2}}{r_j^{n+q}}\iint_{L^{\prime\prime}_j(\bar l)}\left(\frac{l_j-u}{d_j(\bar l)}\right)^{\frac{p}{n}}\psi_q(x,t)^q\varphi_j(\bar l)^{M-q}
 \,\mathrm {d}x\mathrm {d}t
 \\=&a_1\frac{d_j(\bar l)^{q-2}}{r_j^{n+q}}\iint_{L^{\prime\prime}_j(\bar l)}\left[\psi_q\varphi(\bar l)^s\right]^q
\left[\left(\frac{l_j-u}{d_j(\bar l)}\right)^{\frac{p}{n}}\varphi_j(\bar l)^{\frac{M}{2}}\right]
 \,\mathrm {d}x\mathrm {d}t
 \\ \leq& a_1\frac{d_j(\bar l)^{q-2}}{r_j^{n+q}}\int_{t_1-\Theta_j(\bar l)}^{t_1}
 \left(\int_{B_j}\left(\psi_q\varphi_j(\bar l)^s\right)(\cdot,t)^{\delta q}\,\mathrm {d}x\right)^\frac{1}{\delta }
 \\&\times\left(\int_{L^{\prime\prime}_j(t)}\left(\frac{l_j-u}{d_j(\bar l)}\right)\varphi_j(\bar l)
 ^{\frac{M}{2}\frac{n}{p}}\,\mathrm {d}x\right)^\frac{p}{n}\,\mathrm {d}t
 \\ \leq & a_1\frac{d_j(\bar l)^{q-2}}{r_j^{n+q}}r_j^{q-n+\frac{n}{\delta}}
 \esssup_t\left(\int_{L_j^{\prime\prime}(t)}
\frac{l_j-u}{d_j(\bar l)}\varphi_j(\bar l)^{\frac{Mn}{2p}} \,\mathrm {d}x\right)^{\frac{p}{n}}
\\&\times \int_{t_1-\Theta_j(\bar l)}^{t_1}\int_{B_j}|D(\psi_q\varphi_j(\bar l)^s)|^q \,\mathrm {d}x\mathrm {d}t.
  \end{split}\end{equation}
  In the last line, we have used the Sobolev's inequality \eqref{Sobolevq>n}. Since $\delta=\frac{n}{n-p}$, we observe that
  $r_j^{q-n+\frac{n}{\delta}-n-q}=r_j^{-p-n}$ and the inequality \eqref{step1secondestimateq>nq} reads as follows:
   \begin{equation}\begin{split}\label{step1secondestimateq>nq1}
&a_1\frac{d_j(\bar l)^{q-2}}{r_j^{n+q}}\iint_{L^{\prime\prime}_j(\bar l)}\left(\frac{l_j-u}{d_j(\bar l)}\right)^{\frac{p}{n}}\psi_q(x,t)^q\varphi_j(\bar l)^{M-q}
 \,\mathrm {d}x\mathrm {d}t
 \\ \leq & \gamma
 \esssup_t\left(\frac{1}{r_j^n}\int_{L_j^{\prime\prime}(t)}
\frac{l_j-u}{d_j(\bar l)}\varphi_j(\bar l)^{\frac{Mn}{2p}} \,\mathrm {d}x\right)^{\frac{p}{n}}
\times\left[ a_1\frac{d_j(\bar l)^{q-2}}{r_j^{n}}\iint_{Q_j(\bar  l)}|D\tilde v_q|^q \,\mathrm {d}x\mathrm {d}t\right],
  \end{split}\end{equation}
  where $\tilde v_q=\psi_q\varphi_j(\bar l)^s$. We now proceed along the lines of the proofs of \eqref{T1-T5qlaplace} and \eqref{Dpsiresultq}. This yields that
  \begin{equation}\begin{split}\label{T1-T5qlaplaceq>n}
  \esssup_t \frac{1}{r_j^n} \int_{L_j^{\prime\prime}(t)}
\frac{l_j- u}{d_j(\bar l)}\varphi_j(\bar l)^{\frac{Mn}{2p}} \,\mathrm {d}x\leq \gamma\chi+\gamma(B^{-p}+B^{-(p-1)})
    \end{split}\end{equation}
  and
   \begin{equation}\begin{split}\label{Dpsiresultqq>n}
   a_1\frac{d_j(\bar l)^{q-2}}{r_j^n}\iint_{Q_j(\bar l)}
|D\tilde v_q|^q \,\mathrm {d}x\mathrm {d}t\leq \gamma\chi+\gamma(B^{-p}+B^{-(p-1)}).
    \end{split}\end{equation}
    Combining \eqref{step1secondestimate q>n}, \eqref{step1secondestimateq>nq1}, \eqref{T1-T5qlaplaceq>n} and \eqref{Dpsiresultqq>n}, we infer that in the case $q\geq n$ the estimate
    \begin{equation}\begin{split}\label{estimate for q>nintegral}
 \frac{d_j(\bar l)^{q-2}}{r_j^{n+q}}&\iint_{L^{\prime\prime}_j(\bar l)}a(x,t)\left(\frac{l_j-u}{d_j(\bar l)}\right)^{(1+\lambda)(q-1)}\varphi_j(\bar l)^{M-q}
 \,\mathrm {d}x\mathrm {d}t\\&\leq  \gamma\left(\chi+B^{-p}+B^{-(p-1)}\right)^{1+\frac{p}{n}}
  \end{split}\end{equation}
  holds true for a constant $\gamma=\gamma(\text{data},\epsilon_1)$. At this stage, we conclude from \eqref{estimate for pintegral}, \eqref{estimate for q<nintegral}
  and \eqref{estimate for q>nintegral} that the inequality \eqref{double prime claim} holds for the case of $(p,q)$-phase.
  We are left with the task of proving estimates similar to \eqref{step1 second estimate} and \eqref{double prime claim} for the case of $p$-phase.

We now turn our attention to the case of $p$-phase, i.e., $a_1<10[a]_\alpha r_j^\alpha$. Recalling that $Q_j(\bar l)\subseteq Q_{r_j,r_j^2}(z_1)$,
we get $a(x,t)\leq 12[a]_\alpha r_j^\alpha$ for all $(x,t)\in Q_j(\bar l)$.
Noting that $d_j(\bar l)=l_j-\bar l\leq l_j\leq l_0=\frac{1}{4}\omega\leq \frac{1}{2}\|u\|_\infty$ and $p<q\leq p+\alpha$, we use \eqref{step1 initial estimate} to infer that
   \begin{equation}\begin{split}\label{step1 second estimate p-phase}
   &\frac{d_j(\bar l)^{p-2}}{r_j^{n+p}}\iint_{L^\prime_j(\bar l)}\left(\frac{l_j-u}{d_j(\bar l)}\right)^{(1+\lambda)(p-1)}\varphi_j(\bar l)^{M-p}
 \,\mathrm {d}x\mathrm {d}t
 \\&+\frac{d_j(\bar l)^{q-2}}{r_j^{n+q}}\iint_{L^\prime_j(\bar l)}a(x,t)\left(\frac{l_j-u}{d_j(\bar l)}\right)^{(1+\lambda)(q-1)}\varphi_j(\bar l)^{M-q}
 \,\mathrm {d}x\mathrm {d}t
 \\&\leq \epsilon_1^{(p-1)(1+\lambda)}\frac{d_j(\bar l)^{p-2}}{r_j^{n+p}}|L_j(\bar l)|
 \\ &\quad+12[a]_\alpha \|u\|_\infty^{q-p} \frac{d_j(\bar l)^{p-2}}{r_j^{n+q-\alpha}}\iint_{L^\prime_j(\bar l)}
 \left(\frac{l_j-u}{d_j(\bar l)}\right)^{(1+\lambda)(q-1)}\varphi_j(\bar l)^{M-q}
 \,\mathrm {d}x\mathrm {d}t
 \\&\leq \gamma_1\epsilon_1^{(p-1)(1+\lambda)}\chi,
 \end{split}\end{equation}
 where the constant $\gamma_1$ depends only upon the data.
 For a fixed $\epsilon_2>0$, we assert that there exists a constant $\gamma>0$ depending only upon the data, $\epsilon_1$ and $\epsilon_2$ such that
  \begin{equation}\begin{split}\label{double prime claimpphase}
 &\frac{d_j(\bar l)^{p-2}}{r_j^{n+p}}\iint_{L^{\prime\prime}_j(\bar l)}\left(\frac{l_j-u}{d_j(\bar l)}\right)^{(1+\lambda)(p-1)}\varphi_j(\bar l)^{M-p}
 \,\mathrm {d}x\mathrm {d}t
 \\&+\frac{d_j(\bar l)^{q-2}}{r_j^{n+q}}\iint_{L^{\prime\prime}_j(\bar l)}a(x,t)\left(\frac{l_j-u}{d_j(\bar l)}\right)^{(1+\lambda)(q-1)}\varphi_j(\bar l)^{M-q}
 \,\mathrm {d}x\mathrm {d}t
 \\&\leq
  4^n\epsilon_2\chi+\gamma\left(\chi+B^{-p}+B^{-(p-1)}\right)^{1+\frac{p}{n}}.
  \end{split}\end{equation}
  Once again, we treat the two integrals on the left-hand side of \eqref{double prime claimpphase} separately.
  The integral for the $p$-Laplace type can be handled in much the same way as in the case of $(p,q)$-phase, the only difference
being in the estimate of $T_1$.
Taking into account that $p<q\leq p+\alpha$, $d_j(\bar l)\leq\frac{1}{2}\|u\|_\infty$ and
$a_1< 10[a]_\alpha r_j^\alpha$ for all $(x,t)\in Q_j(\bar l)$, we have
\begin{equation}\label{time derivativepphase}|\partial_t\varphi_j(\bar l)|\leq 9\Theta_j(\bar l)^{-1}\leq 9\left(\frac{d_j(\bar l)^{p-2}}{r_j^{p}}+10[a]_\alpha\|u\|_\infty^{q-p}
\frac{d_j(\bar l)^{p-2}}{r_j^{q-\alpha}}\right)\leq\gamma_1\frac{d_j(\bar l)^{p-2}}{r_j^{p}},\end{equation}
 where the constant $\gamma_1$ depends only upon the data.
We now apply Young's inequality, \eqref{time derivativepphase}, \eqref{step1 initial estimate} and \eqref{T2T3} to conclude
that there exists a constant $\gamma_2=\gamma_2(\text{data})$ such that
\begin{equation}\begin{split}\label{T1T2T3pphase}
T_1=&\gamma_2 \frac{1}{r_j^n}\iint_{L_j(\bar l)}\frac{l_j-u}{d_j(\bar l)}\left|\partial_t\varphi_j(\bar l)\right|\varphi_j(\bar l)^{k_1h_1-1}\,\mathrm {d}x\mathrm {d}t
\\ \leq& \gamma_2 \frac{d_j(\bar l)^{p-2}}{r_j^{p+n}}\iint_{L_j(\bar l)}\left(\frac{l_j-u}{d_j(\bar l)}\right)\,\mathrm {d}x\mathrm {d}t
\\ \leq &\gamma_2\frac{d_j(\bar l)^{p-2}}{r_j^{p+n}}|L_j(\bar l)|
+\gamma_2  \frac{d_j(\bar l)^{p-2}}{r_j^{n+p}}\iint_{L_j(\bar l)}\left(\frac{l_j-u}{d_j(\bar l)}\right)^{(1+\lambda)(p-1)}
 \,\mathrm {d}x\mathrm {d}t
  \\
  \leq &\gamma_2\chi.
\end{split}\end{equation}
In the same manner as in the proof of \eqref{estimate for pintegral}, we conclude that the inequality
 \begin{equation}\begin{split}\label{estimate for pintegralpphase}
 &\frac{d_j(\bar l)^{p-2}}{r_j^{n+p}}\iint_{L^{\prime\prime}_j(\bar l)}\left(\frac{l_j-u}{d_j(\bar l)}\right)^{(1+\lambda)(p-1)}\varphi_j(\bar l)^{M-p}
 \,\mathrm {d}x\mathrm {d}t\\&\leq  4^n\epsilon_2\chi+\gamma\left(\chi+B^{-p}+B^{-(p-1)}\right)^{1+\frac{p}{n}}
  \end{split}\end{equation}
  holds in the case of $p$-phase.
  Here, the constant $\gamma$ depends only upon the data, $\epsilon_1$ and $\epsilon_2$.
  The next step is to estimate the second integral on the left-hand side of \eqref{double prime claimpphase}.
  Recalling that $a(x,t)\leq 12[a]_\alpha r_j^\alpha$ for all $(x,t)\in Q_j(\bar l)$, we infer from $p<q\leq p+\alpha$ that
  \begin{equation}\begin{split}\label{qintegralpphse}
  &\frac{d_j(\bar l)^{q-2}}{r_j^{n+q}}\iint_{L^{\prime\prime}_j(\bar l)}a(x,t)\left(\frac{l_j-u}{d_j(\bar l)}\right)^{(1+\lambda)(q-1)}\varphi_j(\bar l)^{M-q}
 \,\mathrm {d}x\mathrm {d}t
 \\&\leq 12[a]_\alpha\frac{d_j(\bar l)^{q-2}}{r_j^{n+q-\alpha}}\iint_{L^{\prime\prime}_j(\bar l)}\left(\frac{l_j-u}{d_j(\bar l)}\right)^{(1+\lambda)(q-1)}\varphi_j(\bar l)^{M-q}
 \,\mathrm {d}x\mathrm {d}t
  \\&\leq \gamma^\prime\frac{d_j(\bar l)^{q-2}}{r_j^{n+p}}\iint_{L^{\prime\prime}_j(\bar l)}\left(\frac{l_j-u}{d_j(\bar l)}\right)^{q-1-\lambda+\lambda q}\varphi_j(\bar l)^{M-q}
 \,\mathrm {d}x\mathrm {d}t,
 \end{split}\end{equation}
 where $\gamma^\prime=12[a]_\alpha$. Recalling that $\frac{l_j-u}{d_j(\bar l)}\geq\epsilon_1$ on $L^{\prime\prime}_j(\bar l)$, we apply Lemma \ref{lemmainequalitypsi-}
 with $\nu=p$ to obtain
   \begin{equation*}\left(\frac{l_j-u}{d_j(\bar l)}\right)^{p-1-\lambda}\leq c(\text{data},\epsilon_1)\cdot\psi_p(x,t)^p\qquad\text{on}\quad L^{\prime\prime}_j(\bar l).\end{equation*}
 Taking into account that $u$ is nonnegative and $l_j\leq l_0\leq\frac{1}{4}\omega$, we deduce $l_j-u\leq \frac{1}{4}\omega\leq \frac{1}{2}\|u\|_\infty$ on $L_j(\bar l)$.
 It follows that the inequality
 \begin{equation}\begin{split}\label{auxiliaryinequality}
 &\left(\frac{l_j-u}{d_j(\bar l)}\right)^{q-1-\lambda+\lambda q}= \left(\frac{l_j-u}{d_j(\bar l)}\right)^{p-1-\lambda+q-p+\lambda q}
 \\&\leq c(\text{data},\epsilon_1)\psi_p(x,t)^p\left(\frac{l_j-u}{d_j(\bar l)}\right)^{q-p}\left(\frac{l_j-u}{d_j(\bar l)}\right)^{\lambda q}
 \\&\leq c(\text{data},\epsilon_1)\|u\|_\infty^{q-p}\psi_p(x,t)^pd_j(\bar l)^{p-q}\left(\frac{l_j-u}{d_j(\bar l)}\right)^{\frac{p}{n}}
 \end{split}\end{equation}
 holds on the set $L^{\prime\prime}_j(\bar l)$. In the last line, we have also used $\lambda=\frac{p}{nq}$. In view of \eqref{auxiliaryinequality}, we use H\"older's inequality
 with exponents $\frac{n}{p}$ and $\frac{n}{n-p}$ to deduce
 \begin{equation*}\begin{split}
&\frac{d_j(\bar l)^{q-2}}{r_j^{n+p}}\iint_{L^{\prime\prime}_j(\bar l)}\left(\frac{l_j-u}{d_j(\bar l)}\right)^{q-1-\lambda+\lambda q}\varphi_j(\bar l)^{M-q}
 \,\mathrm {d}x\mathrm {d}t
 \\ &\leq \gamma\frac{d_j(\bar l)^{p-2}}{r_j^{n+p}}\iint_{L^{\prime\prime}_j(\bar l)}
 \left[\psi_p(x,t)^p \varphi_j(\bar l)^{\frac{1}{2}M-q}\right]\times\left[\left(\frac{l_j-u}{d_j(\bar l)}\right)^{\frac{p}{n}}
 \varphi_j(\bar l)^{\frac{1}{2}M}\right]
 \,\mathrm {d}x\mathrm {d}t
  \\ &\leq \gamma\frac{d_j(\bar l)^{p-2}}{r_j^{n+p}}\int_{t_1-\Theta_j(\bar l)}^{t_1}\left(\int_{L_j^{\prime\prime}(t)}\tilde v_p^\frac{np}{n-p}\,\mathrm {d}x\right)^\frac{n-p}{n}
\left(\int_{L_j^{\prime\prime}(t)}\frac{l_j-u}{d_j(\bar l)}
 \varphi_j(\bar l)^{\frac{Mn}{2p}}
 \,\mathrm {d}x\right)^\frac{p}{n}\mathrm {d}t,
 \end{split}\end{equation*}
 where $\tilde v_p=\psi_p\varphi_j(\bar l)^{\frac{1}{p}(\frac{M}{2}-q)}$. For $t\in (t_1-\Theta_j(\bar l), t_1)$, we apply Sobolev's inequality slice-wise for $\tilde v_p(\cdot,t)$.
 This yields that
 \begin{equation}\begin{split}\label{qintegralpphse1}
&\frac{d_j(\bar l)^{q-2}}{r_j^{n+p}}\iint_{L^{\prime\prime}_j(\bar l)}\left(\frac{l_j-u}{d_j(\bar l)}\right)^{q-1-\lambda+\lambda q}\varphi_j(\bar l)^{M-q}
 \,\mathrm {d}x\mathrm {d}t
  \\ \leq & \gamma
 \esssup_t\left(\frac{1}{r_j^n}\int_{L_j^{\prime\prime}(t)}
\frac{l_j-u}{d_j(\bar l)}\varphi_j(\bar l)^{\frac{Mn}{2p}} \,\mathrm {d}x\right)^{\frac{p}{n}}
\times\left[\frac{d_j(\bar l)^{p-2}}{r_j^{n}}\iint_{Q_j(\bar  l)}|D\tilde v_p|^p \,\mathrm {d}x\mathrm {d}t\right].
 \end{split}\end{equation}
 We can now proceed analogously to the proofs of \eqref{T1-T5qlaplace} and \eqref{Dpsiresultq}. The only difference is that we have to use
 \eqref{T1T2T3pphase} instead of \eqref{T1T2T3}. This gives
  \begin{equation}\begin{split}\label{T1-T5qlaplacepphase}
  \esssup_t \frac{1}{r_j^n} \int_{L_j^{\prime\prime}(t)}
\frac{l_j- u}{d_j(\bar l)}\varphi_j(\bar l)^{\frac{Mn}{2p}} \,\mathrm {d}x\leq \gamma\chi+\gamma(B^{-p}+B^{-(p-1)})
    \end{split}\end{equation}
  and
   \begin{equation}\begin{split}\label{Dpsiresultqpphase}
   \frac{d_j(\bar l)^{p-2}}{r_j^n}\iint_{Q_j(\bar l)}
|D\tilde v_p|^p \,\mathrm {d}x\mathrm {d}t\leq \gamma\chi+\gamma(B^{-p}+B^{-(p-1)}).
    \end{split}\end{equation}
Combining \eqref{qintegralpphse}, \eqref{qintegralpphse1}, \eqref{T1-T5qlaplacepphase} and \eqref{Dpsiresultqpphase}, we infer that the inequality
     \begin{equation}\begin{split}\label{estimate for qintegralpphase}
 &\frac{d_j(\bar l)^{q-2}}{r_j^{n+q}}\iint_{L^{\prime\prime}_j(\bar l)}a(x,t)\left(\frac{l_j-u}{d_j(\bar l)}\right)^{(1+\lambda)(q-1)}\varphi_j(\bar l)^{M-q}
 \,\mathrm {d}x\mathrm {d}t\\&\leq \gamma\left(\chi+B^{-p}+B^{-(p-1)}\right)^{1+\frac{p}{n}}
  \end{split}\end{equation}
   holds true for a constant $\gamma=\gamma(\text{data},\epsilon_1)$. Consequently, we infer from \eqref{estimate for pintegralpphase} and \eqref{estimate for qintegralpphase}
   that \eqref{double prime claimpphase} holds for the case of $p$-phase.

   To summarize what we have proved, we conclude from \eqref{step1 second estimate}, \eqref{double prime claim}, \eqref{step1 second estimate p-phase}
   and \eqref{double prime claimpphase} that the inequalities
   \begin{equation}\begin{split}\label{mainprimeprime}
 &\frac{d_j(\bar l)^{p-2}}{r_j^{n+p}}\iint_{L^{\prime\prime}_j(\bar l)}\left(\frac{l_j-u}{d_j(\bar l)}\right)^{(1+\lambda)(p-1)}\varphi_j(\bar l)^{M-p}
 \,\mathrm {d}x\mathrm {d}t
 \\&+\frac{d_j(\bar l)^{q-2}}{r_j^{n+q}}\iint_{L^{\prime\prime}_j(\bar l)}a(x,t)\left(\frac{l_j-u}{d_j(\bar l)}\right)^{(1+\lambda)(q-1)}\varphi_j(\bar l)^{M-q}
 \,\mathrm {d}x\mathrm {d}t
 \\&\leq
  4^n\epsilon_2\chi+\gamma_2\left(\chi+B^{-p}+B^{-(p-1)}\right)^{1+\frac{p}{n}}
  +\gamma_2\left(\chi+B^{-p}+B^{-(p-1)}\right)^{1+\frac{q}{n}}
  \end{split}\end{equation}
  and
  \begin{equation}\begin{split}\label{mainnoprime}
 &\frac{d_j(\bar l)^{p-2}}{r_j^{n+p}}\iint_{L_j(\bar l)}\left(\frac{l_j-u}{d_j(\bar l)}\right)^{(1+\lambda)(p-1)}\varphi_j(\bar l)^{M-p}
 \,\mathrm {d}x\mathrm {d}t
 \\&+\frac{d_j(\bar l)^{q-2}}{r_j^{n+q}}\iint_{L_j(\bar l)}a(x,t)\left(\frac{l_j-u}{d_j(\bar l)}\right)^{(1+\lambda)(q-1)}\varphi_j(\bar l)^{M-q}
 \,\mathrm {d}x\mathrm {d}t
 \\&\leq \gamma\epsilon_1^{(p-1)(1+\lambda)}\chi+
  4^n\epsilon_2\chi+\gamma_2\left(\chi+B^{-p}+B^{-(p-1)}\right)^{1+\frac{p}{n}}
  \\&\quad+\gamma_2\left(\chi+B^{-p}+B^{-(p-1)}\right)^{1+\frac{q}{n}}
  \end{split}\end{equation}
  hold for all the cases. Here, the constant $\gamma$ depends only  upon the data, while the constant $\gamma_2$ depends on the data, $\epsilon_1$ and $\epsilon_2$.

  In order to prove \eqref{Aj}, we need to treat the third integral on the right-hand side of \eqref{A_j}.
  To this end, we use the Caccioppoli inequality \eqref{Cacformula1} with $(\rho,l,d,\theta)$ replaced by
    $(r_j,l_j,d_j(\bar l),\Theta_j(\bar l))$ to deduce
  \begin{equation*}\begin{split}
\esssup_t&\frac{1}{r_j^n}
\int_{B_j}
G\left(\frac{l_j-u}{d_j(\bar l)} \right)\varphi_j(\bar l)^M\,\mathrm {d}x
  \\
 \leq &\gamma \frac{1}{r_j^n}\iint_{L_j(\bar l)}\frac{l_j-u}{d_j(\bar l)}\left|\partial_t\varphi_j(\bar l)\right|\varphi_j(\bar l)^{M-1}\,\mathrm {d}x\mathrm {d}t
  \\&+\gamma  \frac{d_j(\bar l)^{p-2}}{r_j^{n+p}}\iint_{L_j(\bar l)}\left(\frac{l_j-u}{d_j(\bar l)}\right)^{(1+\lambda)(p-1)}
  \varphi_j(\bar l)^{M-p}\,\mathrm {d}x\mathrm {d}t
  \\&+\gamma  \frac{d_j(\bar l)^{q-2}}{r_j^{n+q}}\iint_{L_j(\bar l)}a(x,t)\left(\frac{l_j-u}{d_j(\bar l)}\right)^{(1+\lambda)(q-1)}
  \varphi_j(\bar l)^{M-q}\,\mathrm {d}x\mathrm {d}t
\\&+\gamma \frac{\Theta_j(\bar l)}{r_j^nd_j(\bar l)^2}\int_{B_j}g^{\frac{p}{p-1}}\,\mathrm {d}x
  +\gamma \frac{\Theta_j(\bar l)}{r_j^nd_j(\bar l)}\int_{B_j}|f|\,\mathrm {d}x
  \\ =:&S_1+S_2+S_3+S_4+S_5,
    \end{split}\end{equation*}
    where the constant $\gamma$ depends only upon the data. Analysis similar to that in the proofs of \eqref{T4} and \eqref{T5} shows that
    $S_4+S_5\leq \gamma(\text{data})(B^{-p}+B^{-(p-1)})$. Moreover, it follows from \eqref{mainnoprime} that the estimate
     \begin{equation*}\begin{split}
 S_2+S_3&\leq \gamma\epsilon_1^{(p-1)(1+\lambda)}\chi+
  4^n\epsilon_2\chi+\gamma_2\left(\chi+B^{-p}+B^{-(p-1)}\right)^{1+\frac{p}{n}}
  \\&\quad+\gamma_2\left(\chi+B^{-p}+B^{-(p-1)}\right)^{1+\frac{q}{n}}
  \end{split}\end{equation*}
  holds true for $\gamma=\gamma(\text{data})$ and $\gamma_2=\gamma_2(\text{data},\epsilon_1,\epsilon_2)$. In order to estimate $S_1$,
 we first decompose
 \begin{equation*}\begin{split}
 S_1=\gamma \frac{1}{r_j^n}\left[\iint_{L_j^{\prime}
 (\bar l)}\cdots+\iint_{L_j^{\prime\prime}
 (\bar l)}\frac{l_j-u}{d_j(\bar l)}\left|\partial_t\varphi_j(\bar l)\right|\varphi_j(\bar l)^{M-1}\,\mathrm {d}x\mathrm {d}t\right]=:S_{1,1}+S_{1,2},
  \end{split}\end{equation*}
  with the obvious meaning of $S_{1,1}$ and $S_{1,2}$. In view of $|\partial_t\varphi_j(\bar l)|\leq 9\Theta_j(\bar l)^{-1}$, we infer from \eqref{step1 initial estimate} that
  the inequality
  \begin{equation*}\begin{split}
 S_{1,1}\leq\gamma \epsilon_1\left(\frac{d_j(\bar l)^{p-2}}{r_j^{n+p}}+a_1\frac{d_j(\bar l)^{q-2}}{r_j^{n+q}}\right)|L_j(\bar l)|\leq \gamma \epsilon_1\chi
         \end{split}\end{equation*}
          holds true for $\gamma=\gamma(\text{data})$.
  To estimate $S_{1,2}$, we again distinguish two cases: $p$-phase and $(p,q)$-phase. We first consider the case of $(p,q)$-phase, i.e., $a_1\geq 10[a]_\alpha r_j^\alpha$.
  Since $p>2$, $q>2$ and $\frac{4}{5}a_1\leq a(x,t)\leq \frac{6}{5}a_1$
for all $(x,t)\in Q_j(\bar l)$, we have
  \begin{equation*}\begin{split}
  S_{1,2}\leq &\gamma \frac{d_j(\bar l)^{p-2}}{r_j^{n+p}}\iint_{L_j^{\prime\prime}
 (\bar l)}\frac{l_j-u}{d_j(\bar l)}\varphi_j(\bar l)^{M-1}\,\mathrm {d}x\mathrm {d}t
 \\&+\gamma a_1\frac{d_j(\bar l)^{q-2}}{r_j^{n+q}}\iint_{L_j^{\prime\prime}
 (\bar l)}\frac{l_j-u}{d_j(\bar l)}\varphi_j(\bar l)^{M-1}\,\mathrm {d}x\mathrm {d}t
 \\ \leq &\gamma_1 \frac{d_j(\bar l)^{p-2}}{r_j^{n+p}}\iint_{L_j^{\prime\prime}
 (\bar l)}\left(\frac{l_j-u}{d_j(\bar l)}\right)^{(p-1)(1+\lambda)}\varphi_j(\bar l)^{M-p}\,\mathrm {d}x\mathrm {d}t
 \\&+\gamma_1 \frac{d_j(\bar l)^{q-2}}{r_j^{n+q}}\iint_{L_j^{\prime\prime}
 (\bar l)}a(x,t)\left(\frac{l_j-u}{d_j(\bar l)}\right)^{(q-1)(1+\lambda)}\varphi_j(\bar l)^{M-q}\,\mathrm {d}x\mathrm {d}t,
    \end{split}\end{equation*}
    where the constant $\gamma_1$ depends on the data and $\epsilon_1$. It follows from \eqref{mainprimeprime} that the inequality
     \begin{equation*}\begin{split}
     S_{1,2}\leq \epsilon_2\gamma_1\chi+\gamma_2\left(\chi+B^{-p}+B^{-(p-1)}\right)^{1+\frac{p}{n}}
  +\gamma_2\left(\chi+B^{-p}+B^{-(p-1)}\right)^{1+\frac{q}{n}}
     \end{split}\end{equation*}
     holds true for $\gamma_1=\gamma_1(\text{data},\epsilon_1)$ and $\gamma_2=\gamma_2(\text{data},\epsilon_1,\epsilon_2)$.
     Next, we consider the case of $p$-phase, i.e., $a_1<10[a]_\alpha r_j^\alpha$. According to \eqref{time derivativepphase} and \eqref{mainprimeprime}, we find that
     \begin{equation*}\begin{split}
  S_{1,2}\leq &\gamma \frac{d_j(\bar l)^{p-2}}{r_j^{n+p}}\iint_{L_j^{\prime\prime}
 (\bar l)}\frac{l_j-u}{d_j(\bar l)}\varphi_j(\bar l)^{M-1}\,\mathrm {d}x\mathrm {d}t
 \\ \leq &\gamma_1 \frac{d_j(\bar l)^{p-2}}{r_j^{n+p}}\iint_{L_j^{\prime\prime}
 (\bar l)}\left(\frac{l_j-u}{d_j(\bar l)}\right)^{(p-1)(1+\lambda)}\varphi_j(\bar l)^{M-p}\,\mathrm {d}x\mathrm {d}t
 \\ \leq& \epsilon_2\gamma_1\chi+\gamma_2\left(\chi+B^{-p}+B^{-(p-1)}\right)^{1+\frac{p}{n}}
  +\gamma_2\left(\chi+B^{-p}+B^{-(p-1)}\right)^{1+\frac{q}{n}},
    \end{split}\end{equation*}
    where the constant $\gamma_1$ depends on the data and $\epsilon_1$ and the constant $\gamma_2$ depends on the data, $\epsilon_1$ and $\epsilon_2$.
    Combining these inequalities, we conclude that the inequality
     \begin{equation}\begin{split}\label{estimate for Gchi}
\esssup_t&\frac{1}{r_j^n}
\int_{B_j}
G\left(\frac{l_j-u}{d_j(\bar l)} \right)\varphi_j(\bar l)^M\,\mathrm {d}x
\\& \leq\gamma\epsilon_1^{(p-1)(1+\lambda)}\chi+
 \epsilon_2\gamma_1\chi+\epsilon_1\gamma \chi+\gamma(B^{-p}+B^{-(p-1)})
  \\&\quad+\gamma_2\left(\chi+B^{-p}+B^{-(p-1)}\right)^{1+\frac{p}{n}}
  +\gamma_2\left(\chi+B^{-p}+B^{-(p-1)}\right)^{1+\frac{q}{n}}
\end{split}\end{equation}
holds for $\gamma=\gamma(\text{data})$, $\gamma_1=\gamma_1(\text{data},\epsilon_1)$ and $\gamma_2=\gamma_2(\text{data},\epsilon_1,\epsilon_2)$.
Combining \eqref{mainnoprime} and \eqref{estimate for Gchi}, we arrive at
\begin{equation}\begin{split}\label{estimate for Ajbarlchi}
A_j(\bar l)& \leq
 \epsilon_2\gamma_1\chi+\epsilon_1\gamma^{\prime\prime} \chi+\gamma^{\prime\prime}(B^{-p}+B^{-(p-1)})
  \\&\quad+\gamma_2\chi^{1+\frac{p}{n}}+\gamma_2\chi^{1+\frac{q}{n}}+\gamma_2\left(B^{-p}+B^{-(p-1)}\right)^{1+\frac{p}{n}}
  +\gamma_2\left(B^{-p}+B^{-(p-1)}\right)^{1+\frac{q}{n}},
\end{split}\end{equation}
where $\gamma^{\prime\prime}=\gamma^{\prime\prime}(\text{data})$, $\gamma_1=\gamma_1(\text{data},\epsilon_1)$ and $\gamma_2=\gamma_2(\text{data},\epsilon_1,\epsilon_2)$.
Our task now is to fix the parameters. We first choose $\epsilon_1=(100\gamma)^{-1}$ and this fixes $\gamma_1=\gamma_1(\epsilon_1)$. Next, we choose $\epsilon_2=(100\gamma_1)^{-1}$.
This also fixes $\gamma_2$ with the choices of $\epsilon_1$ and $\epsilon_2$. At this point, we set $B>8$ large enough to have
\begin{equation}\label{conditionforB}
B^{-p}+B^{-(p-1)}<\frac{1}{100\gamma^{\prime\prime}}\chi.
\end{equation}
Finally, we choose $\chi=\left(\frac{1}{100\gamma_2}\right)^\frac{n}{p}2^{-1-\frac{n}{p}}$. With the choices of $\epsilon_1$, $\epsilon_2$, $\chi$ and $B$,
      we infer from \eqref{estimate for Ajbarlchi} that $A_j(\bar l)\leq \frac{1}{2}\chi$,
      which proves the claim
      \eqref{Aj}. Recalling that $A_j(l)\to+\infty$
    as $l\to l_j$ and $A_j(l)$ is continuous in $(\bar l, l_j)$, then there exists a number $\tilde l\in (\bar l, l_j)$ such that $A_j(\tilde l)=\chi$.
    At this point, we define
    \begin{equation}\label{definitionlj+1}
	l_{j+1}=\begin{cases}
\tilde l,&\quad \text{if}\quad \tilde l<l_j-\frac{1}{4}(\alpha_{j-1}-\alpha_j),\\
	l_j-\frac{1}{4}(\alpha_{j-1}-\alpha_j),&\quad \text{if}\quad \tilde l\geq l_j-\frac{1}{4}(\alpha_{j-1}-\alpha_j).
	\end{cases}
\end{equation}
It is easily seen that
\eqref{li} and \eqref{Aj-1} are valid for $i=j+1$. Next, we claim that \eqref{lj}$_{j+1}$ holds. Since $l_{j+1}\geq \bar l$, we use \eqref{lj}$_{j}$ and $\alpha_{j-1}\geq \alpha_j$ to get
\begin{equation*}\begin{split}
l_{j+1}&>\frac{1}{2}\left(\frac{1}{8}B\alpha_{j-1}+\frac{1}{16}\omega\right)+\frac{1}{16}B\alpha_j+\frac{1}{32}\omega
    \geq \frac{1}{8}B\alpha_j+\frac{1}{16}\omega,
\end{split}\end{equation*}
which is our claim. Therefore, we have now constructed a sequence $\{l_i\}_{i=0}^\infty$ satisfying \eqref{li}-\eqref{lj}.
We just remark that this choice of $\chi$ also determines the value of $\nu_0$ via \eqref{nu0}.

Step 4: \emph{Proof of the inequality $u(x_1,t_1)>2^{-5}\omega$.} We first observe from the construction of $\{l_i\}_{i=0}^\infty$ that
this sequence is decreasing and hence that there exists a number $\hat l\geq  0$ such that
\begin{equation*}\hat l=\lim_{i\to\infty}l_i.\end{equation*}
This also implies that $d_i=l_i-l_{i+1}\to 0$ as $i\to\infty$. Moreover, we claim that $\hat l=u(x_1,t_1)$.
In the case $a_1=0$, we infer from \eqref{assumption for Aj} and \eqref{Aj-1} that
\begin{equation*}\begin{split}
\frac{1}{r_j^{n+p}} &\iint_{Q_j^\prime}(l_j-u)_+^{(1+\lambda)(p-1)}\,\mathrm {d}x\mathrm {d}t\leq A_j(l_{j+1})d_j^{(1+\lambda)(p+1)-(p-2)}
\\&\leq\chi d_j^{(1+\lambda)(p+1)-(p-2)}\to0,
  \end{split}\end{equation*}
  as $j\to\infty$.
  Here $Q_j^\prime=B_{j+1}\times\left(t_1-\frac{4}{9}(1+10[a]_\alpha)^{-1}r_j^p,t_1\right)$.
  In the case $a_1>0$, we choose $j\geq1$ large enough to have $10[a]_\alpha r_j^{q-p}\leq a_1$. It follows from \eqref{assumption for Aj} and \eqref{Aj-1} that
  \begin{equation*}\begin{split}
\frac{1}{a_1^{-1}r_j^{n+q}}&\iint_{Q_j^{\prime\prime}}(l_j-u)_+^{(1+\lambda)(q-1)}\,\mathrm {d}x\mathrm {d}t\leq \frac{5}{4}A_j(l_{j+1})d_j^{(1+\lambda)(q+1)-(q-2)}
\\&\leq \frac{5}{4}\chi d_j^{(1+\lambda)(q+1)-(q-2)}\to0,
  \end{split}\end{equation*}
  as $j\to\infty$. Here $Q_j^{\prime\prime}=B_{j+1}\times\left(t_1-\frac{4}{9}(1+[a]_\alpha^{-1})^{-1}a_1^{-1}r_j^q,t_1\right)$. This shows that $\hat l=u(x_1,t_1)$,
  since $(x_1,t_1)$ is a Lebesgue point of $u$. Furthermore, we prove that for any $j\geq1$ there holds
  \begin{equation}\begin{split}\label{djdj-1}
  d_j\leq &\frac{1}{4}d_{j-1}+\gamma \frac{4^{-j-100}}{B}\omega+
  \gamma\left(r_{j-1}^{p-n}\int_{B_{j-1}} g(y)^{\frac{p}{p-1}}
  \,\mathrm {d}y\right)^{\frac{1}{p}}\\&+\gamma\left(r_{j-1}^{p-n}\int_{B_{j-1}}|f(y)|
  \,\mathrm {d}y\right)^{\frac{1}{p-1}}+\gamma r_{j-1},
  \end{split}\end{equation}
  where the constant $\gamma$ depends only upon the data. In order to prove \eqref{djdj-1}, we assume that for any fixed $j\geq 1$,
  \begin{equation}\begin{split}\label{djdj-1proof}
  d_j>\frac{1}{4}d_{j-1}\qquad\text{and}\qquad d_j>\frac{1}{4}(\alpha_{j-1}-\alpha_j),
   \end{split}\end{equation}
   since otherwise \eqref{djdj-1} follows from \eqref{alpha2}.
   We first observe from $d_j>\frac{1}{4}(\alpha_{j-1}-\alpha_j)$ and \eqref{definitionlj+1} that $A_j(l_{j+1})=A_j(\tilde l)=\chi$.
   In view of $d_j>\frac{1}{4}d_{j-1}$, we see that
\begin{equation*}
\Theta_j
=d_{j}^2\left[\left(\frac{d_{j}}{r_{j}}\right)^p+a_1\left(\frac{d_{j}}{r_{j}}\right)^q\right]^{-1}
<\left(\frac{1}{4}d_{j-1}\right)^2\left[\left(\frac{d_{j-1}}{r_{j-1}}\right)^p+a_1\left(\frac{d_{j-1}}{r_{j-1}}\right)^q\right]^{-1}=\frac{1}{16}\Theta_{j-1}\end{equation*}
and this yields that $Q_j\subseteq Q_{j-1}$ and $\varphi_{j-1}=1$ in $Q_j$. Recalling that $Q_i\subseteq Q_A$ and $Q_i\subseteq Q_{r_i,r_i^2}(z_1)$ hold for
all $i\geq 0$, an argument similar to the one used in step 3 shows that
\begin{equation}\begin{split}\label{estimate for Ajbarlchistep4}
\chi=A_j(l_{j+1})& \leq
 \epsilon_2\gamma_1\chi+\epsilon_1\gamma \chi+\gamma\left(\frac{r_j^{p-n}}{d_j^p}\int_{B_{j}} g(y)^{\frac{p}{p-1}}
  \,\mathrm {d}y+\frac{r_j^{p-n}}{d_j^{p-1}}\int_{B_{j}} |f(y)|
  \,\mathrm {d}y\right)
  \\&\quad+\gamma_2\chi^{1+\frac{p}{n}}+\gamma_2\left(\frac{r_j^{p-n}}{d_j^p}\int_{B_{j}} g(y)^{\frac{p}{p-1}}
  \,\mathrm {d}y+\frac{r_j^{p-n}}{d_j^{p-1}}\int_{B_{j}} |f(y)|
  \,\mathrm {d}y\right)^{1+\frac{p}{n}}
 \\&\quad +\gamma_2\chi^{1+\frac{q}{n}}+\gamma_2\left(\frac{r_j^{p-n}}{d_j^p}\int_{B_{j}} g(y)^{\frac{p}{p-1}}
  \,\mathrm {d}y+\frac{r_j^{p-n}}{d_j^{p-1}}\int_{B_{j}} |f(y)|
  \,\mathrm {d}y\right)^{1+\frac{q}{n}},
\end{split}\end{equation}
where $\gamma=\gamma(\text{data})$, $\gamma_1=\gamma_1(\text{data},\epsilon_1)$ and $\gamma_2=\gamma_2(\text{data},\epsilon_1,\epsilon_2)$
are the constants in \eqref{estimate for Ajbarlchi}. According to the choices of $\epsilon_1$, $\epsilon_2$ and $\chi$,
we infer that the  terms involving $\chi$ on the right-hand side of \eqref{estimate for Ajbarlchistep4} can be re-absorbed into the left-hand side.
This yields that
  \begin{equation*}\begin{split}
  d_j\leq \gamma\left(r_j^{p-n}\int_{B_j}g(y)^{\frac{p}{p-1}}\,\mathrm {d}y\right)^{\frac{1}{p}}
  \leq \gamma\left(r_{j-1}^{p-n}\int_{B_{j-1}}g(y)^{\frac{p}{p-1}}\,\mathrm {d}y\right)^{\frac{1}{p}}
   \end{split}\end{equation*}
   or
    \begin{equation*}\begin{split}
   d_j\leq
  \gamma \left(r_j^{p-n}\int_{B_j}|f(y)|\,\mathrm {d}y\right)^{\frac{1}{p-1}}
  \leq
  \gamma \left(r_{j-1}^{p-n}\int_{B_{j-1}}|f(y)|\,\mathrm {d}y\right)^{\frac{1}{p-1}},
   \end{split}\end{equation*}
   which establishes the estimate \eqref{djdj-1}. At this stage, we let $J>1$ be a fixed integer and sum up the inequality \eqref{djdj-1}
   for $j=1,\cdots,J-1$. This yields that
   \begin{equation*}\begin{split}
 \sum_{j=1}^{J-1} d_j\leq& \frac{1}{4}\sum_{j=1}^{J-1}d_{j-1}+\gamma \frac{1}{B}
   \sum_{j=1}^{J-1} 4^{-j-100}\omega+\gamma\sum_{j=1}^{J-1}\left(r_{j-1}^{p-n}\int_{B_{j-1}}|f(y)|
  \,\mathrm {d}y\right)^{\frac{1}{p-1}}
\\&+ \gamma\sum_{j=1}^{J-1}\left(r_{j-1}^{p-n}\int_{B_{j-1}} g(y)^{\frac{p}{p-1}}
  \,\mathrm {d}y\right)^{\frac{1}{p}}+\gamma  \sum_{j=1}^{J-1}r_{j-1}
  \\ \leq& \frac{1}{4}\sum_{j=1}^{J-1}d_{j-1}+\gamma\frac{1}{B}\omega+\gamma\left(F_p(R)+G_p(R)+ R\right),
   \end{split}\end{equation*}
   where the constant $\gamma$ depends only upon the data. Recalling that $d_j=l_j-l_{j+1}$,
   we pass to the limit $J\to\infty$ to infer that
    \begin{equation*}
    l_1\leq \frac{3}{4}u(x_1,t_1)+\frac{1}{4}l_0+\gamma\frac{1}{B}\omega+\gamma\left(F_p(R)+G_p(R)+ R\right)\leq \frac{3}{4}u(x_1,t_1)+\frac{1}{4}l_0+\gamma\frac{1}{B}\omega,
    \end{equation*}
    where we have used \eqref{omega1violated} in the last step. Taking into account that $l_0=\frac{1}{4}\omega$ and $d_0\leq d_0(\bar l)<\frac{1}{8}\omega$,
    we deduce that for $B>32\gamma$ there holds
     \begin{equation*}\begin{split}
     \frac{\omega}{4}&=l_0=d_0+l_1\leq d_0+\frac{3}{4}u(x_1,t_1)+\frac{1}{4}l_0+\gamma\frac{1}{B}\omega
     \\&<
     \frac{3}{4}u(x_1,t_1)+\left(\frac{1}{8}+\frac{1}{16}+\gamma\frac{1}{B}\right)\omega< \frac{3}{4}u(x_1,t_1)+\frac{7}{32}\omega,
      \end{split}\end{equation*}
      which proves the inequality $u(x_1,t_1)>\frac{1}{24}\omega$. Recalling that $(x_1,t_1)$ is a Lebesgue point of $u$, we conclude that
   \eqref{DeGiorgi1} holds for
    almost everywhere point in $Q_{\frac{1}{4}R}^-(\bar t)$. Finally, according to \eqref{nu0} and \eqref{conditionforB}, 
    the values of $\nu_0$ and $B$ can be determined via
 \begin{equation}\begin{split}\label{nu0B} \nu_0=\frac{\chi}{4C^{n+q}\gamma^\prime}\qquad\text{and}\qquad  
 B=\max\left\{32\gamma,
 \left(\frac{\chi}{8C^{n-p}\gamma^\prime}\right)^{\frac{1}{1-p}},\left(\frac{\chi}{200\gamma^{\prime\prime}}\right)^{\frac{1}{1-p}}\right\}.
 \end{split}\end{equation}  
 Here, $\gamma^\prime$ and $\gamma^{\prime\prime}$ are the constants in \eqref{gammaprime} and \eqref{estimate for Ajbarlchi}, respectively.
The proof of the lemma is now complete.
 \end{proof}
 We remark that from the proof of Lemma \ref{lemmaDeGiorgi1} the constant $A$ satisfies the following condition:
 \begin{equation}\begin{split}\label{firstcondition forA}
         A>\max\left\{4^{100}\left(18\hat\gamma_1^{-1}\right)^{\frac{1}{p-2}}B,\quad 25^{\frac{1}{p-2}}(4^{100}B)^\frac{q-2}{p-2}\right\},
        \end{split}\end{equation}
        where $\hat \gamma_1=(1+10[a]_\alpha\|u\|_\infty^{q-p})^{-1}$ and the constant $B$ is determined in \eqref{nu0B}.
Moreover, we introduce a time level
$\hat t=\bar t-\Theta_\omega\left(\frac{1}{4}R\right)$.
According to Lemma \ref{lemmaDeGiorgi1}, we conclude that if \eqref{omega1} is violated, then
 \begin{equation}\begin{split}\label{hat t lower}
 u(x,\hat  t)>\mu_-+\frac{\omega}{2^5}\qquad\text{for}\qquad x\in B_{\frac{R}{4}}.
  \end{split}\end{equation}
With the help of the inequality \eqref{hat t lower} we can now establish the following result regarding the time propagation of positivity.
\begin{lemma}\label{timeexpandlemma}
Let $u$ be a bounded weak solution to \eqref{parabolic}-\eqref{A} in $\Omega_T$. Assume that \eqref{hat t lower} holds for the time level $\hat t=\bar t-\Theta_\omega\left(\frac{1}{4}R\right)$.
Given $\nu_*\in(0,1)$, there exists a constant $s_*=s_*(\text{data},\nu_*)>5$, such that either
\begin{equation}
\label{omega2}\omega\leq 2^{\frac{2}{p}s_*}G_p(R)+2^{\frac{1}{p-1}s_*}F_p(R)
\end{equation}
or
\begin{equation}\begin{split}\label{time expand}
\left|\left\{x\in B_{\frac{R}{8}}:u(x,t)<\mu_-+\frac{\omega}{2^{s_*}}\right\}\right|\leq \nu_*|B_{\frac{R}{8}}|
\end{split}\end{equation}
holds for any $t\in (\hat t,0)$.
 \end{lemma}
\begin{proof}
We first observe from \eqref{firstcondition forA} that $-\hat t=-\bar t+\Theta_\omega\left(\frac{1}{4}R\right)\leq \frac{8}{9}\Theta_A+\Theta_\omega<\Theta_A$.
In order to apply the logarithmic estimate \eqref{lnCac} over the cylinder $B_{\frac{R}{4}}\times[\hat t,0]$, we must have $-\hat t\leq \frac{1}{16}R^2$.
To this end, we assume that $A>16^{\frac{1}{q-p}}$ and this yields that $\Theta_A\leq A^{q-p}R^2\leq \frac{1}{16}R^2$.
Moreover, we will apply Lemma \ref{logestimatelemma}
with the logarithmic function
\begin{equation}\begin{split}\label{SS0}\Psi^-&=\ln^+\left(\frac{H_k^-}{H_k^--(u-k)_-+c}\right)
=\begin{cases}
	\ln\left(\dfrac{H_k^-}{H_k^-+u-k+c}\right),&\quad k-H_k^-\leq u<k-c,\\
	0,&\quad u\geq k-c.
	\end{cases}\end{split}\end{equation}
Here, we set $k=\mu_-+\frac{1}{2^5}\omega$, $H_k^-=\frac{1}{32}\omega$ and $c=\frac{1}{2^{5+l}}\omega$ where $l\geq1$ will be determined later.
Therefore, we infer that
\begin{equation*}\begin{split}
\esssup_{\hat t<t<0}&\int_{B_{\frac{R}{4}}\times\{t\}}[\Psi^-(u)]^2\phi^q\,\mathrm {d}x
\\ \leq &\int_{B_{\frac{R}{4}}\times\{\hat t\}}[\Psi^-(u)]^2\phi^q\,\mathrm {d}x+\gamma
\iint_{B_{\frac{R}{4}}\times[\hat t,0]}a_0\Psi^-(u)|D\phi|^q\left[(\Psi^-)^\prime(u)\right]^{2-q}
\,\mathrm {d}x\mathrm {d}t
\\&+\gamma \iint_{B_{\frac{R}{4}}\times[\hat t,0]}\Psi^-(u)\left[(\Psi^-)^\prime(u)\right]^{2-p}\left(|D\phi|^p+R^\alpha|D\phi|^q\right)
\,\mathrm {d}x\mathrm {d}t
\\&+\gamma\ln\left(\frac{H}{c}\right)\left(\frac{1}{c}\iint_{B_{\frac{R}{4}}\times[\hat t,0]}|f|\,\mathrm {d}x\mathrm {d}t+
\frac{1}{c^2}\iint_{B_{\frac{R}{4}}\times[\hat t,0]}g^{\frac{p}{p-1}}\,\mathrm {d}x\mathrm {d}t\right)
\\ =&:I_1+I_2+I_3+I_4,
\end{split}\end{equation*}
where $\phi=\phi(x)$ satisfies $0\leq\phi\leq1$ in $B_{\frac{R}{4}}$, $\phi\equiv 1$
in $B_{\frac{R}{8}}$ and $|D\phi|\leq 4R^{-1}$.
We first observe from \eqref{hat t lower} and \eqref{SS0} that $\psi^-(x,\hat t)= 0$ for all $x\in B_{\frac{R}{4}}$,
and hence $I_1=0$. In view of
\eqref{SS0}, we see that $\Psi^-(u)\leq l\ln2$ and $|(\Psi^-)^\prime(u)|^{-1}\leq \omega$. Noting that $-\hat t<\Theta_A\leq A^{q-2}\Theta_\omega$, we get
\begin{equation*}\begin{split}
I_2+I_3\leq 4l(\ln2)(-\hat t)|B_{\frac{R}{4}}|\left(\omega^{p-2}R^{-p}+a_0 \omega^{q-2}R^{-q}\right)\leq \gamma lA^{q-2}|B_{\frac{R}{8}}|.
\end{split}\end{equation*}
To estimate $I_4$, we use the inequality $-\hat t<\Theta_A\leq A^{p-2}\omega^{2-p}R^p$.  It follows that
\begin{equation*}\begin{split}
I_4\leq &\gamma (lA^{p-2})\frac{2^l}{\omega^{p-1}}\left(R^{p-n}\int_{B_{\frac{R}{4}}}|f|\,\mathrm {d}x\right)|B_{\frac{R}{8}}|
\\&+\gamma (lA^{p-2})\frac{2^{2l}}{\omega^{p}}\left(R^{p-n}\int_{B_{\frac{R}{4}}}g^{\frac{p}{p-1}}\,\mathrm {d}x\right)|B_{\frac{R}{8}}|.
\end{split}\end{equation*}
At this stage, we assume that
\begin{equation}\label{omega2assumption}
\omega> 2^{\frac{2}{p}l}G_p(R)+2^{\frac{1}{p-1}l}F_p(R)
\end{equation}
and this gives $I_3\leq \gamma lA^{p-2}|B_{\frac{R}{8}}|$. Consequently, we infer that the inequality
\begin{equation}\begin{split}\label{psi-upper}
\int_{B_{\frac{R}{4}}\times\{t\}}[\Psi^-(u)]^2\phi^q\,\mathrm {d}x\leq \gamma lA^{q-2}|B_{\frac{R}{8}}|
\end{split}\end{equation}
holds for any $t\in(\hat t,0)$. On the other hand,
the left-hand side of \eqref{psi-upper} is estimated below by integrating
over the smaller set
$\left\{x\in B_{\frac{R}{8}}:u<\mu_-+\frac{\omega}{2^{l+5}}\right\}\subset B_{\frac{R}{4}}.$
On such a set, $\phi\equiv 1$ and
\begin{equation*}\begin{split}\Psi^-&=\ln^+\left(\frac{\frac{1}{2^5}\omega}{\frac{1}{2^5}\omega
-(u-k)_-+\frac{1}{2^{l+5}}\omega}\right)
>(l-1)\ln2.\end{split}\end{equation*}
This leads to
\begin{equation}\begin{split}\label{SS4}\int_{B_{\frac{R}{8}}\times\{t\}}[\Psi^-(u)]^2\phi^q dx\geq
(l-1)^2(\ln2)^2
\left|\left\{x\in B_{\frac{R}{8}}:u<\mu_-+\frac{\omega}{2^{l+5}} \right\}\right|,\end{split}\end{equation}
for all $t\in(\hat t,0)$. Combining \eqref{psi-upper} and \eqref{SS4}, we conclude that there exists a constant $\gamma_0>1$ 
depending only upon the data, such that
\begin{equation*}
\left|\left\{x\in B_{\frac{R}{8}}:u<\mu_-+\frac{\omega}{2^{l+5}} \right\}\right|\leq \gamma_0\frac{l}{(l-1)^2}A^{q-2}
|B_{\frac{R}{8}}|.
\end{equation*}
For a fixed $\nu_*\in(0,1)$, we let $l=l(\nu_*,A)$ be an integer such that $l>1+\gamma_0\nu_*^{-1}A^{q-2}$ and choose $s_*=
s_*(\nu_*)=2\gamma_0\nu_*^{-1}A^{q-2}$.
This proves the inequality \eqref{time expand} under the assumption \eqref{omega2assumption}. On the other hand, if \eqref{omega2assumption}
is violated, then we see that \eqref{omega2} holds for such a choice of $s_*$. This completes the proof of Lemma \ref{timeexpandlemma}.
\end{proof}
The crucial result in our development of regularity properties of weak solutions will be the following De Giorgi-type Lemma involving the initial data.
The proof is similar in spirit to Lemma \ref{lemmaDeGiorgi1}. However, in order to handle
the initial data, we need to modify the definition of $A_j(l)$ and the argument is considerably more delicate.
\begin{lemma}\label{lemmaDeGiorgi2}
Let $\xi=2^{-s_*}$, $0<\xi<2^{-5}$, $\widetilde Q_1=B_\frac{R}{8}\times (\hat t,0)$ and $\widetilde Q_2=B_\frac{R}{16}\times \left(\frac{\hat t}{2},0\right)$.
Let $u$ be a bounded weak solution to \eqref{parabolic}-\eqref{A} in $\Omega_T$.
Assume that \eqref{hat t lower} holds.
Then, there exist $\nu_1\in(0,1)$ and $\tilde B>1$ which
can be quantitatively determined a priori only in terms of the data, such that if
\begin{equation}\label{assumptionDegiorgi}\left|\left\{(x,t)\in \widetilde Q_1:u<\mu_-+\xi \omega\right\}
\right|\leq \nu_1|\widetilde Q_1|,\end{equation}
then either
\begin{equation}
\label{DeGiorgi2}u(x,t)>\mu_-+\frac{1}{4}\xi\omega\qquad\text{for}\ \ \text{a.e.}\ \ (x,t)\in \widetilde Q_2\end{equation}
or
\begin{equation}
\label{omega3}\xi\omega\leq \tilde B\left(
F_p(R)+G_p(R)+100R\right).\end{equation}
Here, $\hat t=\bar t-\Theta_\omega\left(\frac{1}{4}R\right)$.
 \end{lemma}
 \begin{proof}Our proof starts with the observation that $\tilde u=u-\mu_-$ is a bounded weak solution to the parabolic double-phase equation
$\partial_t \tilde u-\operatorname{div}\tilde A(x,t,\tilde u,D\tilde u)=f,$
where the vector field $\tilde A(x,t,\tilde u,\xi)=A(x,t,\tilde u+\mu_-,\xi)$ satisfies the structure condition \eqref{A}.
 Once again, we abbreviate $\tilde u$ by $u$.
 We first assume that \eqref{omega3} is violated, that is,
 \begin{equation}
\label{omega1violated1}\xi\omega>\tilde B\left(
F_p(R)+G_p(R)+100R\right),\end{equation}
where $\tilde B>10$ will be chosen later in a universal way. Let $z_1\in \tilde Q_1$  be a Lebesgue point of $u$.
We are reduced to proving \eqref{DeGiorgi2} for $u(x_1,t_1)>\frac{1}{4}\xi\omega$. The proof of this result will be divided into several steps.

Step 1: \emph{Setting up the notations.}
Let $a_1=a(x_1,t_1)$, $r_j=4^{-j-4}R$, $B_j=B_{r_j}(x_1)$ and $K_j=B_j\times(\hat t, t_1)$.
For a sequence $\{l_j\}_{j=0}^\infty$ and a fixed $l>0$, we set $d_{j}(l)=l_j-l$. If $d_{j}(l)>0$, then we define
the parabolic cylinder $Q_j(l)=B_j\times (t_1-\Theta_j(l),t_1)$, where
 \begin{equation}\label{defthetajl}\Theta_j(l)=d_{j}(l)^2\left[\left(\frac{d_{j}(l)}{r_j}\right)^p+a_1\left(\frac{d_{j}(l)}{r_j}\right)^q\right]^{-1}.\end{equation}
 Moreover, we define $\varphi_j(l)=\phi_j(x)\theta_{j,l}(t)$, where
$\phi_j\in C_0^\infty(B_j)$, $\phi_j=1$ on $B_{j+1}$, $|D\phi_j|\leq r_j^{-1}$
 and $\theta_{j,l}(t)$ is a Lipschitz function
satisfies $\theta_{j,l}(t)=1$ in $t\geq t_1-\frac{4}{9}\Theta_j(l)$, $\theta_{j,l}(t)=0$ in $t\leq t_1-\frac{5}{9}\Theta_j(l)$
and
  \begin{equation*}
 \theta_{j,l}(t)=\frac{t-t_1-\frac{5}{9}\Theta_j(l)}{\frac{1}{9}\Theta_j(l)}\qquad\text{in}\qquad
 t_1-\frac{5}{9}\Theta_j(l)\leq t\leq t_1-\frac{4}{9}\Theta_j(l).
 \end{equation*}
By the definition of $\varphi_j(l)$, we see that $\varphi_j(l)=0$ on $\partial_PQ_j(l)$.
 Next, for  $j=-1,0,1,2,\cdots$, we define the sequence $\{\alpha_j\}$ by
  \begin{equation}\begin{split}\label{alpha}\alpha_j=&\int_0^{r_j}\left(r^{p-n}\int_{B_r(x_1)}
  g(y)^{\frac{p}{p-1}} \,\mathrm {d}y
  \right)^{\frac{1}{p}}\frac{\mathrm {d}r}{r}\\&+
  \int_0^{r_j}\left(r^{p-n}\int_{B_r(x_1)}|f(y)| \,\mathrm {d}y
  \right)^{\frac{1}{p-1}}\frac{\mathrm {d}r}{r}+100r_j.
  \end{split}\end{equation}
  According to the definition of $\alpha_j$, we see that $\alpha_j\to 0$ as $j\to\infty$ and there holds $\tilde B\alpha_{j-1}\leq \xi\omega$,
  \begin{equation}\begin{split}\label{alpha1xi}
  \alpha_{j-1}-\alpha_j\geq &\gamma\left(r_j^{p-n}\int_{B_j} g(y)^{\frac{p}{p-1}}
  \,\mathrm {d}y\right)^{\frac{1}{p}}\\&+\gamma\left(r_j^{p-n}\int_{B_j}|f(y)|
  \,\mathrm {d}y\right)^{\frac{1}{p-1}}+200 r_j
  \end{split}\end{equation}
  and
   \begin{equation}\begin{split}\label{alpha2xi}
  \alpha_{j-1}-\alpha_j\leq &\gamma\left(r_{j-1}^{p-n}
  \int_{B_{j-1}} g(y)^{\frac{p}{p-1}}
  \,\mathrm {d}y\right)^{\frac{1}{p}}\\&+\gamma\left(r_{j-1}^{p-n}\int_{B_{j-1}}|f(y)|
  \,\mathrm {d}y\right)^{\frac{1}{p-1}}+\gamma r_{j-1}
  \end{split}\end{equation}
  for all $j=0,1,2,\cdots$, where the constant $\gamma$ depends only upon the data.
  Moreover,
 for the sequence $\{l_j\}_{j=0}^\infty$ and a fixed $l>0$
 we introduce a quantity $A_j(l)$ as follows:
   \begin{itemize}
 \item[$\bullet$]In the case $Q_j(l)\subseteq \widetilde Q_1$, i.e., $t_1-\Theta_j(l)\geq\hat t$, we define $A_j(l)$ by
 \begin{equation}\begin{split}\label{Aj1st}
 A_j(l)= A_j^{(1)}(l)=&\frac{d_j(l)^{p-2}}{r_j^{n+p}}\iint_{L_j(l)}\left(\frac{l_j-u}{d_j(l)}\right)^{(1+\lambda)(p-1)}\varphi_j(l)^{M-p}
 \,\mathrm {d}x\mathrm {d}t
 \\&+\frac{d_j(l)^{q-2}}{r_j^{n+q}}\iint_{L_j(l)}a(x,t)\left(\frac{l_j-u}{d_j(l)}\right)^{(1+\lambda)(q-1)}\varphi_j(l)^{M-q}
 \,\mathrm {d}x\mathrm {d}t
 \\&+\esssup_t\frac{1}{r_j^n}\int_{B_j\times\{t\}}G\left(\frac{l_j-u}{d_j(l)}\right)\varphi_j(l)^M
 \,\mathrm {d}x.
 \end{split}\end{equation}
 \end{itemize}
  \begin{itemize}
 \item[$\bullet$]In the case $Q_j(l)\nsubseteq\widetilde Q_1$, i.e., $t_1-\Theta_j(l)<\hat t$,
 we define $A_j(l)$ by
 \begin{equation}\begin{split}\label{Aj2nd}
 A_j(l)=A_j^{(2)}(l)=&\frac{d_j(l)^{p-2}}{r_j^{n+p}}\iint_{U_j}\left(\frac{l_j-u}{d_j(l)}\right)^{(1+\lambda)(p-1)}\phi_j^{M-p}
 \,\mathrm {d}x\mathrm {d}t
 \\&+\frac{d_j(l)^{q-2}}{r_j^{n+q}}\iint_{U_j}a(x,t)\left(\frac{l_j-u}{d_j(l)}\right)^{(1+\lambda)(q-1)}\phi_j^{M-q}
 \,\mathrm {d}x\mathrm {d}t
 \\&+\esssup_{\hat t<t<t_1}\frac{1}{r_j^n}\int_{B_j\times\{t\}}G\left(\frac{l_j-u}{d_j(l)}\right)\phi_j^M
 \,\mathrm {d}x.
 \end{split}\end{equation}
 \end{itemize}
 Here, the function $G$ is defined in \eqref{G}, $L_j(l)=Q_j(l)\cap \{u\leq l_j\}$ and
 $U_j=K_j\cap \{u\leq l_j\}$.
 Throughout the proof, we keep $\lambda=\frac{p}{nq}$.
 In the following, the sequence $\{l_j\}_{j=0}^\infty$ we constructed will satisfy $l_j\geq\frac{1}{4}\xi\omega$
 for all $j=0,1,2,\cdots$. At this stage, we claim that either $u(x_1,t_1)\geq\frac{1}{4}\xi\omega$ or $A_j(l)\to+\infty$ as $l\to l_j$.
 For a fixed $j\geq0$, let $m_j>0$ be a number satisfying $l_j-1<m_j< l_j$. Thus, $d_j(l)<1$ for any $m_j\leq l<l_j$.
We first consider the case $\Theta_j(m_j)\leq t_1-\hat t$.
Observe that
$\Theta_j(l)\geq\left(r_j^{-p}+a_1r_j^{-q}\right)^{-1}=:\theta_j$ for any $m_j\leq l<l_j$. We infer from \eqref{Aj1st} that
  \begin{equation}\begin{split}\label{assumption for Ajxi1}
  A_j(l)\geq& \frac{d_j(l)^{p-2-(1+\lambda)(p+1)}}{r_j^{n+p}}\iint_{\tilde Q_j}(l_j-u)_+^{(1+\lambda)(p-1)}\,\mathrm {d}x\mathrm {d}t
  \\&+ \frac{d_j(l)^{q-2-(1+\lambda)(q+1)}}{r_j^{n+q}}\iint_{\tilde Q_j}a(x,t)(l_j-u)_+^{(1+\lambda)(q-1)}\,\mathrm {d}x\mathrm {d}t,
  \end{split}\end{equation}
  where $\tilde Q_j=B_{j+1}\times \left(\hat t_1,t_1\right)$ and $\hat t_1=\max\left\{t_1-\frac{4}{9}\theta_j,\hat t\right\}$. 
  Since $\theta_j\leq r_j^2$,
  we get $\tilde Q_j\subseteq Q_{r_j,r_j^2}(z_1)$.
 In the case $a_1<10[a]_\alpha r_j^{q-p}$, we have $\tilde Q_j\supseteq Q_j^\prime=B_{j+1}\times\left(t_1^\prime,t_1\right)$,
 where $t_1^\prime=\max\left\{t_1-\frac{4}{9}(1+10[a]_\alpha)^{-1}r_j^p,\hat t\right\}$.
 If
  \begin{equation*}\begin{split}
  \iint_{Q_j^\prime}(l_j-u)_+^{(1+\lambda)(p-1)}\,\mathrm {d}x\mathrm {d}t=0,
   \end{split}\end{equation*}
   then $u(x_1,t_1)\geq l_j\geq \frac{1}{4}\xi\omega$, since $(x_1,t_1)$ is a Lebesgue point of $u$. On the other hand, if
    \begin{equation*}\begin{split}
  \iint_{Q_j^\prime}(l_j-u)_+^{(1+\lambda)(p-1)}\,\mathrm {d}x\mathrm {d}t>0,
   \end{split}\end{equation*}
   then we conclude from \eqref{assumption for Ajxi1} that
    \begin{equation*}\begin{split}
  A_j(l)\geq& \frac{d_j(l)^{p-2-(1+\lambda)(p+1)}}{r_j^{n+p}}\iint_{Q_j^\prime}(l_j-u)_+^{(1+\lambda)(p-1)}\,\mathrm {d}x\mathrm {d}t\to+\infty
  \end{split}\end{equation*}
  as $l\to l_j$ and this proves the claim. In the case $a_1\geq10[a]_\alpha r_j^{q-p}$, we have
  $\theta_j\geq (1+[a]_\alpha^{-1})^{-1}a_1^{-1}r_j^q$ and $\frac{4}{5}a_1\leq a(x,t)\leq \frac{6}{5}a_1$ for all $(x,t)\in \tilde Q_j$. Hence
  $\tilde Q_j\supseteq Q_j^{\prime\prime}=B_{j+1}\times\left(t_1^{\prime\prime},t_1\right)$,
 where $t_1^{\prime\prime}=\max\left\{t_1-\frac{4}{9}(1+[a]_\alpha^{-1})^{-1}a_1^{-1}r_j^q,\hat t\right\}$.
  If
  \begin{equation*}\begin{split}
  \iint_{Q_j^{\prime\prime}}(l_j-u)_+^{(1+\lambda)(q-1)}\,\mathrm {d}x\mathrm {d}t=0,
   \end{split}\end{equation*}
   then $u(x_1,t_1)\geq l_j\geq \frac{1}{4}\xi\omega$, since $(x_1,t_1)$ is a Lebesgue point of $u$. On the other hand, if
 \begin{equation*}\begin{split}
  \iint_{Q_j^{\prime\prime}}(l_j-u)_+^{(1+\lambda)(q-1)}\,\mathrm {d}x\mathrm {d}t>0,
   \end{split}\end{equation*}
   then we deduce from $Q_j^{\prime\prime}\subseteq \tilde Q_j$ and \eqref{assumption for Aj} that
    \begin{equation*}\begin{split}
  A_j(l)\geq& \frac{4}{5}a_1\frac{d_j(l)^{q-2-(1+\lambda)(q+1)}}{r_j^{n+q}}\iint_{Q_j^{\prime\prime}}(l_j-u)_+^{(1+\lambda)(q-1)}\,\mathrm {d}x\mathrm {d}t\to+\infty
  \end{split}\end{equation*}
  as $l\to l_j$. This proves the claim for the case $\Theta_j(m_j)\leq t_1-\hat t$. We now turn our attention to the case $\Theta_j(m_j)> t_1-\hat t$.
  Noting that $\Theta_j(l)\geq \Theta_j(m_j)> t_1-\hat t$ for $m_j\leq l<l_j$, we infer from \eqref{Aj2nd} that
    \begin{equation}\begin{split}\label{assumption for Ajxi2}
  A_j(l)\geq &\frac{d_j(l)^{p-2-(1+\lambda)(p-1)}}{r_j^{n+p}}\iint_{\hat Q_j}(l_j-u)_+^{(1+\lambda)(p-1)}
 \,\mathrm {d}x\mathrm {d}t,
  \end{split}\end{equation}
  where $\hat Q_j=B_{j+1}\times (\hat t,t_1)$. If
  \begin{equation*}\begin{split}
  \iint_{\hat Q_j}(l_j-u)_+^{(1+\lambda)(q-1)}\,\mathrm {d}x\mathrm {d}t=0,
   \end{split}\end{equation*}
   then $u(x_1,t_1)\geq l_j\geq \frac{1}{4}\xi\omega$, since $(x_1,t_1)$ is a Lebesgue point of $u$. On the other hand, if
 \begin{equation*}\begin{split}
  \iint_{\hat Q_j}(l_j-u)_+^{(1+\lambda)(q-1)}\,\mathrm {d}x\mathrm {d}t>0,
   \end{split}\end{equation*}
   then we deduce from \eqref{assumption for Ajxi2} that $A_j(l)\to+\infty$ as $l\to l_j$, which proves our assertion.
  It is obvious that the lemma holds if $u(x_1,t_1)\geq \frac{1}{4}\xi\omega$. In the following, we assume that
  $A_j(l)\to+\infty$ as $l\to l_j$ holds for all $j=0,1,2,\cdots$.
 From \eqref{defthetajl}, we see that $\Theta_j(l)$ is continuous and
increasing in $[0,l_j)$. In the case $\Theta_j(0)>t_1-\hat t$, we have $Q_j(l)\nsubseteq\widetilde Q_1$ for all $0\leq l<l_j$. Hence $A_j(l)=A_j^{(2)}(l)$
for all $l\in[0,l_j)$, and so $A_j(l)$ is continuous and increasing in $[0,l_j)$. In the case $\Theta_j(0)\leq t_1-\hat t$, there exists a unique $l_j^*\in[0,l_j)$
such that $\Theta_j(l_j^*)=t_1-\hat t$. It is easy to check that $A_j(l)= A_j^{(1)}(l)$ for $0\leq l\leq l_j^*$ and $A_j(l)= A_j^{(2)}(l)$ for $ l_j^* < l<l_j$.
Hence $A_j(l)$ is discontinuous at $l_j^*$.
It is worth pointing out that the domain of $A_j^{(2)}(l)$ extends to $[0,l_j)$ and $A_j^{(2)}(l)\geq A_j^{(1)}(l)$ for all $l\in [0,l_j^*]$.

 Step 2: \emph{Determine the values of $l_0$ and $l_1$.} To start with, we define $l_0=\xi\omega$ and $\bar l=\frac{1}{2}l_0+\frac{1}{4}\tilde B\alpha_0+\frac{1}{8}\xi\omega$.
 Noting that $\tilde B\alpha_0<\xi\omega$, we get $\frac{1}{8}\xi\omega\leq d_0(\bar l)\leq \frac{3}{8}\xi\omega$.
Moreover, we claim that $Q_0(\bar l)\nsubseteq\widetilde Q_1$. To begin with the proof, we define the number
\begin{equation}\label{Dpq}D_{p}=\max\left\{1,\ \tfrac{1}{(p-2)\ln2}\ln\left(\tfrac{1}{(p-2)\ln2}\right)\right\}.\end{equation}
Then, we have $X\leq 2^{(p-2)X}$, provided that $X>D_p$.
Again, we first consider the case of $p$-phase, i.e., $a_0<10[a]_\alpha R^\alpha$. Observe that $a_1\leq 12[a]_\alpha R^\alpha$, $\xi<1$ and $\omega\leq 2\|u\|_\infty$. It follows
that
\begin{equation*}\begin{split}
\Theta_0(\bar l)\geq \left(\frac{3}{8}\xi\omega\right)^2\left[\left(\frac{3}{8}\frac{\xi\omega}{4^{-4}R}\right)^p+a_1\left(\frac{3}{8}\frac{\xi\omega}{4^{-4}R}
\right)^q\right]^{-1}\geq \hat\gamma_1(\xi\omega)^{2-p}R^p,
 \end{split}\end{equation*}
 where $\hat\gamma_1=\left(\frac{3}{8}\right)^{2-p}4^{-4q}\left[1+12[a]_\alpha \left(\frac{3}{4}\|u\|_\infty\right)^{q-p}\right]^{-1}$.
At this point, we choose
\begin{equation}\label{firstconditionforA}A>\max\left\{D_{p}^\frac{1}{q-2},\left[\tfrac{1}{p-2}\log_2\tfrac{1}{\hat\gamma_1}\right]^\frac{1}{q-2}\right\}.\end{equation}
Recalling that $\xi=2^{-s_*}=2^{-2\gamma_0\nu_*^{-1}A^{q-2}}\leq 2^{-2A^{q-2}}$, we have
\begin{equation*}\begin{split}
\Theta_0(\bar l)\geq \hat\gamma_1 2^{(p-2)A^{q-2}}A^{q-2}\omega^{2-p}R^p\geq A^{p-2}\omega^{2-p}R^p>t_1+\Theta_A,
 \end{split}\end{equation*}
 which proves the claim $Q_0(\bar l)\nsubseteq\widetilde Q_1$. Next, we consider the case of $(p,q)$-phase, i.e., $a_0\geq10[a]_\alpha R^\alpha$.
 We infer from \eqref{phase analysis} that $a_1\leq \frac{6}{5}a_0$ and hence
 \begin{equation*}\begin{split}
\Theta_0(\bar l)&\geq \left(\frac{3}{8}\xi\omega\right)^2\left[\left(\frac{3}{8}\frac{\xi\omega}{4^{-4}R}\right)^p+\frac{6}{5}a_0\left(\frac{3}{8}\frac{\xi\omega}{4^{-4}R}
\right)^q\right]^{-1}
\\&\geq \hat\gamma_2(\xi\omega)^2\left[\left(\frac{\xi\omega}{R}\right)^p+a_0\left(\frac{\xi\omega}{R}
\right)^q\right]^{-1},
 \end{split}\end{equation*}
 where $\hat\gamma_2=\frac{5}{6}\left(\frac{3}{8}\right)^{2-q}4^{-4q}$.
 At this stage, we impose a condition for $A$ as follows:
\begin{equation}\label{secondconditionforA}A>\max\left\{D_{p}^\frac{1}{q-2},\left[\tfrac{1}{p-2}\log_2\tfrac{1}{\hat\gamma_2}\right]^\frac{1}{q-2}\right\}.\end{equation}
In view of
$\xi\leq 2^{-2A^{q-2}}$, we conclude that
 \begin{equation*}\begin{split}
\Theta_0(\bar l)&\geq \hat\gamma_2\xi^{2-p}\omega^2\left[\left(\frac{\omega}{R}\right)^p+a_0\left(\frac{\omega}{R}
\right)^q\right]^{-1}>A^{q-2}\Theta_\omega\geq \Theta_A>t_1+\Theta_A.
 \end{split}\end{equation*}
  This proves that $Q_0(\bar l)\nsubseteq\widetilde Q_1$ holds.
  According to the definition of $A_j(l)$, we see that $A_0(\bar l)=A_0^{(2)}(\bar l)$.
  Our task  now is to deduce an upper bound for $A_0(\bar l)$.
  %First, we observe that there exist constants $c_1$ and $c_2>0$ depending only upon the data, such that $c_1|\tilde Q_1|\leq |K_0|\leq c_2|\tilde Q_1|$, where
  %$K_0=B_0\times(\hat t, t_1)$.
  In the case of $p$-phase, i.e., $a_0<10[a]_\alpha R^\alpha$. We infer from $d_0(\bar l)\leq\frac{3}{8}\xi\omega\leq\frac{3}{4}\|u\|_\infty$,
  $l_0-u\leq \xi\omega$
  and \eqref{assumptionDegiorgi}
that
 \begin{equation}\begin{split}\label{A01xi}
&\frac{d_0(\bar l)^{p-2}}{r_0^{n+p}}\iint_{U_0}\left(\frac{l_0-u}{d_0(\bar l)}\right)^{(1+\lambda)(p-1)}\phi_0^{M-p}
 \,\mathrm {d}x\mathrm {d}t
 \\&+\frac{d_0(\bar l)^{q-2}}{r_0^{n+q}}\iint_{U_0}a(x,t)\left(\frac{l_0-u}{d_0(\bar l)}\right)^{(1+\lambda)(q-1)}\phi_0^{M-q}
 \,\mathrm {d}x\mathrm {d}t
 \\&\leq \gamma \frac{d_0(\bar l)^{p-2}}{r_0^{n+p}}\iint_{U_0}\left(\frac{l_0-u}{d_0(\bar l)}\right)^{(1+\lambda)(p-1)}\phi_0^{M-p}
 +\left(\frac{l_0-u}{d_0(\bar l)}\right)^{(1+\lambda)(q-1)}\phi_0^{M-q}
 \,\mathrm {d}x\mathrm {d}t
 \\&\leq \gamma\frac{(\xi\omega)^{p-2}}{R^{n+p}}|U_0|\leq \gamma \frac{A^{2-p}\omega^{p-2}}{R^{n+p}}R^n(-\hat t)\frac{|\tilde Q_1\cap\{u<\xi\omega\}|}{|\tilde Q_1|}\leq\gamma\nu_1,
 \end{split}\end{equation}
  since $\xi^{p-2}<2^{-A^{q-2}}<A^{2-q}<A^{2-p}$ and $-\hat t\leq \Theta_A\leq \omega^{2-p}A^{p-2}R^p$.
  In the case of $(p,q)$-phase, i.e., $a_0\geq10[a]_\alpha R^\alpha$.
  We have $a(x,t)\leq \frac{4}{5}a_0$ for all $(x,t)\in K_0$.
 Moreover, we deduce from $d_0(\bar l)\leq\frac{3}{8}\xi\omega$,
  $l_0-u\leq \xi\omega$, $\xi^{p-2}<A^{2-q}$
  and \eqref{assumptionDegiorgi}
that
 \begin{equation}\begin{split}\label{A02xi}
&\frac{d_0(\bar l)^{p-2}}{r_0^{n+p}}\iint_{U_0}\left(\frac{l_0-u}{d_0(\bar l)}\right)^{(1+\lambda)(p-1)}\phi_0^{M-p}
 \,\mathrm {d}x\mathrm {d}t
 \\&+\frac{d_0(\bar l)^{q-2}}{r_0^{n+q}}\iint_{U_0}a(x,t)\left(\frac{l_0-u}{d_0(\bar l)}\right)^{(1+\lambda)(q-1)}\phi_0^{M-q}
 \,\mathrm {d}x\mathrm {d}t
 \\&\leq \gamma \left(\frac{(\xi\omega)^{p-2}}{R^{n+p}}+a_0\frac{(\xi\omega)^{q-2}}{R^{n+q}}\right)|U_0|
 \\&\leq \gamma\frac{\xi^{p-2}}{R^{n}\Theta_\omega}|U_0|\leq \gamma \frac{1}{A^{q-2}\Theta_\omega}(-\hat t)\frac{|\tilde Q_1\cap\{u<\xi\omega\}|}{|\tilde Q_1|}\leq\gamma\nu_1,
 \end{split}\end{equation}
  since $-\hat t\leq \Theta_A\leq A^{q-2}\Theta_\omega$. Moreover, we use the Caccioppoli estimate \eqref{Cacformula1} with $(l,d,\Theta)$ replaced by $(l_0,d_0(\bar l),t_1-\hat t)$
to deduce that
 \begin{equation*}\begin{split}
    \esssup_{\hat t<t<t_1}& \frac{1}{r_0^n}\int_{B_0\times\{t\}}G\left(\frac{l_0-u}{d_0(\bar l)}\right)\phi_0^M
 \,\mathrm {d}x \\ \leq &\gamma\frac{d_0(\bar l)^{p-2}}{r_0^{p+n}}\iint_{U_0}\left(\frac{l_0-u}{d_0(\bar l)}
 \right)^{(1+\lambda)(p-1)}
 \phi_0^{M-p}\,\mathrm {d}x\mathrm {d}t
 \\ &+\gamma\frac{d_0(\bar l)^{q-2}}{r_0^{q+n}}\iint_{U_0}a(x,t)\left(\frac{l_0-u}{d_0(\bar l)}
 \right)^{(1+\lambda)(q-1)}
 \phi_0^{M-q}\,\mathrm {d}x\mathrm {d}t
   \\&+\gamma \frac{t_1-\hat t}{r_0^nd_0(\bar l)^2}\int_{B_0}g^{\frac{p}{p-1}}\,\mathrm {d}x
  +\gamma \frac{t_1-\hat t}{r_0^nd_0(\bar l)}\int_{B_0}|f|\,\mathrm {d}x
  \\ =&L_1+L_2+L_3+L_4,
 \end{split}\end{equation*}
 since $|D\phi_0|\leq r_0^{-1}$ and $\partial_t\phi_0\equiv0$.
 In view of \eqref{hat t lower}, we infer that the first term on the right-hand side of \eqref{Cacformula1} vanishes.
 According to \eqref{A01xi} and \eqref{A02xi}, we get $L_1+L_2\leq \gamma\nu_1$. To estimate $L_3$ and $L_4$, we observe from
$Q_0(\bar l)\nsubseteq\widetilde Q_1$ that $t_1-\hat t<\Theta_0(\bar l)\leq d_0(\bar l)^{2-p}r_0^p$. It follows from $\frac{1}{8}\xi\omega\leq d_0(\bar l)$
 and \eqref{omega1violated1}
 that
  \begin{equation*}\begin{split}
  L_3+L_4&\leq \gamma \frac{1}{r_0^{n-p}d_0(\bar l)^p}\int_{B_0}g^{\frac{p}{p-1}}\,\mathrm {d}x
  +\gamma \frac{1}{r_0^{n-p}d_0(\bar l)^{p-1}}\int_{B_0}|f|\,\mathrm {d}x
  \\&\leq \gamma (\xi\omega)^{-p}\frac{1}{R^{n-p}}\int_{B_R(x_1)}g^{\frac{p}{p-1}}\,\mathrm {d}x
  +\gamma (\xi\omega)^{-(p-1)}\frac{1}{R^{n-p}}\int_{B_R(x_1)}|f|\,\mathrm {d}x
  \\&\leq \gamma(\tilde B^{-p}+\tilde B^{-(p-1)}).
  \end{split}\end{equation*}
  Consequently, we infer that
  \begin{equation*}\begin{split}
    \esssup_{\hat t<t<t_1}& \frac{1}{r_0^n}\int_{B_0\times\{t\}}G\left(\frac{l_0-u}{d_0(\bar l)}\right)\phi_0^M
 \,\mathrm {d}x  \leq \gamma\nu_1+\gamma(\tilde B^{-p}+\tilde B^{-(p-1)})
 \end{split}\end{equation*}
 and hence that there exists a constant $\gamma_1>0$ depending only on the data, such that
 \begin{equation}\label{definitionofgamma1}A_0(\bar l)=A_0^{(2)}(\bar l)\leq \gamma_1\nu_1+\gamma_1(\tilde B^{-p}+\tilde B^{-(p-1)}).\end{equation}
 Then, we choose $\nu_1=\nu_1(\text{data},\chi)<1$ and $\tilde B=\tilde B(\text{data},\chi)>1$ such that
     \begin{equation}\begin{split}\label{nu1tildeB}
  \nu_1=\frac{\chi}{4\gamma_1}\qquad\text{and}\qquad\gamma_1(\tilde B^{-p}+\tilde B^{-(p-1)})<\frac{\chi}{4}.
    \end{split}\end{equation}
    This implies that $A_0(\bar l)=A_0^{(2)}(\bar l)\leq\frac{1}{2}\chi$. Noting that $\Theta_0(l)\geq\Theta_0(\bar l)>t_1-\hat t$ for all $l\in[\bar l,l_0)$, we have $A_0(l)=A_0^{(2)}(l)$
    for all $l\in[\bar l,l_0)$. Observe that $A_0(l)=A_0^{(2)}(l)$ is continuous in $[\bar l,l_0)$ and $A_0(l)\to+\infty$ as $l\to l_0$. Then, there exists a number $\tilde l\in(\bar l,l_0)$
    such that $A_0(\tilde l)=\chi$. We infer from \eqref{omega1violated1} that for $\tilde B>2$ there holds
  $l_0-\bar l=d_0(\bar l)\geq \frac{1}{8}\xi\omega\geq \frac{1}{4
    \tilde B}\xi\omega>\frac{1}{4}(\alpha_{-1}-\alpha_0),$
    since $\tilde B\alpha_{-1}<\xi\omega$.
    At this point, we set
    \begin{equation}\label{l1xi}
	l_1=\begin{cases}
\tilde l,&\quad \text{if}\quad \tilde l<l_0-\frac{1}{4}(\alpha_{-1}-\alpha_0),\\
	l_0-\frac{1}{4}(\alpha_{-1}-\alpha_0),&\quad \text{if}\quad \tilde l\geq l_0-\frac{1}{4}(\alpha_{-1}-\alpha_0).
	\end{cases}
\end{equation}
Furthermore, we set $Q_0=Q_0(l_1)$ and $d_0=l_0-l_1$. Since $\tilde B\alpha_0<\xi\omega$, we get $l_1\geq  \bar l>\frac{1}{2}\tilde B\alpha_0+\frac{1}{4}\xi\omega$.

Step 3:  \emph{Determine the sequence $\{l_j\}_{j=0}^\infty$.} Suppose that we have chosen two
sequences $l_1,\cdots,l_j$ and $d_0,\cdots,d_{j-1}$ such that for $i=1,\cdots,j$, there holds
    \begin{equation}\label{lixi}\frac{1}{8}\xi\omega+\frac{1}{2}l_{i-1}+\frac{1}{4}\tilde B\alpha_{i-1}<l_i\leq
    l_{i-1}-\frac{1}{4}(\alpha_{i-2}-\alpha_{i-1}),
    \end{equation}
    \begin{equation}\label{Aj-1xi}
    A_{i-1}(l_i)\leq \chi,
     \end{equation}
      \begin{equation}\label{ljxi}
      l_i>\frac{1}{2}\tilde B\alpha_{i-1}+\frac{1}{4}\xi\omega.
       \end{equation}
       Then, we set $Q_i=Q_i(l_{i+1})$, $d_i=d_i(l_{i+1})=l_i-l_{i+1}$, $L_i=L_i(l_{i+1})$, $\varphi_i=\varphi_i(l_{i+1})$ and
       $\Theta_i=\Theta_i(l_{i+1})$ for $i=1,2,\cdots,j-1$. Furthermore, we assert that, with suitable choices of $\chi$ and $\tilde B$, there holds
       \begin{equation}\label{Ajxi}
       A_j(\bar l)\leq \frac{1}{2}\chi,\qquad\text{where}\qquad \bar l=\frac{1}{2}l_j+\frac{1}{4}\tilde B\alpha_j+\frac{1}{8}\xi\omega.
        \end{equation}
    To prove \eqref{Ajxi}, we first observe from \eqref{lixi} and \eqref{ljxi} that $d_j(\bar l)\geq \frac{1}{4}d_{j-1}+\frac{1}{4}\tilde B(\alpha_{j-1}-\alpha_j)$.
     Moreover, we claim that $Q_j(\bar l)\subseteq Q_{r_j,r_j^2}(x_1,t_1)$. In view of \eqref{alpha1xi} and $d_j(\bar l)\geq \frac{1}{4}\tilde B(\alpha_{j-1}-\alpha_j)$, we deduce
     $d_j(\bar l)\geq 25\tilde B r_j$ and hence
     \begin{equation*}\begin{split}
   \Theta_j(\bar l)=d_{j}(\bar l)^2\left[\left(\frac{d_{j}(\bar l)}{r_{j}}\right)^p+a_1\left(\frac{d_{j}(\bar l)}{r_{j}}\right)^q\right]^{-1}
   \leq (25\tilde B)^{2-p}r_j^2<r_j^2,
     \end{split}\end{equation*}
     since $\tilde B>1$. This proves the inclusion $Q_j(\bar l)\subseteq Q_{r_j,r_j^2}(x_1,t_1)$.
    To proceed further, we divide the proof of \eqref{Ajxi} into two cases: $Q_j(\bar l)\nsubseteq\widetilde Q_1$ and $Q_j(\bar l)\subseteq\widetilde Q_1$.

    \S 1. We consider the case $Q_j(\bar l)\nsubseteq\widetilde Q_1$, i.e., $\Theta_j(\bar l)\geq t_1-\hat t$. First, we claim that $Q_{j-1}\nsubseteq\widetilde Q_1$
    and $Q_j(\bar l)\subseteq Q_{j-1}$. In view of $d_j(\bar l)\geq \frac{1}{4}d_{j-1}$, we have
    \begin{equation*}\begin{split}
    \Theta_{j-1}&=\Theta_{j-1}(l_j)=d_{j-1}^2\left[\left(\frac{d_{j-1}}{r_{j-1}}\right)^p+a_1\left(\frac{d_{j-1}}{r_{j-1}}\right)^q\right]^{-1}
    \\&\geq 16d_{j}(\bar l)^2\left[\left(\frac{4d_{j}(\bar l)}{4r_{j}}\right)^p+a_1\left(\frac{4d_{j}(\bar l)}{4r_{j}}\right)^q\right]^{-1}
    =16\Theta_j(\bar l)\geq 16(t_1-\hat t)>t_1-\hat t.
     \end{split}\end{equation*}
     This yields that $Q_{j-1}\nsubseteq\widetilde Q_1$
    and $Q_j(\bar l)\subseteq Q_{j-1}$. Hence, we have $A_{j-1}(l_j)=A_{j-1}^{(2)}(l_j)$ and $A_{j}(\bar l)=A_{j}^{(2)}(\bar l)$.
   Since $u\leq l_j$ on $U_j$ and $\Theta_j(\bar l)\geq t_1-\hat t$, we conclude from \eqref{Aj-1xi} that
   \begin{equation}\begin{split}\label{firstestimate for setU}
  & \left(\frac{d_j(\bar l)^{p-2}}{r_j^{n+p}}+a_1\frac{d_j(\bar l)^{q-2}}{r_j^{n+q}}\right)|U_j|
  \\&\leq \Theta_j(\bar l)^{-1}(t_1-\hat t)\frac{1}{r_j^n}\esssup_{\hat t<t<t_1}\int_{B_{j-1}\cap \{
  u(\cdot,t)\leq l_j\}}\phi_{j-1}^M\,\mathrm {d}x
  \\&\leq\frac{4^n}{r_{j-1}^n}\esssup_{\hat t<t<t_1}\int_{B_{j-1}}G\left(\frac{l_{j-1}-u}{d_{j-1}}\right)\phi_{j-1}^M \,\mathrm {d}x
      \\&  \leq 4^nA_{j-1}^{(2)}(l_j)=4^nA_{j-1}(l_j)\leq 4^n\chi.
     \end{split}\end{equation}
     For a fixed $\epsilon_1>0$, we now decompose
   $U_j=K_j\cap \{u\leq l_j\}=U^\prime_j\cup U^{\prime\prime}_j$, where
   \begin{equation}\begin{split}\label{Udecomposition}
   U^\prime_j=U_j\cap \left\{\frac{l_j-u}{d_j(\bar l)}\leq\epsilon_1\right\}\qquad\text{and}\qquad
   U^{\prime\prime}_j=U_j\setminus U^\prime_j.
   \end{split}\end{equation}
   To proceed further, we shall again distinguish two cases: $p$-phase ($a_1<10[a]_\alpha r_j^\alpha$) and $(p,q)$-phase ($a_1\geq10[a]_\alpha r_j^\alpha$).

   \S 1.1. In the case of $(p,q)$-phase, i.e., $a_1\geq10[a]_\alpha r_j^\alpha$. We observe from $Q_j(\bar l)\subseteq Q_{r_j,r_j^2}(x_1,t_1)$
   that $\frac{4}{5}a_1\leq a(x,t)\leq\frac{6}{5}a_1$ for all $(x,t)\in Q_j(\bar l)$. First, we infer from \eqref{firstestimate for setU} that
    \begin{equation}\begin{split}\label{Uprimexi}
&\frac{d_j(\bar l)^{p-2}}{r_j^{n+p}}\iint_{U^\prime_j}\left(\frac{l_j-u}{d_j(\bar l)}\right)^{(1+\lambda)(p-1)}\phi_j^{M-p}
 \,\mathrm {d}x\mathrm {d}t
 \\&+\frac{d_j(\bar l)^{q-2}}{r_j^{n+q}}\iint_{U^\prime_j}a(x,t)\left(\frac{l_j-u}{d_j(\bar l)}\right)^{(1+\lambda)(q-1)}\phi_j^{M-q}
 \,\mathrm {d}x\mathrm {d}t
 \\&\leq \epsilon_1^{(1+\lambda)(p-1)}\left(\frac{d_j(\bar l)^{p-2}}{r_j^{n+p}}+a_1\frac{d_j(\bar l)^{q-2}}{r_j^{n+q}}\right)|U_j|
 \leq \epsilon_1^{(1+\lambda)(p-1)}4^n\chi.
 \end{split}\end{equation}
 To estimate the integrals on the set $U^{\prime\prime}_j$, we distinguish two cases: $q<n$ and $q\geq n$.

\S 1.1.1. In the case $q<n$. For any fixed $\epsilon_2\in(0,1)$, we use Young's inequality to obtain
  \begin{equation}\begin{split}\label{Uprimeprimexi}
&\frac{d_j(\bar l)^{p-2}}{r_j^{n+p}}\iint_{U^{\prime\prime}_j}\left(\frac{l_j-u}{d_j(\bar l)}\right)^{(1+\lambda)(p-1)}\phi_j^{M-p}
 \,\mathrm {d}x\mathrm {d}t
 \\&+\frac{d_j(\bar l)^{q-2}}{r_j^{n+q}}\iint_{U^{\prime\prime}_j}a(x,t)\left(\frac{l_j-u}{d_j(\bar l)}\right)^{(1+\lambda)(q-1)}\phi_j^{M-q}
 \,\mathrm {d}x\mathrm {d}t
 \\ \leq &\epsilon_2\left(\frac{d_j(\bar l)^{p-2}}{r_j^{n+p}}+a_1\frac{d_j(\bar l)^{q-2}}{r_j^{n+q}}\right)|U_j|
\\ &+c(\epsilon_2)\left[\frac{d_j(\bar l)^{p-2}}{r_j^{n+p}}\iint_{U^{\prime\prime}_j}\left(\frac{l_j-u}{d_j(\bar l)}\right)^{p\frac{n+h_1}{nh_1}}\phi_j^{(M-p)p_*}
 \,\mathrm {d}x\mathrm {d}t\right.
 \\&+\left. a_1\frac{d_j(\bar l)^{q-2}}{r_j^{n+q}}\iint_{U^{\prime\prime}_j}\left(\frac{l_j-u}{d_j(\bar l)}\right)^{q\frac{n+h_2}{nh_2}}\phi_j^{(M-q)q_*}
 \,\mathrm {d}x\mathrm {d}t\right]
 \\ =&:T_1+T_2,
 \end{split}\end{equation}
 where $h_1=\frac{p}{p-1-\lambda}$, $h_2=\frac{q}{q-1-\lambda}$, $p_*=p\frac{n+h_1}{h_1(1+\lambda)(p-1)n}$ and $q_*=q\frac{n+h_2}{h_2(1+\lambda)(q-1)n}$. Since $\lambda=\frac{p}{nq}<\frac{1}{n}$,
 we see that $p_*>1$ and $q_*>1$. In view of \eqref{firstestimate for setU}, we get $T_1\leq 4^n\epsilon_2\chi$. To estimate $T_2$, we apply Lemma \ref{lemmainequalitypsi-}, H\"older's inequality
 and Sobolev's inequality to obtain
 \begin{equation}\begin{split}\label{T2estimatexi}
 T_2\leq &c(\epsilon_2)\left[\frac{d_j(\bar l)^{p-2}}{r_j^{n+p}}\iint_{U^{\prime\prime}_j}v_p^{p+\frac{ph_1}{n}}
 \,\mathrm {d}x\mathrm {d}t
 + a_1\frac{d_j(\bar l)^{q-2}}{r_j^{n+q}}\iint_{U^{\prime\prime}_j}v_q^{q+\frac{qh_2}{n}}
 \,\mathrm {d}x\mathrm {d}t\right]
 \\ \leq & c(\epsilon_2)\left[\frac{d_j(\bar l)^{p-2}}{r_j^{n}}\esssup_{\hat t<t<t_1}\left(\frac{1}{r_j^n}\int_{U^{\prime\prime}_j(t)}\frac{l_j-u}{d_j(\bar l)}\phi_j^{k_1h_1}\,\mathrm {d}x\right)^\frac{p}{n}
 \iint_{U_j}|Dv_p|^p\,\mathrm {d}x\mathrm {d}t
 \right.
 \\ &+\left.a_1\frac{d_j(\bar l)^{q-2}}{r_j^{n}}\esssup_{\hat t<t<t_1}\left(\frac{1}{r_j^n}
 \int_{U^{\prime\prime}_j(t)}\frac{l_j-u}{d_j(\bar l)}\phi_j^{k_2h_2}\,\mathrm {d}x\right)^\frac{q}{n}
 \iint_{U_j}|Dv_q|^q\,\mathrm {d}x\mathrm {d}t
 \right],
  \end{split}\end{equation}
  since $p<n$ and $q<n$.
  Here, $k_1=\frac{(M-p)np_*}{p(n+h_1)}$, $k_2=\frac{(M-q)nq_*}{q(n+h_2)}$, $v_p=\psi_p\phi_j^{k_1}$, $v_q=\psi_q\phi_j^{k_2}$ and
  \begin{equation}\label{defineUprimprime}U_j^{\prime\prime}(t)=\{x\in B_j:u(\cdot,t)\geq l_j\}\cap \left\{x\in B_j:\frac{l_j-u(\cdot,t)}{d_j(\bar l)}>
   \epsilon_1\right\}.\end{equation}
   According to \cite[Lemma 3.4]{Qifanli}, we use Lemma \ref{Cac1} with $(l,d,\Theta,\varphi)$ replaced by
  $(l_j,d_j(\bar l),t_1-\hat t,\phi_j)$ and this yields
   \begin{equation}\begin{split}\label{Cacphi}
   \frac{1}{r_j^n}&\esssup_{\hat t<t<t_1}\int_{U^{\prime\prime}_j(t)}\frac{l_j-u}{d_j(\bar l)}\phi_j^{k_1h_1}\,\mathrm {d}x
   \\ \leq& c(\epsilon_1) \frac{1}{r_j^n}\esssup_{\hat t<t<t_1}
   \int_{B_j}G\left(\frac{l_j-u}{d_j(\bar l)}\right)\phi_j^{k_1h_1}\,\mathrm {d}x
   \\ \leq&\gamma  \frac{d_j(\bar l)^{p-2}}{r_j^{n+p}}\iint_{U_j}\left(\frac{l_j-u}{d_j(\bar l)}\right)^{(1+\lambda)(p-1)}
  \phi_j^{k_1h_1-p}\,\mathrm {d}x\mathrm {d}t
  \\&+\gamma  \frac{d_j(\bar l)^{q-2}}{r_j^{n+q}}\iint_{U_j}a(x,t)\left(\frac{l_j-u}{d_j(\bar l)}\right)^{(1+\lambda)(q-1)}
  \phi_j^{k_1h_1-q}\,\mathrm {d}x\mathrm {d}t
\\&+\gamma \frac{t_1-\hat t}{r_j^nd_j(\bar l)^2}\int_{B_j}g^{\frac{p}{p-1}}\,\mathrm {d}x
  +\gamma \frac{t_1-\hat t}{r_j^nd_j(\bar l)}\int_{B_j}|f|\,\mathrm {d}x
  \\ =&:T_3+T_4+T_5+T_6,
   \end{split}\end{equation}
 since $|D\phi_j|\leq r_j^{-1}$ and $\partial_t\phi_j\equiv0$. According to \eqref{hat t lower}, we find that
 $ u(x,\hat  t)>\xi\omega\geq l_j$ and
 the first term on the right-hand side of \eqref{Cacformula1} vanishes. In view of $d_j(\bar l)\geq \frac{1}{4}d_{j-1}$, $\phi_j\equiv 1$ on $\supp \phi_{j-1}$
 and $A_{j-1}^{(2)}(l_j)=A_{j-1}(l_j)\leq \chi$, we conclude that
 \begin{equation}\begin{split}\label{T3T4xi}
 T_3+T_4 \leq&\gamma  \frac{d_{j-1}^{p-2}}{r_{j-1}^{n+p}}\iint_{U_{j-1}}\left(\frac{l_{j-1}-u}{d_{j-1}}\right)^{(1+\lambda)(p-1)}
  \phi_{j-1}^{M-p}\,\mathrm {d}x\mathrm {d}t
  \\&+\gamma  \frac{d_{j-1}^{q-2}}{r_{j-1}^{n+q}}\iint_{U_{j-1}}a(x,t)\left(\frac{l_{j-1}-u}{d_{j-1}}\right)^{(1+\lambda)(q-1)}
  \phi_{j-1}^{M-q}\,\mathrm {d}x\mathrm {d}t
  \\ \leq &\gamma A_{j-1}^{(2)}(l_j)\leq \gamma\chi,
 \end{split}\end{equation}
 where the constant $\gamma$ depends only upon the data and $\epsilon_1$. We now turn our attention to the
estimates of $T_5$ and $T_6$. Noting that $t_1-\hat t\leq \Theta_j(\bar l)\leq d_j(\bar l)^{2-p}r_j^p$ and $d_j(\bar l)\geq \frac{1}{4}\tilde B(\alpha_{j-1}-\alpha_j)$,
we infer from \eqref{alpha1xi} that
 \begin{equation}\begin{split}\label{T5T6xi}
 T_5+T_6 &\leq \gamma \frac{1}{d_j(\bar l)^p}\left(r_j^{p-n}\int_{B_j}g^{\frac{p}{p-1}}\,\mathrm {d}x\right)
  +\gamma \frac{1}{d_j(\bar l)^{p-1}}\left(r_j^{p-n}\int_{B_j}|f|\,\mathrm {d}x\right)
  \\& \leq \frac{\gamma }{[\tilde B(\alpha_{j-1}-\alpha_j)]^p}\left(r_j^{p-n}\int_{B_j}g^{\frac{p}{p-1}}\,\mathrm {d}x\right)
\\& \ \  +\frac{\gamma }{[\tilde B(\alpha_{j-1}-\alpha_j)]^{p-1}}\left(r_j^{p-n}\int_{B_j}|f|\,\mathrm {d}x\right)
  \\&\leq \gamma(\tilde B^{-p}+\tilde B^{-(p-1)}),
  \end{split}\end{equation}
  where the constant $\gamma$ depends only upon the data and $\epsilon_1$. Consequently, we infer that there exists a constant $\gamma_1=\gamma_1(\epsilon_1,\text{data})$
  such that
   \begin{equation*}\begin{split}
   \frac{1}{r_j^n}&\esssup_{\hat t<t<t_1}\int_{U^{\prime\prime}_j(t)}\frac{l_j-u}{d_j(\bar l)}\phi_j^{k_1h_1}\,\mathrm {d}x
   \leq \gamma_1\chi+\gamma_1(\tilde B^{-p}+\tilde B^{-(p-1)}).
    \end{split}\end{equation*}
Similarly, we have
      \begin{equation*}\begin{split}
   \frac{1}{r_j^n}&\esssup_{\hat t<t<t_1}\int_{U^{\prime\prime}_j(t)}\frac{l_j-u}{d_j(\bar l)}\phi_j^{k_2h_2}\,\mathrm {d}x
   \leq \gamma_1\chi+\gamma_1(\tilde B^{-p}+\tilde B^{-(p-1)}).
    \end{split}\end{equation*}
    Furthermore, we
proceed to estimate the right-hand side of \eqref{T2estimatexi} by
    \begin{equation*}\begin{split}
    &\frac{d_j(\bar l)^{p-2}}{r_j^{n}}
 \iint_{U_j}|Dv_p|^p\,\mathrm {d}x\mathrm {d}t
+a_1\frac{d_j(\bar l)^{q-2}}{r_j^{n}}
 \iint_{U_j}|Dv_q|^q\,\mathrm {d}x\mathrm {d}t
 \\ \leq &\left[\frac{d_j(\bar l)^{p-2}}{r_j^{n}}
 \iint_{U_j}\phi_j^{k_1p}|D\psi_p|^p\,\mathrm {d}x\mathrm {d}t\right.\\&\left.+a_1\frac{d_j(\bar l)^{q-2}}{r_j^{n}}
 \iint_{U_j}\phi_j^{k_2q}|D\psi_q|^q\,\mathrm {d}x\mathrm {d}t\right]
 \\&+ \left[\frac{d_j(\bar l)^{p-2}}{r_j^{n}}
 \iint_{U_j}\phi_j^{(k_1-1)p}\psi_p^p|D\phi_j|^p\,\mathrm {d}x\mathrm {d}t\right.
 \\&\left.+a_1\frac{d_j(\bar l)^{q-2}}{r_j^{n}}
 \iint_{U_j}\phi_j^{(k_2-1)q}\psi_q^q|D\phi_j|^q\,\mathrm {d}x\mathrm {d}t\right]
  =:T_7+T_8,
  \end{split}\end{equation*}
  since $v_p=\psi_p\phi_j^{k_1}$ and $v_q=\psi_q\phi_j^{k_2}$. To estimate $T_7$, we set $M_1=\min\left\{k_1p,k_2q\right\}$. Recalling that
  $\frac{4}{5}a_1\leq a(x,t)$ for all $(x,t)\in Q_j(\bar l)$, we apply Lemma \ref{Cac1} with $(l,d,\Theta,\varphi)$ replaced by
  $(l_j,d_j(\bar l),t_1-\hat t,\phi_j)$. Analysis similar to that in the proof of \eqref{Cacphi}-\eqref{T5T6xi} shows that
 \begin{equation*}\begin{split} T_7 \leq &\frac{5}{4r_j^{n}}\iint_{U_j} \left(d_j(\bar l)^{p-2}
 |D\psi_p|^p+a(x,t)d_j(\bar l)^{q-2}|D\psi_q|^q\right)\phi_j^{M_1}\,\mathrm {d}x\mathrm {d}t
\\ \leq&\gamma  \frac{d_j(\bar l)^{p-2}}{r_j^{n+p}}\iint_{U_j}\left(\frac{l_j-u}{d_j(\bar l)}\right)^{(1+\lambda)(p-1)}
  \phi_j^{M_1-p}\,\mathrm {d}x\mathrm {d}t
  \\&+\gamma  \frac{d_j(\bar l)^{q-2}}{r_j^{n+q}}\iint_{U_j}a(x,t)\left(\frac{l_j-u}{d_j(\bar l)}\right)^{(1+\lambda)(q-1)}
  \phi_j^{M_1-q}\,\mathrm {d}x\mathrm {d}t
\\&+\gamma \frac{t_1-\hat t}{r_j^nd_j(\bar l)^2}\int_{B_j}g^{\frac{p}{p-1}}\,\mathrm {d}x
  +\gamma \frac{t_1-\hat t}{r_j^nd_j(\bar l)}\int_{B_j}|f|\,\mathrm {d}x
  \\ \leq &\gamma\chi+\gamma(\tilde B^{-p}+\tilde B^{-(p-1)}),
   \end{split}\end{equation*}
   where the constant $\gamma$ depends only upon the data. To estimate $T_8$, we observe that
    \begin{equation}\begin{split}\label{psippsiq}\psi_p\leq \gamma d_j(\bar l)^{-\frac{1}{p^\prime}}(l_j-u)_+^{\frac{1}{p^\prime}}\qquad\text{and}\qquad
    \psi_q\leq \gamma d_j(\bar l)^{-\frac{1}{q^\prime}}(l_j-u)_+^{\frac{1}{q^\prime}}
    \end{split}\end{equation}
    where $p^\prime=\tfrac{p}{p-1}$ and $q^\prime=\tfrac{q}{q-1}$. At this point, we use
    Young's inequality, \eqref{psippsiq}, \eqref{firstestimate for setU} and \eqref{T3T4xi} to obtain
    \begin{equation}\begin{split}\label{estimateforT8}
     T_8\leq &\frac{d_j(\bar l)^{p-2}}{r_j^{n+p}}\iint_{U_j}\left(\frac{l_j-u}{d_j(\bar l)}\right)^{p-1}\phi_j^{(k_1-1)p}\,\mathrm {d}x\mathrm {d}t
\\&+a_1\frac{d_j(\bar l)^{q-2}}{r_j^{n+q}}
 \iint_{U_j}\left(\frac{l_j-u}{d_j(\bar l)}\right)^{q-1}\phi_j^{(k_2-1)q}\,\mathrm {d}x\mathrm {d}t
 \\ \leq & \left(\frac{d_j(\bar l)^{p-2}}{r_j^{n+p}}+a_1\frac{d_j(\bar l)^{q-2}}{r_j^{n+q}}\right)|U_j|
\\ &+\frac{d_j(\bar l)^{p-2}}{r_j^{n+p}}\iint_{U_j}\left(\frac{l_j-u}{d_j(\bar l)}\right)^{(p-1)(1+\lambda)}\phi_j^{(k_1-1)p}\,\mathrm {d}x\mathrm {d}t
\\&+\frac{d_j(\bar l)^{q-2}}{r_j^{n+q}}
 \iint_{U_j}a(x,t)\left(\frac{l_j-u}{d_j(\bar l)}\right)^{(q-1)(1+\lambda)}\phi_j^{(k_2-1)q}\,\mathrm {d}x\mathrm {d}t
 \\ \leq &\gamma\chi,
  \end{split}\end{equation}
  since $\frac{4}{5}a_1\leq a(x,t)$ for all $(x,t)\in Q_j(\bar l)$. Combining the above estimates, we arrive at
   \begin{equation*}\begin{split}
 T_2\leq c_1c_2\left[(\chi+\tilde B^{-p}+\tilde B^{-(p-1)})^{1+\frac{p}{n}}+(\chi+\tilde B^{-p}+\tilde B^{-(p-1)})^{1+\frac{q}{n}}\right],
  \end{split}\end{equation*}
  where $c_1=c_1(\epsilon_1,\text{data})$ and $c_2=c_2(\epsilon_2,\text{data})$. Inserting this inequality into \eqref{Uprimeprimexi}, we conclude that
   \begin{equation}\begin{split}\label{Uprimeprimexiq<n}
&\frac{d_j(\bar l)^{p-2}}{r_j^{n+p}}\iint_{U^{\prime\prime}_j}\left(\frac{l_j-u}{d_j(\bar l)}\right)^{(1+\lambda)(p-1)}\phi_j^{M-p}
 \,\mathrm {d}x\mathrm {d}t
 \\&+\frac{d_j(\bar l)^{q-2}}{r_j^{n+q}}\iint_{U^{\prime\prime}_j}a(x,t)\left(\frac{l_j-u}{d_j(\bar l)}\right)^{(1+\lambda)(q-1)}\phi_j^{M-q}
 \,\mathrm {d}x\mathrm {d}t
 \\ &\leq 4^n\epsilon_2\chi+c_1c_2\left[(\chi+\tilde B^{-p}+\tilde B^{-(p-1)})^{1+\frac{p}{n}}+(\chi+\tilde B^{-p}+\tilde B^{-(p-1)})^{1+\frac{q}{n}}\right].
 \end{split}\end{equation}

 \S 1.1.2. We now consider the case $q\geq n$. To this end, we first consider the second term on the left-hand side of \eqref{Uprimeprimexi}.
 Recalling that $\frac{4}{5}a_1\leq a(x,t)\leq\frac{6}{5}a_1$ for all $(x,t)\in Q_j(\bar l)$, we infer from \eqref{lemmainequalitypsi-} that
 \begin{equation}\begin{split}\label{Uprimeprimexiq>n1st}
 &\frac{d_j(\bar l)^{q-2}}{r_j^{n+q}}\iint_{U^{\prime\prime}_j}a(x,t)\left(\frac{l_j-u}{d_j(\bar l)}\right)^{(1+\lambda)(q-1)}\phi_j^{M-q}
 \,\mathrm {d}x\mathrm {d}t
 \\&\leq  \frac{6}{5}a_1\frac{d_j(\bar l)^{q-2}}{r_j^{n+q}}\iint_{U^{\prime\prime}_j}\left(\frac{l_j-u}{d_j(\bar l)}\right)^{q-1-\lambda+\lambda q}\phi_j^{M-q}
 \,\mathrm {d}x\mathrm {d}t
  \\&\leq c(\epsilon_1)a_1\frac{d_j(\bar l)^{q-2}}{r_j^{n+q}}\iint_{U^{\prime\prime}_j}\psi_q^q\left(\frac{l_j-u}{d_j(\bar l)}\right)^{\frac{p}{n}}\phi_j^{M-q}
 \,\mathrm {d}x\mathrm {d}t,
  \end{split}\end{equation}
  since $\lambda=\frac{p}{nq}$. Since $q\geq n$, we conclude from Sobolev's inequality that
  \begin{equation*}\begin{split}
  \left(\int_{B_j}\left(\psi_q\phi_j^{M_2}\right)^{\nu_1q}\,\mathrm {d}x\right)^\frac{1}{\nu_1q}
  \leq \gamma r_j^{1-\frac{n}{q}+\frac{n}{\nu_1q}}\left(\int_{B_j}|D(\psi_q\phi_j^{M_2})|^q\,\mathrm {d}x\right)^\frac{1}{q},
   \end{split}\end{equation*}
   where $\nu_1=\frac{n}{n-p}$ and $M_2=\frac{\frac{1}{2}M-q}{q}$. We use the inequality above and H\"older's inequality to find that
    \begin{equation}\begin{split}\label{Uprimeprimexiq>n2nd}
&a_1\frac{d_j(\bar l)^{q-2}}{r_j^{n+q}}\iint_{U^{\prime\prime}_j}\psi_q^q\left(\frac{l_j-u}{d_j(\bar l)}\right)^{\frac{p}{n}}\phi_j^{M-q}
 \,\mathrm {d}x\mathrm {d}t
 \\& =a_1\frac{d_j(\bar l)^{q-2}}{r_j^{n+q}}\iint_{U^{\prime\prime}_j}\left(\psi_q\phi_j^{M_2}\right)^q\cdot\left(\frac{l_j-u}{d_j(\bar l)}\right)^{\frac{p}{n}}\phi_j^{\frac{M}{2}}
 \,\mathrm {d}x\mathrm {d}t
  \\& \leq a_1\frac{d_j(\bar l)^{q-2}}{r_j^{n+q}}\int_{\hat t}^{t_1}\left(\int_{B_j}\left(\psi_q\phi_j^{M_2}\right)^{q\nu_1} \,\mathrm {d}x\right)^\frac{1}{\nu_1}
\\&\ \ \ \ \times\left( \int_{U^{\prime\prime}_j(t)}\frac{l_j-u}{d_j(\bar l)}\phi_j^{\frac{nM}{2p}}
 \,\mathrm {d}x\right)^\frac{p}{n}\mathrm {d}t
   \\& \leq a_1\frac{d_j(\bar l)^{q-2}}{r_j^{n+q}}r_j^{q-n+\frac{n}{\nu_1}}\int_{\hat t}^{t_1}\int_{B_j}|D(\psi_q\phi_j^{M_2})|^q \,\mathrm {d}x\mathrm {d}t
\\&\ \ \ \ \times\esssup_{\hat t<t<t_1}\left( \int_{U^{\prime\prime}_j(t)}\frac{l_j-u}{d_j(\bar l)}\phi_j^{\frac{nM}{2p}}
 \,\mathrm {d}x\right)^\frac{p}{n}
  \\& = \left[a_1\frac{d_j(\bar l)^{q-2}}{r_j^{n}}\iint_{U_j}|D(\psi_q\phi_j^{M_2})|^q \,\mathrm {d}x\mathrm {d}t\right]
\\&\ \ \ \ \times\esssup_{\hat t<t<t_1}\left(\frac{1}{r_j^n} \int_{U^{\prime\prime}_j(t)}\frac{l_j-u}{d_j(\bar l)}\phi_j^{\frac{nM}{2p}}
 \,\mathrm {d}x\right)^\frac{p}{n}
 \\&=:T_9\cdot T_{10},
  \end{split}\end{equation}
  since $q-n+\frac{n}{\nu_1}=q-p$. The proof of the estimates of $T_9$ and $T_{10}$ follows in a similar manner as the arguments in \S 1.1.1, and we get
  \begin{equation}\begin{split}\label{Uprimeprimexiq>n3rd}
  T_9\leq \gamma(\chi+\tilde B^{-p}+\tilde B^{-(p-1)})\qquad\text{and}\qquad T_{10}\leq \gamma_1(\chi+\tilde B^{-p}+\tilde B^{-(p-1)})^{\frac{p}{n}},
   \end{split}\end{equation}
   where $\gamma=\gamma(\text{data})$ and $\gamma_1=\gamma_1(\epsilon_1,\text{data})$. Combining the estimates \eqref{Uprimeprimexiq>n1st}-\eqref{Uprimeprimexiq>n3rd}, we conclude that
   \begin{equation}\begin{split}\label{Uprimeprimexiq>nqintegral}
 &\frac{d_j(\bar l)^{q-2}}{r_j^{n+q}}\iint_{U^{\prime\prime}_j}a(x,t)\left(\frac{l_j-u}{d_j(\bar l)}\right)^{(1+\lambda)(q-1)}\phi_j^{M-q}
 \,\mathrm {d}x\mathrm {d}t
 \leq  c(\chi+\tilde B^{-p}+\tilde B^{-(p-1)})^{1+\frac{p}{n}},
  \end{split}\end{equation}
  where $c=c(\epsilon_1,\text{data})$. The first term on the left-hand side of \eqref{Uprimeprimexi} can be treated by the same way as in \S 1.1.1.
  More precisely, we get
  \begin{equation}\begin{split}\label{Uprimeprimexiq>npintegral}
 &\frac{d_j(\bar l)^{p-2}}{r_j^{n+p}}\iint_{U^{\prime\prime}_j}\left(\frac{l_j-u}{d_j(\bar l)}\right)^{(1+\lambda)(p-1)}\phi_j^{M-p}
 \,\mathrm {d}x\mathrm {d}t
 \\ \leq &\epsilon_2\frac{d_j(\bar l)^{p-2}}{r_j^{n+p}}|U_j|
+c(\epsilon_2)\frac{d_j(\bar l)^{p-2}}{r_j^{n+p}}\iint_{U^{\prime\prime}_j}\left(\frac{l_j-u}{d_j(\bar l)}\right)^{p\frac{n+h_1}{nh_1}}\phi_j^{(M-p)p_*}
 \,\mathrm {d}x\mathrm {d}t
 \\ \leq& 4^n\epsilon_2\chi+
 c(\epsilon_2)\frac{d_j(\bar l)^{p-2}}{r_j^{n}}\esssup_{\hat t<t<t_1}\left(\frac{1}{r_j^n}\int_{U^{\prime\prime}_j(t)}\frac{l_j-u}{d_j(\bar l)}\phi_j^{k_1h_1}\,\mathrm {d}x\right)^\frac{p}{n}
 \iint_{U_j}|Dv_p|^p\,\mathrm {d}x\mathrm {d}t
  \\ \leq& 4^n\epsilon_2\chi+c_3c_4(\chi+\tilde B^{-p}+\tilde B^{-(p-1)})^{1+\frac{p}{n}},
  \end{split}\end{equation}
  where $c_3=c_3(\epsilon_1,\text{data})$ and $c_4=c_4(\epsilon_2,\text{data})$.

Thus, we can summarize what we have proved so far.
  Combining \eqref{Uprimexi}, \eqref{Uprimeprimexiq<n}, \eqref{Uprimeprimexiq>npintegral} and \eqref{Uprimeprimexiq>npintegral}, we conclude that there exist constants
  $\gamma_1=\gamma_1(\epsilon_1,\text{data})$ and $\gamma_2=\gamma_2(\epsilon_2,\text{data})$, such that the inequality
  \begin{equation}\begin{split}\label{Upqphase}
  &\frac{d_j(\bar l)^{p-2}}{r_j^{n+p}}\iint_{U_j}\left(\frac{l_j-u}{d_j(\bar l)}\right)^{(1+\lambda)(p-1)}\phi_j^{M-p}
 \,\mathrm {d}x\mathrm {d}t
 \\&+\frac{d_j(\bar l)^{q-2}}{r_j^{n+q}}\iint_{U_j}a(x,t)\left(\frac{l_j-u}{d_j(\bar l)}\right)^{(1+\lambda)(q-1)}\phi_j^{M-q}
 \,\mathrm {d}x\mathrm {d}t
 \\
 \leq &\epsilon_1^{(1+\lambda)(p-1)}4^n\chi+
4^n\epsilon_2\chi
\\&+\gamma_1\gamma_2\left[(\chi+\tilde B^{-p}+\tilde B^{-(p-1)})^{1+\frac{p}{n}}+(\chi+\tilde B^{-p}+\tilde B^{-(p-1)})^{1+\frac{q}{n}}\right]
  \end{split}\end{equation}
  holds for the $(p,q)$-phase.

  \S 1.2. In the case of $p$-phase, i.e., $a_1<10[a]_\alpha r_j^\alpha$. We observe from $Q_j(\bar l)\subseteq Q_{r_j,r_j^2}(x_1,t_1)$
   that $a(x,t)\leq 12[a]_\alpha r_j^\alpha$ for all $(x,t)\in Q_j(\bar l)$. Noting that $d_j(\bar l)=l_j-\bar l\leq l_j\leq l_0=\xi\omega\leq \|u\|_\infty$ and $r_j^{\alpha-q}\leq r_j^{-p}$,
   we infer from \eqref{firstestimate for setU} that
    \begin{equation}\begin{split}\label{Uprimexipphase}
&\frac{d_j(\bar l)^{p-2}}{r_j^{n+p}}\iint_{U^\prime_j}\left(\frac{l_j-u}{d_j(\bar l)}\right)^{(1+\lambda)(p-1)}\phi_j^{M-p}
 \,\mathrm {d}x\mathrm {d}t
 \\&+\frac{d_j(\bar l)^{q-2}}{r_j^{n+q}}\iint_{U^\prime_j}a(x,t)\left(\frac{l_j-u}{d_j(\bar l)}\right)^{(1+\lambda)(q-1)}\phi_j^{M-q}
 \,\mathrm {d}x\mathrm {d}t
 \\&\leq  (1+12[a]_\alpha\|u\|_\infty^{q-p})\frac{d_j(\bar l)^{p-2}}{r_j^{n+p}}\iint_{U^\prime_j}\left(\frac{l_j-u}{d_j(\bar l)}\right)^{(1+\lambda)(p-1)}\phi_j^{M-q}
 \,\mathrm {d}x\mathrm {d}t
 \\&\leq \gamma\epsilon_1^{(1+\lambda)(p-1)}\frac{d_j(\bar l)^{p-2}}{r_j^{n+p}}|U_j|
 \leq \gamma\epsilon_1^{(1+\lambda)(p-1)}\chi,
 \end{split}\end{equation}
 where the constant $\gamma$ depends only upon the data. The next step is to estimate the integrals on the set $U^{\prime\prime}_j$.
 Analysis similar to that in the proof of \eqref{Uprimeprimexiq>npintegral} shows that
   \begin{equation}\begin{split}\label{Uprimeprimexipphasepintegral}
 &\frac{d_j(\bar l)^{p-2}}{r_j^{n+p}}\iint_{U^{\prime\prime}_j}\left(\frac{l_j-u}{d_j(\bar l)}\right)^{(1+\lambda)(p-1)}\phi_j^{M-p}
 \,\mathrm {d}x\mathrm {d}t
 \leq 4^n\epsilon_2\chi+c_3c_4(\chi+\tilde B^{-p}+\tilde B^{-(p-1)})^{1+\frac{p}{n}},
  \end{split}\end{equation}
  where $c_3=c_3(\epsilon_1,\text{data})$ and $c_4=c_4(\epsilon_2,\text{data})$. Recalling that $a(x,t)\leq 12[a]_\alpha r_j^\alpha$ for all $(x,t)\in Q_j(\bar l)$, we use $r_j^{\alpha-q}\leq r_j^{-p}$ and Lemma \ref{inequalitypsi-} to deduce
   \begin{equation}\begin{split}\label{Uprimeprimepphase1st}
 &\frac{d_j(\bar l)^{q-2}}{r_j^{n+q}}\iint_{U^{\prime\prime}_j}a(x,t)\left(\frac{l_j-u}{d_j(\bar l)}\right)^{(1+\lambda)(q-1)}\phi_j^{M-q}
 \,\mathrm {d}x\mathrm {d}t
\\& \leq 12[a]_\alpha \frac{d_j(\bar l)^{q-2}}{r_j^{n+p}}\iint_{U^{\prime\prime}_j}\left(\frac{l_j-u}{d_j(\bar l)}\right)^{p-1-\lambda+(q-p)+\lambda q}\phi_j^{M-q}
 \,\mathrm {d}x\mathrm {d}t
 \\& \leq c(\epsilon_1) \frac{d_j(\bar l)^{q-2}}{r_j^{n+p}}\iint_{U^{\prime\prime}_j}\psi_p^p
 \left(\frac{l_j-u}{d_j(\bar l)}\right)^{q-p}
 \left(\frac{l_j-u}{d_j(\bar l)}\right)^{\frac{p}{n}}\phi_j^{M-q}
 \,\mathrm {d}x\mathrm {d}t,
 \end{split}\end{equation}
since $\lambda q=\frac{p}{n}$. Noting that $l_j-u\leq l_j\leq l_0=\xi\omega\leq \|u\|_\infty$ on $U_j$, we proceed to estimate the integral on the right-hand side of \eqref{Uprimeprimepphase1st}
by
\begin{equation}\begin{split}\label{Uprimeprimepphase2st}
 & \frac{d_j(\bar l)^{q-2}}{r_j^{n+p}}\iint_{U^{\prime\prime}_j}\psi_p^p
 \left(\frac{l_j-u}{d_j(\bar l)}\right)^{q-p}
 \left(\frac{l_j-u}{d_j(\bar l)}\right)^{\frac{p}{n}}\phi_j^{M-q}
 \,\mathrm {d}x\mathrm {d}t
 \\&\leq \|u\|_\infty^{q-p}\frac{d_j(\bar l)^{p-2}}{r_j^{n+p}}\iint_{U^{\prime\prime}_j}\psi_p^p
 \left(\frac{l_j-u}{d_j(\bar l)}\right)^{\frac{p}{n}}\phi_j^{M-q}
 \,\mathrm {d}x\mathrm {d}t
  \\&\leq \|u\|_\infty^{q-p}\frac{d_j(\bar l)^{p-2}}{r_j^{n+p}}\int_{\hat t}^{t_1}\left(\int_{B_j}\psi_p^{\frac{np}{n-p}}\phi_j^{\frac{n}{n-p}(\frac{1}{2}M-q)}\,\mathrm {d}x\right)^\frac{n-p}{n}
  \\&\ \ \times
\left( \int_{U^{\prime\prime}_j(t)}\frac{l_j-u}{d_j(\bar l)}\phi_j^{\frac{Mn}{2p}}
 \,\mathrm {d}x\right)^{\frac{p}{n}}\mathrm {d}t.
 \end{split}\end{equation}
In the last line, we have used H\"older's inequality with exponents $\frac{n}{p}$ and $\frac{n}{n-p}$. To proceed further, we set $\tilde v_p=\psi_p\phi_j^{M_3}$, where $M_3=\frac{\frac{1}{2}M-q}{p}$.
Inserting the inequality \eqref{Uprimeprimepphase2st} into \eqref{Uprimeprimepphase1st}, we apply Sobolev's inequality slice-wise for $\tilde v_p(\cdot,t)$ to deduce
\begin{equation}\begin{split}\label{Uprimeprimepphase3rd}
 &\frac{d_j(\bar l)^{q-2}}{r_j^{n+q}}\iint_{U^{\prime\prime}_j}a(x,t)\left(\frac{l_j-u}{d_j(\bar l)}\right)^{(1+\lambda)(q-1)}\phi_j^{M-q}
 \,\mathrm {d}x\mathrm {d}t
\\& \leq c(\epsilon_1)\esssup_{\hat  t<t<t_1}\left(\frac{1}{r_j^n} \int_{U^{\prime\prime}_j(t)}\frac{l_j-u}{d_j(\bar l)}\phi_j^{\frac{Mn}{2p}}
 \,\mathrm {d}x\right)^{\frac{p}{n}}
 \\&\ \ \times\left[\frac{d_j(\bar l)^{p-2}}{r_j^n}\iint_{U_j}|D\tilde v_p|^p\,\mathrm {d}x\mathrm {d}t\right]
 \\ &=:c(\epsilon_1)T_{11}\cdot T_{12}.
 \end{split}\end{equation}
Furthermore, we follow the arguments in \S 1.1.1 to infer that
  \begin{equation}\begin{split}\label{T11T12pphase}
  T_{11}\leq \gamma_1(\chi+\tilde B^{-p}+\tilde B^{-(p-1)})^{\frac{p}{n}}\qquad\text{and}\qquad T_{12}\leq \gamma(\chi+\tilde B^{-p}+\tilde B^{-(p-1)}),
   \end{split}\end{equation}
    where $\gamma=\gamma(\text{data})$ and $\gamma_1=\gamma_1(\epsilon_1,\text{data})$. Consequently, we infer from \eqref{Uprimexipphase}, \eqref{Uprimeprimexipphasepintegral},
    \eqref{Uprimeprimepphase3rd} and \eqref{T11T12pphase} that
    \begin{equation}\begin{split}\label{Upphasexi}
&\frac{d_j(\bar l)^{p-2}}{r_j^{n+p}}\iint_{U_j}\left(\frac{l_j-u}{d_j(\bar l)}\right)^{(1+\lambda)(p-1)}\phi_j^{M-p}
 \,\mathrm {d}x\mathrm {d}t
 \\&+\frac{d_j(\bar l)^{q-2}}{r_j^{n+q}}\iint_{U_j}a(x,t)\left(\frac{l_j-u}{d_j(\bar l)}\right)^{(1+\lambda)(q-1)}\phi_j^{M-q}
 \,\mathrm {d}x\mathrm {d}t
 \\&
 \leq \gamma\epsilon_1^{(1+\lambda)(p-1)}\chi+4^n\epsilon_2\chi+\gamma_3\gamma_4(\chi+\tilde B^{-p}+\tilde B^{-(p-1)})^{1+\frac{p}{n}},
 \end{split}\end{equation}
 where $\gamma=\gamma(\text{data})$, $\gamma_3=\gamma_3(\epsilon_1,\text{data})$ and $\gamma_4=\gamma_4(\epsilon_2,\text{data})$.

 At this point, we can summarize what we have proved from \S 1.1 to \S 1.2. Combining \eqref{Upqphase} and \eqref{Upphasexi}, we conclude that
 \begin{equation}\begin{split}\label{estimateforU}
  &\frac{d_j(\bar l)^{p-2}}{r_j^{n+p}}\iint_{U_j}\left(\frac{l_j-u}{d_j(\bar l)}\right)^{(1+\lambda)(p-1)}\phi_j^{M-p}
 \,\mathrm {d}x\mathrm {d}t
 \\&+\frac{d_j(\bar l)^{q-2}}{r_j^{n+q}}\iint_{U_j}a(x,t)\left(\frac{l_j-u}{d_j(\bar l)}\right)^{(1+\lambda)(q-1)}\phi_j^{M-q}
 \,\mathrm {d}x\mathrm {d}t
 \\
 \leq &\epsilon_1^{(1+\lambda)(p-1)}\gamma\chi+
4^n\epsilon_2\chi
\\&+\gamma_5\gamma_6\left[(\chi+\tilde B^{-p}+\tilde B^{-(p-1)})^{1+\frac{p}{n}}+(\chi+\tilde B^{-p}+\tilde B^{-(p-1)})^{1+\frac{q}{n}}\right],
  \end{split}\end{equation}
  where $\gamma=\gamma(\text{data})$, $\gamma_5=\max\{\gamma_1,\gamma_3\}$ and $\gamma_6
  =\max\{\gamma_2,\gamma_4\}$. Furthermore, we apply Caccioppoli inequality \eqref{Cacformula1}, \eqref{T5T6xi} and \eqref{estimateforU} to deduce
   \begin{equation}\begin{split}\label{CacphiforG}
& \frac{1}{r_j^n}\esssup_{\hat t<t<t_1}
   \int_{B_j}G\left(\frac{l_j-u}{d_j(\bar l)}\right)\phi_j^{M}\,\mathrm {d}x
   \\ \leq&\gamma  \frac{d_j(\bar l)^{p-2}}{r_j^{n+p}}\iint_{U_j}\left(\frac{l_j-u}{d_j(\bar l)}\right)^{(1+\lambda)(p-1)}
  \phi_j^{M-p}\,\mathrm {d}x\mathrm {d}t
  \\&+\gamma  \frac{d_j(\bar l)^{q-2}}{r_j^{n+q}}\iint_{U_j}a(x,t)\left(\frac{l_j-u}{d_j(\bar l)}\right)^{(1+\lambda)(q-1)}
  \phi_j^{M-q}\,\mathrm {d}x\mathrm {d}t
\\&+\gamma \frac{t_1-\hat t}{r_j^nd_j(\bar l)^2}\int_{B_j}g^{\frac{p}{p-1}}\,\mathrm {d}x
  +\gamma \frac{t_1-\hat t}{r_j^nd_j(\bar l)}\int_{B_j}|f|\,\mathrm {d}x
  \\ \leq&\epsilon_1^{(1+\lambda)(p-1)}\gamma\chi+
\epsilon_2\gamma\chi+\gamma(\tilde B^{-p}+\tilde B^{-(p-1)})
\\&+\gamma_5\gamma_6\left[(\chi+\tilde B^{-p}+\tilde B^{-(p-1)})^{1+\frac{p}{n}}+(\chi+\tilde B^{-p}+\tilde B^{-(p-1)})^{1+\frac{q}{n}}\right],
   \end{split}\end{equation}
 where $\gamma=\gamma(\text{data})$, $\gamma_5=\gamma_5(\epsilon_1,\text{data})$ and $\gamma_6=\gamma_6(\epsilon_2,\text{data})$.
 Combining \eqref{estimateforU} and \eqref{CacphiforG}, we infer from  \eqref{Aj2nd} that there exist constants $\tilde\gamma_0=\tilde\gamma_0
 (\text{data})$, $\tilde\gamma_1=\tilde\gamma_1(\epsilon_1,\text{data})$ and $\tilde\gamma_2=\tilde\gamma_2(\epsilon_2,\text{data})$, such that
 \begin{equation}\begin{split}\label{Ajbarlxi}
 A_j(\bar l) \leq&\epsilon_1^{(1+\lambda)(p-1)}\tilde\gamma_0\chi+
\epsilon_2\tilde\gamma_0\chi+\tilde\gamma_0(\tilde B^{-p}+\tilde B^{-(p-1)})
\\&+\tilde\gamma_1\tilde\gamma_2\left[\chi^{1+\frac{p}{n}}+(\tilde B^{-p}+\tilde B^{-(p-1)})^{1+\frac{p}{n}}+\chi^{1+\frac{q}{n}}+(\tilde B^{-p}+\tilde B^{-(p-1)})^{1+\frac{q}{n}}\right],
   \end{split}\end{equation}
   since $A_j(\bar l)= A_j^{(2)}(\bar l)$ in the case $Q_j(\bar l)\nsubseteq\widetilde Q_1$. This establishes an upper bound for $A_j(\bar l)$ for the case
   $Q_j(\bar l)\nsubseteq\widetilde Q_1$.

   \S 2. We consider the case $Q_j(\bar l)\subseteq\widetilde Q_1$, i.e., $\Theta_j(\bar l)< t_1-\hat t$.
   Note that in this case $A_j(\bar l)= A_j^{(1)}(\bar l)$.
   Our task now is to establish an upper bound for $A_j(\bar l)$ similar to
   \eqref{Ajbarlxi}. For a fixed $\epsilon_1>0$, we decompose $L_j(\bar l)=L_j^\prime(\bar l)\cup L_j^{\prime\prime}(\bar l)$,
   where  \begin{equation*}\begin{split}L_j^\prime(\bar l)=L_j(\bar l)\cap\left\{\frac{l_j-u}{d_j(\bar l)}\leq\epsilon_1\right\} \qquad\text{and}\qquad
   L_j^{\prime\prime}(\bar l)=L_j(\bar l)\setminus L_j^\prime(\bar l).
   \end{split}\end{equation*}
   Moreover, we claim that there exists a constant $\gamma$, depending only on the data, such that
   \begin{equation}\begin{split}\label{Qclaim1}
 \left(\frac{d_j(\bar l)^{p-2}}{r_j^{n+p}}+a_1\frac{d_j(\bar l)^{q-2}}{r_j^{n+q}}\right)|L_j(\bar l)|
 \leq 4^n\chi
  \end{split}\end{equation}
   and
    \begin{equation}\begin{split}\label{Qclaim2}
  &\frac{d_j(\bar l)^{p-2}}{r_j^{n+p}}\iint_{L_j(\bar l)}\left(\frac{l_j-u}{d_j(\bar l)}\right)^{(1+\lambda)(p-1)}
 \,\mathrm {d}x\mathrm {d}t
 \\&+\frac{d_j(\bar l)^{q-2}}{r_j^{n+q}}\iint_{L_j(\bar l)}a(x,t)\left(\frac{l_j-u}{d_j(\bar l)}\right)^{(1+\lambda)(q-1)}
 \,\mathrm {d}x\mathrm {d}t
 \leq \gamma \chi.
  \end{split}\end{equation}
  First, we observe that $u\leq l_j$ on $L_j(\bar l)$ and hence $d_{j-1}^{-1}(l_{j-1}-u)\geq 1$ on $L_j(\bar l)$. To prove the claim, we distinguish two cases:
  $Q_{j-1}(l_j)\nsubseteq \tilde Q_1$ and $Q_{j-1}(l_j)\subseteq \tilde Q_1$. In the case  $Q_{j-1}(l_j)\nsubseteq \tilde Q_1$, we observe that $A_{j-1}(l_j)=A_{j-1}^{(2)}(l_j)$,
  $Q_j(\bar l)\subseteq U_{j-1}$ and $\phi_{j-1}(x)=1$ for all $x\in B_j$. Then, we deduce from \eqref{Aj-1xi} that
     \begin{equation*}\begin{split}
 &\left(\frac{d_j(\bar l)^{p-2}}{r_j^{n+p}}+a_1\frac{d_j(\bar l)^{q-2}}{r_j^{n+q}}\right)|L_j(\bar l)|
\\&\leq \Theta_j(\bar l)^{-1}\frac{1}{r_j^n}\Theta_j(\bar l)\esssup_{t_1-\Theta_j(\bar l)<t<t_1}\int_{B_{j-1}\cap \{
  u(\cdot,t)\leq l_j\}}\phi_{j-1}^M\,\mathrm {d}x
  \\&\leq\frac{4^n}{r_{j-1}^n}\esssup_{\hat t<t<t_1}\int_{B_{j-1}}G\left(\frac{l_{j-1}-u}{d_{j-1}}\right)\phi_{j-1}^M \,\mathrm {d}x
      \\&  \leq 4^nA_{j-1}^{(2)}(l_j)=4^nA_{j-1}(l_j)\leq 4^n\chi,
  \end{split}\end{equation*}
  which proves the claim \eqref{Qclaim1}. Taking into account that $d_j(\bar l)\geq \frac{1}{4}d_{j-1}$, we infer from \eqref{Aj-1xi} that
   \begin{equation*}\begin{split}
  &\frac{d_j(\bar l)^{p-2}}{r_j^{n+p}}\iint_{L_j(\bar l)}\left(\frac{l_j-u}{d_j(\bar l)}\right)^{(1+\lambda)(p-1)}
 \,\mathrm {d}x\mathrm {d}t
 \\&+\frac{d_j(\bar l)^{q-2}}{r_j^{n+q}}\iint_{L_j(\bar l)}a(x,t)\left(\frac{l_j-u}{d_j(\bar l)}\right)^{(1+\lambda)(q-1)}
 \,\mathrm {d}x\mathrm {d}t
 \\&\leq \gamma\frac{d_{j-1}^{p-2}}{r_{j-1}^{n+p}}\iint_{U_{j-1}}\left(\frac{l_{j-1}-u}{d_{j-1}}\right)^{(1+\lambda)(p-1)}\phi_{j-1}^{M-p}
 \,\mathrm {d}x\mathrm {d}t
 \\&\ \ +\gamma\frac{d_{j-1}^{q-2}}{r_{j-1}^{n+q}}\iint_{U_{j-1}}a(x,t)\left(\frac{l_{j-1}-u}{d_{j-1}}\right)^{(1+\lambda)(q-1)}\phi_{j-1}^{M-q}
 \,\mathrm {d}x\mathrm {d}t
 \\&  \leq \gamma A_{j-1}^{(2)}(l_j)=\gamma A_{j-1}(l_j)\leq \gamma\chi,
  \end{split}\end{equation*}
    which proves the claim \eqref{Qclaim2}. We now turn our attention to the case $Q_{j-1}(l_j)\subseteq \tilde Q_1$.
    In  this case, we have $A_{j-1}(l_j)=A_{j-1}^{(1)}(l_j)$.
    Recalling that $d_j(\bar l)\geq \frac{1}{4}d_{j-1}$, we obtain
    \begin{equation*}\begin{split}
    \Theta_j(\bar l)&=d_j(\bar l)^2\left[\left(\frac{d_j(\bar l)}{r_j}\right)^p+a_1\left(\frac{d_j(\bar l)}{r_j}\right)^q\right]^{-1}
    \\&\leq \frac{1}{16}d_{j-1}^2\left[\left(\frac{4^{-1}d_{j-1}}{4^{-1}r_{j-1}}\right)^p+a_1\left(\frac{4^{-1}d_{j-1}}{4^{-1}r_{j-1}}\right)^q\right]^{-1}=\frac{1}{16}\Theta_{j-1}.
     \end{split}\end{equation*}
It follows that $Q_j(\bar l)\subseteq Q_{j-1}(l_j)=:Q_{j-1}$ and $\varphi_{j-1}(x,t)=1$ for any $(x,t)\in Q_j(\bar l)$. Then, we conclude from \eqref{Aj-1xi} that
     \begin{equation*}\begin{split}
 &\left(\frac{d_j(\bar l)^{p-2}}{r_j^{n+p}}+a_1\frac{d_j(\bar l)^{q-2}}{r_j^{n+q}}\right)|L_j(\bar l)|
\\&\leq \Theta_j(\bar l)^{-1}\frac{1}{r_j^n}\Theta_j(\bar l)\esssup_{t_1-\Theta_j(\bar l)<t<t_1}\int_{B_{j-1}\cap \{
  u(\cdot,t)\leq l_j\}}\varphi_{j-1}^M\,\mathrm {d}x
  \\&\leq\frac{4^n}{r_{j-1}^n}\esssup_{t}\int_{B_{j-1}}G\left(\frac{l_{j-1}-u}{d_{j-1}}\right)\varphi_{j-1}^M \,\mathrm {d}x
      \\&  \leq 4^nA_{j-1}^{(1)}(l_j)=4^nA_{j-1}(l_j)\leq 4^n\chi
  \end{split}\end{equation*}
  and
   \begin{equation*}\begin{split}
  &\frac{d_j(\bar l)^{p-2}}{r_j^{n+p}}\iint_{L_j(\bar l)}\left(\frac{l_j-u}{d_j(\bar l)}\right)^{(1+\lambda)(p-1)}
 \,\mathrm {d}x\mathrm {d}t
 \\&+\frac{d_j(\bar l)^{q-2}}{r_j^{n+q}}\iint_{L_j(\bar l)}a(x,t)\left(\frac{l_j-u}{d_j(\bar l)}\right)^{(1+\lambda)(q-1)}
 \,\mathrm {d}x\mathrm {d}t
 \\&\leq \gamma\frac{d_{j-1}^{p-2}}{r_{j-1}^{n+p}}\iint_{Q_{j-1}}\left(\frac{l_{j-1}-u}{d_{j-1}}\right)^{(1+\lambda)(p-1)}\varphi_{j-1}^{M-p}
 \,\mathrm {d}x\mathrm {d}t
 \\&\ \ +\gamma\frac{d_{j-1}^{q-2}}{r_{j-1}^{n+q}}\iint_{Q_{j-1}}a(x,t)\left(\frac{l_{j-1}-u}{d_{j-1}}\right)^{(1+\lambda)(q-1)}\varphi_{j-1}^{M-q}
 \,\mathrm {d}x\mathrm {d}t
 \\&  \leq \gamma A_{j-1}^{(1)}(l_j)=\gamma A_{j-1}(l_j)\leq \gamma\chi,
  \end{split}\end{equation*}
  since $d_j(\bar l)\geq \frac{1}{4}d_{j-1}$. This proves \eqref{Qclaim1} and \eqref{Qclaim2} for the case $Q_{j-1}(l_j)\subseteq \tilde Q_1$.
  In order to derive an upper bound for $A_j(\bar l)$, we shall deal with the estimates in the cases of $p$-phase and $(p,q)$-phase separately.

  \S 2.1. In the case of $(p,q)$-phase, i.e., $a_1\geq10[a]_\alpha r_j^\alpha$. Recalling that $Q_j(\bar l)\subseteq Q_{r_j,r_j^2}(x_1,t_1)$,
   we have $\frac{4}{5}a_1\leq a(x,t)\leq\frac{6}{5}a_1$ for all $(x,t)\in Q_j(\bar l)$. First, we observe from \eqref{Qclaim1} that
    \begin{equation}\begin{split}\label{Qprimexi}
&\frac{d_j(\bar l)^{p-2}}{r_j^{n+p}}\iint_{L^\prime_j(\bar l)}\left(\frac{l_j-u}{d_j(\bar l)}\right)^{(1+\lambda)(p-1)}\varphi_j(\bar l)^{M-p}
 \,\mathrm {d}x\mathrm {d}t
 \\&+\frac{d_j(\bar l)^{q-2}}{r_j^{n+q}}\iint_{L^\prime_j(\bar l)}a(x,t)\left(\frac{l_j-u}{d_j(\bar l)}\right)^{(1+\lambda)(q-1)}\varphi_j(\bar l)^{M-q}
 \,\mathrm {d}x\mathrm {d}t
 \\&\leq \epsilon_1^{(1+\lambda)(p-1)}\left(\frac{d_j(\bar l)^{p-2}}{r_j^{n+p}}+a_1\frac{d_j(\bar l)^{q-2}}{r_j^{n+q}}\right)|L_j(\bar l)|
 \leq \epsilon_1^{(1+\lambda)(p-1)}4^n\chi.
 \end{split}\end{equation}
Our task now is to estimate the integrals on the set $L^{\prime\prime}_j(\bar l)$ for the case of $(p,q)$-phase.
We claim that there exist constants $c_1=c_1(\epsilon_1,\text{data})$ and $c_2=c_2(\epsilon_2,\text{data})$, such that
\begin{equation}\begin{split}\label{claimQprimeprimexi}
&\frac{d_j(\bar l)^{p-2}}{r_j^{n+p}}\iint_{L^{\prime\prime}_j(\bar l)}\left(\frac{l_j-u}{d_j(\bar l)}\right)^{(1+\lambda)(p-1)}\varphi_j(\bar l)^{M-p}
 \,\mathrm {d}x\mathrm {d}t
 \\&+\frac{d_j(\bar l)^{q-2}}{r_j^{n+q}}\iint_{L^{\prime\prime}_j(\bar l)}a(x,t)\left(\frac{l_j-u}{d_j(\bar l)}\right)^{(1+\lambda)(q-1)}\varphi_j(\bar l)^{M-q}
 \,\mathrm {d}x\mathrm {d}t
 \\ &\leq 4^n\epsilon_2\chi+c_1c_2\left[(\chi+\tilde B^{-p}+\tilde B^{-(p-1)})^{1+\frac{p}{n}}+(\chi+\tilde B^{-p}+\tilde B^{-(p-1)})^{1+\frac{q}{n}}\right].
 \end{split}\end{equation}
To prove \eqref{claimQprimeprimexi}, we again distinguish two cases: $q<n$ and $q\geq n$.

\S 2.1.1. In the case $q<n$. For any fixed $\epsilon_2>0$, we apply Young's inequality and \eqref{Qclaim1} to deduce
  \begin{equation}\begin{split}\label{Qprimeprimexi}
&\frac{d_j(\bar l)^{p-2}}{r_j^{n+p}}\iint_{L^{\prime\prime}_j(\bar l)}\left(\frac{l_j-u}{d_j(\bar l)}\right)^{(1+\lambda)(p-1)}\varphi_j(\bar l)^{M-p}
 \,\mathrm {d}x\mathrm {d}t
 \\&+\frac{d_j(\bar l)^{q-2}}{r_j^{n+q}}\iint_{L^{\prime\prime}_j(\bar l)}a(x,t)\left(\frac{l_j-u}{d_j(\bar l)}\right)^{(1+\lambda)(q-1)}\varphi_j(\bar l)^{M-q}
 \,\mathrm {d}x\mathrm {d}t
 \\ \leq &4^n\epsilon_2\chi+c(\epsilon_2)\left[\frac{d_j(\bar l)^{p-2}}{r_j^{n+p}}\iint_{L^{\prime\prime}_j(\bar l)}\left(\frac{l_j-u}{d_j(\bar l)}\right)^{p\frac{n+h_1}{nh_1}}\varphi_j(\bar l)^{(M-p)p_*}
 \,\mathrm {d}x\mathrm {d}t\right.
 \\&+\left. a_1\frac{d_j(\bar l)^{q-2}}{r_j^{n+q}}\iint_{L^{\prime\prime}_j(\bar l)}\left(\frac{l_j-u}{d_j(\bar l)}\right)^{q\frac{n+h_2}{nh_2}}\varphi_j(\bar l)^{(M-q)q_*}
 \,\mathrm {d}x\mathrm {d}t\right]=:4^n\epsilon_2\chi+\tilde T_1,
 \end{split}\end{equation}
  where $h_1=\frac{p}{p-1-\lambda}$, $h_2=\frac{q}{q-1-\lambda}$, $p_*=p\frac{n+h_1}{h_1(1+\lambda)(p-1)n}$ and $q_*=q\frac{n+h_2}{h_2(1+\lambda)(q-1)n}$. Since $p<n$ and $q<n$,
  we obtain similar as in \eqref{T2estimatexi} that
 \begin{equation*}\begin{split}
 \tilde  T_1\leq & c(\epsilon_2)\left[\frac{d_j(\bar l)^{p-2}}{r_j^{n}}\esssup_t\left(\frac{1}{r_j^n}\int_{U^{\prime\prime}_j(t)}\frac{l_j-u}{d_j(\bar l)}\varphi_j(\bar l)^{k_1h_1}\,\mathrm {d}x\right)^\frac{p}{n}
 \iint_{Q_j(\bar l)}|D\bar v_p|^p\,\mathrm {d}x\mathrm {d}t
 \right.
 \\ &+\left.a_1\frac{d_j(\bar l)^{q-2}}{r_j^{n}}\esssup_t\left(\frac{1}{r_j^n}
 \int_{U^{\prime\prime}_j(t)}\frac{l_j-u}{d_j(\bar l)}\varphi_j(\bar l)^{k_2h_2}\,\mathrm {d}x\right)^\frac{q}{n}
 \iint_{Q_j(\bar l)}|D\bar v_q|^q\,\mathrm {d}x\mathrm {d}t
 \right],
  \end{split}\end{equation*}
  where $k_1=\frac{(M-p)np_*}{p(n+h_1)}$, $k_2=\frac{(M-q)nq_*}{q(n+h_2)}$, $\bar v_p=\psi_p\varphi_j(\bar l)^{k_1}$, $\bar v_q=\psi_q\varphi_j(\bar l)^{k_2}$ and the set
$U_j^{\prime\prime}(t)$ is defined in \eqref{defineUprimprime}. For $i=1$ or $2$,
we apply Lemma \ref{Cac1} with $(l,d,\Theta)$ replaced by
  $(l_j,d_j(\bar l),\Theta_j(\bar l))$ to obtain
   \begin{equation}\begin{split}\label{Cacvarphi}
   \frac{1}{r_j^n}&\esssup_t\int_{U^{\prime\prime}_j(t)}\frac{l_j-u}{d_j(\bar l)}\varphi_j(\bar l)^{k_ih_i}\,\mathrm {d}x
   \\ \leq& c(\epsilon_1) \frac{1}{r_j^n}\esssup_t
   \int_{B_j}G\left(\frac{l_j-u}{d_j(\bar l)}\right)\varphi_j(\bar l)^{k_ih_i}\,\mathrm {d}x
   \\ \leq&\gamma  \frac{d_j(\bar l)^{p-2}}{r_j^{n+p}}\iint_{L_j(\bar l)}\left(\frac{l_j-u}{d_j(\bar l)}\right)^{(1+\lambda)(p-1)}
  \varphi_j(\bar l)^{k_ih_i-p}\,\mathrm {d}x\mathrm {d}t
  \\&+\gamma  \frac{d_j(\bar l)^{q-2}}{r_j^{n+q}}\iint_{L_j(\bar l)}a(x,t)\left(\frac{l_j-u}{d_j(\bar l)}\right)^{(1+\lambda)(q-1)}
  \varphi_j(\bar l)^{k_ih_i-q}\,\mathrm {d}x\mathrm {d}t
   \\&+\gamma \frac{1}{r_j^n}\iint_{L_j(\bar l)}\frac{l_j-u}{d_j(\bar l)}|\partial_t\varphi_j(\bar l)|\,\mathrm {d}x\mathrm {d}t
\\&+\gamma \frac{\Theta_j(\bar l)}{r_j^nd_j(\bar l)^2}\int_{B_j}g^{\frac{p}{p-1}}\,\mathrm {d}x
  +\gamma \frac{\Theta_j(\bar l)}{r_j^nd_j(\bar l)}\int_{B_j}|f|\,\mathrm {d}x
  \\ =&:\tilde T_2+\tilde T_3+\tilde T_4+\tilde T_5+\tilde T_6,
   \end{split}\end{equation}
   since $\varphi_j(\bar l)=0$ on $\partial_PQ_j(\bar l)$. As in the proof of \eqref{T5T6xi}, we infer from $\Theta_j(\bar l)\leq d_j(\bar l)^{2-p}r_j^p$ and $d_j(\bar l)\geq \frac{1}{4}\tilde B(\alpha_{j-1}-\alpha_j)$ that $\tilde T_5+\tilde T_6\leq \gamma(\tilde B^{-p}+\tilde B^{-(p-1)})$. In view of \eqref{Qclaim2}, we find that $\tilde T_2+\tilde T_3\leq \gamma\chi$.
   To estimate $\tilde T_4$, we use Young's inequality, \eqref{Qclaim1} and \eqref{Qclaim2} to deduce
    \begin{equation}\begin{split}\label{tildeT4xi}
    \tilde T_4&\leq \gamma \iint_{L_j(\bar l)}\frac{l_j-u}{d_j(\bar l)}\cdot\left(\frac{d_j(\bar l)^{p-2}}{r_j^{n+p}}+a_1\frac{d_j(\bar l)^{q-2}}{r_j^{n+q}}
    \right)\,\mathrm {d}x\mathrm {d}t
    \\&\leq \gamma\left(\frac{d_j(\bar l)^{p-2}}{r_j^{n+p}}+a_1\frac{d_j(\bar l)^{q-2}}{r_j^{n+q}}\right)|L_j(\bar l)|
    \\&\ \ +\gamma\frac{d_j(\bar l)^{p-2}}{r_j^{n+p}}\iint_{L_j(\bar l)}\left(\frac{l_j-u}{d_j(\bar l)}\right)^{(1+\lambda)(p-1)}
 \,\mathrm {d}x\mathrm {d}t
 \\&\ \ +\gamma\frac{d_j(\bar l)^{q-2}}{r_j^{n+q}}\iint_{L_j(\bar l)}a(x,t)\left(\frac{l_j-u}{d_j(\bar l)}\right)^{(1+\lambda)(q-1)}
 \,\mathrm {d}x\mathrm {d}t
 \\&\leq  \gamma\chi,
      \end{split}\end{equation}
      since $|\partial_t\varphi_j(\bar l)|\leq \Theta_j(\bar l)^{-1}$ and $a_1\leq \frac{5}{4}a(x,t)$ for all $(x,t)\in Q_j(\bar l)$. Consequently, we infer that
      there exists a constant $\gamma_1=\gamma_1(\epsilon_1,\text{data})$
  such that
       \begin{equation}\begin{split}\label{Cacvarphi1}
   \frac{1}{r_j^n}&\esssup_t\int_{U^{\prime\prime}_j(t)}\frac{l_j-u}{d_j(\bar l)}\varphi_j(\bar l)^{k_ih_i}\,\mathrm {d}x
 \leq \gamma_1\chi+\gamma_1(\tilde B^{-p}+\tilde B^{-(p-1)})
   \end{split}\end{equation}
   holds for all $i=1,2$. Furthermore, we decompose
    \begin{equation}\begin{split}\label{DvT7T7tilde}
&\frac{d_j(\bar l)^{p-2}}{r_j^{n}}
 \iint_{Q_j(\bar l)}|D\bar v_p|^p\,\mathrm {d}x\mathrm {d}t
+a_1\frac{d_j(\bar l)^{q-2}}{r_j^{n}}
 \iint_{Q_j(\bar l)}|D\bar v_q|^q\,\mathrm {d}x\mathrm {d}t
  \\ \leq &\left[\frac{d_j(\bar l)^{p-2}}{r_j^{n}}
 \iint_{L_j(\bar l)}\varphi_j(\bar l)^{k_1p}|D\psi_p|^p\,\mathrm {d}x\mathrm {d}t\right.
 \\& +\left.a_1\frac{d_j(\bar l)^{q-2}}{r_j^{n}}
 \iint_{L_j(\bar l)}\varphi_j(\bar l)^{k_2q}|D\psi_q|^q\,\mathrm {d}x\mathrm {d}t\right]
 \\&+ \left[\frac{d_j(\bar l)^{p-2}}{r_j^{n}}
 \iint_{L_j(\bar l)}\varphi_j(\bar l)^{(k_1-1)p}\psi_p^p|D\varphi_j(\bar l)|^p\,\mathrm {d}x\mathrm {d}t\right.
 \\&
+\left.a_1\frac{d_j(\bar l)^{q-2}}{r_j^{n}}
 \iint_{L_j(\bar l)}\varphi_j(\bar l)^{(k_2-1)q}\psi_q^q|D\varphi_j(\bar l)|^q\,\mathrm {d}x\mathrm {d}t\right]=:\tilde T_7+\tilde T_8.
  \end{split}\end{equation}
  Recalling that
  $\frac{4}{5}a_1\leq a(x,t)$ for all $(x,t)\in Q_j(\bar l)$, we apply Lemma \ref{Cac1} with $(l,d,\Theta)$ replaced by
  $(l_j,d_j(\bar l),\Theta_j(\bar l))$ to get
  \begin{equation}\begin{split}\label{tildeT7}
 \tilde T_7\leq &\left[\frac{d_j(\bar l)^{p-2}}{r_j^{n}}
 \iint_{L_j(\bar l)}\varphi_j(\bar l)^{k_1p}|D\psi_p|^p\,\mathrm {d}x\mathrm {d}t\right.
 \\& +\left.\frac{d_j(\bar l)^{q-2}}{r_j^{n}}
 \iint_{L_j(\bar l)}a(x,t)\varphi_j(\bar l)^{k_2q}|D\psi_q|^q\,\mathrm {d}x\mathrm {d}t\right]
   \\ \leq&\gamma  \frac{d_j(\bar l)^{p-2}}{r_j^{n+p}}\iint_{L_j(\bar l)}\left(\frac{l_j-u}{d_j(\bar l)}\right)^{(1+\lambda)(p-1)}
 \,\mathrm {d}x\mathrm {d}t
  \\&+\gamma  \frac{d_j(\bar l)^{q-2}}{r_j^{n+q}}\iint_{L_j(\bar l)}a(x,t)\left(\frac{l_j-u}{d_j(\bar l)}\right)^{(1+\lambda)(q-1)}
 \,\mathrm {d}x\mathrm {d}t
   \\&+\gamma \frac{1}{r_j^n}\iint_{L_j(\bar l)}\frac{l_j-u}{d_j(\bar l)}|\partial_t\varphi_j(\bar l)|\,\mathrm {d}x\mathrm {d}t
\\&+\gamma \frac{\Theta_j(\bar l)}{r_j^nd_j(\bar l)^2}\int_{B_j}g^{\frac{p}{p-1}}\,\mathrm {d}x
  +\gamma \frac{\Theta_j(\bar l)}{r_j^nd_j(\bar l)}\int_{B_j}|f|\,\mathrm {d}x
  \\ \leq&\gamma\chi+\gamma(\tilde B^{-p}+\tilde B^{-(p-1)}),
   \end{split}\end{equation}
   where the constant $\gamma$ depends only on the data. As in the proof of \eqref{estimateforT8}, we use \eqref{psippsiq}, \eqref{Qclaim1} and \eqref{Qclaim2} to obtain
   \begin{equation}\begin{split}\label{estimatefortildeT8}
    \tilde T_8\leq & \left(\frac{d_j(\bar l)^{p-2}}{r_j^{n+p}}+a_1\frac{d_j(\bar l)^{q-2}}{r_j^{n+q}}\right)|L_j(\bar l)|
\\ &+\frac{d_j(\bar l)^{p-2}}{r_j^{n+p}}\iint_{L_j(\bar l)}\left(\frac{l_j-u}{d_j(\bar l)}\right)^{(p-1)(1+\lambda)}\,\mathrm {d}x\mathrm {d}t
\\&+\frac{d_j(\bar l)^{q-2}}{r_j^{n+q}}
 \iint_{L_j(\bar l)}a(x,t)\left(\frac{l_j-u}{d_j(\bar l)}\right)^{(q-1)(1+\lambda)}\,\mathrm {d}x\mathrm {d}t
 \\ \leq &\gamma\chi,
  \end{split}\end{equation}
  where $\gamma=\gamma(\text{data})$. Inserting the estimates \eqref{tildeT7} and \eqref{estimatefortildeT8} into \eqref{DvT7T7tilde}, we infer that
   \begin{equation}\begin{split}\label{Dvbar}
&\frac{d_j(\bar l)^{p-2}}{r_j^{n}}
 \iint_{Q_j(\bar l)}|D\bar v_p|^p\,\mathrm {d}x\mathrm {d}t
+a_1\frac{d_j(\bar l)^{q-2}}{r_j^{n}}
 \iint_{Q_j(\bar l)}|D\bar v_q|^q\,\mathrm {d}x\mathrm {d}t
  \\ &\leq \gamma (\chi+\tilde B^{-p}+\tilde B^{-(p-1)}).
  \end{split}\end{equation}
  Combining \eqref{Cacvarphi1} and \eqref{Dvbar}, we arrive at
   \begin{equation*}\begin{split}
 \tilde T_1\leq c_1c_2\left[(\chi+\tilde B^{-p}+\tilde B^{-(p-1)})^{1+\frac{p}{n}}+(\chi+\tilde B^{-p}+\tilde B^{-(p-1)})^{1+\frac{q}{n}}\right],
  \end{split}\end{equation*}
  where $c_1=c_1(\epsilon_1,\text{data})$ and $c_2=c_2(\epsilon_2,\text{data})$. Inserting this inequality into \eqref{Qprimeprimexi}, we conclude that
\eqref{claimQprimeprimexi} holds.

\S 2.1.2. In the case $q\geq n$.
The proof of \eqref{claimQprimeprimexi} follows in a similar manner as the arguments in \S 1.1.2 and we sketch the proof.
To start with, we observe that the first term on the left-hand side of \eqref{claimQprimeprimexi} can be treated by the same way as in \S 2.1.1.
  More precisely, we obtain
   \begin{equation}\begin{split}\label{QQprimeprimexiq>npintegral}
 &\frac{d_j(\bar l)^{p-2}}{r_j^{n+p}}\iint_{L_j^{\prime\prime}(\bar l)}\left(\frac{l_j-u}{d_j(\bar l)}\right)^{(1+\lambda)(p-1)}\varphi_j(\bar l)^{M-p}
 \,\mathrm {d}x\mathrm {d}t
  \\ &\leq 4^n\epsilon_2\chi+c_1c_2(\chi+\tilde B^{-p}+\tilde B^{-(p-1)})^{1+\frac{p}{n}},
  \end{split}\end{equation}
  where $c_1=c_1(\epsilon_1,\text{data})$ and $c_2=c_2(\epsilon_2,\text{data})$. It remains to consider the second term on the left-hand side of \eqref{claimQprimeprimexi}.
 Recalling that $a(x,t)\leq\frac{6}{5}a_1$ for all $(x,t)\in Q_j(\bar l)$, we infer from \eqref{lemmainequalitypsi-} that
 \begin{equation}\begin{split}\label{QQprimeprimexiq>n1st}
 &\frac{d_j(\bar l)^{q-2}}{r_j^{n+q}}\iint_{L_j^{\prime\prime}(\bar l)}a(x,t)\left(\frac{l_j-u}{d_j(\bar l)}\right)^{(1+\lambda)(q-1)}\varphi_j(\bar l)^{M-q}
 \,\mathrm {d}x\mathrm {d}t
  \\&\leq c(\epsilon_1)a_1\frac{d_j(\bar l)^{q-2}}{r_j^{n+q}}\iint_{L_j^{\prime\prime}(\bar l)}\psi_q^q\left(\frac{l_j-u}{d_j(\bar l)}\right)^{\frac{p}{n}}\varphi_j(\bar l)^{M-q}
 \,\mathrm {d}x\mathrm {d}t,
  \end{split}\end{equation}
  since $\lambda=\frac{p}{nq}$.
  Moreover, let $\hat v_q=\psi_q\varphi_j(\bar l)^{M_2}$, where $M_2=\frac{\frac{1}{2}M-q}{q}$.
  Noting that $q\geq n$ and $\varphi_j(\bar l)=0$ on $\partial_PQ_j(\bar l)$, we apply Sobolev's inequality to $\hat v_q(\cdot,t)$ slicewise. This yields that
  \begin{equation*}\begin{split}
  \left(\int_{B_j} \hat v_q(\cdot,t)^{\frac{nq}{n-p}}\,\mathrm {d}x\right)^\frac{n-p}{nq}
  \leq \gamma r_j^{1-\frac{p}{q}}\left(\int_{B_j}|D\hat v_q(\cdot,t)|^q\,\mathrm {d}x\right)^\frac{1}{q}.
   \end{split}\end{equation*}
   We use this inequality and H\"older's inequality to estimate the right-hand side of \eqref{QQprimeprimexiq>n1st} by
    \begin{equation}\begin{split}\label{QQprimeprimexiq>n2nd}
&a_1\frac{d_j(\bar l)^{q-2}}{r_j^{n+q}}\iint_{L_j^{\prime\prime}(\bar l)}\psi_q^q\left(\frac{l_j-u}{d_j(\bar l)}\right)^{\frac{p}{n}}\varphi_j(\bar l)^{M-q}
 \,\mathrm {d}x\mathrm {d}t
  \\& \leq a_1\frac{d_j(\bar l)^{q-2}}{r_j^{n+q}}\int_{t_1-\Theta_j(\bar l)}^{t_1}\left(\int_{B_j}\hat v_q(\cdot,t)^{\frac{nq}{n-p}} \,\mathrm {d}x\right)^\frac{n-p}{n}
\\&\ \ \ \ \times\left( \int_{U^{\prime\prime}_j(t)}\frac{l_j-u}{d_j(\bar l)}\varphi_j(\bar l)^{\frac{nM}{2p}}
 \,\mathrm {d}x\right)^\frac{p}{n}\mathrm {d}t
  \\& \leq \left[a_1\frac{d_j(\bar l)^{q-2}}{r_j^{n}}\iint_{Q_j(\bar l)}|D\hat v_q|^q \,\mathrm {d}x\mathrm {d}t\right]
\\&\ \ \ \ \times\esssup_t\left(\frac{1}{r_j^n} \int_{U^{\prime\prime}_j(t)}\frac{l_j-u}{d_j(\bar l)}\varphi_j(\bar l)^{\frac{nM}{2p}}
 \,\mathrm {d}x\right)^\frac{p}{n}
 \\&\leq \gamma_1(\chi+\tilde B^{-p}+\tilde B^{-(p-1)})^{1+\frac{p}{n}},
  \end{split}\end{equation}
   where $\gamma_1=\gamma_1(\epsilon_1,\text{data})$.
   The estimate in the last line can be proved in much the same way as \eqref{Cacvarphi1} and \eqref{Dvbar}.
   Combining \eqref{QQprimeprimexiq>npintegral}-\eqref{QQprimeprimexiq>n2nd}, we obtain the desired estimate \eqref{claimQprimeprimexi}.

   At this stage, we shall derive an upper bound for $A_j(\bar l)$ in terms of $\chi$ and $\tilde B$ for the $(p,q)$-phase.
   First, we infer from \eqref{Qprimexi} and \eqref{claimQprimeprimexi} that
   \begin{equation}\begin{split}\label{QQxixi}
&\frac{d_j(\bar l)^{p-2}}{r_j^{n+p}}\iint_{L_j(\bar l)}\left(\frac{l_j-u}{d_j(\bar l)}\right)^{(1+\lambda)(p-1)}\varphi_j(\bar l)^{M-p}
 \,\mathrm {d}x\mathrm {d}t
 \\&+\frac{d_j(\bar l)^{q-2}}{r_j^{n+q}}\iint_{L_j(\bar l)}a(x,t)\left(\frac{l_j-u}{d_j(\bar l)}\right)^{(1+\lambda)(q-1)}\varphi_j(\bar l)^{M-q}
 \,\mathrm {d}x\mathrm {d}t
 \\&
 \leq \epsilon_1^{(1+\lambda)(p-1)}4^n\chi+4^n\epsilon_2\chi
 \\&\ \ +c_1c_2\left[(\chi+\tilde B^{-p}+\tilde B^{-(p-1)})^{1+\frac{p}{n}}+(\chi+\tilde B^{-p}+\tilde B^{-(p-1)})^{1+\frac{q}{n}}\right],
 \end{split}\end{equation}
  where $c_1=c_1(\epsilon_1,\text{data})$ and $c_2=c_2(\epsilon_2,\text{data})$. To estimate the third term on the right-hand side of \eqref{Aj1st}, we have to refine the estimate \eqref{tildeT4xi}.
  In view of $|\partial_t\varphi_j(\bar l)|\leq \Theta_j(\bar l)^{-1}$ and $a_1\leq \frac{5}{4}a(x,t)$ for all $(x,t)\in Q_j(\bar l)$, we deduce from
  \eqref{Qclaim1} and \eqref{claimQprimeprimexi} that
   \begin{equation}\begin{split}\label{tildeT4xirefinementpq}
  & \frac{1}{r_j^n}\iint_{L_j(\bar l)}\frac{l_j-u}{d_j(\bar l)}|\partial_t\varphi_j(\bar l)|\varphi_j(\bar l)^{M-1}\,\mathrm {d}x\mathrm {d}t
    \\&\leq \epsilon_1\left(\frac{d_j(\bar l)^{p-2}}{r_j^{n+p}}+a_1\frac{d_j(\bar l)^{q-2}}{r_j^{n+q}}\right)|L_j^\prime(\bar l)|
    \\&\ \ + \epsilon_1^{1-(1+\lambda)(p-1)}\frac{d_j(\bar l)^{p-2}}{r_j^{n+p}}\iint_{L_j^{\prime\prime}(\bar l)}\left(\frac{l_j-u}{d_j(\bar l)}\right)^{(1+\lambda)(p-1)}\varphi_j(\bar l)^{M-p}
 \,\mathrm {d}x\mathrm {d}t
 \\&\ \ +
 \epsilon_1^{1-(1+\lambda)(q-1)}
 \frac{d_j(\bar l)^{q-2}}{r_j^{n+q}}\iint_{L_j^{\prime\prime}(\bar l)}a(x,t)\left(\frac{l_j-u}{d_j(\bar l)}\right)^{(1+\lambda)(q-1)}\varphi_j(\bar l)^{M-q}
 \,\mathrm {d}x\mathrm {d}t
 \\&\leq 4^n\epsilon_1\chi+4^n
 \epsilon_1^{1-(1+\lambda)(q-1)}\epsilon_2\chi
 \\&\ \ +
 \epsilon_1^{1-(1+\lambda)(q-1)}c_1c_2\left[(\chi+\tilde B^{-p}+\tilde B^{-(p-1)})^{1+\frac{p}{n}}+(\chi+\tilde B^{-p}+\tilde B^{-(p-1)})^{1+\frac{q}{n}}\right].
      \end{split}\end{equation}
      Combining \eqref{QQxixi} and \eqref{tildeT4xirefinementpq}, we apply Caccioppoli inequality \eqref{Cacformula1} with $(l,d,\Theta)$ replaced by
  $(l_j,d_j(\bar l),\Theta_j(\bar l))$ to obtain
       \begin{equation}\begin{split}\label{CacvarphipqG}
&\frac{1}{r_j^n}\esssup_t
   \int_{B_j}G\left(\frac{l_j-u}{d_j(\bar l)}\right)\varphi_j(\bar l)^M\,\mathrm {d}x
   \\ \leq&\gamma  \frac{d_j(\bar l)^{p-2}}{r_j^{n+p}}\iint_{L_j(\bar l)}\left(\frac{l_j-u}{d_j(\bar l)}\right)^{(1+\lambda)(p-1)}
  \varphi_j(\bar l)^{M-p}\,\mathrm {d}x\mathrm {d}t
  \\&+\gamma  \frac{d_j(\bar l)^{q-2}}{r_j^{n+q}}\iint_{L_j(\bar l)}a(x,t)\left(\frac{l_j-u}{d_j(\bar l)}\right)^{(1+\lambda)(q-1)}
  \varphi_j(\bar l)^{M-q}\,\mathrm {d}x\mathrm {d}t
   \\&+\gamma \frac{1}{r_j^n}\iint_{L_j(\bar l)}\frac{l_j-u}{d_j(\bar l)}|\partial_t\varphi_j(\bar l)|\varphi_j(\bar l)^{M-1}\,\mathrm {d}x\mathrm {d}t
\\&+\gamma \frac{\Theta_j(\bar l)}{r_j^nd_j(\bar l)^2}\int_{B_j}g^{\frac{p}{p-1}}\,\mathrm {d}x
  +\gamma \frac{\Theta_j(\bar l)}{r_j^nd_j(\bar l)}\int_{B_j}|f|\,\mathrm {d}x
  \\ \leq & \epsilon_1^{(1+\lambda)(p-1)}4^n\chi+4^n(\epsilon_1+\epsilon_2)\chi
+4^n
 \epsilon_1^{1-(1+\lambda)(q-1)}\epsilon_2\chi+c_0(\tilde B^{-p}+\tilde B^{-(p-1)})
 \\&\ \ +
 2\epsilon_1^{1-(1+\lambda)(q-1)}c_1c_2\left[(\chi+\tilde B^{-p}+\tilde B^{-(p-1)})^{1+\frac{p}{n}}+(\chi+\tilde B^{-p}+\tilde B^{-(p-1)})^{1+\frac{q}{n}}\right],
   \end{split}\end{equation}
   where $c_0=c_0(\text{data})$, $c_1=c_1(\epsilon_1,\text{data})$ and $c_2=c_2(\epsilon_2,\text{data})$. Consequently, we infer from \eqref{QQxixi} and \eqref{CacvarphipqG}
   that there exist constants $\gamma_0^\prime=\gamma_0^\prime(\text{data})$, $\gamma_1^\prime=\gamma_1^\prime(\epsilon_1,\text{data})$ and
   $\gamma_2^\prime=\gamma_2^\prime(\epsilon_2,\text{data})$ such that
   \begin{equation}\begin{split}\label{AjbarlQpq}
   A_j(\bar l)&\leq 8^n\epsilon_1\chi
+\gamma_1^\prime\epsilon_2\chi+\gamma_0^\prime(\tilde B^{-p}+\tilde B^{-(p-1)})
 \\&\ \ +
\gamma_1^\prime\gamma_2^\prime\left[(\chi+\tilde B^{-p}+\tilde B^{-(p-1)})^{1+\frac{p}{n}}+(\chi+\tilde B^{-p}+\tilde B^{-(p-1)})^{1+\frac{q}{n}}\right].
    \end{split}\end{equation}
     This establishes an upper bound for $A_j(\bar l)$ in the case of the $(p,q)$-phase.

     \S 2.2. In the case of $p$-phase, i.e., $a_1<10[a]_\alpha r_j^\alpha$.
     Our proof follows in a similar manner as the argument in \S 1.2 and we just sketch the proof.
     Noting that $Q_j(\bar l)\subseteq Q_{r_j,r_j^2}(x_1,t_1)$,
   we have $a(x,t)\leq 12[a]_\alpha r_j^\alpha$ for all $(x,t)\in Q_j(\bar l)$.
   In view of $d_j(\bar l)\leq\xi\omega\leq \|u\|_\infty$ and $r_j^{\alpha-q}\leq r_j^{-p}$,
   we deduce from \eqref{Qclaim1} that there exists a constant $\gamma$ depending only upon the data, such that
    \begin{equation}\begin{split}\label{Qprimexipphase}
&\frac{d_j(\bar l)^{p-2}}{r_j^{n+p}}\iint_{L^\prime_j(\bar l)}\left(\frac{l_j-u}{d_j(\bar l)}\right)^{(1+\lambda)(p-1)}\varphi_j(\bar l)^{M-p}
 \,\mathrm {d}x\mathrm {d}t
 \\&+\frac{d_j(\bar l)^{q-2}}{r_j^{n+q}}\iint_{L^\prime_j(\bar l)}a(x,t)\left(\frac{l_j-u}{d_j(\bar l)}\right)^{(1+\lambda)(q-1)}\varphi_j(\bar l)^{M-q}
 \,\mathrm {d}x\mathrm {d}t
 \\&\leq (1+12[a]_\alpha\|u\|_\infty^{q-p})\epsilon_1^{(1+\lambda)(p-1)}\frac{d_j(\bar l)^{p-2}}{r_j^{n+p}}|L_j(\bar l)|
 \leq \gamma\epsilon_1^{(1+\lambda)(p-1)}\chi.
 \end{split}\end{equation}
Next, we consider the estimate on the set $L^{\prime\prime}_j(\bar l)$.
 Analysis similar to that in the proof of \eqref{QQprimeprimexiq>npintegral} shows that there exist constants $c_1=c_1(\epsilon_1,\text{data})$ and $c_2=c_2(\epsilon_2,\text{data})$ such that
   \begin{equation}\begin{split}\label{Qprimeprimexipphasepintegral}
 &\frac{d_j(\bar l)^{p-2}}{r_j^{n+p}}\iint_{L^{\prime\prime}_j(\bar l)}\left(\frac{l_j-u}{d_j(\bar l)}\right)^{(1+\lambda)(p-1)}\varphi_j(\bar l)^{M-p}
 \,\mathrm {d}x\mathrm {d}t
 \\&\leq 4^n\epsilon_2\chi+c_1c_2(\chi+\tilde B^{-p}+\tilde B^{-(p-1)})^{1+\frac{p}{n}}.
  \end{split}\end{equation}
  Recalling that $\lambda q=\frac{p}{n}$ and $a(x,t)\leq 12[a]_\alpha r_j^\alpha$ for all $(x,t)\in Q_j(\bar l)$, we use $r_j^{\alpha-q}\leq r_j^{-p}$ and Lemma \ref{inequalitypsi-} to deduce
   \begin{equation}\begin{split}\label{Qprimeprimepphase1st}
 &\frac{d_j(\bar l)^{q-2}}{r_j^{n+q}}\iint_{L^{\prime\prime}_j(\bar l)}a(x,t)\left(\frac{l_j-u}{d_j(\bar l)}\right)^{(1+\lambda)(q-1)}\varphi_j(\bar l)^{M-q}
 \,\mathrm {d}x\mathrm {d}t
 \\& \leq c(\epsilon_1) \frac{d_j(\bar l)^{q-2}}{r_j^{n+p}}\iint_{L^{\prime\prime}_j(\bar l)}\psi_p^p
 \left(\frac{l_j-u}{d_j(\bar l)}\right)^{q-p}
 \left(\frac{l_j-u}{d_j(\bar l)}\right)^{\frac{p}{n}}\varphi_j(\bar l)^{M-q}
 \,\mathrm {d}x\mathrm {d}t,
 \end{split}\end{equation}
 To proceed further, we set $\hat v_p=\psi_p\varphi_j(\bar l)^{M_3}$, where $M_3=\frac{\frac{1}{2}M-q}{p}$.
In view of $l_j-u\leq \xi\omega\leq \|u\|_\infty$ on $L_j(\bar l)$ and $\varphi_j(\bar l)=0$ on $\partial_PQ_j(\bar l)$,
 we apply H\"older's inequality and Sobolev's inequality slicewise. This leads us to
\begin{equation}\begin{split}\label{Qprimeprimepphase2st}
 & \frac{d_j(\bar l)^{q-2}}{r_j^{n+p}}\iint_{L^{\prime\prime}_j(\bar l)}\psi_p^p
 \left(\frac{l_j-u}{d_j(\bar l)}\right)^{q-p}
 \left(\frac{l_j-u}{d_j(\bar l)}\right)^{\frac{p}{n}}\varphi_j(\bar l)^{M-q}
 \,\mathrm {d}x\mathrm {d}t
 \\&\leq \|u\|_\infty^{q-p}\frac{d_j(\bar l)^{p-2}}{r_j^{n+p}}\iint_{L^{\prime\prime}_j(\bar l)}\hat v_p^p
 \left(\frac{l_j-u}{d_j(\bar l)}\right)^{\frac{p}{n}}\varphi_j(\bar l)^\frac{M}{2}
 \,\mathrm {d}x\mathrm {d}t
  \\&\leq \|u\|_\infty^{q-p}\frac{d_j(\bar l)^{p-2}}{r_j^{n+p}}\int_{t_1-\Theta_j(\bar l)}^{t_1}\left(\int_{B_j}\hat v_p^{\frac{np}{n-p}}\,\mathrm {d}x\right)^\frac{n-p}{n}
  \\&\ \ \ \ \times
\left( \int_{U^{\prime\prime}_j(t)}\frac{l_j-u}{d_j(\bar l)}\varphi_j(\bar l)^{\frac{Mn}{2p}}
 \,\mathrm {d}x\right)^{\frac{p}{n}}\mathrm {d}t
 \\& \leq \gamma_1\esssup_t\left(\frac{1}{r_j^n} \int_{U^{\prime\prime}_j(t)}\frac{l_j-u}{d_j(\bar l)}\varphi_j(\bar l)^{\frac{Mn}{2p}}
 \,\mathrm {d}x\right)^{\frac{p}{n}}
 \\&\ \ \ \ \times\left[\frac{d_j(\bar l)^{p-2}}{r_j^n}\iint_{Q_j(\bar l)}|D\hat v_p|^p\,\mathrm {d}x\mathrm {d}t\right],
 \end{split}\end{equation}
    where $\gamma_1=\gamma_1(\epsilon_1,\text{data})$. The last line can be handled in much the same way as \eqref{Cacvarphi1} and \eqref{Dvbar}, the only difference
being in the analysis of $\tilde T_4$ in \eqref{Cacvarphi}.
In fact, we have $a_1<10[a]_\alpha r_j^\alpha$ in the case of $p$-phase. We deduce from $d_j(\bar l)\leq \|u\|_\infty$, $r_j^{\alpha-q}\leq r_j^{-p}$,
\eqref{Qclaim1} and \eqref{Qclaim2} that
 \begin{equation}\begin{split}\label{tildeT4xipphase}
    \tilde T_4
    &=\gamma \frac{1}{r_j^n}\iint_{L_j(\bar l)}\frac{l_j-u}{d_j(\bar l)}|\partial_t\varphi_j(\bar l)|\,\mathrm {d}x\mathrm {d}t
    \\&\leq \gamma (1+10[a]_\alpha\|u\|_\infty^{q-p})\frac{d_j(\bar l)^{p-2}}{r_j^{n+p}}\iint_{L_j(\bar l)}\frac{l_j-u}{d_j(\bar l)}\,\mathrm {d}x\mathrm {d}t
    \\&\leq \gamma\frac{d_j(\bar l)^{p-2}}{r_j^{n+p}}|L_j(\bar l)|
 +\gamma\frac{d_j(\bar l)^{p-2}}{r_j^{n+p}}\iint_{L_j(\bar l)}\left(\frac{l_j-u}{d_j(\bar l)}\right)^{(1+\lambda)(p-1)}
 \,\mathrm {d}x\mathrm {d}t
 \\&\leq  \gamma\chi,
      \end{split}\end{equation}
      where the constant $\gamma$ depends only upon the data. Therefore, we have
       \begin{equation}\begin{split}\label{ppahseCacvarphi1}
   \frac{1}{r_j^n}&\esssup_t\int_{U^{\prime\prime}_j(t)}\frac{l_j-u}{d_j(\bar l)}\varphi_j(\bar l)^\frac{Mn}{2p}\,\mathrm {d}x
 \leq \gamma_1(\chi+\tilde B^{-p}+\tilde B^{-(p-1)})
   \end{split}\end{equation}
   and
    \begin{equation}\begin{split}\label{pphaseDvbar}
&\frac{d_j(\bar l)^{p-2}}{r_j^{n}}
 \iint_{Q_j(\bar l)}|D\hat v_p|^p\,\mathrm {d}x\mathrm {d}t
\leq \gamma (\chi+\tilde B^{-p}+\tilde B^{-(p-1)}),
  \end{split}\end{equation}
  where $\gamma=\gamma(\text{data})$ and $\gamma_1=\gamma_1(\epsilon_1,\text{data})$.
Inserting \eqref{ppahseCacvarphi1} and \eqref{pphaseDvbar} into \eqref{Qprimeprimepphase2st}, we deduce from \eqref{Qprimeprimepphase1st} that
    \begin{equation}\begin{split}\label{primeprimeQpphasexi}
&\frac{d_j(\bar l)^{q-2}}{r_j^{n+q}}\iint_{L_j^{\prime\prime}(\bar l)}a(x,t)\left(\frac{l_j-u}{d_j(\bar l)}\right)^{(1+\lambda)(q-1)}\varphi_j(\bar l)^{M-q}
 \,\mathrm {d}x\mathrm {d}t
 \\&
 \leq 4^n\epsilon_2\chi+\gamma_3\gamma_4(\chi+\tilde B^{-p}+\tilde B^{-(p-1)})^{1+\frac{p}{n}},
 \end{split}\end{equation}
  where $\gamma=\gamma(\text{data})$, $\gamma_3=\gamma_3(\epsilon_1,\text{data})$ and $\gamma_4=\gamma_4(\epsilon_2,\text{data})$.
  Consequently, we infer from \eqref{Qprimexipphase}, \eqref{Qprimeprimexipphasepintegral}
and \eqref{primeprimeQpphasexi} that
    \begin{equation}\begin{split}\label{Qpphasexi}
&\frac{d_j(\bar l)^{p-2}}{r_j^{n+p}}\iint_{L_j(\bar l)}\left(\frac{l_j-u}{d_j(\bar l)}\right)^{(1+\lambda)(p-1)}\varphi_j(\bar l)^{M-p}
 \,\mathrm {d}x\mathrm {d}t
 \\&+\frac{d_j(\bar l)^{q-2}}{r_j^{n+q}}\iint_{L_j(\bar l)}a(x,t)\left(\frac{l_j-u}{d_j(\bar l)}\right)^{(1+\lambda)(q-1)}\varphi_j(\bar l)^{M-q}
 \,\mathrm {d}x\mathrm {d}t
 \\&
 \leq \gamma\epsilon_1^{(1+\lambda)(p-1)}\chi+4^n\epsilon_2\chi+\gamma_3\gamma_4(\chi+\tilde B^{-p}+\tilde B^{-(p-1)})^{1+\frac{p}{n}},
 \end{split}\end{equation}
 where $\gamma=\gamma(\text{data})$, $\gamma_3=\gamma_3(\epsilon_1,\text{data})$ and $\gamma_4=\gamma_4(\epsilon_2,\text{data})$.

 Our task now is to estimate the third term on the right-hand side of \eqref{Aj1st}. To this end, we need to improve the inequality \eqref{tildeT4xipphase}.
 In view of $a_1<10[a]_\alpha r_j^\alpha$ and \eqref{Qprimeprimexipphasepintegral}, we have
 \begin{equation}\begin{split}\label{tildeT4xirefinementpphase}
  & \frac{1}{r_j^n}\iint_{L_j(\bar l)}\frac{l_j-u}{d_j(\bar l)}|\partial_t\varphi_j(\bar l)|\varphi_j(\bar l)^{M-1}\,\mathrm {d}x\mathrm {d}t
    \\&\leq \epsilon_1 (1+10[a]_\alpha\|u\|_\infty^{q-p})
   \frac{d_j(\bar l)^{p-2}}{r_j^{n+p}}|L_j^\prime(\bar l)|
    \\&\ + \epsilon_1^{1-(1+\lambda)(p-1)}
     (1+10[a]_\alpha\|u\|_\infty^{q-p})\frac{d_j(\bar l)^{p-2}}{r_j^{n+p}}
\iint_{L_j^{\prime\prime}(\bar l)}\left(\frac{l_j-u}{d_j(\bar l)}\right)^{(1+\lambda)(p-1)}\varphi_j(\bar l)^{M-p}
 \,\mathrm {d}x\mathrm {d}t
 \\&\leq \gamma_0\epsilon_1\chi+\gamma_0
 \epsilon_1^{1-(1+\lambda)(p-1)}\epsilon_2\chi
+\epsilon_1^{1-(1+\lambda)(p-1)}\gamma_1\gamma_2(\chi+\tilde B^{-p}+\tilde B^{-(p-1)})^{1+\frac{p}{n}},
      \end{split}\end{equation}
 where $\gamma_0=\gamma_0(\text{data})$, $\gamma_1=\gamma_1(\epsilon_1,\text{data})$ and $\gamma_2=\gamma_2(\epsilon_2,\text{data})$.
 Taking \eqref{Qpphasexi} and \eqref{tildeT4xirefinementpphase} into account,
 we use Caccioppoli inequality \eqref{Cacformula1} with $(l,d,\Theta)$ replaced by
  $(l_j,d_j(\bar l),\Theta_j(\bar l))$ to obtain
       \begin{equation}\begin{split}\label{CacvarphipphaseG}
&\frac{1}{r_j^n}\esssup_t
   \int_{B_j}G\left(\frac{l_j-u}{d_j(\bar l)}\right)\varphi_j(\bar l)^M\,\mathrm {d}x
  \\ &\leq \gamma_0\epsilon_1\chi+\gamma_1\epsilon_2\chi+
  \gamma_0(\chi+\tilde B^{-p}+\tilde B^{-(p-1)})
+\gamma_1\gamma_2(\chi+\tilde B^{-p}+\tilde B^{-(p-1)})^{1+\frac{p}{n}}.
   \end{split}\end{equation}
   Consequently, we infer from \eqref{Qpphasexi} and \eqref{CacvarphipphaseG}
   that there exist constants $\gamma_0^{\prime\prime}=\gamma_0^{\prime\prime}(\text{data})$, $\gamma_1^{\prime\prime}=\gamma_1^{\prime\prime}(\epsilon_1,\text{data})$ and
   $\gamma_2^{\prime\prime}=\gamma_2^{\prime\prime}(\epsilon_2,\text{data})$, such that
   \begin{equation}\begin{split}\label{AjbarlQpphase}
   A_j(\bar l)&\leq \gamma_0^{\prime\prime}\epsilon_1\chi+\gamma_1^{\prime\prime}\epsilon_2\chi+
  \gamma_0^{\prime\prime}(\chi+\tilde B^{-p}+\tilde B^{-(p-1)})
+\gamma_1^{\prime\prime}\gamma_2^{\prime\prime}(\chi+\tilde B^{-p}+\tilde B^{-(p-1)})^{1+\frac{p}{n}}.
    \end{split}\end{equation}
     This establishes an upper bound for $A_j(\bar l)$ in the case of the $p$-phase.

     Finally, we can summarize what we have proved from \S 1 to \S 2. Combining \eqref{Ajbarlxi}, \eqref{AjbarlQpq} and \eqref{AjbarlQpphase}, we conclude that there exist constants
     $\hat\gamma_0=\hat\gamma_0(\text{data})$, $\hat\gamma_1=\hat\gamma_1(\epsilon_1,\text{data})$ and $\hat\gamma_2=\hat\gamma_2(\epsilon_2,\text{data})$, such that
       \begin{equation}\begin{split}\label{Ajbarlalltogether}
   A_j(\bar l)&\leq \hat\gamma_0\epsilon_1\chi
+\hat\gamma_1\epsilon_2\chi+\hat\gamma_0(\tilde B^{-p}+\tilde B^{-(p-1)})
 \\&\ \ +
\hat\gamma_1\hat\gamma_2\left[\chi^{1+\frac{p}{n}}+(\tilde B^{-p}+\tilde B^{-(p-1)})^{1+\frac{p}{n}}+\chi^{1+\frac{q}{n}}+(\tilde B^{-p}+\tilde B^{-(p-1)})^{1+\frac{q}{n}}\right].
    \end{split}\end{equation}
    At this point, we first choose $\epsilon_1=(15\hat\gamma_0)^{-1}$ and this also fixes $\hat\gamma_1=\hat\gamma_1(\epsilon_1)$. Next, we set $\epsilon_2=(15\hat\gamma_1)^{-1}$.
    This determines quantitatively $\hat\gamma_2=\hat\gamma_2(\epsilon_1)$. Moreover, we set
     \begin{equation}\begin{split}\label{chidefxi}
     \chi=\min\left\{(15\hat\gamma_1\hat\gamma_2)^{-\frac{n}{p}},(15\hat\gamma_1\hat\gamma_2)^{-\frac{n}{q}}\right\}.
      \end{split}\end{equation}
Note that this particular choice of $\chi$ determines $\nu_1$ from \eqref{nu1tildeB}. Finally, we choose $\tilde B$ so large that
 \begin{equation}\begin{split}\label{tildeBdefxi}
\tilde B^{-p}+\tilde B^{-(p-1)}\leq \min\left\{\frac{1}{15\hat\gamma_0}\chi,\left(\frac{1}{15\hat\gamma_1\hat\gamma_2}\chi\right)^\frac{1}{1+\frac{p}{n}},
\left(\frac{1}{15\hat\gamma_1\hat\gamma_2}\chi\right)^\frac{1}{1+\frac{q}{n}}\right\}.
 \end{split}\end{equation}
With these choices of $\epsilon_1$, $\epsilon_2$, $\chi$ and $\tilde B$, we infer from \eqref{Ajbarlalltogether} that $A_j(\bar l)\leq \frac{1}{2}\chi$.

Our next goal is to construct $l_{j+1}$. Before proceeding further,
    we observe that $l_j-\bar l=d_j(\bar l)>\frac{1}{4}\tilde B(\alpha_{j-1}-\alpha_j)>$.
    We also recall that $A_i(l)\to+\infty$ as $l\to l_i$ holds for all $i=0,1,2,\cdots$.
    Next, we distinguish five cases.

  (\romannumeral 1) In the case $\Theta_j(0)> t_1-\hat t$, we find that $A_j(l)=A_j^{(2)}(l)$ is continuous in $[0,l_j)$,
and there exists a number
 $\tilde l\in (\bar l, l_j)$ such that $A_j(\tilde l)=A_j^{(2)}(\tilde l)=\chi$.
    Then, we set
    \begin{equation}\label{def lj+1xicontinuous}
	l_{j+1}=\begin{cases}
\tilde l,&\quad \text{if}\quad \tilde l<l_j-\frac{1}{4}(\alpha_{j-1}-\alpha_j),\\
	l_j-\frac{1}{4}(\alpha_{j-1}-\alpha_j),&\quad \text{if}\quad \tilde l\geq l_j-\frac{1}{4}(\alpha_{j-1}-\alpha_j).
	\end{cases}\end{equation}
It can be easily seen that $A_j(l_{j+1})=A_j^{(2)}(l_{j+1})\leq A_j^{(2)}(\tilde l)=\chi$.

(\romannumeral 2) In the case $\Theta_j(0)\leq t_1-\hat t$ and $\bar l\in(l_j^*,l_j)$, we have $A_j(l)=A_j^{(2)}(l)$ for any $l\in(\bar l,l_j)$.
Since $A_j^{(2)}(l)$ is continuous and increasing, there exists a number
 $l^\prime\in (\bar l, l_j)$ such that $A_j(l^\prime)=A_j^{(2)}(l^\prime)=\chi$.
    At this point, we define
    \begin{equation}\label{def lj+1xi}
	l_{j+1}=\begin{cases}
l^\prime,&\quad \text{if}\quad l^\prime<l_j-\frac{1}{4}(\alpha_{j-1}-\alpha_j),\\
	l_j-\frac{1}{4}(\alpha_{j-1}-\alpha_j),&\quad \text{if}\quad  l^\prime\geq l_j-\frac{1}{4}(\alpha_{j-1}-\alpha_j).
	\end{cases}\end{equation}
It follows from $\bar l<l_j-\frac{1}{4}(\alpha_{j-1}-\alpha_j)$ that $A_j(l_{j+1})=A_j^{(2)}(l_{j+1})\leq A_j^{(2)}(l^\prime)=\chi$.

(\romannumeral 3)
In the case $\Theta_j(0)\leq t_1-\hat t$, $\bar l\in(0,l_j^*]$ and $A_j^{(1)}(l_j^*)> \chi$, we have $A_j(l)=A_j^{(1)}(l)$ for any $l\in[\bar l,l_j^*]$.
Noting that
$A_j^{(1)}(l)$ is continuous and increasing in $(0,l_j^*]$, there exists a number
$l^{\prime\prime}\in (\bar l,l_j^*)$ such that $A_j^{(1)}(l^{\prime\prime})=\chi$. In this case, we define $l_{j+1}$ via the formula
 \begin{equation}\label{def lj+1xilowlow}
	l_{j+1}=\begin{cases}
l^{\prime\prime},&\quad \text{if}\quad l^{\prime\prime}<l_j-\frac{1}{4}(\alpha_{j-1}-\alpha_j),\\
	l_j-\frac{1}{4}(\alpha_{j-1}-\alpha_j),&\quad \text{if}\quad  l^{\prime\prime}\geq l_j-\frac{1}{4}(\alpha_{j-1}-\alpha_j).
	\end{cases}\end{equation}
Moreover, we observe that $A_j(l_{j+1})=A_j^{(1)}(l_{j+1})\leq A_j^{(1)}(l^{\prime\prime})=\chi$.

(\romannumeral 4)
In the case $\Theta_j(0)\leq t_1-\hat t$, $\bar l\in(0,l_j^*]$ and $A_j^{(2)}(l_j^*)<\frac{\chi}{2}$,
we choose $\hat l\in(l_j^*,l_j)$ such that $A_j(\hat l)=A_j^{(2)}(\hat l)=\chi$, since $A_j(l)=A_j^{(2)}(l)$ is continuous and increasing in $(l_j^*,l_j)$.
At this point, we define $l_{j+1}$ by
    \begin{equation}\label{def lj+1xilowhigh}
	l_{j+1}=\begin{cases}
\hat l,&\quad \text{if}\quad \hat l<l_j-\frac{1}{4}(\alpha_{j-1}-\alpha_j),\\
	l_j-\frac{1}{4}(\alpha_{j-1}-\alpha_j),&\quad \text{if}\quad \hat l\geq l_j-\frac{1}{4}(\alpha_{j-1}-\alpha_j).
	\end{cases}\end{equation}
It is easy to check that $A_j(l_{j+1})\leq A_j^{(2)}(\hat l)=\chi$.

(\romannumeral 5)
In the case $\Theta_j(0)\leq t_1-\hat t$, $\bar l\in(0,l_j^*]$, $A_j^{(1)}(l_j^*)\leq \chi$ and $A_j^{(2)}(l_j^*)\geq\frac{\chi}{2}$,
we introduce $l_{j+1}$ via the formula
    \begin{equation}\label{def lj+1xidiscontinuous}
	l_{j+1}=\begin{cases}
l_j^*,&\quad \text{if}\quad l_j^*<l_j-\frac{1}{4}(\alpha_{j-1}-\alpha_j),\\
	l_j-\frac{1}{4}(\alpha_{j-1}-\alpha_j),&\quad \text{if}\quad l_j^*\geq l_j-\frac{1}{4}(\alpha_{j-1}-\alpha_j).
	\end{cases}\end{equation}
According to the definition of $A_j(l)$, we find that $A_j(l_{j+1})=A_j^{(1)}(l_{j+1})\leq A_j^{(1)}(l_j^*)\leq \chi$.
On the other hand, we remark that the case $l_j^*<l_j-\frac{1}{4}(\alpha_{j-1}-\alpha_j)$ is equivalent to $l_{j+1}=l_j^*$, $d_j>\frac{1}{4}(\alpha_{j-1}-\alpha_j)$ and in this case
\begin{equation}\label{special}A_j^{(2)}(l_{j+1})=A_j^{(2)}(l_j^*)\geq\frac{\chi}{2}.\end{equation}
This finishes the construction of $l_{j+1}$.

According to the construction of $l_{j+1}$, we
check at once that \eqref{lixi} and \eqref{Aj-1xi}
hold with $i=j+1$. In view of \eqref{ljxi}$_j$, we conclude from $\alpha_{j-1}\geq\alpha_j$ that
\begin{equation*}\begin{split}
l_{j+1}>&\bar l=\frac{1}{2}l_j+\frac{1}{4}\tilde B\alpha_j+\frac{1}{8}\xi\omega
       \\&>\frac{1}{2}\left(\frac{1}{2}\tilde B\alpha_{j-1}+\frac{1}{4}\xi\omega\right)+\frac{1}{4}\tilde B\alpha_j+\frac{1}{8}\xi\omega
       \\&=\frac{1}{4}\tilde B(\alpha_j+\alpha_{j-1})+\frac{1}{4}\xi\omega\geq \frac{1}{2}\tilde
       B\alpha_j+\frac{1}{4}\xi\omega
\end{split}\end{equation*}
and hence that \eqref{ljxi}$_{j+1}$ holds.

Step 4: \emph{Proof of the inequality $u(x_1,t_1)>\frac{1}{4}\xi\omega$.}
Repeating the arguments as in Step 3, we can determine a sequence of numbers $\{l_i\}_{i=0}^\infty$ satisfying \eqref{lixi}-\eqref{ljxi}.
Noting that the sequence $\{l_i\}_{i=0}^\infty$ is decreasing and has a lower bound, we infer that the limitation of $l_i$ exists.
It follows that $d_i\to0$ as $i\to\infty$. Now, we define
$$w=\lim_{i\to\infty}l_i,$$
and assert that $w=u(x_1,t_1)$.
First, we consider the case $a_1=0$. We take $i$ so large that $d_i<1$ and $r_i^p<t_1-\hat t$.
According to the definition of $A_j(l)$ and \eqref{Aj-1xi}, we have
\begin{equation*}\begin{split}
\frac{1}{r_j^{n+p}} &\iint_{Q_j^\prime}(w-u)_+^{(1+\lambda)(p-1)}\,\mathrm {d}x\mathrm {d}t\leq A_j(l_{j+1})d_j^{(1+\lambda)(p+1)-(p-2)}
\\&\leq\chi d_j^{(1+\lambda)(p+1)-(p-2)}\to0,
  \end{split}\end{equation*}
  as $j\to\infty$.
  Here $Q_j^\prime=B_{j+1}\times\left(t_1-\frac{4}{9}(1+10[a]_\alpha)^{-1}r_j^p,t_1\right)$.
  In the case $a_1>0$, we choose $j\geq1$ large enough to have $10[a]_\alpha r_j^{q-p}\leq a_1$
  and $a_1^{-1}r_j^q<t_1-\hat t$. We infer from the definition of $A_j(l)$ and \eqref{Aj-1xi} that
  \begin{equation*}\begin{split}
\frac{1}{a_1^{-1}r_j^{n+q}}&\iint_{Q_j^{\prime\prime}}(w-u)_+^{(1+\lambda)(q-1)}\,\mathrm {d}x\mathrm {d}t\leq \frac{5}{4}A_j(l_{j+1})d_j^{(1+\lambda)(q+1)-(q-2)}
\\&\leq \frac{5}{4}\chi d_j^{(1+\lambda)(q+1)-(q-2)}\to0,
  \end{split}\end{equation*}
  as $j\to\infty$. Here $Q_j^{\prime\prime}=B_{j+1}\times\left(t_1-\frac{4}{9}(1+[a]_\alpha^{-1})^{-1}a_1^{-1}r_j^q,t_1\right)$. This
  proves the claim $w=u(x_1,t_1)$,
  since $(x_1,t_1)$ is a Lebesgue point of $u$.

Next, we claim that for any $j\geq1$,
  \begin{equation}\begin{split}\label{djdj-1xi}
  d_j\leq &\frac{1}{4}d_{j-1}+
  \gamma\left(r_{j-1}^{p-n}\int_{B_{j-1}} g(y)^{\frac{p}{p-1}}
  \,\mathrm {d}y\right)^{\frac{1}{p}}\\&+\gamma\left(r_{j-1}^{p-n}\int_{B_{j-1}}|f(y)|
  \,\mathrm {d}y\right)^{\frac{1}{p-1}}+\gamma r_{j-1},
  \end{split}\end{equation}
  where the constant $\gamma$ depends only upon the data.
For any fixed $j\geq 1$, we assume that
  \begin{equation}\begin{split}\label{djdj-1proofxi}
  d_j>\frac{1}{4}d_{j-1}\qquad\text{and}\qquad d_j>\frac{1}{4}(\alpha_{j-1}-\alpha_j),
   \end{split}\end{equation}
   since otherwise, the inequality \eqref{djdj-1xi} holds immediately.
First, we consider the cases (\romannumeral 1)-(\romannumeral 4). In these cases, $l_{j+1}$ is defined via \eqref{def lj+1xicontinuous}-\eqref{def lj+1xilowhigh}.
   In view of $d_j>\frac{1}{4}(\alpha_{j-1}-\alpha_j)$,
   we find that
   $A_j(l_{j+1})=A_j^{(2)}(\tilde l)=\chi$, $A_j(l_{j+1})=A_j^{(2)}(l^\prime)=\chi$, $A_j(l_{j+1})=A_j^{(1)}(l^{\prime\prime})=\chi$ or $A_j(l_{j+1})=A_j^{(2)}(\hat l)=\chi$.
   Hence, we have $A_j(l_{j+1})=\chi$ for these cases.
    According to \eqref{djdj-1proofxi}, we can
  repeat the arguments from Step 3 with $\bar l$ and $d_j(\bar l)$ replaced by $l_{j+1}$ and $d_j$, respectively.
  More precisely, we establish the following estimate similar to \eqref{Ajbarlalltogether}:
  \begin{equation*}\begin{split}
  \chi=A_j(l_{j+1})&\leq
\hat\gamma_0 \left(\frac{r_j^{p-n}}{d_j^p}\int_{B_j}g^{\frac{p}{p-1}}\,\mathrm {d}x+
  \frac{r_j^{p-n}}{d_j^{p-1}}\int_{B_j}|f|\,\mathrm {d}x\right)
  \\
  &+\hat\gamma_0\epsilon_1\chi
+\hat\gamma_1\epsilon_2\chi+\hat\gamma_1\hat\gamma_2\chi^{1+\frac{p}{n}}+\hat\gamma_1\hat\gamma_2\chi^{1+\frac{q}{n}}
\\&+\hat\gamma_1\hat\gamma_2\left(
 \frac{r_j^{p-n}}{d_j^p}\int_{B_j}g^{\frac{p}{p-1}}\,\mathrm {d}x+
 \frac{r_j^{p-n}}{d_j^{p-1}}\int_{B_j}|f|\,\mathrm {d}x\right)^{1+\frac{p}{n}}
 \\&+\hat\gamma_1\hat\gamma_2\left(
 \frac{r_j^{p-n}}{d_j^p}\int_{B_j}g^{\frac{p}{p-1}}\,\mathrm {d}x+
 \frac{r_j^{p-n}}{d_j^{p-1}}\int_{B_j}|f|\,\mathrm {d}x\right)^{1+\frac{q}{n}},
  \end{split}\end{equation*}
  where $\epsilon_1$, $\epsilon_2$, $\chi$, $\hat\gamma_0$, $\hat\gamma_1$ and $\hat\gamma_2$ are the constants in \eqref{Ajbarlalltogether}.
According to the choices of $\epsilon_1$, $\epsilon_2$, $\chi$, $\hat\gamma_0$, $\hat\gamma_1$ and $\hat\gamma_2$ in Step 3, we find that
$\hat\gamma_0\epsilon_1\chi
+\hat\gamma_1\epsilon_2\chi+\hat\gamma_1\hat\gamma_2\chi^{1+\frac{p}{n}}+\hat\gamma_1\hat\gamma_2\chi^{1+\frac{q}{n}}\leq\frac{1}{3}\chi.$
 Consequently, we infer that either
  \begin{equation}\begin{split}\label{dddd}
  d_j\leq \gamma\left(r_j^{p-n}\int_{B_j}g^{\frac{p}{p-1}}\,\mathrm {d}x\right)^{\frac{1}{p}}
  \qquad\text{or}\qquad
  d_j\leq
  \gamma \left(r_j^{p-n}\int_{B_j}|f|\,\mathrm {d}x\right)^{\frac{1}{p-1}},
   \end{split}\end{equation}
   which proves the desired estimate \eqref{djdj-1xi}.
 Finally, we consider the case (\romannumeral 5). In this case, $l_{j+1}$ is defined via \eqref{def lj+1xidiscontinuous} and we see that
   $A_j^{(1)}(l_j^*)\leq \chi$ and $A_j^{(2)}(l_j^*)\geq \frac{1}{2}\chi$.
   In view of $d_j=l_j-l_{j+1}>\frac{1}{4}(\alpha_{j-1}-\alpha_j)$, we infer from \eqref{def lj+1xidiscontinuous} and \eqref{special} that $l_{j+1}=l_j^*$ and
   $A_j^{(2)}(l_{j+1})\geq \frac{1}{2}\chi$.
   Taking into account that $d_j>\frac{1}{4}d_{j-1}$,
   we obtain
    \begin{equation*}\begin{split}
    \Theta_{j-1}&=d_{j-1}^2\left[\left(\frac{d_{j-1}}{r_{j-1}}\right)^p+a_1\left(\frac{d_{j-1}}{r_{j-1}}\right)^q\right]^{-1}
   >16d_{j}^2\left[\left(\frac{4d_{j}}{4r_{j}}\right)^p+a_1\left(\frac{4d_{j}}{4r_{j}}\right)^q\right]^{-1}
   \\&=16d_{j}(l_j^*)^2\left[\left(\frac{d_{j}(l_j^*)}{r_{j}}\right)^p+a_1\left(\frac{d_{j}(l_j^*)}{r_{j}}\right)^q\right]^{-1}
   =16\Theta_j(l_j^*)=16(t_1-\hat t)>t_1-\hat t.
     \end{split}\end{equation*}
   This implies that $Q_{j-1}\nsubseteq\widetilde Q_1$; therefore, $\chi\geq A_{j-1}(l_j)=A_{j-1}^{(2)}(l_j)$.
According to \eqref{djdj-1proofxi}, we can
  repeat the arguments from Step 3, \S 1 with $\bar l$, $Q_j(\bar l)$ and $A_j(\bar l)$ replaced by $l_j^*$, $Q_j$ and $A_j^{(2)}(l_{j+1})$, respectively. More precisely, we establish the estimate similar to \eqref{Ajbarlxi} as follows:
  \begin{equation*}\begin{split}
  \frac{1}{2}\chi\leq A_{j}^{(2)}(l_{j+1})&\leq
\tilde\gamma_0 \left(\frac{r_j^{p-n}}{d_j^p}\int_{B_j}g^{\frac{p}{p-1}}\,\mathrm {d}x+
  \frac{r_j^{p-n}}{d_j^{p-1}}\int_{B_j}|f|\,\mathrm {d}x\right)
  \\
  &+\tilde\gamma_0(\epsilon_1
+\epsilon_2)\chi+\tilde\gamma_1\tilde\gamma_2\chi^{1+\frac{p}{n}}+\tilde\gamma_1\tilde\gamma_2\chi^{1+\frac{q}{n}}
\\&+\tilde\gamma_1\tilde\gamma_2\left(
 \frac{r_j^{p-n}}{d_j^p}\int_{B_j}g^{\frac{p}{p-1}}\,\mathrm {d}x+
 \frac{r_j^{p-n}}{d_j^{p-1}}\int_{B_j}|f|\,\mathrm {d}x\right)^{1+\frac{p}{n}}
 \\&+\tilde\gamma_1\tilde\gamma_2\left(
 \frac{r_j^{p-n}}{d_j^p}\int_{B_j}g^{\frac{p}{p-1}}\,\mathrm {d}x+
 \frac{r_j^{p-n}}{d_j^{p-1}}\int_{B_j}|f|\,\mathrm {d}x\right)^{1+\frac{q}{n}},
  \end{split}\end{equation*}
where $\epsilon_1$, $\epsilon_2$, $\chi$, $\tilde\gamma_0$, $\tilde\gamma_1$ and $\tilde\gamma_2$ are the constants in \eqref{Ajbarlxi}.
According to the choices of $\epsilon_1$, $\epsilon_2$, $\chi$, $\tilde\gamma_0$, $\tilde\gamma_1$ and $\tilde\gamma_2$ in Step 3, we find that
$\tilde\gamma_0(\epsilon_1
+\epsilon_2)\chi+\tilde\gamma_1\tilde\gamma_2\chi^{1+\frac{p}{n}}+\tilde\gamma_1\tilde\gamma_2\chi^{1+\frac{q}{n}}\leq\frac{1}{3}\chi$.
Consequently, we arrive at
\eqref{dddd} for this special case. This completes the proof of the inequality \eqref{djdj-1xi}.

   Let $J>1$ be a fixed integer. We sum up the inequality \eqref{djdj-1xi} for $j=1,\cdots,J-1$
   and obtain
   \begin{equation*}\begin{split}
   l_1-l_J\leq& \frac{1}{3}d_0+\gamma\sum_{j=1}^{J-1}\left(r_{j-1}^{p-n}\int_{B_{j-1}}|f(y)|
  \,\mathrm {d}y\right)^{\frac{1}{p-1}}
\\&+ \gamma\sum_{j=1}^{J-1}\left(r_{j-1}^{p-n}\int_{B_{j-1}} g(y)^{\frac{p}{p-1}}
  \,\mathrm {d}y\right)^{\frac{1}{p}}+ \gamma\sum_{j=1}^{J-1}r_{j-1}.
   \end{split}\end{equation*}
   Recalling that $d_0=l_0-l_1$, $l_0=\xi\omega$ and $d_0\leq \frac{3}{8}\xi\omega$, we apply \eqref{omega1violated1} to infer that there exists
   a constant $C_0$ depending only upon the data, such that
  \begin{equation}\begin{split}\label{uupperboundxi}
\xi\omega\leq
 & \frac{4}{3}d_0+l_J
  +\gamma\int_0^{R}\left(\frac{1}{r^{n-p}}\int_{B_r(x_1)}f(y)
\,\mathrm {d}y\right)^{\frac{1}{p-1}}\frac{1}{r}\,\mathrm {d}r
  \\&+\gamma\int_0^{R}\left(\frac{1}{r^{n-p}}\int_{B_r(x_1)}g(y)^{\frac{p}{p-1}}
\,\mathrm {d}y\right)^{\frac{1}{p}}\frac{1}{r}\,\mathrm {d}r+\gamma R
\\ \leq & \frac{4}{3}d_0+l_J+C_0\frac{1}{\tilde B}\xi\omega.
   \end{split}\end{equation}
  Passing to the limit $J\to\infty$, we conclude from \eqref{uupperboundxi} that
   \begin{equation*}\begin{split}
   u(x_1,t_1)=\lim_{J\to\infty}l_J\geq\frac{1}{2}\xi\omega-C_0\frac{1}{\tilde B}\xi\omega>\frac{1}{4}\xi\omega,
    \end{split}\end{equation*}
    provided that we choose $\tilde B>4C_0$. This establishes the desired inequality $u(x_1,t_1)>\frac{1}{4}\xi\omega$.
Finally, we summarize the precise values of the
parameters $\nu_1$ and $\tilde B$.
First, the parameter $\nu_1$ is determined via \eqref{nu1tildeB}, where the quantity $\chi$ in \eqref{nu1tildeB}
    is fixed in terms of \eqref{chidefxi}. More precisely, we have
     \begin{equation}\begin{split}\label{specifynu1}\nu_1=\frac{\chi}{4\gamma_1},\end{split}\end{equation}
     where $\gamma_1$ is the constant from \eqref{definitionofgamma1} and depends only upon the data.
     This also implies that the constant $\nu_1$ is independent of $\xi$.
     According to \eqref{nu1tildeB} and \eqref{tildeBdefxi}, we choose
     \begin{equation*}\begin{split}&\tilde B=
     \max\left\{4C_0,\ \left(\frac{1}{8\gamma_1}\chi\right)^\frac{1}{1-p},\ \left(\frac{1}{30\hat\gamma_0}\chi\right)^\frac{1}{1-p},\
     2^\frac{1}{p-1}\left(\frac{1}{15\hat\gamma_1\hat\gamma_2}\chi\right)^\frac{1}{(1+\frac{p}{n})(1-p)},\
2^\frac{1}{p-1}\left(\frac{1}{15\hat\gamma_1\hat\gamma_2}\chi\right)^\frac{1}{(1+\frac{q}{n})(1-p)}
     \right\},\end{split}\end{equation*}
     where $C_0$ is the constant in \eqref{uupperboundxi}. With these choices of $\nu_1$ and $\tilde B$, we conclude that the lemma holds.
   \end{proof}
    With the help of the proceeding De Giorgi-type lemma, we can now derive an oscillation estimate of $u$
 in a smaller cylinder of the form \eqref{osc1st}.
The next proposition is our main result in this section.
\begin{proposition}\label{1st proposition}
Let $\tilde Q_0=Q_{\frac{1}{16}R_0}^-(0)$ and
let $u$ be a bounded weak solution to \eqref{parabolic}-\eqref{A} in $\Omega_T$.
Assume that there exists a time level $-\frac{8}{9}\Theta_A\leq \bar t\leq 0$ such that \eqref{1st} holds.
There exist $0<\xi_1<2^{-5}$ and $B_1>1$ depending only upon the data and $A$ such that
\begin{equation*}\begin{split}
\essosc_{\tilde Q_0} u\leq (1-4^{-1}\xi_1)\omega+B_1\xi_1^{-1}\left(
F_p(R_0)+G_p(R_0)+R_0\right).
 \end{split}\end{equation*}
\end{proposition}
\begin{proof}
First, we assume that
\eqref{omega1} is violated. According to Lemma \ref{lemmaDeGiorgi1}, we obtain \eqref{hat t lower}
and this enables us to use Lemma \ref{timeexpandlemma}.
At this point, we take $\nu_*=\nu_1$ according to \eqref{specifynu1} in Lemma \ref{timeexpandlemma}. Moreover,
we choose $s_*=2\gamma_0\nu_*^{-1}A^{q-2}$
and $\xi_1=2^{-s_*}$.
If \eqref{omega2} is violated, then
we infer from Lemma \ref{timeexpandlemma} that
\begin{equation*}\left|\left\{(x,t)\in \widetilde Q_1:u<\mu_-+\xi_1\omega\right\}
\right|\leq \nu_1|\widetilde Q_1|,\end{equation*}
where $\tilde Q_1=B_{\frac{R}{8}}\times(\hat t,0)$. This is actually the condition \eqref{assumptionDegiorgi} for Lemma \ref{lemmaDeGiorgi2}. If \eqref{omega3} is violated, then the inequality \eqref{DeGiorgi2} reads
\begin{equation*}\begin{split}
\essosc_{\tilde Q_0} u\leq (1-4^{-1}\xi_1)\omega,
 \end{split}\end{equation*}
 since 
 $R=\frac{9}{10}R_0$ and
 $\frac{1}{2}\bar t+\Theta_\omega \left(\frac{1}{16}R_0\right)<\left(\frac{5}{18}\right)^p\Theta_\omega\left(\frac{1}{4}R\right)<\frac{1}{2}\Theta_\omega\left(\frac{1}{4}R\right)$.
 On the other hand, if either \eqref{omega1}, \eqref{omega2} or \eqref{omega3} holds, then we conclude that the inequality
\begin{equation*}\begin{split}
\essosc_{\tilde Q_0} u\leq\omega\leq B_1\xi_1^{-1}\left(
F_p(R)+G_p(R)+R\right)
 \end{split}\end{equation*}
 holds for a constant $B_1=B_1(\text{data},A)=100\max\left\{B,\tilde B,2^{\frac{2}{p-1}s_*}\right\}$.
 It is now obvious that the proposition holds.
\end{proof}
 \section{The second alternative}
The aim of this section is to derive the decay estimate of the essential oscillation for the second alternative.
Throughout this section, we assume that \eqref{2nd} holds for all $-\frac{8}{9}\Theta_A\leq\bar t\leq 0$
and the constant $\nu_0$ is determined in \eqref{nu0B}.
We start with the following lemma, which is a standard result that can be found in \cite[Lemma 4.7]{KMS}.
\begin{lemma}\label{endalterfirstlemma}
Let $-\frac{8}{9}\Theta_A\leq\bar t\leq 0$
and let $u$ be a bounded weak solution to \eqref{parabolic}-\eqref{A} in $\Omega_T$.
Assume that \eqref{2nd} holds.
Then there exists a time level
$t^*\in \left[\bar t-\Theta_\omega(R),\bar t-\frac{1}{2}\nu_0
 \Theta_\omega(R)\right]$
 such that
 \begin{equation}\label{2nd alternative begin}
 \left|\left\{x\in B_{R}:u(x,t^*)>\mu_+-\frac{\omega}{4}\right\}\right|\leq
 \left(\frac{1-\nu_0}{1-\frac{1}{2}\nu_0}\right)|B_{R}|,\end{equation}
 where $\nu_0$ is the constant claimed by Lemma \ref{lemmaDeGiorgi1}.
\end{lemma}
With the help of this lemma, we establish the following result regarding the time propagation of positivity.
\begin{lemma}\label{time 2nd}
Let $u$ be a bounded weak solution to \eqref{parabolic}-\eqref{A} in $\Omega_T$.
There exists a positive constant $s_1$
that can be determined a priori only in terms of the data such that
 either
\begin{equation}\begin{split}\label{2ndomega assumption1}
\omega\leq 2^{\frac{s_1}{p-1}}F_p(R_0)
+2^{\frac{2s_1}{p}}G_p(R_0)
\end{split}\end{equation}
or
\begin{equation}\label{2nd measure estimate}
\left|\left\{x\in B_{R}:u(x,t)>\mu_+-\frac{\omega}{2^{s_1}}\right\}\right|\leq
\left(1-\left(\frac{\nu_0}{2}\right)^2\right)|B_{R}|
\end{equation}
for all $t\in \left[t^*,\bar t\right]$. Here, $t^*$ is the time level stated in Lemma \ref{endalterfirstlemma} and satisfies \eqref{2nd alternative begin}.
\end{lemma}
\begin{proof}
We first observe that $\bar t-t_*\leq \Theta_\omega(R)<\Theta_A<R^2$. This enables
us to use Lemma \ref{logestimatelemma} over the cylinder $B_R\times \left[t^*,\bar t\right]$.
Let $k=\mu_+-\frac{1}{4}\omega$ and $c=2^{-2-l}\omega$, where $l\geq2$ is to be determined later.
We consider the logarithmic function defined by
\begin{equation*}\begin{split}\Psi^+(u)=\ln^+\left(\frac{\frac{1}{4}\omega}
{\frac{1}{4}\omega-(u-k)_++c}\right).\end{split}\end{equation*}
Next, we take a smooth cutoff function $0\leq\phi(x)\leq1$, defined in $B_R$, satisfying
$\phi\equiv 1$ in $B_{(1-\sigma)R}$ and $|D\phi|\leq (\sigma R)^{-1}$, where $\sigma\in (0,1)$ is to be determined.
With these choices, we apply the logarithmic estimate \eqref{lnCac} to obtain
\begin{equation}\begin{split}\label{psi+log}
&\esssup_{t\in[t^*,\bar t]}\int_{B_{(1-\sigma)R}\times\{t\}}[\Psi^+(u)]^2\,\mathrm{d}x
\\&\leq
\int_{B_R\times\{t^*\}}[\Psi^+(u)]^2\,\mathrm{d}x+\gamma\frac{1}{\sigma^q R^q}\iint_{B_R\times[t^*,\bar t]}a_0\Psi^+(u)
[(\Psi^+)^\prime(u)]^{2-q}\,\mathrm{d}x\mathrm{d}t
\\&\quad +\gamma\left(\frac{1}{\sigma^p R^p}+R^\alpha\frac{1}{\sigma^q R^q}\right)\iint_{B_R\times[t^*,\bar t]}\Psi^+(u)
[(\Psi^+)^\prime(u)]^{2-p}\,\mathrm{d}x\mathrm{d}t
\\&\quad+\gamma(\ln 2)l\frac{2^{2+l}}{\omega^{p-1}}R^p\int_{B_R}|f|\,\mathrm{d}x+
\gamma(\ln 2)l\frac{2^{4+2l}}{\omega^p}R^p\int_{B_R}g^{\frac{p}{p-1}}\,\mathrm{d}x
\\&=:L_1+L_2+L_3+L_4+L_5,
\end{split}\end{equation}
since $\bar t-t_*\leq \Theta_\omega(R)\leq \omega^{2-p}R^p$.
Taking into account that $\Psi^+(u)\leq l\ln2$, we conclude from \eqref{2nd alternative begin} that
\begin{equation*}\begin{split}
L_1\leq (l\ln2)^2\left|\left\{x\in B_R:u(x,t^*)>\mu_+-\frac{\omega}{4}\right\}\right|
\leq l^2(\ln^22)\left(\frac{1-\nu_0}{1-\frac{1}{2}\nu_0}\right)|B_R|.
\end{split}\end{equation*}
In view of $\Psi^+(u)\leq l\ln2$, $[(\Psi^+)^\prime(u)]^{-1}\leq 4^{-1}\omega$ and $\bar t-t_*\leq \Theta_\omega(R)$, we deduce
\begin{equation*}\begin{split}
L_2+L_3\leq \gamma l\frac{1}{\sigma^q}\left(a_0\frac{\omega^{q-2}}{R^q}+\frac{\omega^{p-2}}{R^p}\right)(\bar t-t^*)|B_R|\leq
\gamma \frac{l}{\sigma^q}|B_R|,
\end{split}\end{equation*}
where the constant $\gamma$ depends only upon the data.
At this stage, we assume that
\begin{equation}\begin{split}\label{omega assumption}
\omega>2^{\frac{l}{p-1}}\left(R^{p-n}\int_{B_R}|f|\,\mathrm{d}x\right)^\frac{1}{p-1}
+2^{\frac{2l}{p}}\left(R^{p-n}\int_{B_R}g^{\frac{p}{p-1}}\,\mathrm{d}x\right)^\frac{1}{p}
\end{split}\end{equation}
and hence,
$L_4+L_5\leq \gamma l|B_R|\leq \gamma \sigma^{-q}l|B_R|.$
Combining the estimates for $L_1$-$L_5$, we arrive at
\begin{equation*}\begin{split}
\int_{B_{(1-\sigma)R}\times\{t\}}[\Psi^+(u)]^2\,\mathrm{d}x\leq
l^2(\ln^22)\left(\frac{1-\nu_0}{1-\frac{1}{2}\nu_0}\right)|B_R|+\gamma \frac{l}{\sigma^q}|B_R|
\end{split}\end{equation*}
for all $t\in[t^*,\bar t]$.
On the other hand, the left-hand side of \eqref{psi+log} can be estimated below by integrating over the smaller set
\begin{equation*}\begin{split}
S=\left\{x\in B_{(1-\sigma)R}:u(x,t)>\mu_+-\frac{\omega}{2^{l+2}}\right\},
\end{split}\end{equation*}
and this yields that
\begin{equation*}\begin{split}
\int_{B_{(1-\sigma)R}\times\{t\}}[\Psi^+(u)]^2\,\mathrm{d}x\geq (l-1)^2(\ln^22)|S|.
\end{split}\end{equation*}
Consequently, we infer that the estimate
\begin{equation*}\begin{split}
&\left|\left\{x\in B_R:u(x,t)>\mu_+-\frac{\omega}{2^{l+2}}\right\}\right|
\leq
|S|+|B_R\setminus B_{(1-\sigma)R}|
\\&\quad\leq \left[\left(\frac{l}{l-1}\right)^2\left(\frac{1-\nu_0}{1-\frac{1}{2}\nu_0}\right)+\gamma \frac{l}{\sigma^q(l-1)^2}
+\gamma\sigma\right]|B_R|
\end{split}\end{equation*}
holds  for all $t\in[t^*,\bar t]$. At this point, we choose $\gamma \sigma\leq\frac{3}{8}\nu_0^2$ and then $l$ so large that
\begin{equation*}\begin{split}
\left(\frac{l}{l-1}\right)^2<\left(1-\frac{1}{2}\nu_0\right)(1+\nu_0)\qquad\text{and}\qquad
\gamma\frac{1}{\sigma^ql}\leq\frac{3}{8}\nu_0^2.
\end{split}\end{equation*}
This proves the inequality \eqref{2nd measure estimate} with $s_1=l+2$. On the other hand, if \eqref{omega assumption} is violated, then we obtain
\eqref{2ndomega assumption1} for such a choice of $s_1$. This proves the lemma.
\end{proof}
Our next goal is to establish a De Giorgi-type lemma for the second alternative. To this end,
we introduce concentric parabolic cylinders
$Q_A^{(1)}=B_{R}\times\left(-\frac{1}{2}\Theta_A,0\right]$
and
$Q_A^{(2)}=B_{\frac{1}{2}R}\times\left(-\frac{1}{4}\Theta_A,0\right]$,
where $A$ satisfies \eqref{firstcondition forA}, \eqref{firstconditionforA} and \eqref{secondconditionforA}.
The proof of the De Giorgi type lemma is based on the Kilpel\"ainen-Mal\'y technique and
follows in a similar manner as the proof of Lemma \ref{lemmaDeGiorgi1}.
 \begin{lemma}\label{lemmaDeGiorgi1+}
 Let $0<\xi<4^{-1}$ and
let $u$ be a bounded weak solution to \eqref{parabolic}-\eqref{A} in $\Omega_T$.
There exist constants $\nu_2\in(0,1)$ and $\hat B>1$, depending only upon the data, such that if
\begin{equation}\label{QA1}\left|\left\{(x,t)\in Q_A^{(1)}:u\geq\mu_+-\xi\omega\right\}\right|\leq \nu_2|Q_A^{(1)}|,\end{equation}
then either
\begin{equation}
\label{DeGiorgi1+}u(x,t)<\mu_+-\frac{1}{4}\xi\omega\qquad\text{for}\ \ \text{a.e.}\ \ (x,t)\in Q_A^{(2)}\end{equation}
or
\begin{equation}
\label{omega1+}\xi\omega\leq \hat B\left(F_p(R)
+G_p(R)+100R\right).\end{equation}
Here, $A=\hat B\xi^{-1}$ and satisfies \eqref{firstcondition forA}, \eqref{firstconditionforA} and \eqref{secondconditionforA}.
 \end{lemma}
 \begin{proof} First, we define $v=\mu_+-u$ and note that $v$ is nonnegative in $Q_A$.
 Observe that $v$ is a bounded weak solution to the
parabolic double-phase equation
\begin{equation*}\partial_t v-\operatorname{div}\hat A(x,t,v,Dv)=-f,\end{equation*}
where the vector field $\hat A(x,t,v,\xi)=-A(x,t,\mu_+-v,-\xi)$ satisfies the structure condition
 \begin{equation*}
	\begin{cases}
	 \big\langle \hat A(x,t,v,\xi),\xi\big \rangle\geq C_0\left(|\xi|^p+a(x,t)|\xi|^q\right),\\
	|\hat A(x,t,v,\xi)|\leq C_1\left(|\xi|^{p-1}+a(x,t)|\xi|^{q-1}\right)+g.
	\end{cases}
\end{equation*}
Initially, we assume that the constant
$\hat B$ satisfies \eqref{firstcondition forA}, \eqref{firstconditionforA} and \eqref{secondconditionforA}.
Next, we assume that \eqref{omega1+} is violated, that is,
 \begin{equation}
\label{omega1violated+}\xi\omega>\frac{3}{4}\hat B\left(\frac{1}{3\hat B}\xi\omega
+F_p(R)+G_p(R)+100R
\right).\end{equation}
Fix $z_1=(x_1,t_1)\in Q_A^{(2)}$ and assume that $(x_1,t_1)$ is a Lebesgue point of $v$.
Our problem reduces to showing that $v(x_1,t_1)>\frac{1}{6}\xi\omega$.

Step 1: \emph{Setting up the notations.}
Let $a_1=a(x_1,t_1)$, $r_j=4^{-j}C^{-1}R$
 and $B_j=B_{r_j}(x_1)$ where $C>4$ is to be determined.
For a sequence $\{l_j\}_{j=0}^\infty$ and a fixed $l>0$, we define $d_{j}(l)=l_j-l$. If $d_{j}(l)>0$, then we introduce
a parabolic cylinder $Q_j(l)=B_j\times (t_1-\Theta_j(l),t_1)$, where
 \begin{equation*}\Theta_j(l)=d_{j}(l)^2\left[\left(\frac{d_{j}(l)}{r_j}\right)^p+a_1\left(\frac{d_{j}(l)}{r_j}\right)^q\right]^{-1}.\end{equation*}
Moreover, we set $\varphi_j(l)=\phi_j(x)\theta_{j,l}(t)$, where
$\phi_j\in C_0^\infty(B_j)$, $\phi_j=1$ on $B_{j+1}$, $|D\phi_j|\leq r_j^{-1}$
 and $\theta_{j,l}(t)$ is a Lipschitz function
satisfies $\theta_{j,l}(t)=1$ in $t\geq t_1-\frac{4}{9}\Theta_j(l)$, $\theta_{j,l}(t)=0$ in $t\leq t_1-\frac{5}{9}\Theta_j(l)$
and
  \begin{equation*}
 \theta_{j,l}(t)=\frac{t-t_1-\frac{5}{9}\Theta_j(l)}{\frac{1}{9}\Theta_j(l)}\qquad\text{in}\qquad
 t_1-\frac{5}{9}\Theta_j(l)\leq t\leq t_1-\frac{4}{9}\Theta_j(l).
 \end{equation*}
It is easy to check that $\varphi_j(l)=0$ on $\partial_PQ_j(l)$.
 Next, for  $j=-1,0,1,2,\cdots$, we introduce the sequence $\{\alpha_j\}$ by
  \begin{equation}\begin{split}\label{alphaxi+}\alpha_j=&\frac{4^{-j-100
  }}{3\hat B}\xi\omega+ \frac{3}{4}\int_0^{r_j}\left(r^{p-n}\int_{B_r(x_1)}
  g(y)^{\frac{p}{p-1}} \,\mathrm {d}y
  \right)^{\frac{1}{p}}\frac{\mathrm {d}r}{r}\\&+
  \frac{3}{4}\int_0^{r_j}\left(r^{p-n}\int_{B_r(x_1)}|f(y)| \,\mathrm {d}y
  \right)^{\frac{1}{p-1}}\frac{\mathrm {d}r}{r}+75r_j.
  \end{split}\end{equation}
  According to the definition of $\alpha_j$, we see that $\alpha_j\to 0$ as $j\to\infty$. Moreover, we infer from \eqref{omega1violated+}
  and \eqref{alphaxi+} that
   $\hat B\alpha_{j-1}\leq\xi\omega$,
  \begin{equation}\begin{split}\label{alpha1xi+}
  \alpha_{j-1}-\alpha_j\geq &\frac{4^{-j-100}}{\hat B}\xi\omega+\gamma\left(r_j^{p-n}\int_{B_j} g(y)^{\frac{p}{p-1}}
  \,\mathrm {d}y\right)^{\frac{1}{p}}\\&+\gamma\left(r_j^{p-n}\int_{B_j}|f(y)|
  \,\mathrm {d}y\right)^{\frac{1}{p-1}}+200 r_j
  \end{split}\end{equation}
  and
   \begin{equation}\begin{split}\label{alpha2xi+}
  \alpha_{j-1}-\alpha_j\leq & \frac{4^{-j-100}}{\hat B}\xi\omega+\gamma\left(r_{j-1}^{p-n}
  \int_{B_{j-1}} g(y)^{\frac{p}{p-1}}
  \,\mathrm {d}y\right)^{\frac{1}{p}}\\&+\gamma\left(r_{j-1}^{p-n}\int_{B_{j-1}}|f(y)|
  \,\mathrm {d}y\right)^{\frac{1}{p-1}}+\gamma r_{j-1}
  \end{split}\end{equation}
  for all $j=0,1,2,\cdots$, where the constant $\gamma$ depends only upon the data.
  Moreover, we define a quantity $A_j(l)$ by
 \begin{equation}\begin{split}\label{A_jxi+}
 A_j(l)=&\frac{d_j(l)^{p-2}}{r_j^{n+p}}\iint_{L_j(l)}\left(\frac{l_j-v}{d_j(l)}\right)^{(1+\lambda)(p-1)}\varphi_j(l)^{M-p}
 \,\mathrm {d}x\mathrm {d}t
 \\&+\frac{d_j(l)^{q-2}}{r_j^{n+q}}\iint_{L_j(l)}a(x,t)\left(\frac{l_j-v}{d_j(l)}\right)^{(1+\lambda)(q-1)}\varphi_j(l)^{M-q}
 \,\mathrm {d}x\mathrm {d}t
 \\&+\esssup_t\frac{1}{r_j^n}\int_{B_j\times\{t\}}G\left(\frac{l_j-v}{d_j(l)}\right)\varphi_j(l)^M
 \,\mathrm {d}x,
 \end{split}\end{equation}
 where $M>q$, $G$ is defined in \eqref{G} and $L_j(l)=Q_j(l)\cap \{v\leq l_j\}\cap \Omega_T$.
 Throughout the proof, we keep $\lambda=\frac{p}{nq}$.
 It is easy to check that $A_j(l)$ is continuous in $l<l_j$. In the following, the sequence $\{l_j\}_{j=0}^\infty$ we constructed will satisfy $l_j\geq\frac{1}{4}\xi\omega$
 for all $j=0,1,2,\cdots$.
 Analysis similar to that in the proof of Lemma \ref{lemmaDeGiorgi1}, Step 1 shows
 that either $v(x_1,t_1)\geq\frac{1}{4}\xi\omega$ or $A_j(l)\to+\infty$ as $l\to l_j$.
 It is obvious that the lemma holds if $v(x_1,t_1)\geq \frac{1}{4}\xi\omega$. In the following, we assume that
  $A_j(l)\to+\infty$ as $l\to l_j$ holds for all $j=0,1,2,\cdots$.

 Step 2: \emph{Determine the values of $l_0$, $l_1$ and $\nu_2$.} Initially, we
 set $l_0=\xi\omega$ and $\bar l=\frac{1}{2}l_0+\frac{1}{4}
 \hat B\alpha_0+\frac{1}{8}\xi\omega$.
 Recalling that $\hat B\alpha_0< \xi\omega$, we deduce $\frac{1}{8}\xi\omega\leq d_0(\bar l)\leq\frac{3}{8}\xi\omega$ and
hence $d_0(\bar l)\leq\frac{3}{4}\|u\|_\infty$. Next, we claim that
 $Q_0(\bar l)\subseteq Q_A^{(1)}$.
In the $(p,q)$-phase, i.e., $a_0\geq 10[a]_\alpha R^\alpha$, we observe from \eqref{phase analysis} that
$\frac{4}{5}a_0\leq a_1\leq \frac{6}{5}a_0$. Recalling that $r_0=C^{-1}R$ and $\xi=\hat B A^{-1}>A^{-1}$, we have
\begin{equation*}\begin{split}
\Theta_0(\bar l)&\leq d_{0}(\bar l)^2\left[\left(\frac{Cd_{0}(\bar l)}{R}\right)^p+\frac{4}{5}a_0\left(\frac{Cd_{0}
(\bar l)}{R}\right)^q\right]^{-1}
\leq \left(\frac{1}{8}\right)^{2-q}\frac{5}{4}C^{-p}\Theta_A<\frac{1}{4}\Theta_A,
\end{split}\end{equation*}
provided that we choose $C>5^\frac{1}{p}\cdot 8^\frac{q-2}{p}$.
Next, we consider the case of $p$-phase, i.e., $a_0<10[a]_\alpha R^\alpha$.
Since $\omega\leq 2\|u\|_\infty$, $R<1$ and $q\leq p+\alpha$, we get $\Theta_A\geq (1+\tilde\gamma_1)^{-1}A^{p-2}\omega^{2-p}R^p$,
where $\tilde\gamma_1=10[a]_\alpha2^{q-p}\|u\|_\infty^{q-p}$.
Since $d_0(\bar l)\geq \frac{1}{8}\xi\omega>\frac{1}{8}A^{-1}\omega$, we obtain
 \begin{equation*}\begin{split}
\Theta_0(\bar l)&\leq C^{-p}d_{0}(\bar l)^{2-p}R^p\leq 8^{p-2}C^{-p}(1+\tilde\gamma_1)\Theta_A<\frac{1}{4}\Theta_A,
\end{split}\end{equation*}
provided that we choose $C>4^\frac{1}{p}\cdot 8^\frac{p-2}{p}(1+\tilde\gamma_1)^\frac{1}{p}$.
At this stage, we choose
$C>5^\frac{1}{p}\cdot 8^\frac{q-2}{p}(1+\tilde\gamma_1)^\frac{1}{p}.$
This choice of $C$ ensures that $\Theta_0(\bar l)<\frac{1}{4}\Theta_A$ and hence
$Q_0(\bar l)\subseteq Q_A^{(1)}$.

At this point, we fix a number $\chi\in(0,1)$ which will be chosen later.
We claim that $A_0(\bar l)\leq\frac{1}{2}\chi$.
Once again, we consider two cases: $p$-phase and $(p,q)$-phase. In the case of $p$-phase, i.e., $a_0< 10[a]_\alpha R^\alpha$,
we have $a(x,t)\leq a_0+|a(x,t)-a_0|\leq 12[a]_\alpha R^\alpha$ for all $(x,t)\in Q_0(\bar l)$, since $Q_0(\bar l)\subseteq Q_A^{(1)}\subseteq Q_A\subseteq Q_{R,R^2}$.
In view of $l_0-v\leq \xi\omega$, $\frac{1}{8}\xi\omega\leq d_0(\bar l)\leq\frac{3}{8}\xi\omega\leq\frac{3}{4}\|u\|_\infty$ and $q\leq p+\alpha$, we infer from \eqref{QA1}
that
 \begin{equation}\begin{split}\label{A01xi+}
&\frac{d_0(\bar l)^{p-2}}{r_0^{n+p}}\iint_{L_0(\bar l)}\left(\frac{l_0-v}{d_0(\bar l)}\right)^{(1+\lambda)(p-1)}\varphi_0(\bar l)^{M-p}
 \,\mathrm {d}x\mathrm {d}t
 \\&+\frac{d_0(\bar l)^{q-2}}{r_0^{n+q}}\iint_{L_0(\bar l)}a(x,t)\left(\frac{l_0-v}{d_0(\bar l)}\right)^{(1+\lambda)(q-1)}\varphi_0(\bar l)^{M-q}
 \,\mathrm {d}x\mathrm {d}t
 \\&\leq \gamma C^{n+q}\hat B^{q-2}\frac{|Q_0(\bar l)\cap \{v\leq l_0\}|}{A^{p-2}\omega^{2-p}R^{n+p}}
\\& \leq \gamma C^{n+q}\hat B^{q-2}\frac{| Q_A^{(1)}\cap \{v\leq \xi\omega\}|}{| Q_A^{(1)}|}\leq \gamma C^{n+q}\hat B^{q-2}\nu_2,
 \end{split}\end{equation}
 since $|Q_A^{(1)}|\leq c_nA^{p-2}\omega^{2-p}R^{n+p}$ and $Q_0(\bar l)\subseteq Q_A^{(1)}$.
 In the case of $(p,q)$-phase, i.e., $a_0\geq10[a]_\alpha R^\alpha$,
 we have $\frac{4}{5}a_0\leq a(x,t)\leq \frac{6}{5}a_0$ for all $(x,t)\in Q_0(\bar l)$, since $Q_0(\bar l)\subseteq Q_{R,R^2}$. According to
 $l_0-v\leq \xi\omega$ and $\frac{1}{8}\xi\omega\leq d_0(\bar l)\leq\frac{3}{8}\xi\omega$, we infer from \eqref{QA1}
that
  \begin{equation}\begin{split}\label{A02xi+}
&\frac{d_0(\bar l)^{p-2}}{r_0^{n+p}}\iint_{L_0(\bar l)}\left(\frac{l_0-v}{d_0(\bar l)}\right)^{(1+\lambda)(p-1)}\varphi_0(\bar l)^{M-p}
 \,\mathrm {d}x\mathrm {d}t
 \\&+\frac{d_0(\bar l)^{q-2}}{r_0^{n+q}}\iint_{L_0(\bar l)}a(x,t)\left(\frac{l_0-v}{d_0(\bar l)}\right)^{(1+\lambda)(q-1)}\varphi_0(\bar l)^{M-q}
 \,\mathrm {d}x\mathrm {d}t
 \\&\leq \gamma C^{n+q}d_0(\bar l)^{-2}\left(\frac{d_0(\bar l)^p}{R^{n+p}}+a_0\frac{d_0(\bar l)^q}{R^{n+q}}\right)|Q_0(\bar l)\cap \{v\leq l_0\}|
 \\& \leq \gamma C^{n+q}\hat B^{q-2}\frac{| Q_A^{(1)}\cap \{v\leq \xi\omega\}|}{| Q_A^{(1)}|}\leq \gamma C^{n+q}\hat B^{q-2}\nu_2,
 \end{split}\end{equation}
 since $Q_0(\bar l)\subseteq Q_A^{(1)}$.
According to the proofs of \eqref{A01xi+} and \eqref{A02xi+}, we conclude that
\begin{equation}\label{timederivativeL_0xi+}\left(\frac{d_0(\bar l)^{p-2}}{r_0^{n+p}}+a_1\frac{d_0(\bar l)^{q-2}}{r_0^{n+q}}\right)
|L_0(\bar l)|\leq \gamma C^{n+q}\hat B^{q-2}\nu_2,\end{equation}
where $a_1=a(x_1,t_1)$. Furthermore, we use the Caccioppoli estimate \eqref{Cacformula1} with $(u,l,d,\Theta)$ replaced by $(v,l_0,d_0(\bar l),\Theta_0(\bar l))$
to deduce that
 \begin{equation*}\begin{split}
    \esssup_t& \frac{1}{r_0^n}\int_{B_0\times\{t\}}G\left(\frac{l_0-v}{d_0(\bar l)}\right)\varphi_0(\bar l)^M
 \,\mathrm {d}x \\ \leq &\gamma\frac{d_0(\bar l)^{p-2}}{r_0^{p+n}}\iint_{L_0(\bar l)}\left(\frac{l_0-v}{d_0(\bar l)}
 \right)^{(1+\lambda)(p-1)}
 \varphi_0(\bar l)^{M-p}\,\mathrm {d}x\mathrm {d}t
 \\ &+\gamma\frac{d_0(\bar l)^{q-2}}{r_0^{q+n}}\iint_{L_0(\bar l)}a(x,t)\left(\frac{l_0-v}{d_0(\bar l)}
 \right)^{(1+\lambda)(q-1)}
 \varphi_0(\bar l)^{M-q}\,\mathrm {d}x\mathrm {d}t
 \\&+\gamma\frac{1}{r_0^n} \iint_{L_0(\bar l)}\frac{l_0-v}{d_0(\bar l)}\varphi_0(\bar l)^{M-1}|\partial_t\varphi_0(\bar l)|\,\mathrm {d}x\mathrm {d}t
   \\&+\gamma \frac{\Theta_0(\bar l)}{r_0^nd_0(\bar l)^2}\int_{B_0}g^{\frac{p}{p-1}}\,\mathrm {d}x
  +\gamma \frac{\Theta_0(\bar l)}{r_0^nd_0(\bar l)}\int_{B_0}|f|\,\mathrm {d}x
  \\ =&L_1+L_2+L_3+L_4+L_5,
 \end{split}\end{equation*}
 since $|D\varphi_0(\bar l)|\leq r_0^{-1}$ and $\varphi_0(\bar l)=0$ on $\partial_PQ_0(\bar l)$. In view of \eqref{A01xi+} and \eqref{A02xi+},
 we deduce $L_1+L_2\leq  \gamma C^{n+q}\hat B^{q-2}\nu_2$.
To estimate $L_3$, we use \eqref{timederivativeL_0xi+} and
$|\partial_t\varphi_0(\bar l)|\leq 9\Theta_0(\bar l)^{-1}$ to deduce that
 \begin{equation*}\begin{split}
 L_3&\leq \gamma\frac{1}{r_0^n}\Theta_0(\bar l)^{-1}|L_0(\bar l)|
\leq
 \gamma\left(\frac{d_0(\bar l)^{p-2}}{r_0^{n+p}}+a_1\frac{d_0(\bar l)^{q-2}}{r_0^{n+q}}\right)
|L_0(\bar l)|\leq \gamma C^{n+q}\hat B^{q-2}\nu_2.
  \end{split}\end{equation*}
Observe that $\Theta_0(\bar l)\leq d_0(\bar l)^{2-p}r_0^p$ and $d_0(\bar l)\geq \frac{1}{8}\xi\omega$.
We infer from \eqref{omega1violated+} that
 \begin{equation*}\begin{split}
 L_4+L_5\leq \gamma\frac{d_0(\bar l)^{-p}}{r_0^{n-p}}\int_{B_0}g^\frac{p}{p-1}\,\mathrm {d}x+\gamma\frac{d_0(\bar l)^{1-p}}{r_0^{n-p}}\int_{B_0}|f|\,\mathrm {d}x\leq
 \gamma C^{n-p}\left(\hat B^{-p}+\hat B^{-(p-1)}\right),
  \end{split}\end{equation*}
  where the constant $\gamma$ depends only upon the data.
  Consequently, we conclude from \eqref{A01xi+}, \eqref{A02xi+} and the estimates for $L_1$-$L_5$ that
  there exists a constants $\hat \gamma_0$ depending only upon the data, such that
  \begin{equation*}\begin{split}
  A_0(\bar l)\leq \hat\gamma_0 C^{n+q}\hat B^{q-2}\nu_2+\hat\gamma_0 C^{n-p}\left(\hat B^{-p}+\hat B^{-(p-1)}\right).
   \end{split}\end{equation*}
    At this point, we  fix a number $\chi\in(0,1)$ which will be chosen later in a universal way.
    Then, we choose $\nu_2=\nu_2(\text{data},\chi)<1$ and $\hat B=\hat B(\text{data},\chi)>1$ such that
     \begin{equation}\begin{split}\label{nu0xi+}
  \hat\gamma_0 C^{n+q}\hat B^{q-2}\nu_2=\frac{\chi}{4}\qquad\text{and}\qquad
  \hat\gamma_0 C^{n-p}(\hat B^{-p}+\hat B^{1-p})<\frac{\chi}{4}.
    \end{split}\end{equation}
    This leads to $A_0(\bar l)\leq\frac{1}{2}\chi$. Taking into account that $A_0(l)\to+\infty$
    as $l\to l_0$ and $A_0(l)$ is continuous in $(\bar l, l_0)$, then there exists a number $\tilde l\in (\bar l, l_0)$ such that $A_0(\tilde l)=\chi$.
    From \eqref{omega1violated+}, we infer that for $\hat B>2$ there holds
  $l_0-\bar l\geq \frac{1}{8}\xi\omega\geq \frac{1}{4
    \hat B}\xi\omega>\frac{1}{4}(\alpha_{-1}-\alpha_0),$
    since $\hat B\alpha_{-1}<\xi\omega$.
    At this point, we set
    \begin{equation}\label{l1xi+}
	l_1=\begin{cases}
\tilde l,&\quad \text{if}\quad \tilde l<l_0-\frac{1}{4}(\alpha_{-1}-\alpha_0),\\
	l_0-\frac{1}{4}(\alpha_{-1}-\alpha_0),&\quad \text{if}\quad \tilde l\geq l_0-\frac{1}{4}(\alpha_{-1}-\alpha_0).
	\end{cases}
\end{equation}
Moreover, we define $Q_0=Q_0(l_1)$ and $d_0=l_0-l_1$. Since $\hat B\alpha_0<\xi\omega$, we have $l_1\geq  \bar l>\frac{1}{2}\hat B\alpha_0+\frac{1}{4}\xi\omega$.

Step 3:  \emph{Determine the sequence $\{l_j\}_{j=0}^\infty$.} Suppose that we have chosen two
sequences $l_1,\cdots,l_j$ and $d_0,\cdots,d_{j-1}$ such that for $i=1,\cdots,j$, there holds
    \begin{equation}\label{lixi+}\frac{1}{8}\xi\omega+\frac{1}{2}l_{i-1}+\frac{1}{4}\hat B\alpha_{i-1}<l_i\leq
    l_{i-1}-\frac{1}{4}(\alpha_{i-2}-\alpha_{i-1}),
    \end{equation}
    \begin{equation}\label{Aj-1xi+}
    A_{i-1}(l_i)\leq \chi,
     \end{equation}
      \begin{equation}\label{ljxi+}
      l_i>\frac{1}{2}\hat B\alpha_{i-1}+\frac{1}{4}\xi\omega.
       \end{equation}
       Then, we define $Q_i=Q_i(l_{i+1})$, $d_i=d_i(l_{i+1})=l_i-l_{i+1}$, $L_i=L_i(l_{i+1})$, $\varphi_i=\varphi_i(l_{i+1})$ and
       $\Theta_i=\Theta_i(l_{i+1})$ for $i=1,2,\cdots,j-1$. Moreover, we claim that, with suitable choices of $\chi$ and $\hat B$, there holds
       \begin{equation}\label{Ajxi+}
       A_j(\bar l)\leq \frac{1}{2}\chi,\qquad\text{where}\qquad \bar l=\frac{1}{2}l_j+\frac{1}{4}\hat B\alpha_j+\frac{1}{8}\xi\omega.
        \end{equation}
        First, we claim that the inclusions
        $Q_i\subseteq Q_A$ and $Q_i\subseteq Q_{r_i,r_i^2}(x_1,t_1)$ hold for $i=0,1,\cdots,j-1$.
        In the case of $p$-phase, i.e., $a_0<10[a]_\alpha R^\alpha$, we have
       $\Theta_A\geq \hat\gamma_1A^{p-2}\omega^{2-p}R^p$,
        where $\hat \gamma_1=(1+10[a]_\alpha2^{q-p}\|u\|_\infty^{q-p})^{-1}=(1+\tilde \gamma_1)^{-1}$.
        In view of \eqref{alpha1xi+}, we
        see that $\alpha_{i-1}-\alpha_i\geq\frac{1}{\hat B}4^{-i-100}\xi\omega$.
         It follows from \eqref{lixi+} that $d_i\geq \frac{1}{\hat B}4^{-i-101}\xi\omega=
        4^{-i-101}A^{-1}\omega$ and hence
        \begin{equation*}\begin{split}
        \Theta_i=d_i^2\left[\left(\frac{d_i}{r_i}\right)^p+a_1\left(\frac{d_i}{r_i}\right)^q\right]^{-1}\leq d_i^{2-p}r_i^p\leq 4^{101(p-2)}A^{p-2}\omega^{2-p}C^{-p}R^p<\frac{1}{4}\Theta_A,
        \end{split}\end{equation*}
        provided that we choose
$C>4^{101}\hat\gamma^{-\frac{1}{p}}$.
        This implies that $Q_i\subseteq Q_A$ for the $p$-phase. In the case of $(p,q)$-phase, i.e., $a_0\geq10[a]_\alpha R^\alpha$, we have
        $\frac{4}{5}a_0\leq a_1\leq \frac{6}{5}a_0$. We infer from $d_i>4^{-i-101}A^{-1}\omega$ that
         \begin{equation*}\begin{split}
        \Theta_i&\leq \frac{5}{4}d_i^2\left[\left(\frac{d_i}{r_i}\right)^p+a_0\left(\frac{d_i}{r_i}\right)^q\right]^{-1}
       \\&\leq\frac{5}{4}\left(4^{-i-101}A^{-1}\omega\right)^2\left[\left(\frac{4^{-i-101}A^{-1
      }\omega}{4^{-i}C^{-1}R}\right)^p+a_0
        \left(\frac{4^{-i-101}A^{-1}\omega}{4^{-i}C^{-1}R}\right)^q\right]^{-1}
       \\& \leq \frac{5}{4}4^{101(q-2)}C^{-p}\Theta_A<\frac{1}{4}\Theta_A,
        \end{split}\end{equation*}
 provided that we choose $C>4^{101\frac{q}{p}}$.
         This leads to $Q_i\subseteq Q_A$ for the $(p,q)$-phase.
         At this point, we choose
          \begin{equation}\label{definitionofCxi+}
          C=4^{101\frac{q}{p}}(1+10[a]_\alpha2^{q-p}\|u\|_\infty^{q-p})^\frac{1}{p}.\end{equation}
         On the other hand, we infer from \eqref{alpha1xi+} and \eqref{lixi+} that
          $d_i=l_i-l_{i+1}\geq \frac{1}{4}(\alpha_{i-1}-\alpha_i)\geq 50r_i.$ This leads to
         $\Theta_i\leq(50r_i)^2\left(50^p+a_150^q\right)^{-1}\leq 50^{2-p}r_i^2<r_i^2$ and hence $Q_i\subseteq Q_{r_i,r_i^2}(x_1,t_1)$.
       We now turn our attention to the proof of \eqref{Ajxi+}.
        Next, we observe from \eqref{lixi+} and \eqref{ljxi+} that $d_j(\bar l)\geq\frac{1}{4}d_{j-1}+\frac{1}{4}\hat B(\alpha_{j-1}-\alpha_j)$.
It is easy to check that $Q_j(\bar l)\subseteq Q_{r_j,r_j^2}(x_1,t_1)$ and $Q_j(\bar l)\subseteq Q_{j-1}\subseteq Q_A$.

Repeating the arguments in the proof of Lemma \ref{lemmaDeGiorgi1} Step 3,
we conclude that for any fixed $\epsilon_1,\epsilon_2\in(0,1)$, there exist constants
$\hat\gamma=\hat\gamma(\text{data})$, $\hat\gamma_1=\hat\gamma_1(\text{data},\epsilon_1)$ and $\hat\gamma_2=\hat\gamma_2(\text{data},\epsilon_1,\epsilon_2)$, such that
\begin{equation}\begin{split}\label{estimate for Ajbarlchixi+}
A_j(\bar l)& \leq
 \epsilon_2\hat\gamma_1\chi+\epsilon_1\hat\gamma \chi+\hat\gamma(\hat B^{-p}+\hat B^{-(p-1)})
  \\&\quad+\hat\gamma_2\chi^{1+\frac{p}{n}}+\hat\gamma_2\chi^{1+\frac{q}{n}}+\hat\gamma_2\left(\hat B^{-p}+\hat B^{-(p-1)}\right)^{1+\frac{p}{n}}
  +\left(\hat B^{-p}+\hat B^{-(p-1)}\right)^{1+\frac{q}{n}}.
\end{split}\end{equation}
Our task now is to fix the parameters. We first choose $\epsilon_1=(100\hat\gamma)^{-1}$ and this fixes $\hat\gamma_1=\hat\gamma_1(\epsilon_1)$. Next, we set $\epsilon_2=(100\hat\gamma_1)^{-1}$.
This also fixes $\hat\gamma_2$ with the choices of $\epsilon_1$ and $\epsilon_2$. At this point, we set $\hat B>8$ large enough to have
\begin{equation}\label{conditionforBxi+}
\hat B^{-p}+\hat B^{-(p-1)}<\frac{1}{100\hat\gamma}\chi.
\end{equation}
Finally, we set $\chi=\left(\frac{1}{100\hat\gamma_2}\right)^\frac{n}{p}2^{-1-\frac{n}{p}}$. With the choices of $\epsilon_1$, $\epsilon_2$, $\chi$ and $\hat B$,
      we infer from \eqref{estimate for Ajbarlchixi+} that $A_j(\bar l)\leq \frac{1}{2}\chi$,
      which proves the claim
      \eqref{Ajxi+}. Recalling that $A_j(l)\to+\infty$
    as $l\to l_j$ and $A_j(l)$ is continuous in $(\bar l, l_j)$, then there exists a number $\tilde l\in (\bar l, l_j)$ such that $A_j(\tilde l)=\chi$.
    At this point, we define
    \begin{equation}\label{definitionlj+1xi+}
	l_{j+1}=\begin{cases}
\tilde l,&\quad \text{if}\quad \tilde l<l_j-\frac{1}{4}(\alpha_{j-1}-\alpha_j),\\
	l_j-\frac{1}{4}(\alpha_{j-1}-\alpha_j),&\quad \text{if}\quad \tilde l\geq l_j-\frac{1}{4}(\alpha_{j-1}-\alpha_j).
	\end{cases}
\end{equation}
It is easily seen that
\eqref{lixi+}-\eqref{ljxi+} are valid for $i=j+1$. Consequently, we have now constructed a sequence $\{l_i\}_{i=0}^\infty$ satisfying \eqref{lixi+}-\eqref{ljxi+}.

Step 4: \emph{Proof of the inequality $v(x_1,t_1)>\frac{1}{4}\xi\omega$.} We first observe from the construction of $\{l_i\}_{i=0}^\infty$ that
this sequence is decreasing and hence that there exists a number $\hat l\geq  0$ such that
\begin{equation*}\hat l=\lim_{i\to\infty}l_i.\end{equation*}
This also implies that $d_i=l_i-l_{i+1}\to 0$ as $i\to\infty$. Moreover, we infer that $\hat l=v(x_1,t_1)$.
At this point, we claim that for any $j\geq1$ there holds
  \begin{equation}\begin{split}\label{djdj-1xi+}
  d_j\leq &\frac{1}{4}d_{j-1}+\gamma \frac{4^{-j-100}}{\hat B}\xi\omega+
  \gamma\left(r_{j-1}^{p-n}\int_{B_{j-1}} g(y)^{\frac{p}{p-1}}
  \,\mathrm {d}y\right)^{\frac{1}{p}}\\&+\gamma\left(r_{j-1}^{p-n}\int_{B_{j-1}}|f(y)|
  \,\mathrm {d}y\right)^{\frac{1}{p-1}}+\gamma r_{j-1},
  \end{split}\end{equation}
  where the constant $\gamma$ depends only upon the data. In order to prove \eqref{djdj-1xi+}, we assume that for any fixed $j\geq 1$,
  \begin{equation}\begin{split}\label{djdj-1proof}
  d_j>\frac{1}{4}d_{j-1}\qquad\text{and}\qquad d_j>\frac{1}{4}(\alpha_{j-1}-\alpha_j),
   \end{split}\end{equation}
   since otherwise \eqref{djdj-1xi+} follows from \eqref{alpha2xi+}.
   We first observe from $d_j>\frac{1}{4}(\alpha_{j-1}-\alpha_j)$ and \eqref{definitionlj+1xi+} that $A_j(l_{j+1})=A_j(\tilde l)=\chi$.
   In view of $d_j>\frac{1}{4}d_{j-1}$, we see that
$\Theta_j
<\frac{1}{16}\Theta_{j-1}$
and this yields that $Q_j\subseteq Q_{j-1}$ and $\varphi_{j-1}=1$ in $Q_j$. Recalling that $Q_i\subseteq Q_A$ and $Q_i\subseteq Q_{r_i,r_i^2}(z_1)$ hold for
all $i\geq 0$, an argument similar to the one used in step 3 shows that
\begin{equation}\begin{split}\label{estimate for Ajbarlchistep4xi+}
\chi=A_j(l_{j+1})& \leq
 \epsilon_2\hat\gamma_1\chi+\epsilon_1\hat\gamma \chi+\hat\gamma\left(\frac{r_j^{p-n}}{d_j^p}\int_{B_{j}} g(y)^{\frac{p}{p-1}}
  \,\mathrm {d}y+\frac{r_j^{p-n}}{d_j^{p-1}}\int_{B_{j}} |f(y)|
  \,\mathrm {d}y\right)
  \\&\quad+\hat\gamma_2\chi^{1+\frac{p}{n}}+\hat\gamma_2\left(\frac{r_j^{p-n}}{d_j^p}\int_{B_{j}} g(y)^{\frac{p}{p-1}}
  \,\mathrm {d}y+\frac{r_j^{p-n}}{d_j^{p-1}}\int_{B_{j}} |f(y)|
  \,\mathrm {d}y\right)^{1+\frac{p}{n}}
 \\&\quad +\hat\gamma_2\chi^{1+\frac{q}{n}}+\hat\gamma_2\left(\frac{r_j^{p-n}}{d_j^p}\int_{B_{j}} g(y)^{\frac{p}{p-1}}
  \,\mathrm {d}y+\frac{r_j^{p-n}}{d_j^{p-1}}\int_{B_{j}} |f(y)|
  \,\mathrm {d}y\right)^{1+\frac{q}{n}},
\end{split}\end{equation}
where $\hat\gamma=\hat\gamma(\text{data})$, $\hat\gamma_1=\hat\gamma_1(\text{data},\epsilon_1)$ and $\hat\gamma_2=\hat\gamma_2(\text{data},\epsilon_1,\epsilon_2)$
are the constants in \eqref{estimate for Ajbarlchixi+}. According to the choices of $\epsilon_1$, $\epsilon_2$ and $\chi$,
we infer that the  terms involving $\chi$ on the right-hand side of \eqref{estimate for Ajbarlchistep4xi+} can be re-absorbed into the left-hand side.
This implies that
  \begin{equation*}\begin{split}
  d_j\leq\gamma\left(r_{j-1}^{p-n}\int_{B_{j-1}}g(y)^{\frac{p}{p-1}}\,\mathrm {d}y\right)^{\frac{1}{p}}
\qquad\text{or}\qquad
   d_j\leq
  \gamma \left(r_{j-1}^{p-n}\int_{B_{j-1}}|f(y)|\,\mathrm {d}y\right)^{\frac{1}{p-1}},
   \end{split}\end{equation*}
   which establishes the estimate \eqref{djdj-1xi+}. At this stage, we let $J>1$ be a fixed integer and sum up the inequality \eqref{djdj-1xi+}
   for $j=1,\cdots,J-1$. Consequently, we infer that
   \begin{equation*}\begin{split}
 \sum_{j=1}^{J-1} d_j\leq& \frac{1}{4}\sum_{j=1}^{J-1}d_{j-1}+\gamma \frac{1}{\hat B}
   \sum_{j=1}^{J-1} 4^{-j-100}\xi\omega+\gamma\sum_{j=1}^{J-1}\left(r_{j-1}^{p-n}\int_{B_{j-1}}|f(y)|
  \,\mathrm {d}y\right)^{\frac{1}{p-1}}
\\&+ \gamma\sum_{j=1}^{J-1}\left(r_{j-1}^{p-n}\int_{B_{j-1}} g(y)^{\frac{p}{p-1}}
  \,\mathrm {d}y\right)^{\frac{1}{p}}+\gamma  \sum_{j=1}^{J-1}r_{j-1}
  \\ \leq& \frac{1}{4}\sum_{j=1}^{J-1}d_{j-1}+\gamma\frac{1}{\hat B}\xi\omega+\gamma\left(F_p(R)+G_p(R)+ R\right),
   \end{split}\end{equation*}
   where the constant $\gamma$ depends only upon the data. Recalling that $d_j=l_j-l_{j+1}$,
   we pass to the limit $J\to\infty$ to deduce that
    \begin{equation*}
\frac{3}{4}(l_1-v(x_1,t_1))-\frac{1}{4}d_0\leq\gamma\frac{1}{\hat B}\xi\omega+\gamma\left(F_p(R)+G_p(R)+ R\right)\leq\hat  C\frac{1}{\hat B}\xi\omega,
    \end{equation*}
    where we have used \eqref{omega1violated+} in the last step. Recalling that $l_0=\xi\omega$ and $d_0\leq d_0(\bar l)<\frac{3}{8}\xi\omega$,
    we deduce that for $\hat B>8\hat  C$ there holds $v(x_1,t_1)\geq \frac{1}{3}\xi\omega$.
    This proves the desired inequality. Finally, we fix the values of parameters $\nu_2$ and $\hat B$. According to \eqref{nu0xi+} and \eqref{conditionforBxi+}, we choose
     \begin{equation}\label{definehatB}
     \hat B=\max\left\{8\hat C,\  \left(\frac{200\hat\gamma}{\chi}\right)^\frac{1}{p-1},\ \left(\frac{8\hat\gamma_0C^{n-p}}{\chi}\right)^\frac{1}{p-1}\right\},
       \end{equation}
       where the constant $C$ is defined in \eqref{definitionofCxi+}. From \eqref{nu0xi+}, we define $\nu_2$ by
        \begin{equation}\label{definenu2}
        \nu_2=\frac{\chi}{4\hat\gamma_0  C^{n+q}\hat B^{q-2}}.
        \end{equation}
        With these choices of $\nu_2$ and $\hat B$, we conclude that the lemma holds.
 \end{proof}
 The next lemma asserts that the estimate of the measure of level sets in the form \eqref{QA1} may hold for any choice of $\nu_2$.
 According to Lemma \ref{time 2nd}, we remark that if \eqref{2ndomega assumption1} is violated, then the slicewise estimate
\begin{equation}\label{2nd measure estimate++}
\left|\left\{x\in B_{R}:u(x,t)>\mu_+-\frac{\omega}{2^{s_1}}\right\}\right|\leq
\left(1-\left(\frac{\nu_0}{2}\right)^2\right)|B_{R}|
\end{equation}
holds
for all $t\in \left(-\frac{1}{2}\Theta_A,0\right]$. With the help of this inequality we can now prove the following lemma.
\begin{lemma}\label{DeGiorgi3}
Let $u$ be a bounded weak solution to \eqref{parabolic}-\eqref{A} in $\Omega_T$. Assume that \eqref{2nd measure estimate++} holds for
all $t\in \left[-\frac{1}{2}\Theta_A,0\right]$.
For every $\bar\nu\in(0,1)$, there exists a positive integer $q_*=q_*(\text{data},\bar\nu)$ such that either
\begin{equation}\begin{split}\label{2st omega assumption}
\omega\leq 2^{s_1+q_*}F_p(R_0)+2^{s_1+q_*}G_p(R_0)
\end{split}\end{equation}
or
\begin{equation}\label{measure estimate 2nd}\left|\left\{(x,t)\in Q_A^{(1)}:u\geq\mu_+-\frac{\omega}{2^{s_1+q_*}}\right\}
\right|\leq \bar\nu|Q_A^{(1)}|,\end{equation}
where
$s_1>2$ is the constant claimed by Lemma \ref{time 2nd} and
$A=2^{s_1+q_*}\hat B$.
\end{lemma}
\begin{proof}
In the following, we abbreviate $Q=Q_A^{(1)}$.
Let $q_*\in \mathbb{N}_+$ to be determined in the course of the proof.
For $j>2$ and $A=2^{s_1+q_*}\hat B$, we define $Q^\prime=B_{R_0}\times(-\Theta_A,0]$,
\begin{equation*}\begin{split}
A_j(t)&=\left\{x\in B_{R}:u(x,t)>\mu_+-\frac{\omega}{2^j}\right\}
\end{split}
\end{equation*}
and
\begin{equation*}Q_j^+=\left\{(x,t)\in Q:u(x,t)>\mu_+-\frac{\omega}{2^j}\right\}.\end{equation*}
Moreover, we set $k_j=\mu_+-2^{-j}\omega$ for $j=s_1,s_1+1,\cdots, s_1+q_*$.
Initially, we assume that for any $j>2$,
\begin{equation}\begin{split}\label{1st omega assumption}
\omega\geq 2^j\left(R^{p-n}\int_{B_R}|f|\,\mathrm{d}x
\right)^{\frac{1}{p-1}}+2^j\left(R^{p-n}\int_{B_R}g^{\frac{p}{p-1}}\,\mathrm{d}x
\right)^{\frac{1}{p}}.
\end{split}\end{equation}
Since $j\leq s_1+q_*$, we observe that if \eqref{1st omega assumption} is violated,
then \eqref{2st omega assumption} holds. What is left is to show that \eqref{measure estimate 2nd} holds.
We divide the proof of \eqref{measure estimate 2nd} into two steps.

Step 1: \emph{Proof of a gradient estimate.} We claim that there exists a constant $\gamma=\gamma(\text{data})$, such that
\begin{equation}\begin{split}\label{claimDu}
\iint_{Q_j^+}|Du|^p \,\mathrm {d}x\mathrm {d}t\leq \gamma\frac{1}{R^p}\left(\frac{\omega}{2^j}\right)^p|Q|.
\end{split}\end{equation}
Take a cutoff function $0\leq\varphi\leq1$,
such that $\varphi=1$ in $Q$, $\varphi=0$ on $\partial_PQ^\prime$,
\begin{equation*}\begin{split}
|D\varphi|\leq 10\frac{1}{R}\qquad\text{and}\qquad 0\leq
\frac{\partial \varphi}{\partial t}\leq 2\Theta_A^{-1}=2R^{-p}\left(\frac{\omega}{A}\right)^{p-2}+2a_0R^{-q}\left(\frac{\omega}{A}\right)^{q-2}.
\end{split}\end{equation*}
We consider the Caccioppoli estimate \eqref{Caccioppoli} for the truncated
functions $(u-k_j)_+$ over the cylinder $Q^\prime$
and obtain
\begin{equation}\begin{split}\label{DuCac}
&\iint_{Q^\prime}|D(u-k_j)_+|^p\varphi^q \,\mathrm {d}x\mathrm {d}t+\iint_{Q^\prime}a(x,t)|D(u-k_j)_+
|^q\varphi^q \,\mathrm {d}x\mathrm {d}t\\
&\leq \gamma\iint_{Q^\prime} (u-k_j)_+^p|D\varphi|^p+a(x,t)(u-k_j)_+^q|D\varphi|^q
\,\mathrm {d}x\mathrm {d}t
\\&\ \ +\gamma
\iint_{Q^\prime} (u-k_j)_+^2|\partial_t\varphi^q|\,\mathrm {d}x\mathrm {d}t
\\&\ \ +\gamma
\iint_{Q^\prime} |f|(u-k_j)_+\,\mathrm {d}x\mathrm {d}t+\gamma
\iint_{Q^\prime} g^{\frac{p}{p-1}}\,\mathrm {d}x\mathrm {d}t
\\ &=:T_1+T_2+T_3+T_4
.\end{split}\end{equation}
First, we address the estimates for the lower-order terms. In view of $u-k_j\leq 2^{-j}\omega$, we infer from \eqref{1st omega assumption} that
\begin{equation}\begin{split}\label{T3estimate}
T_3\leq \gamma\left(\frac{\omega}{2^jR^p}\right)
\left(R^{p-n}\int_{B_R}|f|\,\mathrm{d}x
\right)|Q|\leq \gamma\frac{1}{R^p}\left(\frac{\omega}{2^j}\right)^p|Q|.
\end{split}\end{equation}
Moreover, we infer from \eqref{1st omega assumption} that
\begin{equation}\begin{split}\label{T4estimate}
T_4=
\gamma\left(\frac{1}{R^p}\left(\frac{\omega}{2^j}\right)^p|Q|\right)\left[
\left(\frac{\omega}{2^j}\right)^{-p}
R^{p-n}\int_{B_R}g^{\frac{p}{p-1}}\,\mathrm{d}x\right]
\leq \gamma\frac{1}{R^p}\left(\frac{\omega}{2^j}\right)^p|Q|.
\end{split}\end{equation}
To proceed further, we distinguish two cases: $p$-phase and $(p,q)$-phase. In the case of $p$-phase, i.e., $a_0<10[a]_\alpha R_0^\alpha$.
Since $Q^\prime=B_{R_0}\times(-\Theta_A,0]\subseteq Q_{R_0,R_0^2}$, we have $a(x,t)\leq 12[a]_\alpha R_0^\alpha$. Noting that $(u-k_j)_+^q\leq \|u\|_\infty^{q-p}(u-k_j)_+^p$ in $Q^\prime$, we deduce
\begin{equation}\begin{split}\label{T1estimatepphase}
T_1&=\gamma\iint_{Q^\prime} (u-k_j)_+^p|D\varphi|^p+a(x,t)(u-k_j)_+^q|D\varphi|^q
\,\mathrm {d}x\mathrm {d}t
\\&\leq\gamma\left(\frac{\omega}{2^j}\right)^p\frac{1}{R^p}|Q|+\gamma[a]_\alpha\|u\|_\infty^{q-p}
R^\alpha\left(\frac{\omega}{2^j}\right)^p\frac{1}{R^q}|Q|
\\&\leq \gamma\left(\frac{\omega}{2^j}\right)^p\frac{1}{R^p}|Q|.
\end{split}\end{equation}
To estimate $T_2$, we observe that $A=2^{s_1+q_*}\hat B>2^{s_1+q_*}>2^j$ and hence
\begin{equation}\begin{split}\label{T2estimatepphase}
T_2&=\gamma
\iint_{Q^\prime} (u-k_j)_+^2|\partial_t\varphi^q|\,\mathrm {d}x\mathrm {d}t
\\&\leq\gamma\left(\frac{\omega}{2^j}\right)^2\left(\frac{\omega}{A}\right)^{p-2}\frac{1}{R^p}|Q|+\gamma[a]_\alpha
R^\alpha\left(\frac{\omega}{2^j}\right)^2\left(\frac{\omega}{A}\right)^{q-2}\frac{1}{R^q}|Q|
\\&\leq \gamma\left(\frac{\omega}{2^j}\right)^p\frac{1}{R^p}|Q|,
\end{split}\end{equation}
since $\omega\leq2\|u\|_\infty$. Inserting \eqref{T3estimate}-\eqref{T2estimatepphase} into \eqref{DuCac}, we arrive at
\begin{equation}\begin{split}\label{Du}
\iint_{Q^\prime}|D(u-k_j)_+|^p\varphi^q \,\mathrm {d}x\mathrm {d}t\leq \gamma\frac{1}{R^p}\left(\frac{\omega}{2^j}\right)^p|Q|,
\end{split}\end{equation}
which proves the claim \eqref{claimDu}. We now turn our attention to the case of $(p,q)$-phase, i.e., $a_0\geq10[a]_\alpha R_0^\alpha$.
Recalling that $Q^\prime\subseteq Q_{R_0,R_0^2}$, we get $\frac{4}{5}a_0\leq a(x,t)\leq \frac{6}{5}a_0$.
It follows that
\begin{equation}\begin{split}\label{upperboundpqphase}
\iint_{Q^\prime}a(x,t)|D(u-k_j)_+|^q\varphi^q \,\mathrm {d}x\mathrm {d}t\geq
\frac{4}{5}a_0\iint_{Q_j^+}|Du|^q \,\mathrm {d}x\mathrm {d}t.
\end{split}\end{equation}
In order to estimate $T_1$ and $T_2$ for the $(p,q)$-phase, we distinguish between the cases:
$\left(\frac{\omega}{2^j}\right)^pR^{-p}\leq a_0\left(\frac{\omega}{2^j}\right)^qR^{-q}$
and $\left(\frac{\omega}{2^j}\right)^pR^{-p}>a_0\left(\frac{\omega}{2^j}\right)^qR^{-q}$.
In the case $\left(\frac{\omega}{2^j}\right)^pR^{-p}> a_0\left(\frac{\omega}{2^j}\right)^qR^{-q}$, we find that
$\left(\frac{\omega}{A}\right)^{p-2}R^{-p}> a_0\left(\frac{\omega}{A}\right)^{q-2}R^{-q}$, since $A>2^j$.
In view of $(u-k_j)_+\leq 2^{-j}\omega$, we deduce
\begin{equation}\begin{split}\label{T1estimatepqphase}
T_1&=\gamma\iint_{Q^\prime} (u-k_j)_+^p|D\varphi|^p+a(x,t)(u-k_j)_+^q|D\varphi|^q
\,\mathrm {d}x\mathrm {d}t
\\&\leq\gamma\left(\frac{\omega}{2^j}\right)^p\frac{1}{R^p}|Q|+\gamma a_0
\left(\frac{\omega}{2^j}\right)^q\frac{1}{R^q}|Q|
\\&\leq \gamma
\left(\frac{\omega}{2^j}\right)^p\frac{1}{R^p}|Q|
\end{split}\end{equation}
and
\begin{equation}\begin{split}\label{T2estimatepqphase}
T_2&=\gamma
\iint_{Q^\prime} (u-k_j)_+^2|\partial_t\varphi^q|\,\mathrm {d}x\mathrm {d}t
\\&\leq\gamma\left(\frac{\omega}{2^j}\right)^2\left(\frac{\omega}{A}\right)^{p-2}\frac{1}{R^p}|Q|+\gamma a_0\left(\frac{\omega}{2^j}\right)^2\left(\frac{\omega}{A}\right)^{q-2}\frac{1}{R^q}|Q|
\\&\leq \gamma\left(\frac{\omega}{2^j}\right)^2\left(\frac{\omega}{A}\right)^{p-2}\frac{1}{R^p}|Q|<
\gamma
\left(\frac{\omega}{2^j}\right)^p\frac{1}{R^p}|Q|.
\end{split}\end{equation}
Inserting \eqref{T3estimate}, \eqref{T4estimate}, \eqref{T1estimatepqphase} and \eqref{T2estimatepqphase} into \eqref{DuCac},
we get the estimate \eqref{Du} and this proves the claim \eqref{claimDu}. Finally, we consider the case $\left(\frac{\omega}{2^j}\right)^pR^{-p}\leq a_0\left(\frac{\omega}{2^j}\right)^qR^{-q}$. We first observe from \eqref{T3estimate} and \eqref{T4estimate} that
\begin{equation}\begin{split}\label{T3T4estimatepqphase}
T_3+T_4
\leq \gamma\frac{1}{R^p}\left(\frac{\omega}{2^j}\right)^p|Q|\leq \gamma a_0\frac{1}{R^q}\left(\frac{\omega}{2^j}\right)^q|Q|.
\end{split}\end{equation}
From $(u-k_j)_+\leq 2^{-j}\omega$ and $A>2^j$, we conclude that
\begin{equation}\begin{split}\label{T1estimatepqphase2}
T_1&=\gamma\iint_{Q^\prime} (u-k_j)_+^p|D\varphi|^p+a(x,t)(u-k_j)_+^q|D\varphi|^q
\,\mathrm {d}x\mathrm {d}t
\\&\leq\gamma\left(\frac{\omega}{2^j}\right)^p\frac{1}{R^p}|Q|+\gamma a_0
\left(\frac{\omega}{2^j}\right)^q\frac{1}{R^q}|Q|
\\&\leq \gamma a_0
\left(\frac{\omega}{2^j}\right)^q\frac{1}{R^q}|Q|
\end{split}\end{equation}
and
\begin{equation}\begin{split}\label{T2estimatepqphase2}
T_2&=\gamma
\iint_{Q^\prime} (u-k_j)_+^2|\partial_t\varphi^q|\,\mathrm {d}x\mathrm {d}t
\\&\leq\gamma\left(\frac{\omega}{2^j}\right)^2\left(\frac{\omega}{A}\right)^{p-2}\frac{1}{R^p}|Q|+\gamma a_0\left(\frac{\omega}{2^j}\right)^2\left(\frac{\omega}{A}\right)^{q-2}\frac{1}{R^q}|Q|
\\&\leq\gamma\left(\frac{\omega}{2^j}\right)^p\frac{1}{R^p}|Q|+\gamma a_0
\left(\frac{\omega}{2^j}\right)^q\frac{1}{R^q}|Q|\\&\leq \gamma a_0
\left(\frac{\omega}{2^j}\right)^q\frac{1}{R^q}|Q|.
\end{split}\end{equation}
Combining \eqref{upperboundpqphase}, \eqref{T3T4estimatepqphase}, \eqref{T1estimatepqphase2} and \eqref{T2estimatepqphase2} with \eqref{DuCac},
we arrive at
\begin{equation}\begin{split}\label{claimDu11}
a_0\iint_{Q_j^+}|Du|^q \,\mathrm {d}x\mathrm {d}t\leq \gamma a_0\frac{1}{R^q}\left(\frac{\omega}{2^j}\right)^q|Q|.
\end{split}\end{equation}
According to \eqref{claimDu11}, the claim \eqref{claimDu} follows from the H\"older's inequality.

Step 2: \emph{Proof of \eqref{measure estimate 2nd}.}
We apply a De Giorgi-type lemma (see \cite[Lemma 2.2]{Di93}) to the function $u(\cdot,t)$,
for all $-\frac{1}{2}\Theta_A\leq t\leq0$, and with $l=\mu_+-2^{-(j+1)}\omega$
and $k=\mu_+-2^{-j}\omega$. From \eqref{2nd measure estimate++}, we deduce that for $j\geq s_1$,
\begin{equation*}\begin{split}
\left|\left\{x\in B_{R}:u(x,t)\leq \mu_+-\frac{\omega}
{2^j}\right\}\right|=|B_{R}|-|A_j(t)|\geq \left(\frac{\nu_0}{2}\right)^2|B_{R}|
\end{split}\end{equation*}
holds
for all $t\in \left(-\frac{1}{2}\Theta_A,0\right]$.
This yields that
the estimate
\begin{equation}\begin{split}\label{precedinginequality}
\frac{\omega}{2^{j+1}}|A_{j+1}(t)|\leq \gamma\frac{1}{\nu_0^2}\frac{R^{n+1}}{|B_R|}
\int_{A_s(t)\setminus A_{s+1}(t)}|Du|\,\mathrm{d}x
\leq \gamma \frac{R}{\nu_0^2}
\int_{A_s(t)\setminus A_{s+1}(t)}|Du|\,\mathrm{d}x
\end{split}\end{equation}
holds for any $-\frac{1}{2}\Theta_A\leq t\leq0$.
Integrating the inequality \eqref{precedinginequality}
over the interval $\left(-\frac{1}{2}\Theta_A,0\right]$, we infer from \eqref{claimDu} that
\begin{equation*}\begin{split}
\frac{\omega}{2^{j+1}}|Q_{j+1}^+|&\leq \gamma\frac{ R}{\nu_0^2}\iint_{Q_j^+\setminus Q_{j+1}^+}|Du|\,\mathrm{d}x\mathrm{d}t
\\&\leq \gamma\frac{ R}{\nu_0^2}\left(\iint_{Q_j^+\setminus Q_{j+1}^+}|Du|^p\,\mathrm{d}x\mathrm{d}t\right)^{\frac{1}{p}}
|Q_j^+\setminus Q_{j+1}^+|^{\frac{p-1}{p}}
\\&\leq \gamma\frac{1}{\nu_0^2}\left(\frac{\omega}{2^j}\right)|Q|^{\frac{1}{p}}
|Q_j^+\setminus Q_{j+1}^+|^{\frac{p-1}{p}}.
\end{split}\end{equation*}
Consequently, we infer that
\begin{equation*}\begin{split}
|Q_{j+1}^+|^{\frac{p}{p-1}}\leq \gamma\nu_0^{-\frac{2p}{p-1}}|Q|^{\frac{1}{p-1}}|Q_j^+\setminus Q_{j+1}^+|.
\end{split}\end{equation*}
Furthermore, we add up these inequalities for $j=s_1,s_1+1,\cdots,s_1+q_*-1$ and deduce that
\begin{equation*}\begin{split}
(q_*-1)|Q_{s_1+q_*}^+|^{\frac{p}{p-1}}\leq \gamma \nu_0^{-\frac{2p}{p-1}}|Q|^{\frac{p}{p-1}},
\end{split}\end{equation*}
where the constant $\gamma$ depends only upon the data.
From the definition of $Q_j^+$, the above inequality reads
\begin{equation*}\begin{split}
\left|\left\{(x,t)\in Q:u(x,t)>\mu_+-\frac{\omega}{2^{s_1+q_*}}\right\}\right|\leq \gamma\frac{1}{\nu_0^2}
\frac{1}{(q_*-1)^{\frac{p-1}{p}}}|Q|.
\end{split}\end{equation*}
At this stage, we take $q_*=q_*(\text{data},\bar\nu)\geq2$ according to
\begin{equation}\begin{split}\label{q*}
q_*=q_*(\bar\nu)=2+\left(\frac{\gamma}{\nu_0^2\bar\nu}\right)^\frac{p}{p-1}.
\end{split}\end{equation}
For such a choice of $q_*$, the inequality \eqref{measure estimate 2nd} holds true.
The proof of the lemma is now complete.
\end{proof}
We are now in a position to establish a decay estimate of the form \eqref{osc1st} for the second alternative
and the following proposition is our main result in this section.
\begin{proposition}\label{2nd proposition}
Let $\tilde Q_0=Q_{\frac{1}{16}R_0}^-(0)$
and let $u$ be a bounded weak solution to \eqref{parabolic}-\eqref{A} in $\Omega_T$.
There exist $0<\xi_2<1$ and $B_2>1$ depending only upon the data such that
\begin{equation*}\begin{split}
\essosc_{\tilde Q_0} u\leq (1-4^{-1}\xi_2)\omega+B_2\xi_2^{-1}\left(
F_p(R_0)+G_p(R_0)+R_0\right).
 \end{split}\end{equation*}
\end{proposition}
\begin{proof}
First, we assume that
\eqref{2ndomega assumption1} is violated.
Let $s_1>2$ be the constant claimed by Lemma \ref{time 2nd}.
According to Lemma \ref{time 2nd}, we get \eqref{2nd measure estimate++}
and this enables us to apply Lemma \ref{DeGiorgi3}.
At this stage, we take $\bar\nu=\nu_2$ in Lemma \ref{DeGiorgi3}, where $\nu_2$ is the constant in \eqref{definenu2}. This also fixes
$\xi=2^{-s_1-q_*}$, where $q_*$ is the constant in \eqref{q*}.
If \eqref{2st omega assumption} is violated, then
we infer from Lemma \ref{DeGiorgi3} that
\begin{equation*}\left|\left\{(x,t)\in Q_A^{(1)}:u\geq\mu_+-\xi\omega\right\}
\right|\leq \nu_2|Q_A^{(1)}|,\end{equation*}
where $Q_A^{(1)}=B_{R}\times\left(-\frac{1}{2}\Theta_A,0\right]$. This is actually the condition \eqref{QA1} for Lemma \ref{lemmaDeGiorgi1+}. If \eqref{omega1+} is violated, then the inequality \eqref{DeGiorgi1+} reads
\begin{equation*}\begin{split}
\essosc_{\tilde Q_0} u\leq (1-4^{-1}\xi)\omega,
 \end{split}\end{equation*}
 since $\tilde Q_0\subseteq B_{\frac{1}{2}R}\times\left(-\frac{1}{4}\Theta_A,0\right]=Q_A^{(2)}$.
 On the other hand, if either \eqref{2ndomega assumption1}, \eqref{omega1+} or \eqref{2st omega assumption} holds, then we conclude that the inequality
\begin{equation*}\begin{split}
\essosc_{\tilde Q_0} u\leq\omega\leq B_2\xi^{-1}\left(
F_p(R_0)+G_p(R_0)+R_0\right)
 \end{split}\end{equation*}
 holds for a constant $B_2=100\max\left\{\hat B,2^{\frac{2s_1}{p-1}},2^{s_1+q_*}\right\}$.
 Here, $\hat B$ is the constant defined in \eqref{definehatB}.
We have thus proved the proposition.
\end{proof}
\section{Proof of the main result}
In this section, we give the proof of Theorem \ref{main1}.
Our argument follows the idea from
\cite{BDG1, KMS, Qifanli}. We just sketch the proof.
First, we set up an iteration scheme.
Let $\{\xi_1,B_1\}$ and $\{\xi_2,B_2\}$ be the parameters stated by Propositions \ref{1st proposition} and \ref{2nd proposition},
respectively.
Let $\omega_0=\omega$, $Q_0=Q_{R_0,R_0^2}$ and $\Theta_A^{(0)}=\Theta_A$. Moreover, we set
\begin{equation}\begin{split}\label{define omega1}
\omega_1= \eta\omega_0+\gamma\left(
F_p(R_0)+G_p(R_0)+R_0\right),
 \end{split}\end{equation}
 where $\eta=\max\left\{1-4^{-1}\xi_1,\ 1-4^{-1}\xi_2\right\}$ and $\gamma=A^{\frac{q-2}{p-2}}\max\left\{B_1\xi_1^{-1},\ B_2\xi_2^{-1}\right\}$.
 Then, we have $\frac{1}{2}<\eta<1$ and $\omega_1>A^{\frac{q-2}{p-2}} R_0$.
 Let
$Q_1=Q_{\frac{1}{16}R_0}^-(0)=B_{\frac{1}{16}R_0}\times\left(-\Theta_{\omega_0}\left(\frac{1}{16}R_0\right),0\right].$
  According to Propositions \ref{1st proposition} and \ref{2nd proposition}, we conclude that
   \begin{equation}\begin{split}\label{osc1}
   \essosc_{Q_1}u=\esssup_{Q_1}u-\essinf_{Q_1}u\leq\omega_1.
    \end{split}\end{equation}
Furthermore, we define
$\delta=\frac{1}{16}2^{-\frac{3}{p}}\eta^{\frac{q-2}{p}}A^{\frac{2-q}{p}}$, $R_1=\delta R_0$, $\tilde R_1=\frac{9}{10}R_1$, \begin{equation*}\begin{split}\Theta_A^{(1)}=\left(\frac{\omega_1}{A}\right)^2\left[\left(\frac{\omega_1}{A\tilde R_1}\right)^p
    +a_0\left(\frac{\omega_1}{A\tilde R_1}\right)^q\right]^{-1},\qquad
     Q_2=B_{\frac{1}{16}R_1}\times\left(-\Theta_{\omega_1}\left(\frac{1}{16}R_1\right),0\right]
    \end{split}\end{equation*}
    and $\widehat Q_2=B_{R_1}\times\left(-\Theta_A^{(1)},0\right]$.
At this stage, we claim that $\Theta_A^{(1)}\leq \tilde R_1^2$ and $\widehat Q_2\subseteq Q_1$. In view of $\omega_1>A^{\frac{q-2}{p-2}} R_0>
A^{\frac{q-2}{p-2}} \tilde R_1$, we deduce
\begin{equation*}\begin{split}\Theta_A^{(1)}=\left(\frac{\omega_1}{A}\right)^2\left[\left(\frac{\omega_1}{A\tilde R_1}\right)^p
    +a_0\left(\frac{\omega_1}{A\tilde R_1}\right)^q\right]^{-1}<A^{-(q-p)}\tilde R_1^2<\tilde R_1^2.
    \end{split}\end{equation*}
Since $\omega_1>\eta \omega_0$ and $R_1=\delta R_0$, we have
    \begin{equation*}\begin{split}\Theta_A^{(1)}&<\left(\frac{\omega_1}{A}\right)^2\left[\left(\frac{\omega_1}{A R_1}\right)^p
    +a_0\left(\frac{\omega_1}{AR_1}\right)^q\right]^{-1}
    <\eta^{2-q}\left(\frac{\omega_0}{A}\right)^2\left[\left(\frac{\omega_0}{A \delta R_0}\right)^p
    +a_0\left(\frac{\omega_0}{A\delta R_0}\right)^q\right]^{-1}
    \\&\leq \eta^{2-q}A^{q-2}(16\delta)^p\Theta_{\omega_0}\left(\frac{1}{16}R_0\right)<\Theta_{\omega_0}\left(\frac{1}{16}R_0\right),
    \end{split}\end{equation*}
    which proves the claim $\widehat Q_2\subseteq Q_1$. In the case $\essosc_{Q_1}u\geq \frac{1}{2}\omega_1$.
    We can now proceed analogously to the proofs of Propositions \ref{1st proposition} and \ref{2nd proposition}.
This leads us to
       \begin{equation}\begin{split}\label{oscillation}
   \essosc_{Q_2}u\leq\omega_2,\qquad\text{where}\qquad
 \omega_2= \eta\omega_1+\gamma\left(
F_p(R_1)+G_p(R_1)+R_1\right).
    \end{split}\end{equation}
    In the case $\essosc_{Q_1}u< \frac{1}{2}\omega_1$. It follows immediately that \eqref{oscillation} holds, since $\eta>\frac{1}{2}$.
    At this point, we set $R_k=\delta R_{k-1}=\delta^kR_0$, $\omega_k= \eta\omega_{k-1}+\gamma\left(
F_p(R_{k-1})+G_p(R_{k-1})+R_{k-1}\right)$ and
\begin{equation*}\begin{split}
     Q_k=B_{\frac{1}{16}R_{k-1}}\times\left(-\Theta_{\omega_{k-1}}\left(\frac{1}{16}R_{k-1}\right),0\right].
    \end{split}\end{equation*}
Repeating the previous argument leads to
    \begin{equation*}\begin{split}
    \essosc_{Q_n}u\leq \omega_n
    \end{split}\end{equation*}
    for all $n=0,1,\cdots$.
    Our task now is to establish a continuity estimate for the weak solutions.
Fix $\bar\rho\in(0,R_0)$, $\bar r\in (0,\bar\rho)$,
    and let $0<b<\min\{\delta,\eta\}$. Then, there exist two integers $k,l\in\mathbb{N}$ such that
    \begin{equation}\label{kl}k-1<\frac{1}{\ln\delta}\ln\frac{\bar \rho}{R_0}\leq k\qquad\text{and}\qquad
      l-1<\frac{1}{\ln b}\ln\frac{\bar r}{\bar \rho}\leq l.
    \end{equation}
    At this point, we set $n=k+l$ and note that
    \begin{equation}\begin{split}\label{etank}
    \eta^{n-k}=\eta^l=b^{\alpha l}\leq \left(\frac{\bar r}{\bar \rho}\right)^\alpha,\qquad\text{where}\quad
    \alpha=\frac{\ln \eta}{\ln b}.
     \end{split}\end{equation}
On the other hand, we infer from $\omega_k= \eta\omega_{k-1}+\gamma\left(
F_p(R_{k-1})+G_p(R_{k-1})+R_{k-1}\right)$ and \eqref{etank} that $\omega_k\leq \hat\omega_0$ and
     \begin{equation}\begin{split}\label{omegan}
\omega_n&\leq \eta^{n-k}\omega_k+\gamma\frac{1}{1-\eta}\left(
F_p(R_{k})+G_p(R_{k})+R_{k}\right)
\\&\leq \left(\frac{\bar r}{\bar \rho}\right)^\alpha\hat\omega_0+\gamma\frac{1}{1-\eta}
\left(F_p(\bar\rho)+G_p(\bar\rho)+\bar\rho\right)
\end{split}\end{equation}
holds for any $0<\bar r<\bar \rho<R_0$ and $n=k+l$, where $k$ and $l$ satisfy \eqref{kl}.
Here, we set
\begin{equation*}\hat\omega_0=2\|u\|_\infty+A^{\frac{q-p}{p-2}}R_0+\gamma\frac{1}{1-\eta}\left(
F_p(R_0)+G_p(R_0)+R_0\right).\end{equation*}
Furthermore, we take $0<\bar r<R_0$ and
choose $\bar \rho=R_0^{1-\beta}\bar r^\beta$.
Here $\beta\in(0,1)$ is a fixed constant.
In view of \eqref{kl}, we define $k_*$ and $l_*$ by
 \begin{equation*}k_*-1<\frac{\beta}{\ln\delta}\ln\frac{\bar r}{R_0}\leq k_*\qquad\text{and}\qquad
      l_*-1<\frac{(1-\beta)}{\ln b}\ln\frac{\bar r}{R_0}\leq l_*.
    \end{equation*}
    For $n_*=k_*+l_*$, we infer from \eqref{omegan} that
    \begin{equation*}\begin{split}
\omega_{n_*}\leq \left(\frac{\bar r}{R_0}\right)^{\alpha(1-\beta)}\hat\omega_0+\gamma
\left(F_p(R_0^{1-\beta}\bar r^\beta)+G_p(R_0^{1-\beta}\bar r^\beta)+R_0^{1-\beta}\bar r^\beta\right)
\end{split}\end{equation*}
and the integer $n_*$ satisfies
\begin{equation}\label{n*}
n_*-2<\left(\frac{\beta}{\ln\delta}+\frac{1-\beta}{\ln b}\right)\ln\frac{\bar r}{R_0}\leq n_*.
\end{equation}
At this stage, we define $b_0=\hat\omega_0^{2-p}(1+\|a\|_\infty)^{-1}$ and $\bar Q_n=B_{R_n}\times(-b_0R_n^q,0]$.
It is easy to check that $\bar Q_n\subset Q_n$. Moreover, we set
 \begin{equation*}\begin{split}\sigma=\beta+(1-\beta)\frac{\ln\delta}{\ln b}\qquad\text{and}\qquad
 r=\delta^2R_0\left(\frac{\bar r}{R_0}\right)^\sigma.\end{split}\end{equation*}
It follows easily that
 \begin{equation*}\begin{split}
 \bar r=R_0\left(\delta^{-2}R_0^{-1}r\right)^{\frac{1}{\sigma}}\qquad\text{and}\qquad
  \bar \rho=R_0^{1-\beta}\bar r^\beta=\delta^{-2\frac{\beta}{\beta+\frac{1-\beta}{\ln b}\ln\delta}}r^{\frac{\beta}{\sigma}}
  R_0^{\frac{(1-\beta)\ln \delta}{\sigma\ln b}}\leq c(\delta)r^{\frac{\beta}{\sigma}}.
 \end{split}\end{equation*}
According to \eqref{n*}, we find that
 \begin{equation*}\begin{split}
 r=R_0\exp\left(\sigma\ln\frac{\bar r}{R_0}+2\ln\delta\right)<R_0e^{n_*\ln\delta} =R_{n_*}
 \end{split}\end{equation*}
 and hence, $Q_{r,b_0r^q}\subset \bar Q_{n_*}\subset Q_{n_*}$.
 Consequently, we deduce the continuity estimate as follows:
 \begin{equation*}\begin{split}
 &\essosc_{Q_{r,b_0r^q}}u\leq \essosc_{\bar Q_{n_*}}u\leq\essosc_{Q_{n_*}}u\leq \omega_{n_*}
 \\&\leq \gamma\left(\frac{r}{R_0}\right)^{\frac{\alpha(1-\beta)}{\sigma}}\hat\omega_0+\gamma
\left(F_p(cr^{\frac{\beta}{\sigma}}
)+G_p(cr^{\frac{\beta}{\sigma}})+r^{\frac{\beta}{\sigma}}\right).
 \end{split}\end{equation*}
 Thus, we conclude that the weak solution $u$ is locally continuous in $\Omega_T$. This completes the proof of Theorem \ref{main1}.

 Finally, we consider the case when $a(x_0,t_0)>0$, $f\in K_q$ and $g\in \tilde K_q$. As already mentioned in Remark \ref{remarka>0}, the weak solution $u$
 is not uniformly continuous in the set $\{a(x,t)>0\}$. Repeating the previous argument, we are thus led to the following version of Theorem \ref{main1}:
\begin{theorem}\label{main2}
 Let $u$ be a locally bounded weak solution to the parabolic double-phase equation \eqref{parabolic}
 in the sense of Definition \ref{weak solution}, where the vector field $A$ fulfills the structure conditions \eqref{A}.
 Let $z_0=(x_0,t_0)\in\Omega_T$, $a(x_0,t_0)>0$ and $R_0>0$ be such that
 \begin{equation*}Q_{R_0,R_0^2}(x_0,t_0)\Subset\Omega_T\qquad\text{and}\qquad R_0\leq \left(\frac{a(x_0,t_0)}{10[a]_\alpha}\right)^\frac{1}{\alpha}.\end{equation*}
 Assume that $2<p<q\leq p+\alpha$, $q<n$, $f\in K_q$ and $g\in \tilde K_q$.
Then,
there exists a constant $\delta_2=\delta_2(\text{data},R_0)<1$, such that the oscillation
estimate
  \begin{equation}\label{theorem2oscillation1}
 \essosc_{Q_{r,c_0r^q}(z_0)}u\leq \gamma\left(\frac{r}{R_0}\right)^{\beta_1}+\gamma
\left(a_0^{-\frac{1}{q-1}}F_q(cr^{\beta_2}
)+a_0^{-\frac{1}{q-1}}G_q(cr^{\beta_2})+r^{\beta_2}\right),
 \end{equation}
 holds for any $r<\delta_2R_0$. Here, $a_0=a(x_0,t_0)$ and the constants $c_0$, $\beta_1$, $\beta_2$, $\gamma$ and $c$
 depend only upon the data.
 \end{theorem}
 Here, the inequality \eqref{theorem2oscillation1} provides an explicit dependence of the essential oscillation
 of $u$ on the parameter $a_0$. The proof of Theorem \ref{main2} can be handled in much the same way,
the only difference being in the analysis of the terms involving $f$ and $g$. In fact, we are concerned with the $(p,q)$-phase
and we can apply the previous argument with $f$ and $g$ replaced by $f/a_0$ and $g/a_0$, respectively. The proof is quite similar to that of
 Theorem \ref{main1} and so is omitted.

\bibliographystyle{abbrv}

\end{document}